\documentclass[12pt]{amsart}
\usepackage{amsmath,amsthm,amscd,amssymb,latexsym,amsfonts,wasysym,xypic}
\usepackage{fancyhdr}                     
\usepackage{mathrsfs}                     
\usepackage{geometry,tabularx,shapepar}
\usepackage[all,2cell]{xy}
\usepackage{amsbsy}   
\usepackage{mathbbol}     

\usepackage{extarrows}

\usepackage{graphicx}     
\usepackage{graphpap}
\usepackage{epsfig}
\usepackage{rotating}
\usepackage{pst-all}
\usepackage{pstricks-add}
\usepackage{stmaryrd}
\DeclareMathAlphabet{\mathpzc}{OT1}{pzc}{m}{it}

\usepackage[T1]{fontenc}
\usepackage{frcursive}

\usepackage{hyperref}

\def\a{\alpha}
\def\b{\beta}
\def\cd{\cdot}
\def\d{\delta}

\def\e{\epsilon}
\def\f{\frac}                          
\def\g{\gamma}
\def\G{\Gamma}
\def\k{\kappa}

\def\l|{\left|}
\def\la{\lambda}
\def\La{\Lambda}
\def\lra{\longrightarrow}              
\def\dllra{\Longleftrightarrow}        

\def\Om{\Omega}
\def\om{\omega}

\def\ov{\overline}                   
\def\ox{\otimes}                     
\def\p{\partial}                     

\def\r|{\right|}
\def\s{\sigma}
\def\sbq{\subseteq}                  
\def\supth{{\text{th}}}              

\def\tl{\tilde}
\def\un{\underline}                  

\def\vp{\varphi}

\def\wh{\widehat}                    

\def\wt{\widetilde}                  
\def\x{\times}                       
\def\z{\zeta}
\def\({\left(}
\def\){\right)}
\def\[{\left[}
\def\]{\right]}
\def\<{\left<}
\def\>{\right>}

\newcommand{\bk}[1]{\langle #1\rangle}

\def\ra{\rightarrow}
\newcommand{\rz}{\raisebox{.2ex}{*}}

\def\SA{\mathcal A}
\def\SB{\mathcal B}
\def\SC{\mathcal C}
\def\SD{\mathcal D}
\def\SE{\mathcal E}
\def\SF{\mathcal F}
\def\SG{\mathcal G}
\def\SH{\mathcal H}
\def\SI{\mathcal I}
\def\SJ{\mathcal J}
\def\SK{\mathcal K}
\def\SL{\mathcal L}
\def\SM{\mathcal M}
\def\SN{\mathcal N}
\def\SO{\mathcal O}
\def\SP{\mathcal P}

\def\SR{\mathcal R}
\def\SS{\mathcal S}
\def\ST{\mathcal T}
\def\SU{\mathcal U}
\def\SV{\mathcal V}

\def\SX{\mathcal X}
\def\SY{\mathcal Y}

\def\Fa{\mathfrak a}
\def\Fb{\mathfrak b}

\def\Ff{\mathfrak f}
\def\Fg{\mathfrak g}
\def\Fh{\mathfrak h}
\def\Fi{\mathfrak i}
\def\Fj{\mathfrak j}
\def\Fk{\mathfrak k}
\def\Fl{\mathfrak l}

\def\Fq{\mathfrak q}
\def\Fr{\mathfrak r}
\def\Fs{\mathfrak s}
\def\Ft{\mathfrak t}
\def\Fu{\mathfrak u}
\def\Fv{\mathfrak v}
\def\Fw{\mathfrak w}

\def\Fy{\mathfrak y}
\def\Fz{\mathfrak z}

\def\FA{\mathfrak A}
\def\FB{\mathfrak B}
\def\FC{\mathfrak C}
\def\FD{\mathfrak D}
\def\FE{\mathfrak E}
\def\FF{\mathfrak F}
\def\FG{\mathfrak G}
\def\FH{\mathfrak H}
\def\FI{\mathfrak I}
\def\FJ{\mathfrak J}
\def\FK{\mathfrak K}
\def\FL{\mathfrak L}
\def\FM{\mathfrak M}

\def\FR{\mathfrak R}
\def\FS{\mathfrak S}

\def\FZ{\mathfrak Z}

\newcommand{\BA}{\ensuremath{\mathbf A}}

\newcommand{\BS}{\ensuremath{\mathbf S}}
\newcommand{\BT}{\ensuremath{\mathbf T}}
\newcommand{\BU}{\ensuremath{\mathbf U}}

\def\sbbn{{\mathbb n}}

\def\bbc{{\mathbb C}}

\def\bbf{{\mathbb F}}

\def\bbm{{\mathbb M}}
\def\bbn{{\mathbb N}}

\def\bbq{{\mathbb Q}}
\def\bbr{{\mathbb R}}
\def\bbs{{\mathbb S}}

\def\bbz{{\mathbb Z}}

\def\ssm{\smallsetminus}

\newcommand{\bsm}[1]{\boldsymbol{#1}}    

\DeclareMathOperator{\img}{Im}

\newcommand{\cmn}{$C^\infty(\bbr^m,\bbr^n)$}

\newcommand{\dcmn}{C^\infty(\bbr^m,\bbr^n)}

\newcommand{\cm}{$C^\infty(\bbr^m,\bbr)$}

\newcommand{\dscmn}{SC^\infty(\bbr^m,\bbr^n)}

\newcommand{\norm}[1]{\left\|#1\right\|}

\def\k{\kappa}

\newcommand{\ad}{\operatorname{ad}}    
\def\x{\times}                         
\newcommand{\expp}{\operatorname{EXP}} 
\newcommand{\card}{\operatorname{card}} 
\newcommand{\nes}{\operatorname{nes}}    
\newcommand{\st}{\operatorname{st}}       
\newcommand{\open}{\operatorname{open}}    
\newcommand{\Hom}{\operatorname{Hom}}       
\newcommand{\bs}{\ensuremath{\boldsymbol}}   
\newcommand{\en}{\operatorname{End}}        
\newcommand{\loc}{\operatorname{loc}}      
\newcommand{\stan}{\operatorname{stan}}     
\newcommand{\dist}{\operatorname{dist}}      
\newcommand{\intt}{\operatorname{int}}      
\newcommand{\gl}{\operatorname{gl}}        

\newtheorem{thm}{Theorem}[section]
\newtheorem{lem}{Lemma}[section]
\newtheorem{proposition}{Proposition}[section]
\newtheorem{definition}{Definition}[section]
\newtheorem{cor}{Corollary}[section]

\theoremstyle{definition}
\newtheorem{remark}{Remark}[section]



\title{Regularity and Nearness Theorems for Families of Local Lie Groups
  }
\author{Tom McGaffey}
\date{}

\begin{document}

\maketitle

\begin{abstract}
     In this work, we prove three  types of results with the strategy that, together, the author believes  these should imply the local version of Hilbert's Fifth problem. In a separate development, we  construct a nontrivial topology for rings of map germs on Euclidean spaces. First, we develop a framework for the  theory of (local) nonstandard Lie groups and within that framework prove a nonstandard result that implies that a family of local Lie groups that converge in a pointwise sense must then differentiability converge, up to coordinate change, to an analytic local Lie group, see corollary \ref{cor: C^0 precpt fam in Gp is C^k precpt}.
 The second result  essentially says that a pair of mappings that almost satisfy the properties defining a local Lie group must have a local Lie group nearby, see proposition \ref{prop: standard almost->near}. Pairing the above two results, we get the principal standard consequence of the above work, corollary \ref{cor: standard reg + almost->near}, which can be roughly described as follows. If we have pointwise equicontinuous family of mapping pairs (potential local Euclidean topological group structures), pointwise approximating a (possibly differentiably unbounded) family of differentiable (sufficiently approximate) almost  groups, then the original family has, after appropriate coordinate change, a local Lie group as a limit point. The third set of results give nonstandard renditions of equicontinuity criteria for families of differentiable functions, see theorem \ref{thmbasreg}. These results are critical in the proofs of the principal results of this thesis as well as the standard interpretations of the main results here. Following this material, we have a long chapter constructing a Hausdorff topology on the ring of real valued  map germs on Euclidean space. This topology has good properties with respect to convergence and composition.
See the detailed introduction to this chapter for the motivation and description of this topology.

\end{abstract}

\tableofcontents

\section{Introduction: History, Summary, Context}

\pagenumbering{arabic} \setcounter{page}{1}   

  We begin with a summary of contents as well as a perspective, historical and motivational.

\subsection{Content and objectives}In the first part of this paper we give  proofs of the following results. The first result is a regularity result.  Suppose that $(\SG,\psi,\nu)$
 is an SC$^\circ$ ${}^\s$local *Lie group (see definition \ref{def: sigma local *Lie group}) defined on some standard neighborhood of
 $0$ in $\bbr^n$ for $n\in\bbn$. (See preliminaries for definitions.) Then  there is a
 homeomorphic change of coordinates on some standard neighborhood of  $0$ in $\bbr^n$
 such that the standard part of $(\SG,\psi,\nu)$ is an analytic local Lie group in
 the new coordinates. (This is theorem \ref{thm: Main nonst thm}.) Note that the SC$^\circ$ condition only guarantees that the
 standard part will be a continuous local topological group with respect to the
 Euclidean topology; i.e., a locally Euclidean topological group. (Without this
 condition, the standard part of a ${}^\s$local *Lie group can be quite pathological, if it exists at all.) Some dramatic standard consequences follow immediately: eg., and crudely, $C^0$ precompact subsets of local Lie groups are in fact $C^k$ precompact for any integer $k$ (with respect to special coordinates); see the results in section \ref{sec: stan conseq of main reg thm}, especially corollary \ref{cor: C^0 precpt fam in Gp is C^k precpt}.

The (generalized) local Fifth problem of Hilbert asks if a general locally Euclidean local group has a homeomorphic change of coordinates making the local group into a local Lie group in the new coordinates. (We say generalized for, as defined, local Lie groups generally are not neighborhoods of the identity in (global) Lie groups.)
Given this statement, we find that a corollary of the above nonstandard result is a statement
asserting that the (generalized) local Fifth of Hilbert is implied
by a density result: ie., of local Lie groups in local Euclidean local topological groups. In attempting to prove the density result, we have proved the following almost implies near result, see chapter \ref{chap: partial solution to approx prob}. We define the notion of an almost local Euclidean $C^k$ group: an appropriate  pair of differentiable maps $(\psi,\nu)$  that are $s$-almost groups for some $s>0$ (roughly: instead of satisfying the equations defining a group structure, they satisfy inequalities that are pointwise $s$-close to these equalities)  , see definition \ref{def: s-almost gps}. Given this notion and given that we have some bound on derivatives of putative grouplike objects $(\psi,\nu)$, we prove that for every $r>0$, there is $s>0$ such that if $(\psi,\nu)$ is an $s$-almost group, then there is a $C^k$ (local) Lie group in an $r$ neighborhood of $(\psi,\nu)$ (in the $C^k$ topology.) This is a standard statement but the proof is nonstandard; see the specific statement in proposition \ref{prop: standard almost->near}, chapter \ref{chap: partial solution to approx prob}. Note that this result is new (much stronger than previous results along this line) and properly construed was not considered possible in some circles, see Ruh, \cite{Ruh1987} p. 563.
Now given (1): a nonstandard rendition of this almost implies near result (ie., corollary \ref{cor: NSA reg+almost->near}), (2): the nonstandard version of the main regularity result along with (3): a curious nonstandard smoothness result in the appendix (corollary \ref{cor: SC^j sim SC^k}); we can prove the surprising standard result given in corollary \ref{cor: standard reg + almost->near}.  {\it Roughly this says the following. Suppose that $\FC$ is a family of potential (but not!)  Euclidean local topological groups (ie., continuous pairs $(\psi,\nu)$ as described above) having $C^0$ precompactness properties (see definition \ref{def: equicont fam}). Suppose further that this family is pointwise  approximated by a good family of $C^k$ almost groups (whose derivatives are not necessarily bounded) that contains $s$-almost groups for arbitrarily small $s>0$.  Then, in fact, in appropriate coordinates, the family $\FC$ contains a sequence whose pointwise limit  is a local Lie group.}

One can also prove  an elementary approximation result getting that $C^k$ approximations of our local Euclidean group are approximate groups (this result has been omitted as it is not informative and awaits other results before it can be useful). The author believes that this simple approximation result along with the regularity-near results just mentioned should  enable us to prove that one can approximate locally Euclidean local topological groups by local Lie groups, ie., get the density result alluded to earlier. The argument is simple: approximate the local group by almost $C^k$ groups and then find a local Lie group close to this almost group (using some version of the almost implies near), then use some version of the regularity material to show that these approximations have bounded derivatives.  But the author is having problems with technical issues and to date has not finished this part. If the proof can be finished, then the density result along with the regularity result will have the local Fifth problem as a corollary.

The majority of the work in this paper is the proof of the main regularity theorem which  consists of two parts, both fairly
elementary with respect to complexity and depth of background
knowledge. If $\SG$ is our nonstandard local Lie group, the first part gives a proof that for $x\in\Fg$ (the
*Lie algebra of $\SG$) $ad_x$ is a nearstandard linear map. This is proved in chapter \ref{chap: pf that ad is sc0}. The
second part, see chapter \ref{chap: st(exp) is loc homeo}, uses this fact to prove that the standard part of the
exponential map for $(\Fg,\SG)$ is a local homeomorphism. In the
last brief part, chapter \ref{chap: Main NS reg thm and stan version}, we use again that $ad_x$ is nearstandard to show
that the Hausdorff series (ie., the *Lie group product in the
new coordinates) has analytic standard part.
Tieing together the regularity results on the $\ad$ map, the exponential map and the Hausdorff series, we easily get the main result, \ref{thm: Main nonst thm}.

It should be mentioned that the appendix of this paper, chapter \ref{chap: appendix: S-smoothness}, contains  technical  results that are new and  eg.,  critical to the proof of the almost implies near result noted above. These are results of nonstandard analysis. One part of the statement of the primary result, theorem \ref{thmbasreg}, can be be described as saying that  any internal function that is pointwise infinitesimally close to an internal S-smooth function (its internal derivatives are finite) is itself S-smooth; crudely: a super pointwise restriction implies differentiability restrictions. Note the corollaries of this result, especially corollary \ref{cor: 2nd stan cor of appendix thm}.

\subsection{History of the Fifth problem and NSA}
Hilbert stated his Fifth problem within the context of Sophus Lie's work on local transformation groups in the late nineteenth century. Roughly speaking, he asked if a local group acting continuously could be given coordinates for which the action would be analytic. See Richard Palais' contribution to the article, \cite{PalaisGleasonBio5th}, for this and the following historical remarks. With the advent of the modern formulation, ie., in terms of actions by (global) groups, the general group action problem was found to be false. A restricted question in terms of the group acting on itself, ie., the group's own product structure, had some possibility of holding if the group had a reasonably nice topology. That is, if  the group was assumed to be locally Euclidean, then the problem was solved in the affirmative.
More specifically, first  von Neumann  solved  the compact case in 1933, then the abelian case was proved  by Pontryagin in 1939, followed by the solvable case (Chevally in 1941). Finally, after more than 10 years, a proof of the general case followed from the work of a pair of papers (Montgomery and Zippin \cite{MontZip1952} and  Gleason \cite{Gleason1952}) that appeared in a 1952 issue of the Annals of Mathematics . Very briefly, Montgomery and Zippin showed that a locally Euclidean group could not have ``small subgroups'', while  Gleason proved that such groups have continuous, injective homomorphisms into Lie groups and then invoked a result of E. Cartan which implied that such groups have smooth structures.


A proof by Jacoby of the result for locally Euclidean local groups appeared in the Annals of Mathematics in 1957, \cite{Jacoby1957}. But as Peter Olver, \cite{Olver1996}, clearly
demonstrates, this proof is critically flawed, as will be discussed below and in in detail in chapter \ref{chap: error in Jacoby}.
As Olver also notes in this paper, in the intervening years substantial theory has come to
rely on this local version of Hilbert's Fifth.
Meanwhile, many years later, a much shortened nonstandard proof of the Fifth problem by Hirschfeld appeared in the Transactions of the AMS, \cite{Hirschfeld1990}. Hirschfeld followed the approach of the (standard) proof by Montgomery-Zippin-Gleason; but the nonstandard tools allowed great simplification of the original proof. Hirschfeld was able to show that the set of one parameter subgroups has a vector space structure by a straightforward  identification of these with a quotient of the set of infinitesimal elements of a given magnitude scale by those with infinitely smaller scale. Using this, he was able to then define a homomorphism from the group to the group of automorphisms of this vector space, essentially by (nonstandard)  hand. Generally, his arguments followed Gleason but substituted a careful analysis with infinitesimals in the place of Gleason's functional analysis.

The author of the present paper discovered  Olver's analysis of local topological groups (in particular explaining the failure of Jacoby's proof), \cite{Olver1996},  and began to think about a totally different (nonstandard) approach to the (until recently) open local version of the Fifth problem. About the same time we learned that Dr van den Dries had written up notes on a nonstandard proof of the Fifth problem (of which Hirschfeld mentions in his paper) and enquired about the possibility of obtaining a copy of such. During this correspondence, the author informed Dr van den Dries on the open local Fifth problem.
During this time, we began a correspondence with I. Goldbring, a student of Dr van den Dries, with respect to  the author's work on the main regularity theorem in this paper.
In the intervening years, Goldbring has produced a proof of the local Fifth problem, \cite{GoldbringFifthProblem2010} apparently following the approach of Hirschfeld, avoiding the problem of nonglobalizability of local Euclidean local groups (see below and chapter \ref{chap: error in Jacoby}) that doomed Jacoby's approach. Hence, in a strong sense, the proof, in spite of its nonstandard detour,  is modelled on the original proof of 1952.
As summarized above, the approach in the present paper could not be more different from this.




We need some remarks on why the  local Fifth problem is fundamentally different than the (global) Fifth problem.
Crudely, some local topological groups can
not be neighborhoods of the identity in  topological groups:  many-fold associativity follows from associativity for topological groups, but for local groups it does not, as it involves global topological considerations. For details see Chapter \ref{chap: error in Jacoby}
of this paper. The upshot of this is that, apparently many years
after Jacoby's paper appeared, its argument was found to
depend on this flawed assumption that the local group embedded in a (global) group. For a careful exposition, see Olver's paper
\cite{Olver1993}. Malcev had published a paper in 1941 specifying nontrivial
conditions necessary for `globalizing' a local group. See \cite{Malcev1941}, but
especially see Olver's lucid extension of Malcev's result. Chapter \ref{chap: error in Jacoby}
 of the present paper gives some relevant details. Curiously, Pontryagin
refers in his book to the paper of Malcev \cite{Pontryagin1986} p138--39 stating
that local groups are not always locally isomorphic to a
neighborhood of the identity in a global topological group, and
nobody seemed to have noticed this at the time.
Furthermore,  nobody seems to have investigated the relation of Pontryagin's embedding of local Lie groups with trivial center into global Lie groups and Olver's critical counterexamples.

For an overall understanding of local Lie groups and
the position of Jacoby's flawed result, see Olver's excellent paper,
\cite{Olver1993}. Good exposition on Hilbert's Fifth Problem are given by
Kaplansky's texts \cite{Kaplansky1971} and \cite{Kaplansky1977}, and the book of Montgomery and Zippin \cite{MontZip1955}. A nonstandard rendition of the problem was
written by van den Dries, but the author has not seen it.

\subsection{Strategy}
  Here we summarize the major strategies in this paper. Overall, and from a standard point of view, the strategic approach here is to show that locally Euclidean local groups are limits of local Lie groups and that the objects in these limiting processes can be regularized by appropriate coordinate changes to preserve their smoothness. Finally, strategic use of nonstandard mathematics allows us to avoid sequences and limiting arguments. Below, we summarize from the nonstandard perspective the major strategies involved.
\subsubsection{Main nonstandard theorem}
The strategy of the proof follows from the
insight that if we can prove that the $*$Lie algebra of the
${}^\s$local *LG (see \ref{subsec: intro: *Lie alg structure}) is nearstandard  i.e., that the
*Lie bracket $[\ ,\ ]:\rz\bbr^n_{\nes}\x\rz\bbr^n_{\nes}\mapsto\rz\bbr^n_{\nes}$,
then we can prove our change of coordinates $\rz\exp^{-1}$ is an
$S$-homeomorphism and that the group law in the new coordinates, the
*Hausdorff series, $\rz H$ series, is $S$-analytic. For the
definition of an {\bf $S$-homeomorphism} see sections \ref{subsec: notion of S-property} and \ref{sec: fact on S-homeos and pf finish}, definition \ref{def: S-homeo}. For
preliminaries on the dual use of the {\bf*Hausdorff} (aka
{\bf*CHD series}) see sections \ref{sec: H-series estimates, S-lemma} and \ref{sec: prdct is S-analy after coord change}. To prove that {\bf $\rz\exp$} is
an $S$-homeomorphism, we had to prove estimates on the $\rz H$ series
(see 6.2.2) and prove a subtle NSA fact (section \ref{sec: fact on S-homeos and pf finish}). The proof that the
$\rz H$ series is {\bf $S$-analytic} (see section \ref{subsec: S-analyticity}, and \ref{sec: prdct is S-analy after coord change}) is
straightforward.

\subsubsection{Lie bracket S-continuity} Yet we must still prove that
the *Lie bracket, two derivatives above a group operation that is
only assumed to have continuous standard part, is in fact continuous
(at the standard level). The proof depends only on
the intertwining formulas of the three canonical maps
$a_g: h\mapsto
ghg^{-1}, Ad_g(\nu) = d(a_g)\bigm|_{g=e}\quad\text{and}\quad
ad_\nu:w\to[v,w]$. Note that we are using these, and what follows,
on the internal level. See section \ref{sec: Ad is SC0}, expression \ref{diagram: 1st intertwining formula} for the first formula and section \ref{sec: transl pf of ad is SCo to Gin}, expression \ref{diagram: 2nd intertwining formula} for the second formula. Using these
formulas we reduce the problem to a question about the asymptotic
behavior of the internal Euclidean exponential map, $\expp$, and
from this to the elementary differential equation it satisfies,
written in terms of the differential of the product map on$\rz
Gl_n$, see section \ref{sec: diff geom of *Gl_n}.

\subsubsection{Almost implies near}
  The work in chapter \ref{chap: partial solution to approx prob} was initially motivated by the paper of Anderson, \cite{Anderson1986}. Nonetheless, understanding  the argument of Spakula and Zlatos in  \cite{SpakulaZlatos2004} was instrumental in our  construction of  the proper nonstandard equicontinuity argument necessary for a good ``almost implies near'' result for local topological groups. To complete the proof of this fact, we also needed the S-smoothness material in the appendix, chapter \ref{chap: appendix: S-smoothness} (see the material on S-smoothness below).
  Note that the almost implies near strategy is used also for almost associativity in chapter \ref{chap: error in Jacoby}. 
\
\subsubsection{Summary and prospects} This paper is structured as follows.
Chapter \ref{chap: intro:nonstan top and calc on E^n} gives basic preliminaries on  nonstandard analysis and eg., nonstandard calculus on Euclidean space. Chapter \ref{chap: intro: loc gp setup}
gives the basics on local groups, local Lie groups and their Lie algebras and some nonstandard renditions of these. In both chapters We have to prove basic groundwork
material as some does not exist or is not clear in the literature. Chapter \ref{chap: pf that ad is sc0} gives the
 preliminary one dimensional material, the Ad lemma, the preliminary intertwining
 formulas and finally a proof of our linchpin result: that $\ad$ is $SC^o$. Chapter \ref{chap: st(exp) is loc homeo} gives a proof that the exp map is
SC$^o$ and involves a fair amount of NSA work. Chapter \ref{chap: Main NS reg thm and stan version} starts with
a proof, using that $\ad$ is S-continuous, that the group product in the new coordinates, the *CHD
series, is $S$-analytic, uses this and the previous chapter to
give a short proof of the main nonstandard regularity theorem. Note that section \ref{sec: stan conseq of main reg thm} covers a fairly
compelling standard corollary of this main theorem.  Chapter \ref{chap: partial solution to approx prob}
considers the density question and chapter \ref{chap: error in Jacoby} gives an account on why local groups don't embed in (global) groups along with a nonstandard result on global associativity.
Chapter \ref{chap: appendix: S-smoothness} gives some generally useful technical results on S-smoothness that were helpful in chapter \ref{chap: partial solution to approx prob}.

 It seems possible to the
author that a more general result is possible, namely that the {\bf
condition on $S$-continuity of the internal product can be relaxed
to $S$-Borel measurability} of the product. Also, it seems to the
author that these results can be extended to Lie groups over
$p$-adic fields, which he will pursue as time allows. Finally, as a
method for showing that weak regularity implies strong regularity,
these tools appear to have much broader application than the usual
standard tools

\subsection{Topologizing map germs}
   The last extensive chapter has an, in house, summary of its contents, see section \ref{sec: intro: top on map germs}.  Here, we will merely note our motivation and then give a descriptive summary of results. This work was partially motivated by the author's belief that germs (of functions, for example) had not been properly covered by nonstandard analysts. Although Robinson, see his book \cite{Robinson1970} and his paper on germs, \cite{Robinson1969}, had given nonstandard presentations of germs, we believed that the capacity of nonstandard tools to analyze this area had not really been utilized. Furthermore, with respect to the present  work on families local Lie groups, we felt that a nonstandard study of families of germs of local Lie groups would help to understand their nature. We therefore began a study of families of germs of mappings from the point of view of nonstandard mathematics,  and immediately the desire for a good ambient topology for these seemed critical  before we could proceed to a nonstandard study of families of germs of topological groups.  The last chapter is the result of our analysis to this point.

   Briefly summarizing our results, we are able to construct a Hausdorff topology on the ring of germs at $0$ of real valued functions on $\bbr^n$ that has the following properties. A convergent net of germs of continuous functions has a limit point that is the germ of a continuous function. Furthermore, ring operations as well as left and right composition are continuous in this topology. The topology is defined in terms of a *supremum norm on a ball about $0$ of infinitesimal radius $\d$. Nonetheless, we prove that the topology is independent of the choice of $\d$ and also prove results that give close connections with standard convergence.
    For a much more detailed description of the contents of chapter \ref{chap: topologizing germs}, we refer the reader to the extensive introduction to this chapter beginning on page \pageref{sec: intro: top on map germs}.

   We need to note that a much updated and expanded version of this chapter now exists on the arXiv, \cite{McGafGerms2012arXiv1206.0473M}. Although of possibly independent interest, the Hardy field constructions and several other constructions in chapter \ref{chap: topologizing germs} have been left out of the updated paper. On the other hand, the updated paper has essentially completed the objectives of the original. There is interesting material left out of both that will appear at a later date.

\section{Preliminaries: Nonstandard topology and calculus}\label{chap: intro:nonstan top and calc on E^n}

\subsection{A brief description of NSA and local NS calculus}\label{sec: brief deecrip nsa nsa calc}

For those familiar with nonstandard analysis (NSA), this section can be
referred back to for some notation, definitions and a few basic
results in local differential calculus, appropriately transferred.
In introducing NSA, instead of giving a rigorous, and therefore obscure, approach to
to its foundations, we will instead begin with a crude and
descriptive introduction to the basic structures and tools of the
superstructure approach. Then we will include a basic index of
definitions, usually along with notation. We will then follow with
the basic nonstandard local calculus that is needed.

For NSA, my standard is the text of Stroyan and Luxemburg \cite{StrLux76}, but Henson in \cite{Henson1997} and  Lindstrom in \cite{Lindstrom1988} are  good user
friendly introduction, and the article of Farkas and Szabo \cite{FarkasSzabo} give a good picture on how nonstandard methods simplify proofs.
For a clear exposition on the relation between ultrafilters and the faithful transfer of
     theorems to ultrapower models, see Barnes and Mack, \cite{BarnMack1975} p.62-64.
 One way to motivate and explain nonstandard models of mathematical objects is to think of them as a
   "structure". See Ballard, \cite{Ballard1994}, for a short introduction to the model theory of structures
   along with a topological introduction to recent versions of
   nonstandard models and more generally  Di Nasso, \cite{DiNasso1999}, for an overview and analysis of the various approaches to a nonstandard mathematics.
\subsubsection{Brief overview of NSA}\label{subsec: nsa overview}

    We will begin with the notion of structure from model theory, the birthplace of nonstandard mathematics and briefly try to give a model theoretic view of nonstandard math. We will rapidly segue into the ultrapower idea and spend the lion's share of this introduction on developing a picture of the prototypical example: the nonstandard real numbers and its internal and external subsets.

    Loosely speaking, a structure $\{\SX,\SR\}$ consists of a pair: an ambient (variable) set $\SX$ along with
     a formal (fixed) collection, $\SR$ of relations (unary, binary,...) and
    operations on (unary, binary,...)
    on this set and possibly some canonical elements, $a_0, b_0,...$, of that
    set. Think of an ordered topological group $\{\SG,\SR\}$ as a set $\SG$ coming from a bag of many such, along
    with $\SR=\{e,\cdot,< \}$, where  $e$, denotes the identity, $\cdot$, the product binary
    operation, and the order,$<$ (binary relation); the formal set of symbols, $\SR$, structuring any and all such $\SG$ as ordered groups.  Or think of an ordered field $\{\bbf,\SS\}$ as a set $\bbf$ along with symbols for the two binary operations and the relation symbol and including two canonical symbols representing the additive and multiplicative identities.  These object are
    not an ordered group, respectively ordered field, unless all of compatibility axioms among
    the collection $\SR$, respectively $\SS$, are satisfied. These can be expressed
    \textbf{formally}; irrespective of the particular set $\SX$, respectively $\bbf$, using the syntax of (usually) first order logic.

    For a given set and structure on it, $\{\SX,\SR\}$,
    in our (standard) world, we can construct richer (nonstandard)
    models $\{\rz\SX,\SR\}$ that satisfy all of the theorems, suitably interpreted, that our
    standard model satisfies. In fact, we have a formal transfer correspondence between the two `models'. For example, if our ordered field above is the pair $\{\wt{\bbq},\SS\}$ where along with the compatibility axioms to get the full set of axioms getting a characteristic $0$ totally ordered field (hence containing $\bbq$), then $\{\rz\wt{\bbq},\SS\}$ would be a characteristic $0$ totally ordered field. But it would now have vastly more elements and
    with this larger (model theoretically isomorphic!) object, we find that those theorems which might be
    conceptually or proofwise difficult in the original structure are often much less so in
    the enriched model. Given that there is a formal correspondence (transfer) between the set of theorems of $\{\SX,\SR\}$ and those of $\{\rz\SX,\SR\}$, this allows the strategic possibly of proving theorems in the richer $\{\rz\SX,\SR\}$ and then transferring them back to $\{\SX,\SR\}$.  (See the example with respect to continuity below.)

     With respect to  the ultrapower method for acquiring these enriched models, we have the following outline followed by the less abstract constructions for $\rz\bbr$, the nonstandard real numbers.
     Let $\SF_{\SJ}$ is a suitable nonprincipal
    ultrafilter (see the discussion below) on a sufficiently large set $\SJ$, and if, for
    $f,g:\SJ\ra\SX$, we define $f\stackrel{\SF}{\sim}g$ if
     $\{j\in\SJ:f(j)=g(j)\}\in\SF$, then the four properties of  ultrafilters makes
     this a particularly nice equivalence relation on the set of
     maps from $\SJ$ to $\SX$ (below we will write these as sequences in $\SX$ indexed by $\SJ$). In fact,
     $\rz\SX\doteq\SX\!\diagup\;\lower4pt\hbox{$\stackrel{\SF}{\sim}$}$
     has precisely the same (first order) mathematical properties as
     $\SX$,  once they are suitably defined for $\SF_{\SJ}$ -equivalences classes of $\SJ$- sequences of
     elements of $\SX$. In particular, $\rz\SX$ satisfies all of the structural
     features of $\SR$,  along with all
     of the logical consequences, eg., theorems, when, loosely speaking, constants and sets being quantified over
     are replaced with the appropriate internal (see below) analog in the nonstandard universe.

     Let us give a simple but in some ways prototypical example of an ultrapower, a nonstandard model, $\rz\bbr$, for the real numbers, $\bbr$, and we will show how the properties of our ultrafilter make $\rz\bbr$ a totally ordered field properly containing the totally ordered field $\bbr$; eg., verify that $\rz\bbr$ thickens and extends $\bbr$. We begin with a definition of an ultrafilter on $\bbn$. This will be a collection of subsets, $\SF$, of $\bbn$ that satisfy the following four properties: (1) if $A,B\subset\bbn$ with $A\subset B$ and $A\in\SF$, then $B\in\SF$, (2) if $A,B\in\SF$, then $A\cap B\in\SF$, (3) the empty set is \textbf{not} an element of $\SF$, and finally the maximality property (4): if $A\subset\bbn$, then precisely one of $A$ or $\bbn\smallsetminus A$ is in $\SF$. To get a richer $\rz \SX$ from a given $\SX$, at least when $\SX$ is infinite, we need that $\SF$ be (5) nonprincipal; ie., satisfies the additional property that $\cap\{J:J\in\SF\}$ is empty. (Note that nonprincipal filters on eg., $\bbn$, ie., collections satisfying (1), (2), (3) and (5) are easily had: one example, the Frechet filter, is the collection of all $A\subset\bbn$ such that $\bbn\ssm A$ is finite. To get $\SF$  that satisfy property (4), in addition to the other four properties, we must invoke Zorn's lemma; and as such can't have constructive examples of such ultrafilters.)

     To get our ultrapower of $\bbr$ with respect to $\SF$, we introduce an equivalence on the set of sequences $\bbr^\bbn=\{(a_i):i\in\bbn\}$ via our ultrafilter (as noted in the previous paragraph); namely, if $(a_i),(b_i)\in\bbr^\bbn$, we declare that $(a_i)\stackrel{\SF}{\approx}(b_i)$, ie., are in the same $\SF$ equivalence class, if $\{i:a_i=b_i\}\in\SF$ and we let $\rz\bbr=\bbr^\bbn/\SF$ denote the set of $\SF$ equivalence classes. (The fact that $\stackrel{\SF}{\approx}$  gives an equivalence relation follows from the definition of ultrafilter.) Typical of the ultrapower construction, we lift all functions, relations, operations, etc., that are defined on $\bbr$, to $\bbr^\bbn$ componentwise and verify that that they push down to $\rz\bbr$ as functions, relations, operations with precisely the same (finitely stated) properties, ie., they transfer to $\rz\bbr$.

     For example, first lift the product on $\bbr$ to $\bbr^\bbn$ by defining $(a_i)\cdot(b_i)\dot=(a_i\cdot b_i)$. Then lift the ordering from $\bbr$ by $(a_i)<(b_i)$ if $a_i<b_i$ for all $i\in\bbn$. This will just give a partially ordered ring (with lots of zero divisors.)
     Next, we find that we can, in a well defined manner, push these  down to the set of equivalence classes, ie., $\rz\bbr$, and magically (via the properties of ultrafilters) get a totally ordered field as follows.
     First, if $\bk{a_i}\in\rz\bbr$ denote the equivalence class containing $(a_i)\in\bbr^\bbn$, let's show that $+$ and $<$ descend in a well defined way to $\rz\bbr$. Define $\bk{a_i}+\bk{b_i}=\bk{a_i+b_i}$ and $\bk{a_i}<\bk{b_i}$ if $\{i:a_i<b_i\}\in\SF$. Is this well defined, ie., do they respect equivalence classes? Suppose $(a_i)\stackrel{\SF}{\approx}(a'_i)$ and $(b_i)\stackrel{\SF}{\approx}(b'_i)$, we must verify that $(a_i)+(b_i)\stackrel{\SF}{\approx}(a_i')+(b'_i)$. But, by definition, if $I=\{i:a_i=a_i'\}$ and $J=\{i:b_i=b_i'\}$, then $I,J\in\SF$ and so property (2) above implies that $K\dot=I\cap J\in\SF$, that is, as $K\subset L\dot=\{i:a_i+b_i=a'_i+b'_i\}$, then $L$ is in $\SF$ by property (1) above, as we wanted. Similarly, (with the same hypotheses and notation) we verify that $\bk{a_i}<\bk{b_i}$ is well defined: if $A=\{i:a_i<b_i\}\in\SF$, we must have that $B=\{i:a_i'<b_i'\}\in\SF$. But clearly, if $i\in C\dot=I\cap J\cap A$, then $a_i'<b_i'$, ie., $C\subset B$ and as $C\in\SF$ by (repeated use of) property (2), then $B\in\SF$ by property (1). Continuing with these verifications, we get that $\rz\bbr$ is a partially ordered ring.

     Let's verify that $\rz\bbr$ in fact is totally ordered by $<$.
     At this point, we need a property of nonprincipal ultrafilters that follows from the above four properties. If $A_1,\ldots,A_k$ are pairwise disjoint subsets of $\bbn$  with $A_1\cup\cdots\cup A_k=\bbn$, then precisely one of the $A_j$'s is in $\SF$. From this we will get immediately that the partial order $<$ is in fact a total order on $\rz\bbr$. For given $\bk{a_i},\bk{b_i}\in\rz\bbr$ and $I=\{i:a_i<b_i\}$, $J=\{i:a_i=b_i\}$ and $K=\{i:a_i>b_i\}$. Then $I,J$ and $K$ are clearly disjoint with  $I\cup J\cup K=\bbn$ and so the previous statement says that precisely one of $I,J$ or $K$ is in $\SF$, ie., by definition precisely one of $\bk{a_i}<\bk{b_i}$, $\bk{a_i}=\bk{b_i}$ or $\bk{a_i}>\bk{b_i}$ holds, as we asserted. One can go on to verify that $\rz\bbr$ is a totally ordered field that contains an isomorphic copy of $\bbr$, ie., the set of equivalence classes of constant sequences, ie., those elements $\bk{a_i}$ satisfying $\{i:a_i=a\}\in\SF$. Of course, if we denote such a sequence by $\bk{a}$, we have a field injection $a\mapsto\bk{a}:\bbr\ra\rz\bbr$. We will denote this embedded field of standard real numbers by $^\s\bbr$.

     Let's demonstrate that $\rz\bbr$ is essentially `thicker' and `longer' than the embedded copy of $\bbr$. We will first show that it is thicker at $0$. Here is where the nonprincipal assumption plays a direct role because it implies that finite subsets of $\bbn$ cannot be elements of $\SF$, hence their complements the cofinite subsets must be elements of $\SF$. Given this, suppose that $a_1,a_2,\ldots\in\bbr$ are positive and $a_j\ra 0$ as $j\ra\infty$ and let $\a$ denote $\bk{a_i}\in\rz\bbr$. Then if $a\in\bbr$ is positive, $I\dot=\{i:0<a_i<a\}$ is clearly cofinite in $\bbn$ and so $I\in\SF$. But this says that for any positive real number $a$, we have that $\bk{0}<\a<\bk{a}$, ie., $\a$ is positive in $\rz\bbr$, but smaller than any `standard' real number; by definition $\a$ is a positive infinitesimal. One can verify that these are quite numerous, and with a little more work, see that the image of $\bbr$ in $\rz\bbr$ is in a strong sense discrete in $\rz\bbr$. But $^\s\bbr$  is also bounded in $\rz\bbr$ in the following sense. If $b_i\in\bbr$ is a sequence of positive numbers with $b_i$ unboundedly increasing and $\b$ denotes $\bk{b_i}\in\rz\bbr$, then an identical argument shows that if $a\in\bbr$ is any positive number, then $\bk{a}<\b$, ie., $\b$ is a positive infinite element of $\rz\bbr$ (by definition). So one might say that $^\s\bbr$ is bounded between $-\a$ and $\a$ for any such infinite positive number $\a$.

     In the previous example, we used the properties of $\SF$ directly. Yet much of the groundwork theory for the model theoretic approach to NSA is to
      allow one to avoid ever more complicated  arguments involving equivalence classes
      of sequences. Instead, one wishes to be able to use the enriched universe by deploying a small number of basic
      principles.
     For example, in working with `internal sets' (see our discussion on internal subsets of $\rz\bbr$ below) one uses the internal definition principle instead of equivalence classes
      of sequences. The internal definition principle exists in a range of generalities; see Keisler, \cite{Keisler1976} p46, for a transparent version,
       see Henson, \cite{Henson1997} p31, for a more involved version.  Note that, inherent in the logical transfer of structure is the
     fact that some subsets of $\rz\SX$, the external ones (again see below), don't
     faithfully carry over the logical consequences of $\SR$. The internal subsets support these
     transferred statements. (As we further develop our example around the transfer of $\bbr$ and its collection of subsets, $\SP(\bbr)$, we will give some idea on how this works.)   But, internal subsets are still remarkably numerous, and the external subsets
     coupled with internality of transferred statements imply the
     powerful overflow phenomena; see below.

     Continuing with the above concrete construction of $\rz\bbr$; let's give examples of internal and external sets in order to get a sense of the difference between internal and external sets, see why only the internal sets lift all standard properties to the nonstandard level, and also get a crude sense of how to deal with the `levels problem' (see below). In order to do this we must move up to the next rung in the set theoretic universe (ie., sets whose elements are themselves nontrivial sets); we will look at the transfer of $\SP(\bbr)$, the collection of subsets of $\bbr$. First of all, we will follow the recipe used when trying to lift the properties of  $\bbr$ to $\rz\bbr$; that is, we will lift componentwise and then, in pushing down use $\SF$ as before. Now as $\SP(\bbr)$ denotes the the collection of subsets of $\bbr$; then elements of $\rz\SP(\bbr)$ should be $\SF$ equivalence classes of elements of $\SP(\bbr)^\bbn$. Denoting the elements of $\SP(\bbr)^\bbn$ by $(A_i)$,  we should, according to the recipe for $\rz\bbr$, define $(A_i)\stackrel{\SF}{\approx}(B_i)$ precisely if $\{i:A_i=B_i\}\in\SF$. (One can check that this does give an equivalence relation, again via the use of the properties of $\SF$.) According to our recipe for lifting relations, eg., $<$, on $\bbr$ to $\rz\bbr$, and here `is a subset of' is a relation on $\SP(\bbr)$, we define $\bk{A_i}\subset\bk{B_i}$ if $\{i:A_i\subset B_i\}\in\SF$.

     This works perfectly. Note though a possibly confusing point (the first manifestation of `problem of levels'): early on we defined $\rz\bbr$ to consist of a set of equivalence classes, but now we also have an `element' of $\rz\SP(\bbr)$, apparently given by the (equivalence class of the constant sequence $(\bbr)$, denoted $\bk{\bbr}$). That is, we have two manifestations of the nonstandard reals, as a set consisting of the nonstandard reals and as an element of a higher level nonstandard set $\rz\SP(\bbr)$; these need to be the reconciled as with our ordinary sets.  But once we follow our recipe and define the lift of the `is an element of' relation of set membership, all will fit together. So if $\a=\bk{a_i}$ is an element of $\rz\bbr$ and $\SA=\bk{A_i}$, then following our mantra, we define $\bk{a_i}\in\bk{A_i}$ if $\{i:a_i\in A_i\}\in\SF$. (Once more ultrafilter properties get this to be a well defined relation on $\rz\bbr\x\rz\SP(\bbr)$.) So every element of $\rz\SP(\bbr)$ can be seen as a set of nonstandard reals; in particular $\bk{a_i}\in\bk{\bbr}$ if and only if $\{i:a_i\in\bbr\}\in\SF$, but nonstandard numbers $\bk{a_i}$ are defined up to $\SF$ equivalence, ie., this element of $\rz\SP(\bbr)$ contains precisely the same elements as $\rz\bbr$. Note also analogous to the inclusion $\bbr\ra\rz\bbr$ of the standard real numbers (the image being the isomorphic copy denote by $^\s\bbr$), there is the set of `standard subsets' of $\rz\bbr$, $^\s\SP(\bbr)\subset\rz\SP(\bbr)$, given (as before) by equivalence classes of constant sequences. In particular, note that, although as a subset of $\rz\SP(\bbr)$, $^\s\SP(\bbr)$ is external, all of its elements must be internal. With a little more detail, it's clear that all elements of $^\s\SP(\bbr)$ are of the form $\SA=\bk{A}$, for some $A\subset\bbr$,  and so as such $\bk{a_i}\in\SA$ if and only if $\{i:a_i\in A\}\in\SF$.

     Just so the reader may see that indeed these definitions of $\subset$ and $\in$ on the ultrapower (ie., nonstandard) level are consistent with each other (and as a further demonstration of the effectiveness of the properties of a nonprincipal ultrafilter), let's verify that $\bk{A_i}\subset\bk{B_i}$ $\Leftrightarrow$ $\bk{a_i}\in\bk{A_i}$ implies that $\bk{a_i}\in\bk{B_i}$. (Note that it follows immediately that two `good' nonstandard sets, ie., elements of $\rz\SP(\bbr)$, are equal if and only if they have the same elements. This is the transfer of a basic property of sets: they are equal if and only if they contain the same elements.) To verify $\Rightarrow$, let $I=\{i:a_i\in A_i\}$ and $J=\{i:A_i\subset B_i\}$. Then, our hypothesis and definitions imply that $I$ and $J$ are in $\SF$ and so $I\cap J\in\SF$. But clearly $I\cap J\subset K\dot=\{i:a_i\in B_i\}$ and so property (1) implies that $K\in\SF$. We will  prove $\Leftarrow$ by contradiction: suppose that the conclusion does not hold, ie., that $\bk{A_i}\nsubseteq\bk{B_i}$. Then it can't be true that $K\dot=\{i: A_i\subset B_i\}\in\SF$ and so by property (4), $K^c$, the complement of $K$ in $\bbn$, is an element of $\SF$. So, by definition of $K^c$, if $i\in K^c$, then there is $a_i\in A_i$ with $a_i\notin B_i$. Define an element $\ov{\a}=\bk{\ov{a}_i}\in\rz\bbr$ as follows: if $i\in K^c$, let $\ov{a}_i=a_i$, and for all other $i$ define $a_i$ arbitrarily. As $K^c\subset\{i:\ov{a}_i\in A_i\}$, it's clear that $\ov{\a}\in\bk{A_i}$; so it suffices to show that $\ov{\a}\notin\bk{B_i}$. But $K'\dot=\{i:\ov{a}_i\in B_i\}\subset K$, as by definition $\ov{a}_i\notin B_i$ for $i\in K^c$. Given this, suppose that $\ov{\a}\in\bk{B_i}$, then $K'\in\SF$ and so by property (1) $K\in\SF$,  contradicting this assumption.

     But now we find that there are subsets  of $\rz\bbr$ that are not of the form $\bk{A_i}$, ie., $\rz\SP(\bbr)\subsetneqq\SP(\rz\bbr)$. For example, the set, $\mu(0)$, of infinitesimal in $\rz\bbr$ cannot be written in the form $\bk{A_i}$. Let's indicate why this is true and at the same time give some idea on why transfer works for sets of the form $\bk{A_i}$, our internal subsets of $\rz\bbr$, ie., elements of $\rz\SP(\bbr)$ and not for sets like $\mu(0)$, ie., the external subsets of $\rz\bbr$, symbolically elements of $\SP(\rz\bbr)\smallsetminus\rz\SP(\bbr)$. Recall how we `transferred' all relations, operations, etc., from $\bbr$ to $\rz\bbr$: componentwise and then take $\SF$ equivalence classes. The identical process works for elements of $\rz\SP(\bbr)$; in particular, let's consider the `transfer' of supremum. If $A\subset\bbr$ is bounded above, then the completeness of $\bbr$ says that $\sup A$ exists in $\bbr$. First, we need to define the transfer of bounded above for elements of $\rz\SP(\bbr)$. We say that $\bk{A_i}$ is *bounded above (* to indicate the transfer of this property) if $\{i:A_i\;\text{is bounded above}\}\in\SF$. One can verify that this is well defined and has all of the properties of the standard notion bounded above. An important note here: we did not demand a uniform bound (for indices in some element of $\SF$).  Given this and with some work, one can verify that if we follow our recipe and for $\SA=\bk{A_i}\in\rz\SP(\bbr)$ that is *bounded above define  $\rz\sup\SA\dot=\bk{\sup A_i}$, then $\rz\sup\SA$ is clearly in $\rz\bbr$ and *sup has all of the properties of the standard supremum, ie., the properties of supremum transfer to the internal subsets of $\rz\bbr$. (Note that this *supremum can be an infinite element of $\rz\bbr$!) In particular, if $\SA$ is *bounded above and $\Fs=\bk{s_i}$ is the *supremum of $\SA$, then $2\Fs\notin\SA$ and if $\Fa$ is a *upper bound for $\SA$ such that $\Fa/2$ is also a *upper bound, then $\Fa$ cannot be $\rz\sup\SA$. Now suppose, by way of contradiction, that $\mu(0)\in\rz\SP(\bbr)$ and note clearly that $\mu(0)$ is bounded above. Therefore, $\Fa=\rz\sup(\mu(0))\in\rz\bbr$. But $\Fa$ can't be infinitesimal as $2\Fa$ would also be infinitesimal, violating a (transferred) property of supremum. Therefore, $\Fa$ must be noninfinitesimal, but then $\Fa/2$ is also noninfinitesimal, eg., a smaller *upper bound for $\mu(0)$, contradicting that $\Fa$ is the least such. Hence, $\mu(0)$ cannot be internal, eg., cannot carry the transferred properties of bounded subsets of $\rz\bbr$ unlike the internal subsets, where in fact these properties can be transferred using our all inclusive recipe.
     Note that hidden in this discussion is the first example of the `overflow principle'. Namely, suppose that $\SB$ is an internal subset of $\rz\bbr$ that contains the infinitesimals, then clearly it must contain noninfinitesimals. We will return to this a little later.

     Returning to the general overview, depending on how carefully one chooses $\SF_{\SJ}$, internal sets
     become  so densely numerous that various intensities of an
     extremely useful compactness phenomena, \textbf{saturation} may occur\label{page: saturation}. {\it If $\varsigma$ denotes a given
     cardinality,
     we say that $\{\rz\SX,\SR\}$ satisfies $\varsigma^+$ saturation if
     given a set $\SS\doteq\{A_i: A_i\subset \rz\SX\; \text{is internal}\;\forall
     i\in\SI\}$ such that $card\SI\leq\varsigma$ and $\SS$ has the
     finite intersection property, then  $ \cap_{i\in \SI}A_i $ is
     nonempty}, in fact, usefully fat. Monads occur when $\SS$ is the
     collection of neighborhoods of, for example, a point in a topological space.
     With these, expressions of continuity, Hausdorff-ness, etc,
     become intuitively simple. The various degrees of saturation do not come into their own until one begins building a nonstandard model of whole
     communities of mathematical objects. See below.

     Let's look at saturation in the example we are considering above. We begin with the trivial finite intersection property statement about the set $\FI$ of  open symmetric intervals around $0$ in $\bbr$, if  $k\in\bbn$ and $I_1,\ldots,I_k\in\FI$, then $I_1\cap\cdots\cap I_k$ is nonempty. (Countable) saturation of our nonstandard real numbers, $\rz\bbr$, then says that $\cap\{\rz I: I\in\FI\}$ is nonempty, in fact, quite numerous. We can see this directly from our work above and in fact see that  this set is precisely the (external) set of infinitesimals, $\mu(0)$. Clearly, if $\Fi\in\mu(0)$, then $|\Fi|<\bk{a}$ for every positive $a\in\bbr$, ie., $\Fi\in\rz I$ for each $I\in\FI$. On the other hand, if $\Fi$ is not in $\mu(0)$, then $|\Fi|$ is greater than some positive standard number $\bk{a}$ and so is not in, eg., $\rz[-a/2,a/2]$.

     Let's get a hint at how infinitesimal (ie., elements in our enriched structure) simplify the description of continuity and at the same time say a bit more about `overflow'. In order to do this, we need to extend our family of nonstandard sets. If $\xi,\z\in\rz\bbr$ with $\xi-\z\in\mu(0)$, then $\xi,\z$ are said to be infinitesimally close (with respect to the metric topology on $\bbr$) and we write $\xi\sim \z$ to denote this. Now just as we defined the transfer of the set of all subsets of $\bbr$, ie., $\rz\SP(\bbr)$, we can similarly define the transfer of the set of all real valued  functions, $F(\bbr,\bbr)$, on $\bbr$ to be $F(\bbr,\bbr)^\bbn$ modulo the equivalence relation defined by $\SF$; it works as before. And just as we saw that $\rz\SP(\bbr)$ can be seen to be (once *set membership is properly defined) to be (the internal) subsets of $\rz\bbr$, we can identically get that elements of $\rz F(\bbr,\bbr)$ can be seen to be the set of  internal functions mapping $\rz\bbr$ to itself. That is, we extend componentwise and then mod by $\SF$: if $\Ff\dot=\bk{f_i}\in\rz F(\bbr,\bbr)$, and $\xi=\bk{x_i}\in\rz\bbr$, define $\Ff(\xi)=\bk{f_i(x_i)}$. As before, this is well defined and is indeed an element of $F(\rz\bbr,\rz\bbr)$. In particular, those internal functions that are equivalence classes of constant sequences, ie., those $\bk{f_i}$ where there is $f\in F(\bbr,\bbr)$ such that $\{i:f_i=f\}\in\SF$ are precisely the set of transfers, $\rz f$, of elements  $f\in F(\bbr,\bbr)$, ie., the `standard elements' in $\rz F(\bbr,\bbr)$. Note now that just as $\rz[1/2,1]$ is far richer than $[1/2,1]$, so is $\rz f$ much more that $f$, eg., $\rz f$ is now defined on all of $\rz\bbr$ so that its asymptotic behaviors in the large and the small are explicitly revealed. For example, we now have the capacity to give the nonstandard characterization of continuity of $f$ at $0$: if $\xi\sim \rz 0$, then $\rz f(\xi)\sim\rz f(0)$. With this and a few metric properties of nonstandard numbers, proving continuity become greatly streamlined. This kind of simplification of proofs is often the case; again see the paper \cite{FarkasSzabo}.

     We can see here a little of the use of `overflow' and the `internal definition principle' in the verification that this nonstandard condition is indeed equivalent to the usual definition of continuity of $f$at $0$. Suppose that $\rz f$ satisfies the above condition and let $r\in\bbr$ be a positive number. The internal definition principle implies that $\SO\dot=\{\xi\in\rz\bbr:|\rz f(\xi)-\rz f(0)|<\rz r\}$ is an internal set (see the next paragraph). Actually, using our recipe (now getting a bit involved) and writing $\xi=\bk{x_i}$, we can see that this is just $\bk{\{x_i\in\bbr:|f(x_i)-f(0)|<r\}}$, an internal set by definition. But by hypothesis, $\SO$ contains the infinitesimals and hence must contain noninfinitesimals (this is overflow). In fact, it contains a noninfinitesimal interval implying that there is a positive $s\in\bbr$ such that $|x|<s$ implies that $\rz x$ is an element of $\SO$.

     We will finish  the  analysis of $\rz\bbr$  with a remark on the internal definition principle. Above, we needed to know that the nonstandard set $\{\xi\in\rz\bbr:|\rz f(\xi)-\rz f(0)|<\rz r\}$ is internal to carry out the above argument.{\it For this set, the internal definition principle says: the relations and operations $(<,-,\in, \text{function})$ are all formally defined, the `constants' $(\rz r,\rz f,\rz\bbr)$ are all internal;  hence  set formation involving these extended operations and relations with internal entities returns internal sets.} It's a recursive type of definition. Of course, in this case we could see directly that it is of the form $\bk{S_i}$ for `good' sets $S_i$ and therefore internal, but this can get quite involved. The internal definition principle allows one to check if a set is internal by the constituents in its definition, a very useful shortcut.

       It is important to
     note that the standard model fits consistently; standard structures lift into the nonstandard setting  via the embedding into the
     nonstandard model by sending an element, set, function to the equivalence class of the corresponding constant
     sequence.

     In many situations this ultrapower   extension of a particular object is sufficient.
     See for example, van den Dries and Wilkie's greatly simplified proof of Gromov's theorem on groups of polynomial
     growth, \cite{DriesWilkie1984}.
     But what if a nonstandard model needs to be more encompassing of families of objects,
     their families of maps, functionals between these families,
     etc.
     One long standing solution to this extended enterprise is the superstructure approach of Robinson and Zakon. Over a ground object
     of ones choosing, one builds a tower of objects, a superstructure,
     in the manner of building the universe of set theory.
      In the examples given above, we have already begun this process in passing form the construction of $\rz\bbr$ to the next level the construction of $\rz\SP(\bbr)$ and then pulling these together by extending the notion of set membership.
     Generally, if we want to build this up to functionals on function spaces, etc., we need to iterate this procedure, and extend the set membership relations (make sense of the transfer of a set whose elements are sets whose elements are sets whose.....) This is the problem of levels mentioned above that was adequately solved by Robinson and Zakon in the manner we indicated.
     Abstractly, for a given `base' $S$ (above we were working with $S=\bbr$) we build the superstructure \textbf{V(S)}\;$=\cup_{n\in\bbn}\text{ \textbf{V(S)}}_n$
     on this base. (See Lindstrom, \cite{Lindstrom1988} p.23) where
     $\text{\textbf{V(S)}}_{n+1}=\text{\textbf{V(S)}}_n\cup\SP(\text{\textbf{V(S)}}_n)$,
     For example, a  collection of subsets of eg., $C^\infty(\bbr,\bbr)$, and so for example a germ of a subset of $C^\infty(\bbr,\bbr)$, an equivalence class of family of germs of subsets of $C^\infty(\bbr,\bbr)$, etc. is an element of
    $\text{\textbf{V(S)}}_n$ for some $n\in\bbn$ (see, for example Rubio, \cite{Rubio1994} pp. 19-22) and therefore
    its transfer is a standard element of $\rz\text{\textbf{V(S)}}_n$. Looking at our examples above one should not be surprised to see that the internal elements in this transferred tower must be precisely the elements of $\rz\text{\textbf{V(S)}}_n$ for some $n\in\bbn$.

\subsection{Three principles and working tools}\label{sec: 3 principles and tools of nsa}
  After the introduction above, we will briefly discuss the main three working principles of nonstandard mathematics. The above discussion should be sufficient to make the descriptions below understandable.
  We will then list, with very brief description, the nonstandard notations we will use.
\subsubsection{First principle: transfer}\label{subsec: transfer principle} There are three principles
that make $\rz V(S)$ particularly useful. The first, sometimes called
the {\bf transfer principle} says, roughly, that {\it any true statement
in the standard world has a precise counterpart about the
corresponding {\bf internal} objects in $\rz V(S)$}.
We saw this at work in (partially) verifying that $\rz\bbr$ is a totally ordered field.
 For this paper, there
is a basic theorem in Lie group theory that a continuous
homomorphism of Lie groups is $C^\infty$. The transfer theorem
(roughly) implies that a *continuous homomorphism of *Lie groups is
$\rz C^\infty$. Basically, the $\rz$ that is qualifying the three terms continuous, Lie
group and $C^\infty$ says we are talking about internal objects and
hence guarantees the validity of the statement in $\rz V(S)$. Sometimes
we will say that we are *transforming (or transferring) a
particular object or statement; in such cases we are invoking the
transfer theorem. Sometimes we will use it without such a remark.

Note that we will use reverse transfer at critical points in the proofs of the standard consequences of nonstandard results, eg., in corollarys \ref{cor: C^0 precpt fam in Gp is C^k precpt} and \ref{cor: standard reg + almost->near} (see also the curious corollary \ref{cor: 2nd stan cor of appendix thm} of the theorem in the appendix). These will typically be of the following form. If $B$ is a set satisfying $\rz B$ is nonempty, then $B$ is nonempty. This curious strategy is helpful as it is often much easier to show that $\rz B$ is nonempty, than to show that $B$ is nonempty!

\subsubsection{Second principle: saturation}\label{subsec: saturation}
 The second principle is that of (sufficient)
``saturation'' of $\rz V(S)$. We have described consequences of saturation, as well as a brief description on page \pageref{page: saturation}.
There are a variety of types and degrees of saturation. A $\rz V(S)$
big enough to have nonzero *polynomials that vanish at all standard real numbers is an enlargement.
Generally, if $\k$ is a cardinal larger than the countable cardinal,
then there exists $\rz V(S)$ that are ``$\k$-saturated.'' As described above,
this means that if $\{\SA_j:j\in J\}$ is a set of internal elements
of $\rz V(S)$ that has the finite intersection property and $\card(J)$
is less that $\k$, then $\bigcap_{j\in J}\SA_j$ is nonempty. Note
that this intersection is typically an external set. For example, let
$\k$ be greater than the cardinality of $\SA=\{A:A\subset
C^\infty(\bbr^m,\bbr^n)\}$, let $f_0\in C^\infty(\bbr^m,\bbr)$ and
$\SU_{f_0}=\{\SU\subset C^\infty(\bbr^m,\bbr):\SU$ is a neighborhood
of $f_0$ in some $C^\infty$-topology\}. Then
$\{\rz \SU:\SU\in\SU_{f_0}\}$ is a collection of internal sets and the
cardinality of the collection is less than $\k$; so if  $\rz V(S)$ has
``$\k$-saturation'' then
$$
\mu(f_0) = \bigcap\{\rz \SU:\SU\in\SU_{f_0}\}\neq\emptyset,\quad
\text{as an {\bf external} subset of $\rz \SA$.}
$$
Descriptively, $\mu(f_0)$ consists of all elements of
$\rz C^\infty(\bbr^m,\bbr)$ that are {\bf infinitesimally} close to $f_0$
in the given topology. This is an example of a {\bf monad} for the
given topology. In this paper, we will assume $\k$-saturation of
$\rz V(S)$ for $\k$ big enough for our purposes. Also, as all
topological spaces here are Hausdorff, if $x$ and $y$ are two
standard elements and $x\neq y$, then $\mu(x)\cap\mu(y)=\emptyset$.

\subsubsection{Third principle: overflow}\label{subsec: overflow}
We  said a little about this last principle, but overflow is quite important in putting external sets to work. Often external sets are used in simplifying standard notions; eg., the set of infinitesimals, $\mu(0)\subset\rz\bbr$ in the simple nonstandard characterization of continuity (see above discussion). Overflow exists in various guises in NSA; e.g., see
{\bf overflow} in Lindstrom, \cite{Lindstrom1988} p12 or {\bf Cauchy's Principle} in Stroyan and Luxemburg, \cite{StrLux76}
p188. It basically says that {\it if one has an internal statement $p(x)$
such that $x$ is a free variable and $p(x)$ holds for $x$ in an
external set $E$, then $p(x)$ holds for all $x$ in an internal set
containing $E$}; this also includes a use of the internal definition principle.
This principle is used repeatedly here.
We use this principle explicitly in 2.2.12 and 2.2.9
and implicitly in 6.5 where it is critical in constructing a
standard homeomorphism from an internal map. In chapter \ref{chap: partial solution to approx prob}, it is used repeatedly and is also important in the appendix, chapter \ref{chap: appendix: S-smoothness}.

\subsection{Nonstandard tools specific to this paper}

\subsubsection{Definitions of NSA working tools}\label{subsec: nsa working tools}
 We follow this sketch of basic
principles of NSA with a listing of definitions of working tools
from NSA that will be used here. This list is obviously queued by
the notation for the given tool. Again some of the definitions are
heuristic.
\begin{itemize}
\item [$\rz X\lra$]  If $X$ is a standard set, then this is the
corresponding internal set in the nonstandard universe $\rz V(S)$. If
$x\in S$, then $\rz x$ is a point in $^\s S\subset \rz S$.
\item[$^\s X\lra$]  If $X$ is a set in the standard universe,
$^\s X$ is the external set in the nonstandard universe given by
$\{\rz x:x\in X\}$, e.g., $^\s\bbn$ is the (external!) set of standardly
finite integers in $\rz \bbn$.
\item[$a\sim b\lra$] If $X$ is a standard set with a topology
$\tau$, and $a$, $b\in \rz X$, then $a\sim b$ ($a$ is infinitesimally
close to $b$ with respect to the topology $\tau$) holds when $a\in
\rz U$ if and only if $ b\in \rz U$, for all open set $U$ in $\tau$.
\item[$a\nsim b\lra$]  $a\sim b$ is not satisfied.
\item[$a\lnsim b\lra$] $a<b$ and $a\nsim b$.
\item[$X_{\nes}\longrightarrow$] If $X$ is a set in the nonstandard universe,
set $\rz X_{nes}=\{x\in X:$ there is a (standard) element $\rz z$ in $X$ such that
$x\sim\rz z\}$. Obviously, sometimes $X_{\nes}=\emptyset$ (is empty);
e.g., if $\om>0$ is infinite, and if $X=\{a\in\rz \bbr:a\ge\om\}$, then
$X_{\nes}=\emptyset$. The elements of $X_{\nes}$ are called the {\bf
nearstandard points of} $X$.
\item[$^\circ x$ or $\st x\lra$] If $x\in X_{\nes}$, so that
$x\sim\rz z$ for some standard point $\rz z$ in $X$, then $^\circ x=z$,
or$\ st(x)=z$. (This is well defined as monads are disjoint).
\item[$\mu(x_0)$ or $\mu_{x_0}(X)\lra$] If $(x,\tau)$ is a topological
 space, the $\mu(x_0)$ or $\mu_{x_0}(X)$ is the set $\{x\in X:x\sim
 x_0\}$. This is called the {\bf monad of $x_0$} in $X$ (with respect
 to $\tau$). As already noted this is
 $\bigcap\{\rz U:U\in\tau\}\neq\emptyset$ by sufficient saturation.
\item[$\rz X_{\nes}\smallsetminus\mu(x_0)\lra$] This is $\{x\in X_{\nes}:x\nsim x_0\}$.
\item[$^\s\open\lra$] This is a set of the form $\rz U$, where $U$ is an
open set in the given topology.

For mappings, these refine to
\item[$\rz f\lra$] If $f:X\lra Y$ is a standard map, then identifying
$f$ with its graph $\G_f\subset X\x Y$, $\rz f$ is defined to be the
internal map with graph $\rz \G_f\subset\rz(X\x Y)=\rz X\x\rz Y$.
\item[$^\s f\lra$] Again, identifying $f$ with its graph $\G_f$,
this is defined to be the external map with graph $^\s\G_f$, an
external subset of $*\G_f$. In particular,  $^\s\G=\rz \G_f\cap\
{}^\s(X\x Y)$, could be empty.
\item[$^\s$local $\lra$] A description of the domain of an internal function: it is a standard open set.
\end{itemize}

\subsubsection{Pertinent NSA facts for this paper}\label{subsec: pertinent nsa facts}
 We need some final remarks on the NSA needed for this
paper. We will generally be working with internal maps
$f:\rz U\to\rz \bbr^{n}_{\nes}$, where $U\subset\bbr$ is an open
neighborhood of $0$, and attempting to prove that their standard
parts, $^\circ f:\;^\circ\!(*U)\to{}^\circ\!(f(*U))$ have nice properties.
First of all, note that $^\circ(\rz U)=\ov U$, the topological closure
of $U$.  The introduction in Wicks, \cite{Wicks1991}, covers the basic facts of topology from a nonstandard perspective very well. See also Lindstr{\o}m, \cite{Lindstrom1988}, p52--57. Even if $U$ is
connected, simply connected, etc. $\ov U$ may not have any of these
properties. We will assume that our $U$ are convex to prevent this.
Our arguments will be local and so we will be able to restrict our
consideration to a convex subset. Second, our ``map''  $^\circ\!f$
may not even be a function. For example, it might send distinct
points in a monad to points in distinct monads. $S$-continuity
 will prevent this. It might also send nearstandard points to
non nearstandard points. For example, we will be working with
*bilinear maps: $B:V\x V\to V$ (e.g., our *Lie bracket --- here
$V$ is an internal vector space over $\rz\bbr$ such that $V_{\nes}$ is
well defined). If $B\mid V_{\nes}\x V_{\nes}$ does not have image in
$V_{\nes}$, then $^\circ\!B$ will not be defined.

\subsubsection{Transferring maps and their domains}\label{subsec: transferring maps and domains}
We need to also say something about the domains of standard parts of
internal maps. If  $\Ff:\rz U\to\rz\bbr$ is an internal map as above with
well defined standard part $^\circ\Ff:{}^\circ\!(\rz U)\to\bbr$; i.e.,
$^\circ\Ff:\ov U\to\bbr$, we will restrict to the original open set $U$.
There are two reasons for this. First of all, sometimes we will need
to consider internal maps $\Ff$ of the form $\rz h$, for some standard
$h:U\to\bbr$. In this case, we get $^\circ\!(\rz h)=h$ but now extended to
$\ov U$ via it limiting behavior. This process gets these limiting
values automatically --- if they exist! In this paper, worry about
this is not needed. In particular, for this paper we will always
have  $^\circ\!(\rz h)=h$ by restricting our domain to the original open
$U$. The second reason is that we need our map $^\circ\Ff$ to be
defined on an open $U$ in order to consider its regularity
properties without needless boundary technicalities. {\bf Therefore,
henceforth when we write $^\circ\Ff$ for some internal
$\Ff:\rz U\to\rz\bbr$, it will be understood that we are considering
$^\circ\Ff\mid U$.} Later in the text, we will typically remedy this  by considering $\rz f$ not on $\rz U$ but on its `nearstandard part', denoted $U^\mu$, or $\rz U_{nes}$,  the set of all points in $\rz U$ infinitesimally close to a point of $U$. In such cases we will use that $\rz U_{nes}=\cup\{\mu(x):x\in U\}=\cup\{\rz K: K\subset U\;\text{is compact}\}$.

Further tools along with notations are defined in the text.

\subsection{Nonstandard calculus} \label{sec: nonstandard calculus}
In this section we describe the nonstandard
calculus required for this paper. Some of this is in the literature.
We turn now to the basic nonstandard differential calculus that will
be needed.

\subsubsection{Transfer of calculus on Euclidean space}\label{subsec: transfer of calc on E^n}
All internal differential calculus will
be {\bf${}^\s$local near $0$} in $\rz\bbr^n$, i.e., on {\bf standard
neighborhoods of $0$}, i.e., on sets $\rz U$ where $U$ is a
neighborhood of $0$. Here $n\in{}^\s\bbn$. As we will work with only
a finite number of such neighborhoods, we don't have to worry about
externality creeping in. Our nonstandard calculus will follow Stroyan and Luxemburg, \cite{StrLux76}
but be in the spirit of Lutz and Goze (geometry in internal set theory), \cite{LutzGoze1981}. {\bf By the
transfer theorem, all standard differential calculus constructions
have parallel NS, nonstandard, copies. For example, the existence
and properties of tangent spaces and their morphisms, the
differentials of smooth maps, will be asserted to exist in the
nonstandard domain without proof.} Occasionally, their existence
will be asserted by transfer. The work arises from the further
assertion that these have special properties beyond those that arise automatically via transfer. These properties will often be used in verifying such assertions.


\subsubsection{Nonstandard metric properties at the
tangent space level.}\label{subsec: intro ns metrc tngnt sp E^n}

\begin{remark}
   The development in this subsection was originally motivated by the following problem. Given a finite dimensional internal vector space $V$ over $\rz\bbr$  (eg., the *tangent space at the identity of an internal Lie group), how does one define $V_{nes}$? In fact, generally, $V_{nes}$, the set of nearstandard vectors in $V$, cannot be defined. The solution to this question was central to the undertaking in this paper: how could we prove that the the *bilinear map $\ad:V\x V\ra V$ is S-continuous (the critical technical result of this paper) if we did not have an unambiguous framework within which to define $V_{nes}$? (An internal bilinear map is S-continuous precisely when it sends nearstandard pairs of vectors to nearstandard vectors.)
   For an example of such a vector space, let $\Fq\in\rz\bbr$ be such that $^\s\bbr\cap(\Fq\cdot^\s\!\!\bbr)$ is empty and let $V=\{\Fr(1,\Fq):\Fr\in\rz\bbr\}$. $V$ is internally isomorphic to $\rz\bbr$, but no element of $V$ besides $0$ is standard. So unless one considers the extrinsic embedding of $V$, it has no nearstandard points (beyond those given by $\Fr\sim 0$)! The nonstandard calculus (on $\rz\bbr^n$!) developed in this section allowed the author to canonically (standardly) translate the (standard) metric structure on $\rz\bbr^n$ up to the (transferred) tangent spaces from which a natural notion of nearstandard follows, solving my problem. We should be clear on this. The argument here needed the transfer of the natural metrical identification of $\bbr^n$ with $T_0\bbr^n$ via the identification of the canonical frame on $\bbr^n$ with the that on $T_0\bbr^n$. As this is the transfer of a standard isomorphism, it automatically gives a correspondence between nearstandard vectors. Note that the transfer of the process that canonically identifies a general $n$ dimensional real vector space, $V$,  with its tangent space at $0$, $T_0V$, to the class of internal vector spaces over $\rz\bbr$  is not good enough for unambiguously defining the notion of a nearstandard vector!
\end{remark}

To begin, if $U$ is a standard neighborhood of $0$ in $\bbr^n$, then
$\rz C^\infty(\rz U,\rz \bbr)$ is just the *transfer of $C^\infty(U,\bbr)$.
So $\rz C^\infty(\rz U,\rz\bbr)=\{\Ff:\rz U\to\rz\bbr:\Ff$ is internal and all internal
derivatives of $\Ff$ are  *continuous on $\rz U\}$. (Sometimes we will
write this space as $\rz C^\infty(U,\bbr)$, sometimes as
$\rz C^\infty(\rz U,\rz \bbr)$.) Similarly, we define, for $k\in\bbn$, $\rz C^k(U,\bbr)$ to be those internal $\Ff:\rz U\ra\rz\bbr$ with the property that all internal partial derivatives up to order $k$ on $\rz U$ are *continuous.   We must qualify here in order to have a good
notion of nearstandard: all internal derivatives will be with respect to the
standard basis on $\rz\bbr^n$, i.e., the *transfer of the canonical
frame on $\bbr^n$. That is, if $(e_1,\dots,e_n)$ is the canonical
frame  on $\bbr^n$ and $x\in\rz U$, and if $1\le j\le n$ then we have
that $(\rz \p_jf)(x)\doteq\rz \f d{dt}\bigm|_{t=0}$ ($f(x+te_j)$).

Let $\rz T_xU$ be the internal tangent space to
$\rz U$ at $x$ and $\rz TU=\cup\{\rz T_xU:x\in\rz U\}$ be the internal tangent
bundle to $\rz U$. Note that
$$
\rz \Hom(\rz TU\ox\rz TU,\rz TU) = \rz (\Hom(TU\ox TU, TU))
$$
and therefore standard elements on the right hand side of this equality give us
standard metric tensors over $*U$ on the left hand side. These are elements of
$^\s(\Hom(TU\ox TU, TU))$. For this paper, the notions of
nearstandard tangent vectors and nearstandard differentials of
$\rz C^\infty$ maps is critical and we proceed to define these.

\subsubsection{Nearstandard tangent vectors.}\label{subsec: nearst tngnt vectors E^n}
 For the given
canonical frame $(e_1,\dots,e_n)$ for $\bbr^n$, there corresponds the {\bf
constant canonical sections} on $TU$, denoted by $\p_1,\dots,\p_n$.
As vector fields these act on smooth $f:U\to\bbr$ as already defined;
namely $\p_j\!\bigm|_x(f)=\p_jf(x)$.
These *transfer to give the
corresponding canonical standard frame on $\rz TU$, now over $\rz\bbr$. We
will call this a {\bf standard frame on $\rz TU$}. Using these,
$\bs(\bs\rz\BT\BU\bs)_{\nes}$ can be defined in two equivalent ways.
(That these are equivalent is just the transfer of basic tensor facts.)
For each $x\in U$ we have the usual inner product $\<\  \>_x:TU\x
TU\to\bbr$ defined by $\<\p_i\!\bigm|_x,\p_j\!\bigm|_x\>_x=\d_{ij}$,
extended via $\bbr$-bilinearity. The collection of these for each
$x\in U$ will give the constant metric tensor,  $\<\  \>$, over $U$.
If $\nu\in T_xU$, then $|\bs\nu|_{\bf
x}\doteq\sqrt{\<\nu,\nu\>_x}:T_xU\to\bbr$ gives the usual norm on
$T_xU$. Note then that as $|\p_i|_x=1$, then taking *transfer we get that
for every $\xi\in\rz U$, $\rz|\rz\p_i|_\xi=1$. \textbf{This allows us an unambiguous}
\textbf{definition of} $\bsm{\bs(\bs\rz\BT_\xi\BU\bs)_{\nes}}$ for each $\xi\in\rz U$, \textbf{the} $\bsm{\bbr_{nes}}$\textbf{-module of nearstandard tangent vectors on} $\bsm{\rz U}$, as the
$\{\nu\in\rz T_\xi U:|\nu|_\xi\in\rz\bbr_{\nes}\}$. We can then show that for each $\xi\in\rz U$,
$$
(\rz T_\xi U)_{\nes}=\{\sum_ja_j\rz\p_j\bigm|_\xi:a_j\in\rz\bbr_{\nes}\}.
$$

We define
$\bs(\bs\rz \BT\BU\bs)_{\nes}\doteq\bigcup_{\xi\in\rz U}(\rz T_\xi U)_{\nes}$.
(Note that we have to be careful here for we are including tangent vectors over points whose standard parts (if they exist) are in $\ov{U}$. These problems can be avoided by considering *tangent vectors lying over $\xi\in\rz U$ that are infinitesimally close to points of $^\s U$. This is just, following Wick, \cite{Wicks1991}, p.6, $U^\mu=\cup\{\mu(x):x\in U\}$.)
A more natural approach is given by noting that
$(\rz TU)_{\nes}\subset\rz TU$ and the trivialization of $\rz TU$ is a {\bf
standard} trivialization that defines the topology and is outlined
as follows. First note that the standard natural bundle
trivialization $TU\overset{p}{\lra}U\x\bbr^n$ defines the (smooth)
topology on $TU$ in terms of the standard product of the Euclidean
topologies on $U\subset\bbr^n$. Note also that there is a canonical map
$t_x:T_xU\stackrel{\cong}{\ra}T_0U$ translating a vector at
$x$ to the corresponding one in $T_0U$. This is just the
differential at $x$ of the map $v\to x-v$. In fact $t_x$ is the
restriction of the unique map $t:TU\lra T_oU\x
U:\nu_x\lra(t_x(\nu_x),x)$ which is a bundle $\cong$. Given the
natural identification
$F:T_o\bbr^n\stackrel{\cong}{\ra}\bbr^n$, we get
 $p=(F\x1_U)\circ t$. *Transferring this canonical identification, we get a canonical natural standard
*bundle isomorphism $\rz TU\lra\rz T_0U\x\rz U$. Note that as $\rz p$ is standard and defines
the topology, it carries nearstandard points to nearstandard points;
i.e., $(\rz TU)_{\nes}=p((\rz T_0U)_{\nes}\x\rz U_{\nes})$. It is easy to see
that this is the same as the above definition of $(\rz TU)_{\nes}$.
(Note here we are talking about nearstandard points {\bf and} the
near standard tangent vectors at those points.)

\subsubsection{Infinitesimal tangent vectors.}\label{subsec: intro:infinites tngnt vctors}

As $U$ is open, in the following our considerations will be restricted to $\xi\in\rz U$ that are infinitesimally close to points of $^\s U$. For the moment, we will denote this set by $\rz U^\mu$ (see the introductory topology material in Wicks, \cite{Wicks1991}) and the set of nearstandard tangent vectors to such points by $(\rz TU^\mu)_{nes}$. Note that these correspond precisely to $\rz U^\mu\x\rz\bbr^n_{nes}$ under the trivialization $\rz p$.
Given this, as $p$ is a topological equivalence the following definition makes sense.
\begin{definition}\label{def: infinites close elts in *T(R^n)}
 If $x\in U$ we \textbf{define the} {\bf monad of a
point} $\bsm{\nu_x\in(\rz TU)_{\nes}}$ to be $\rz p^{-1}(\mu(\rz p(\nu_x))$.   For $\nu_x\in\rz T_x\bbr^n,\;\;w_y\in\rz T_y\bbr^n$, this
is equivalent to defining $\bsm{\nu_x\sim w_y}$
 if and only if
 $\rz p(\nu_x)\sim
\rz p(\nu_y)$, that is, if and only if $|\rz t_x(\nu_x)-\rz t_y(\nu_y)|_o\sim 0$ and $x\sim y$.
\end{definition}
 Note
that with respect to the first approach writing $\nu_x=\sum a_i\p_i\bigm|_x$ and
$w_y=\sum b_i\p_i\bigm|_y$ we see that $\nu_x\sim w_y\dllra x\sim y$
and $\forall i$ $a_i\sim b_i$, i.e., the monads on $(\rz TU)_{\nes}$
are defined via the standard coordinate chart trivializations.

We will now define the (${}^\s$local) {\bf *differentiable structure
on $\rz (TU)$}. For later purposes, we will give two  definitions (whose equivalence as above follows from the transfer of equivalent notions: in terms of using the trivialization to define differentiable maps or to define differentiable curves). We
define $\bf f:\bs\rz\BT\BU\lra\bs\rz\bbr$ \textbf{(internal) to be}
$\bs\rz\bbc^{\bs\infty}$ $\dllra f\circ\rz p^{-1}:\rz(U\x\bbr^m)\lra\rz\bbr$ is
$\rz C^\infty$. We can also define a
 \textbf{curve $c:\rz\bbr,0\lra\bs\rz\BT\BU$ to be $\bs\rz\bbc^{\bs\infty}$} $\dllra$ the curve
$\rz p\circ c$ is $\rz C^\infty$. As above, since $\rz p$ is standard, $\rz TU$
carries a ${}^\s C^\infty$-structure and as we shall see later an
$SC^\infty$ structure. The ${}^\s C^\infty$-structure will not be
used and there will be no further mention of it.

\subsubsection{Differentiable structures and maps.} \label{subsec: intro:diff struct and maps}
Note that we now have enough machinery to well define the following notion. Suppose that $\vp$, $\psi\in\rz C^\infty(U,V)$, then we have that ($\rz d\vp$ is the internal differential of $\vp$, etc.) $\rz d\vp$, $\rz d\psi\in\rz C^\infty\Hom(TU,TV)$ ($\rz C^\infty$ maps, *linear on *fibers covering $\rz C^\infty$ maps).
The definition will be given in multiple (obviously) equivalent formulations.
\begin{definition}
 We say $\rz d\phi$ is infinitesimally close to $\rz d\psi$ on $\rz U$, written $\rz\bf d\bs\vp\bs\sim\rz\bf d\bs\psi\dllra$ for all $x\in\rz U$, for all $\nu\in(\rz T_xU)_{\nes}$, $\rz d\vp(\nu)\sim\rz d\psi(\nu)$.
 That is, in {\bf standard} local coordinates on $U$ and $V$, we have $\rz \p_i\vp^j(x)\sim\rz \p_i\psi^j(x)$ for all $i$, $j$ and $x\in\rz U$.
   This is the same as saying that for all $\xi\in\rz U$ and
$\forall\nu\in(\rz T_\xi U)_{\nes}$, $d\vp_\xi(\nu)$ is in the monad of $d\psi_\xi(\nu)$.
\end{definition}

As standard local coordinate trivializations
preserve monads, then this should be clear.
Note that generally we have the following definition.
Let $F$, $G$ be
$\rz C^\infty$ bundle mappings: $\rz TU\lra\rz TV$, i.e., linear fiber
mappings covering a $\rz C^\infty$ map: $U\lra V$. Then similar to
above we say that $\bsm{F\sim G} $ if $\forall
u\in(\rz TU)_{\nes}$, $F(u)\sim G(u)$ and as above this is equivalent
to $F^i_j(u)\sim G^i_j(u)$ for all $i,j$ where these are the
components of $F$ and $G$ for a given {\bf standard} trivialization.

\subsection{Nonstandard functions and S-regularity}\label{sec: intro: NS fcns and S-regularity}
   Nonstandard mathematics is useful for standard mathematics if we can build subtle connections beyond the formal transfer theorem. Here we give an introductions to our attempts at such connections.
\subsubsection{The notions S-property where S is continuity, smoothness, etc.}\label{subsec: notion of S-property}
We want to give a brief introduction to regularity properties of
internal maps with respect to the standard world. (After all, proving
such regularity criteria is what this paper is about.) The idea here
is to find useful conditions to impose on a map at the nonstandard
level that force the needed regularity properties on the standard
part of the map. Suppose we are given an internal map $f:X\to Y$
where the internal sets $X$, $Y$ have $X_{\nes}$, $Y_{\nes}$ well
defined and $\mu_x(X)$, $\mu_y(Y)$ well defined for $x\in X_{\nes}$
and $y\in Y_{\nes}$. Then we say that $f$ is {\bf$\BS$-continuous},
written $f\in SC^0$ or $f\in SC^0(X,Y)$, if $f:X_{\nes}\to
Y_{\nes}$ and if $\forall x\in X_{\nes}$,
$f(\mu(x)\subset\mu(f(x))$. See Wicks, \cite{Wicks1991} p.7, for a  wrapup of the basic definitions. If
$X=\rz M$, $Y=\rz N$ for Hausdorff spaces $M$ and $N$ and $f=\rz g$, then
this is the nonstandard version of the continuity of $f$, only that
in this case it needs to be checked just for $x\in{}^\s X$ (checking at all points gets uniform continuity) It
should not be too surprising to find that for such an internal $f$
as above, not only does $^\circ\!f:{}^\circ\!X_{\nes}\to {}^\circ\! Y_{\nes}$ exist, but it is $C^\circ$(continuous).

It is straightforward to give condition on an internal map to guarantee that its standard part is continuous. But such conditions are not so clear if we want the standard part to be a homeomorphism. In this situation and others it is useful to fall back on the default position. {\bf That is we say that an internal map $f$ is $dS\!-\!P$ if its standard part, $^\circ f$, has property $P$}. For example, given $X$ and $Y$ as above and internal $f$ as above, we say that $f$ is an {\bf$dS$-homeomorphism}, if $^\circ f:{}^\circ X_{\nes}\to {}^\circ Y_{\nes}$ is a homeomorphism. Note that there are a variety of nonstandard conditions on $f$ that could force $^\circ f$ to be a homeomorphism. In section \ref{sec: fact on S-homeos and pf finish} we will develop one of these. There is a comparable
definition for {\bf$dS$-diffeomorphisms}, but first we need to say
something about our two definitions of sets of internal maps whose
standard parts are continuous, $SC^0$ and $dSC^0$. Our perspective
will always be directed towards a particular type of regularity of
the standard parts of internal maps. Obviously, $SC^0\subset dSC^0$
and the particular internal nature of the maps in $SC^0$ is clearly
defined so that we can actually work with these on the internal
level. But sometimes we just need to know that particular maps on
the internal level have the kinds of standard parts specified
without knowing their nature on the internal level. For example, let
$\bf d\BS\BA^{\bs\om}$ be those internal maps (unconcerned with
domain and range at the moment) whose standard parts are well
defined (real) analytic maps. We will give a working definition for
a (functorial!) set of internal maps, denoted by $\BS\BA^{\bs\om}$
and prove that indeed the internal maps satisfying these properties
have analytic standard parts. We also have other subsets of
$dSA^{\om}$, $^{\bs\s}\BA^{\bs\om}$ which is the $*$-transfer of
standard analytic maps and {\bf SPoly} our notation for the set of internal polynomial
maps of $^\s$finite degree with nearstandard coefficients. We will
not need the full strength of $SA^{\om}$ in this paper, here we will need some control in the manner that
$^{\s}A^{\om}$ and $S$Poly  interact in $dSA^{\om}$.

\subsubsection{Properties of standard part map}\label{subsec: props of standard part map}

Before moving on to $SC^\infty$ maps, we would like to point out some well known NSA facts. $\rz\bbr_{\nes}$ is not a field, but is a subring of the field $\rz\bbr$. Our internal maps will be $\rz\bbr_{\nes}^n$ valued for some standard $n$ and hence will form on $\rz\bbr_{\nes}$-module. As such, the "taking the standard part" operation commutes with the module operations, i.e. if $\a$, $\b\in\rz\bbr_{\nes}$ and $f$ and $g$ are $\rz\bbr_{\nes}$-valued, then
$$
^{\circ}\!(\a f +\b g) =\; ^\circ\!\a ^{\circ}\!f + {}^{\circ}\!\b\;{}^{\circ}g.
$$
If our maps are $\rz\bbr_{\nes}$-valued, then in fact the operation of
taking the standard part  commutes with the algebra operations,
i.e., if $f$ and $g$ are $\rz\bbr_{\nes}$-valued then
$$
^{\circ}\!(f\cd g) = (^{\circ}\!f)\cd(^{\circ}\!g)
$$
as $\bbr$-valued maps. Finally if $U^{\open}\subset\bbr^m$,
$V^{\open}\subset\bbr$ and if $f\in SC^0(\rz U,\rz V)$ and $g\in
SC^0(\rz V,\rz\bbr^p)$, then $g\circ\!f\in SC^0(\rz U,\bbr^p)$ and
$^{\circ}\!(g\circ f) = (^{\circ}\!g){\circ}(^{\circ}\!f)$ as elements of
$C^0(U,\bbr^p)$. These properties will be used without further
mention.

\subsubsection{S-smoothness.}\label{subsec: intro: S-smoothness}
 Before proceeding we want to cover what is needed with respect to $SC^\infty$ maps defined in standard neighborhoods of $0$ in
$\rz\! \bbr^n$. In order to define this external set of nice internal functions, we first need some notation.
 If $k\in{}^\s\bbn$, then a weight $k$, $m$-multiindex is an ordered $m$-tuple
$\a=(\a_1,\dots,\a_m)$ such that $\a_i\ge0$ are integers for all $i$
and $|\a|\doteq\a_1+\dots+\a_m=k$.
If $U$ is a neighborhood (eg., bounded) of $0$ in $\bbr^m$ and if $\Ff\in\rz\! C^\infty(U,\rz\! \bbr^m)$,
then the $\a^{\supth}$ {\bf internal derivative of} $\Ff$ at $x\in\rz\! U$,
$(\rz\! \p^\a)(\Ff)(x)$ is well defined in $\rz\! \bbr^m$. Similarly, by transfer, if $p\in\bbn$ we say that $\Ff:\rz U\ra\rz\bbr^m$ is $\rz C^p$, written $\Ff\in\rz C^p(U,\bbr^m)$ if for all multiindices $\a$ with $|\a|\leq p$, we have that $\rz\p^\a\Ff:\rz U\ra\rz\bbr^m$ exists and is *continuous.
 With this we have the following definition.
\begin{definition}\label{def: SC^k fcns, k in N and infty}
 Suppose that  we have an internal map $\Ff\in\rz C^\infty(U,\bbr^m)$. Then we
say that $\Ff\in SC^\infty(\rz\! U,\bbr^m)$ if for each $(k,m)$-multi
index (weight $k$, $m$-multi index), the map $(\rz\! \p^\a)\Ff:\rz U\lra\rz\! \bbr^n$
is $SC^0$.
For $p\in\bbn$, suppose that  we have that $\Ff\in\rz C^p(U,\bbr^m)$. Then we say that $\Ff\in SC^p(U,\bbr^m)$ if for all multiindices $\a$ with $|\a|\leq p$, we have that $\rz\p^\a\Ff:\rz U\ra\rz\bbr^m$ is $SC^0$.
 (We include here the empty multi index; i.e., that
$\Ff:U\lra\rz\! \bbr^n$ is $SC^0$, e.g., the image of $\Ff$ is nearstandard).
\end{definition}

  We need to list some basic facts about S-smooth functions.
\begin{lem}[S-smooth facts]\label{lem: S-smooth facts}
One can then show that if $\Ff\in SC^\infty(\rz\! U,\bbr^m)$, the following is true
\begin{enumerate}
\item $^\circ \Ff\in C^\infty(U,\bbr^m)$
\item For every $(k,m)$-multi index $\a$, $^\circ((\rz\! \p^\a)\Ff)\in C^\infty(U,\bbr)$
\item For every $(k,m)$-multi index $\a$, as maps, $\p^\a({}^\circ \Ff)={}^\circ((\rz\! \p^\a)\Ff)$
\item If $f\in C^\infty(U,\bbr^m)$, then $\rz\! f\in SC^\infty(\rz U,\rz \bbr)$
\item If\; $\Ff\in SC^\infty(\rz U,\rz V)$ and $\Fg\in SC^\infty(\rz V,\rz \bbr^p)$, then $\Fg\circ \Ff\in SC^\infty(\rz U,\rz \bbr^p)$
\end{enumerate}
\end{lem}
\begin{proof}
These facts follow from the literature and the work in the appendix. See Stroyan and Luxemburg \cite{StrLux76} p.
96--109, 269--270, for the best background nonstandard material on this topic . Their text is encyclopedic and faithful to Robinson,
but has many misprints, so care must be taken in reading.
For a careful proof of (3) and associated equivalences, from which the others essentially follow, see this paper's appendix, chapter \ref{chap: appendix: S-smoothness}.
\end{proof}

\subsubsection{S-analyticity.}\label{subsec: S-analyticity}

We also need to talk about $S$-analytic maps on $\rz\! U$. We will draw our material for (standard) analytic functions from the text of Krantz and Parks, \cite{Krantz1992AnalyticFcns}.
  If $U$ is a contractible neighborhood of $0$ in $\bbr^m$, recall that $f:U\ra\bbr$ is analytic if each $a\in U$ has a neighborhood in which $f$ can be written as a convergent power series, and we will denote the set of such by $A^\om(U)$. An $\bbr^m$-valued map $f=(f_1,f_2,\ldots,f_m)$ on $U$ is analytic on $U$ if all of its components $f_j$ are in $A^\om(U)$.  Transferring , we say that an \textbf{internal $\bsm{\Ff:\rz U\ra\rz\bbr}$ is *analytic, written $\bsm{\Ff\in\rz A^\om(\rz\bbr)}$} if each $\xi\in\rz U$ has a *neighborhood where it can be written as a *convergent internal power series. These can be quite pathological in general; for example for $\a\in\rz\bbn$, the map  $\xi\mapsto\rz\cos(\a\xi)\in\rz A^\om(\rz\bbr)$ and has a well defined standard part, but for many infinite such $\a$'s, this standard part is not even Lebesgue measurable, see Stroyan and Luxemburg, \cite{StrLux76} p.218. Nonetheless, there are useful conditions that can force good regularity behavior on these. One type of restriction is in the subset of S-analytic internal functions defined below. In the definition, we will use the Taylor expansion for the power series representation for analytic functions. (Later we will give a second useful criterion for forcing regularity on elements of $\rz A^\om$.)
\begin{definition}\label{def: S-analytic}
Suppose that $f\in SC^\infty(\rz U,\rz \bbr^n)$. For $k\in\rz \bbn$, let $\rz T^kf:\rz U\ra\rz \bbr^m$ be the $k^{\supth}$ order $\rz\! $Taylor polynomial of $f$ centered at $0$ and let $\rz R^kf:\rz U\lra\rz \bbr^m$ be the $k^{\supth}$ order remainder term $f-T^kf$. Then we say that  {\bf$f$ is $S$-analytic on $U$} denoted {\bf$f\in SA^\om(\rz U,\rz \bbr^n)$}, if the following holds: a)  for every $ k\in\; ^\s\bbn$, $\rz R^kf$ and $\rz T^kf\in SC^\infty(\rz U,\rz \bbr^n)$ and b) for all $ k>\infty$  and finite $\a$,  and for all $ x\in \rz U$, $(\rz \p^\a)(\rz R^kf)(x)\sim0$.
\end{definition}
\begin{remark}
   Note that the above definition is implied  if the following holds: $a'$) for every $ k\in\rz \bbn$, $\rz T^kf\in SC^\infty(\rz U,\rz \bbr^n)$ and $b'$) for all $ k>\infty$  and finite $\a$,  and for all $ x\in \rz U$, $(\rz \p^\a)(\rz R^kf)(x)\sim0$
\end{remark}
The following lemma contains information useful to the argument of this paper.
\begin{lem}[$S$-analytic maps]\label{lem: S-analytic maps} The follows assertions hold.
\begin{enumerate}
\item[A)] If $f\in SA^\om(\rz U)$, then $^\circ f\in A^\om(U)$.
\item[B)] If $f\in$SPoly$(\rz U,\rz \bbr)$, then $f\in SA^\om(\rz U,\rz \bbr)$.
\item[C)] If $f\in$SPoly$(\rz U,\rz \bbr)$ and $g\in{}^\s A^\om(\rz V,\rz \bbr^p)$ with $f(\rz U)\subset \rz V)$, then the composition map $g\circ f\in dSA^\om\cap SC^\infty(\rz U,\rz \bbr^p)$.
\end{enumerate}
\end{lem}

\begin{remark} As they will not be needed for this paper we will not prove the following facts; $^\s A^\om\subset SA^\om$ and $SA^\om$ is closed under composition. Yet the conditions on $SA^\om$ are strong. If $f\in SC^\infty$ is such that $^\circ(\p^\a\rz\! f)=0$ for all $\rz\! $multiindices $\a$, $f$ is still typically not in $SA^\om$. The tail end Taylor conditions are quite stringent.
\end{remark}

Proof: A) We must show that given a finite multi index $\b$ then the
following is true. If $x\in U$ and $\e_0$ is positive in $\bbr$, then there exists $k_0\in\bbn$ such that $|\p^\b(R^k({}^\circ f)(x)|<\e_0$ if $k>k_0$. Let
$A=\{k\in\rz \bbn:|\p^\b(\rz\! R^kf)(\rz\! x)|<\frac{ *\e_0}{2}\}$. From b) in the
definition of $SA^\om$, it follows that $\rz \bbn_\infty\subset A$. By
definition $A$ is internal and therefore by overflow there is
$k_0\in{}^\s\bbn $ such that $[k_0,\rz\! \infty)\subset A$, eg., if
$k\in\bbn$ and $k\ge k_0$, then $|\rz\! \p^\b(\rz\! R^kf)(\rz\! x)|<
\frac{\e_0}{2}$. As $\e_0$ is standard this implies
$^\circ(\rz\! \p^\b(\rz\! R^kf)(\rz\! x)|)<\e_0$.
But by definition for $k\in\bbn$
$\rz\! Rf^k\in SC^\infty$ and this implies
$^\circ(\rz\! \p^\b(\rz\! R^kf)(\rz\! x)=\p^\b(^\circ(\rz\! R^kf))(x)$ by the third
property of $SC^\infty$ maps. So we have that
$$
|\p^\b(^\circ(\rz\! R^kf))(x)| < \e_0\quad\text{for }\ k>k_0.
$$
But, if $k\in\bbn$, $^\circ(\rz\! R^kf)={}^\circ(f-\rz\! T^kf)={}^\circ
f-T^k({}^\circ f)\doteq R^k({}^\circ f)$, since
$^\circ(\rz\! T^kf)=T^k(^\circ f)$ because for $f\in SC^\infty$,
$^\circ(\rz\! \p^\a(f)=\p^\a(^\circ f)$ for finite $\a$, again by
property 3). That is, $|\p^\a(R^k(^\circ f))(x)|<\e_0$ as we needed
to show.

B) As $f$ is a nearstandard polynomial, we have that $\rz\! T^kf=f$ or a
truncation of $f$ depending on the value of $k\in\rz \bbn$. Hence as a
nearstandard polynomial is $SC^\infty$, then condition a) is
satisfied. For $k\in\rz\bbn$ an infinite integer, $\rz\! R^kf=0$ as $f$ is of finite degree, so
certainly $\rz\! \p^\b(\rz\! R^kf)=0$, so condition b) is clear.

C) First of all, its clear that elements of SPoly or of $^\s A^\om$
are $SC^\infty$. But then by the properties of $SC^\infty$, $g\circ
f\in SC^\infty$. But then we know that $^\circ(g\circ f)=(^\circ
g)\circ(^\circ f)$ and as composition of analytic maps are analytic
we have that $g\circ f\in dSA^\om$.

\subsection{Properties of *$Gl_n$, *End$_n$, *Exp}\label{sec: props *Gl_n,*End_n,*Exp}

\subsubsection{Nonstandard definitions}\label{subsec: def *End_n, *Gl_n}
For the next section, see the lucid introduction to Lie groups in Warner's text \cite{Warner1971}.
 Here we need to talk about the internal topology of $\rz \en_n=\rz \en(\bbr^n)$, the $\rz\! $ transfer
  of $\en_n\doteq\{A:\bbr^n\to\bbr^n:A$ is $\bbr$-linear\}, and $\rz Gl_n=\rz Gl(\bbr^n)$ the  $\rz\! $ transfer
 of $Gl_n\doteq\{A\in\en_n:A$ is invertible\}. First note that $\rz \en_{n,\nes}\doteq\{A\in\rz \en_n:A$ is a nearstandard linear map\} and $\rz Gl_{n,\nes}=\rz Gl_n\cap\rz \en_{n,\nes}$. This all follows from the topological vector space identification of $\en_n$ with $\bbr^{n^2}$ by the identification of a linear map with its corresponding matrix with respect to the canonical basis of $\bbr^n$, and therefore of $Gl_n$ with the corresponding open dense subset of  $\bbr^{n^2}$. It is not hard to see an $A\in\rz \en_n$ is actually in $\rz \en_{n,\nes}$ if $A:\rz \bbr^n_{\nes}\to\rz \bbr^n_{\nes}$ and $A\in\rz Gl_{n,\nes}$ if $A\in\rz \en_{n,\nes}$ and $\rz \det A\neq0$, $\rz \det A$ being the determinant of $A$ with respect to a standard basis on $\rz \bbr$. Note, if $\rz \det A\sim0$, then $^\circ A$ is not in $Gl_n$.

We have the composition map $C:\en_n\x\en_n\to\en_n$. As this is a
polynomial map, it is analytic and therefore $\rz C$ is $S$-analytic on
the nearstandard parts of $\rz \en_n\x\rz \en_n$, i.e., on
$\rz \en_{n,\nes}\x\rz \en_{n,\nes}$. In particular, if
$A\in\rz \en_{n,\nes}$, then $L_A=$ left multiplications by
$A:\rz \en_n\to\rz\! \en_n$, $L_A(B)=A\circ B$ (composition) is $S$-analytic
on $\rz \en_{n,\nes}$. For the same reason, $R_A:\rz \en_n\to\rz \en_n$,
$R_A(B)=B\circ A$ is $S$-analytic when restricted to $\rz \en_
{n,\nes}$.

\subsubsection{Properties of the nonstandard EXP map}\label{subsec: props of *Exp}
 Finally we want to look at $\expp:\en_n\to Gl_n$, $\expp(A)\doteq\sum^\infty_{J=0}\f1{J!}A^J$. $\expp$ is an analytic diffeomorphism from a neighborhood $U$ of $0$ to a neighborhood $V$ of Id, the identity map, in $Gl_n$. This follows from e.g., that $d(\expp)_0=1_{\en_n}$, e.g., a linear isomorphism. Here we are using the canonical identification of $T_0\en_n$ with $\en_n$ --- see the discussion earlier in this section. $Gl_n$ is our protypical Lie group with $\en_n$ the Lie algebra of $Gl_n$, denoted $LA(Gl_n)$, this Lie algebra having product given by the Lie bracket
  \begin{align}
    [\; , \;]:\en_n\x\en_n\to\en_n:(A,B)\to A\circ B-B\circ A
  \end{align}
   (the commutator). Here $\expp$ is the corresponding exponential map tying the Lie algebra to its (local) Lie group. At this point all we need to know is that $\rz\! \expp:\rz\! \en_n\to Gl_n$ is an $dS$-analytic $S$-diffeomorphism. (Note here that this means in particular (and as alluded to earlier) that $^\circ(\rz\! \expp)$ is a diffeomorphism (on a neighborhood of $0$). More specific facts needed for $\rz\! $Lie groups and $\rz\! $Lie algebras will be developed in chapters \ref{chap: pf that ad is sc0}, \ref{chap: st(exp) is loc homeo} and \ref{chap: Main NS reg thm and stan version} as they are needed.

\subsection{Why NSA? The perspective behind this paper}
   In this chapter, we have attempted to give an overview of a typical construction of a nonstandard setting, of the general working principles as well as those specific tools. We have even attempted to give some idea why one might think that it's capacities are worth the trouble to develop some skill in this curious discipline. But we have not discussed  why the author was motivated by this nonstandard approach to this problem. It seems that this belief in the capacities of nonstandard mathematics for this problem comes from the author's view of it's capacities in the study of asymptotic behavior of (say parametric) families of geometric structures. In particular, those structures that at each (parametric) instant subtly integrate   multileveled objects. For example, in considering families of Lie groups, each element of the family has a Lie algebra, an exponential map relating these, various canonical `representations' ($Ad,ad$, etc.) and so on.
   How does one grasp the sophisticated behaviors of these asymptotics. Recall in our introduction to NSA, we described how infinitesimals simplify the criterion for continuity. {\it In a way, this example sells NSA short while it is trying to make it understandable.} The example with the simplification of the formal description of continuity once one enriches the real line (and all of those objects defined on it) is a clarifying but undramatic example of the possibilities of such `enrichment'. In this example, we have pulled a typical function up to a realm where its asymptotic properties are not only revealed; but to view, ascertain and categorize these behaviors, we have the full structured framework of the real numbers, the algebra of functions defined there, etc., all nicely lifted to this enhanced arena. And this enhancement becomes ever more dramatic the more sophisticated the framework we are `transferring'. Referring again to the above example of families of Lie groups (the subject of this paper), the asymptotics are fully revealed by examining the nonstandard Lie groups in the transferred family and as noted above (and played out in chapters \ref{chap: pf that ad is sc0} to \ref{chap: Main NS reg thm and stan version}) we can bring into play those elements of the structure theory (now transferred) to find regularities in these nonstandard objects and hence (back in the standard world) in the families themselves.

   On the level of model theory, Henson and Keisler, \cite{HensonKeisler1986}, have written on this phenomenon. In the context here, what they say is the following. A standard approach to (the standard version of) the main theorem here would mean building families of `good' coordinate changes for our equicontinuous family of local Lie groups carefully orchestrated, via an argument that must use, eg., asymptotic families of commutative diagrams from Lie theory, to have `stringently good' asymptotic properties. This is a tall order, if it can be done. But on the nonstandard level, we essentially deal with fixed individual ideal asymptotic elements of these families and bring to bear all of the framework of Lie theory (now at this ideal/nonstandard level) in the task of finding one good set of coordinates for this solitary ideal group, a feasible project.

\section{The local group setup}\label{chap: intro: loc gp setup}

\subsection{Local topological groups}\label{sec: loc top gps}

\subsubsection{Foundational material}\label{subsec: intro to LTG background}
For {\bf local topological groups}, {\bf LTG}'s, in general, we follow Montgomery and Zippin \cite{MontZip1955} (eg.,p. 31--35) and Pontryagin \cite{Pontryagin1986} (eg., p. 137--143). Olver, \cite{Olver1996},  is a good reference for the varieties of local Lie groups and how they relate to (global) Lie groups. On locally Euclidean LTG's, we follow Kap1ansky \cite{Kaplansky1971} (eg.,p. 87). As such after a homeomorphic change of coordinates in some neighborhood of the identity, our group will be modeled on $\bbr^n$ with the identity being the point $0$.  The Euclidean setting will also facilitate the development of the crucial notions of infinitesimal and nearstandard (see the previous section).

 Heuristically, we will have a nonstandard internal local Lie group modeled
 on a standard neighborhood of $0$ in $\rz \bbr^n$. Its standard part will be our continuous
  locally Euclidean local topological group. First we give the following definition.
\begin{definition}\label{def: loc Euclid top gp}
 Our Euclidean local topological group,LTG, will be given by a quadruple $G=(U,\psi,\nu,0)$ where $U$
  is a neighborhood of $0$ in $\bbr^n$, $\psi\in C^0(U\x U\to\bbr^n)$,
    $\nu\in C^0(U\to\bbr^n)$ such that $\forall x$, $y$, $z$ where all expressions are defined
\begin{enumerate}
\item[a)] $\psi(x,\psi(y,z))=\psi(\psi(x,y),z)$\quad associativity,
\item[b)] $\psi(0,x)=\psi(x,0)=x$ \qquad\ \ \ \ \  $0$ is the identity,
\item[c)] $\psi(\nu(x),x)=\psi(x,\nu(x))=0$\quad\ $\nu(x)$ is the inverse of $x$.
\end{enumerate}
\end{definition}
Note that we may also write $x\circ y$ or $xy$ for $\psi(x,y)$,
$x(yz)$ for $\psi(x,\psi(y,z))$, $x^{-1}$ for $\nu(x)$, etc. As the
identity is at the origin we will sometimes write $0$ instead of $e$
when it seems appropriate.
\subsubsection{Equivalent local Lie groups}\label{subsec: equivalence of loc top gps}
 In the previous paragraph we used the qualifying phrase ``where all expressions are defined.'' As a local group is defined on neighborhoods of the identity giving various representatives, the mapping defining the product of elements or the inverse of elements may need restrictions to be well defined. For example, if $U$ is a neighborhood of the identity, then the map $\nu:x\to x^{-1}$ may have $\nu(U)\not\sbq U$. In this case, we choose a symmetric neighborhood of $0$; say $V=U\cap U^{-1}$ which is actually a neighborhood of $e$ (see \cite{MontZip1955} pp.32, 33 and \cite{Pontryagin1986}  pp. 137--143). So, e.g., to deal with these problems, we restrict repeatedly to different neighborhoods of $e$. In a strong sense (\cite{Pontryagin1986}, p. 138--141), the LTG's defined by these various representative neighborhoods of $e$ are isomorphic. (That is, there is a smaller neighborhood of $0$ where the two local groups coincide. Note that Duistermaat and Kolk, \cite{DuistermaatKolk1999} p.31 give a brief and clear definition of germ equivalence of local Lie groups. If one changes `local Lie' to `locally Euclidean topological' we get the modern rendition. Their definition looks different, but is the same once one realizes that here we are carrying the various change of coordinates representatives along in our calculations.) Critical here is that we clutter an argument with a finite sequence of qualifying restrictions, each time restricting to another isomorphic local topological group, LTG. As long as there is only a finite number of restrictions, the final restriction will be equivalent to the original. Pontryagin argues in his book \cite{Pontryagin1986} p. 431, that for clarity's sake one can forego these qualifications in such arguments. In this paper, there is no infinite sequence of restrictions. (We do work with families of LTG's but in this case the domains of definition is uniformly fixed.) On the Lie algebra level, such exists in the Hausdorff series work, but the Lie algebra is a globally defined object. Hence, with the exceptions of the $\mu$-exp lemma, and the S-lemma, We will generally not use the qualifications for such restrictions. (These restrictions may affect the globalizability properties of $\SG$, see Olver's paper, \cite{Olver1996}, but do not affect the proof.)

\subsubsection{Standardly local internal Lie groups}
For this paper, a LTG as defined above will be the standard part of a  ${}^\s$local $\rz$Lie group (${}^\s\loc\rz LG$) $\SG$ defined on $\rz \bbr^n$. So we need to give this definition.

\begin{definition}\label{def: sigma local *Lie group}
 A ${}^\s$local $\rz$Lie group (${}^\s\loc\rz LG$) will be defined in terms of representatives on neighborhoods $U$ of $0$ in $\bbr^n$ as follows. By definition $\SG=(\rz U,\tl\psi,\tl\nu,0)$ where a representative is defined on a standard neighborhood $\rz U$ of $0$ in $\rz\bbr^n$ so that $\tl\psi\in\rz C^\infty(U\x U,\bbr^n)$ and $\tl\nu\in\rz C^\infty(U,\bbr^n)$ satisfy the conditions  a)--c) above. Furthermore, in order to insure that $\psi={}^\s\tl\psi$ and $\nu={}^\s\tl\nu$ are $C^0$ and satisfy a)--c), we must impose $S$-continuity on $\tl\psi$. Also, in order to insure that the standard part of $\tl\psi$ has the the properties of a local topological group (see Montgomery and Zippin, \cite{MontZip1955}, p 32 ) in the topological sense, we also impose on $\tl\psi$ the condition that right and left $\rz\! $multiplications are $S$-homeomorphisms where defined. We also assume that right and left $\rz\! $ multiplications are $\rz$ local diffeomorphisms,  a quite weak assumption (see Olver \cite{Olver1996}). We write  ${\bf\SG\in{}^\s\loc SC^0\rz LG}$ to denote a ${}^\s\loc\rz\! LG$ with these properties.
 \end{definition}

{\bf From this point until the finish of the proof of the main regularity theorem, this paper will be concerned with the
properties of a fixed $\SG\in{}^\s\loc SC^0\rz LG$ on a representative
neighborhood of the identity.} We need some observations. As our
${}^\s\loc\rz LG$ is internal, it will carry all of the properties and
structures of a standard local LG, but now  *transferred to the
nonstandard universe. Therefore, note it has an internal Lie algebra
with a $\rz\bbr$ bilinear, antisymmetric bracket satisfying the Jacobi
identity, an internal exponential map, etc.

\subsubsection{Structure of the  *Lie algebra of the local *Lie group}\label{subsec: intro: *Lie alg structure}

Before we proceed with the first lemma, we need to say something about the {\bf internal Lie algebra (*LA) $\Fg$ of $\SG$}. First of all, as a $\rz\bbr$ vector space of $\rz\bbr$ dimension $n$, it is the $\rz\! $tangent space of $\SG$ at $e$. But as a $\rz\! $topological space, $\SG$ is a standard neighborhood of $0$ in $\rz\bbr^n$. {\bf So there is a {\bf standard} identification of $\Fg$ as a *vector space with $\rz\bbr^n$.} (See the local differential calculus section.) Therefore, the following are well defined.
\begin{definition}
  The $\rz\bbr_{nes}$-submodules of $\Fg$ of nearstandard and infinitesimal vectors are defined respectively by
\begin{align}
  \Fg_{\nes}=\{v\in\Fg:|v|\in\rz\bbr_{nes}\}\\
  \mu_{\Fg}=\mu_{\Fg}(0)=\{v\in\Fg:|v|\sim 0\}
\end{align}
\end{definition}
 Let $[\ ,\ ]:\Fg\x\Fg\to\Fg$ be the *Lie algebra product of $\Fg$. By *transfer,  $[\ ,\ ]$ is a $\rz\bbr$-bilinear map satisfying anticommutativity and the Jacobi identity. As $\Fg_{\nes}$ is well defined, it makes sense to ask if $[\ ,\ ]$ is a {\bf nearstandard map}, that is if $[\ ,\ ]:\Fg_{\nes}\x\Fg_{\nes}\to\Fg_{\nes}$. We will prove this in the next section; it {\bf will be our pivotal regularity result}.

 Also by *transfer, there is the $\rz C^\infty$ map $\exp:(\Fg,0)\to(\SG,e)$
satisfying $\rz\! d(\exp)_0=1_{\rz\! \bbr^n}$ and therefore is a *local
diffeomorphism; see the previous section. Note here that $\Fg$ has a
standard group structure, $\Fg_{gp}$, given by vector space
addition. In the next part we will use that in the one dimensional
case, $\exp$ is a group isomorphism.

\subsection{The $\mu$-exp lemma}\label{sec: mu-exp lemma}

 The proof of the $\mu$-exp lemma, that the exponential map sends infinitesimal vectors
 to infinitesimal group elements and no other vectors close to $0$,
 depends on the one dimensional case and the Restriction lemma. We
 begin with the Restriction lemma.

\begin{lem}[Restriction Lemma]\label{lem: Restriction lem}
\it Suppose that we have $\SG\in{}^\s\loc SC^{0}\rz LG$ and $\Fg=\rz LA(\SG)$. Suppose that $\Fh<\Fg$ is a *subalgebra. Let $U$ a neighborhood of $0$, so that $\rz U$ is a set of definition of $\SG$, $\SH'=\exp(\Fh)$ where defined and $\SH=\SH'\cap U$. Then $\SH$ is a set of definitions of the $^\s\loc$\rz\! Lie subgroup of $\SG$
given by $\exp(\Fh)$ and $\SH$ is $SC^{0}$ with the restriction topology.
\end{lem}

\begin{proof}
 This means of definition for subgroups of local group is well known (see Montgomery and Zippin, \cite{MontZip1955} pp.33-34, and Kirillov, \cite{Kirillov1976} p.99). We need to prove that $\SH$ is an ${}^\s\loc SC^{0}$ *subgroup. But if we can show that $h_1\sim h_2$ in $\SH\Rightarrow h_1\sim h_2$  in $\SG$, then the result will follow from $\SG\in SC^{0}$. Yet for the restriction topology $\mu^{\SH}(h)=\SH\cap\mu(h)$, where $\mu^{\SH}(h)$ is the monad of $h\in\SH$ in the restriction topology. But $h_1\overset{\SH}{\sim}h_2$ if and only if $\mu^{\SH}(h_1)\cap\mu^{\SH}(h_2)\not=\emptyset$; i.e., $\mu(h_1)\cap\SH\cap\mu(h_2)\not=\emptyset$ eg., $\mu(h_1)\cap\mu(h_2)\not=\emptyset$, i.e., $h_1\sim h_2$ in $\SG$, as we wanted to show.
\end{proof}

\subsubsection{One dimensional case}\label{subsec: one dim'l case}
  In this part, we will be working with one dimensional local LG's,
 $\mathbf{LG^1}$, and the $^\s$local restrictions of their internal transfers, $\rz LG^1$.
  All will be modeled on $\rz \bbr,0$. Hence if $\SG$ is one such, then
   $\wh\SG\doteq\SG_{\nes}\smallsetminus\mu(0)$, the set of nearstandard noninfinitesimal elements is well defined. In particular, if $(\rz\bbr,\cd)$ denotes a model for a 1-dimensional internal  $^\s$local $\rz LG$, then $(\wh\bbr,\cd)$ is well defined, eg., if $(\rz\bbr,+)_{\stan}$ is the standard (local) LG structure on $\rz\bbr$ in the internal universe, then $\wh\bbr$ will be well defined. If $\SG_1$, $\SG_2$ are in $\rz LG^1$ and $\vp:\SG_1\to\SG_2$ is a $\rz LG$ isomorphism which is also a (local) $S$-homeomorphism, we say that $\vp$ is an {\bf $S$-equivalence}.

We will prove after two preliminary lemmas that if $\SG\in\rz LG^1$ is
$SC^0$ and $\La$ is its *exponential map, then $\La$ is an
$S$-equivalence. We begin with the following

\begin{lem}\label{lem: 1-dim-a}
 Suppose that $\G:(\rz\bbr,+)_{\stan}\to(\rz\bbr,+)_{\stan}$ is a $\rz LG^1$ isomorphism.
 Then $\G$ is an $S$-equivalence if and only if  $ d\G_0\in\wh\bbr$.
\end{lem}

\begin{proof}
 Any $LG^1$ isomorphism $\G:(\bbr,+)\to(\bbr,+)$ is given by $t\mapsto\bar kt$ for some $\bar k\neq0$ in $\bbr$ as it is covered on the Lie algebra level by an $\bbr$-linear map. Therefore by *transfer an $\rz LG^1$ isomorphism $\G:(\rz\bbr,+)_{\stan}\to(\rz \bbr,+)_{\stan}$ is given by $t\mapsto\bar kt$ for some $\bar k\in\rz \bbr\smallsetminus\{0\}$. But the magnification map $t\mapsto\bar kt$ is an $S$-equivalence $\dllra\bar k\in\wh\bbr$ and as $d\G_0=\bar k$, we have that $\G$ is an $S$-equivalence $\dllra d\G_0\in\wh\bbr$.
\end{proof}

\begin{lem}\label{lem: 1-dim-b}
 Suppose that $(\rz\bbr,\cd)\in{}^\s\loc$SC$^\circ\rz LG^1$
such that $\vp:(\rz \bbr,\cd)\to(\rz \bbr,+)_{\stan}$ is an $S$-equivalence.
Then $\rz d\vp_0$ is nearstandard which implies that  there is $ c\in \wh \bbr$ such
that $\rz d\vp_0$ is dilation by $c$.
\end{lem}

\begin{proof} $(\rz \bbr,\cd)$ is ${}^\s\loc$SC$^\circ$ implies that
$^\circ(\rz\bbr,\cd)$ is a $C^o$ local topological group modeled in a
neighborhood of $0$ in $\bbr$. Also $\vp$ is an $S$-equivalence
implies that $^\circ \vp$ is a local group isomorphism that is a
local homeomorphism and therefore a homeomorphism on the domain of
definition of $^\circ(\rz\bbr,\cd)$. So $^\circ \vp$ is a homeomorphic
 identification of the local group $^\circ(\rz\bbr,\cd)$ with a
 Euclidean neighborhood of zero of the standard group structure on
 $\bbr$, ie. $(\bbr,+)_{\stan}$. By a very special case of a
 theorem in Lie group theory, (see [Warner], p. 95, Theorem 3.20) this embedded
 $^\circ(\rz\bbr,\cd)$ has a smooth group structure compatible with $(\bbr,+)_{\stan}$. So now we have that $^\circ
 \vp :{}^\circ \: (\rz\bbr,\cd)\to (\bbr,+)_{\stan}$ is a $C^o$
 group isomorphism  of (local) Lie groups. Therefore, we can invoke, a
 special case of another result in Lie groups, (see [Warner], p.109, Theorem 3.39) to assert
 that $^\circ \vp$ is, in fact $C^{\infty}$. (Note that Warner proves
 both of these results locally, in some neighborhood of the
 identity, then moves them out globally to the full group, but here we are
 using only the local conclusions of of these results.) But then $\rz d\vp_0$
 is a nearstandard *linear map and the conclusion follows from the previous lemma.
\end{proof}

\begin{lem}[One dimensional lemma]\label{lem: 1-dim-last}
 Suppose that $\SG\in{}^\s\loc$SC$^\circ\rz LG^1$, that $\Fg=\rz LA(\SG)$
and $\La:\Fg\to\SG$ is the *exponential map. Then for $v\in\Fg$
small enough,
 $\nu\sim0$ if and only if $\La(\nu)\sim0$.
\end{lem}

\begin{proof} By hypothesis, $(\SG,\cd)$ is $^\s$locally $S$-equivalent to $(\rz\bbr,+)_{\stan}$,
 i.e., on some standard neighborhood of the identity, there exists an $S$-equivalence
  $\vp:(\SG,\cd)\to(\rz\bbr,+)_{\stan}$. Letting  $(\bbr,\cd)$ stand for a locally Euclidean
   model of an object in $LG^1$, and letting
   $\vp:(\rz\bbr,\cd)\to(\rz\bbr,+)_{\stan}$ be an $S$- equivalence given
   by the hypothesis, we see that the previous lemma implies that $\rz d\vp_0$ is  dilation by an element of $\wh{\bbr}$.
    Now $\G\doteq\vp\circ\La:(\rz\bbr,+)_{\stan}\to(\rz\bbr,+)_{\stan}$ is a $\rz LG^1$  isomorphism
     and $d\G_0=d\vp_0$, as $d\La_0=$Id; i.e., $d\G_0\in\wh\bbr$. But then Lemma \ref{lem: 1-dim-a}
      implies that $\G$ is an $S$-equivalence. So $\La=\vp^{-1}\circ\G$ must also be an $S$-equivalence,
       e.g., a $^\s$local $S$-homeomorphism. Therefore, for $\nu\in\Fg$ small enough
        $v\sim0\Longleftrightarrow\La(v)\sim0$.
\end{proof}
\subsubsection{Exp preserves infinitesimals}\label{subsec: Exp preserves infinitesimals}
  We will now use the above one dimensional results to give a proof of the first main result toward our goal. Let $\SG\in{}^\s\loc SC^0\rz LG$, modeled on $\bbr^n$ and $\Fg$ be it's *Lie algebra.
\begin{lem}[$\mu$-exp lemma]\label{lem: mu-exp lem}
 Suppose that $v$ is in a small enough standard neighborhood of $0$
in $\Fg$. Then $\exp:\Fg\to\SG$ satisfies $\exp(v)\sim e\dllra
v\sim0$.
\end{lem}
\begin{remark}
Before we prove this we will need to point out the functorality of
$\exp$ from the category of Lie algebras to that of $(\loc)$ Lie
groups. Let LGp be the category of $\loc$LG's and morphisms and LA
the category of Lie Algebras and morphisms. Then $\exp$ is a
functor: LA$\to$LGp; see Kirillov, \cite{Kirillov1976} p.103, or Warner, \cite{Warner1971} p.104]. In particular,
considering inclusion morphisms, we find that {\bf if for a given
Lie algebra $L$, $\exp^L$ denotes its exponential map and $L_2$ is a
Lie subalgebra of $L_1$, then $\exp^{L_2}=\exp^{L_1}\bigm|_{L_2}$.}
With this, we turn to the proof.
\end{remark}

\begin{proof}[$\mu$-exp lemma proof] First note that if $\rz\! \bbr$-dimension of $\Fg$ is $1$ and $\SG$ is
$SC^0$, then this is the 1-D lemma.

Next consider the general case when $\SG\subset\rz\bbr^n,0$. Let
$v\in \Fg_{\nes}$, $L^v\doteq\rz\bbr\cd v$, $\ell^v=\exp(L^v)$, where
defined and $\exp^v=\exp|L^v$, so that $\exp^v:L^v\to\ell^v$ is the
exponential map of a one dimensional *Lie algebra. As $\Fg$ is a
standard $\rz\bbr^n$, for some standard $n$, and $L^\nu$ is a
*subspace, if $w\sim0$ in $\Fg$, then $w\sim0\in L^v$ (and
conversely). But if $w\in L^v$, then $w\sim0\dllra\exp^v(w)\sim0$ in
$\ell^v$, by the 1-D lemma. By the restriction lemma,
$\mu_0(\ell^v)=\mu_0(\SG)\cap\ell^v$ which implies $\exp^v(v)\sim0$
in $\ell^v\dllra\exp(v)\sim0$ in $\SG$, as we wanted.
\end{proof}

\subsubsection{Perspective}\label{subsec: perspective: mu-exp lem}
 Two comments are in order. First, although the $\exp$ map is usually defined in terms of the one parameter subgroups, 1PSGps; and approaches to the Fifth problem, both standard and nonstandard, have been through 1PSGps, we avoid them. Dealing with the relationship between the $SC^0$ properties of these $\rz\! C^\infty$ maps and the putative $SC^1$ regularity of these seemed to be more work than exploiting the functorality of $\exp$.

Second, the local one dimensional case of the Fifth problem goes at
least back to the 19$^{\supth}$ century. After all, this problem was
first conceived of locally. Nonetheless, we felt uncomfortable
invoking the  local one dimensional case without proof. But as it
was not explicitly an integral part of the struggle to finish the
general proof in the middle of the 20$^{\supth}$ century, we have
left our proof in the preliminaries section.

\section{Proof that ad is S-continuous}\label{chap: pf that ad is sc0}

In this chapter we prove that $\ad$ is
nearstandard. The ingredients of the proof consists of the *transfers of two basic formulas (see Warner, \cite{Warner1971} p 114) that intertwine homomorphisms on the Lie group level with the lifted homomorphism on the Lie group level,  along with the $Ad$ Lemma. Note a technicality here: as our group object $\SG$ is defined locally only,  then we will routinely need to restrict to a smaller neighborhood of $0$, to get well defined expressions, in particular with the mappings $a_g$ and $Ad_g$, we shall shortly use. (Of course this caution applies similarly to our S-continuous *groups.) On the Lie algebra level, this is not a problem.

\subsection{Ad is S-continuous}\label{sec: Ad is SC0}

The first intertwining formula is
crucial in the proof of the $Ad$ Lemma. The second intertwining formula is the bridge to
reducing the problem to a question about singularities of one
parameter subgroups of $\rz Gl_n$. This second part depends on the the fact that if $\SA\in\rz Gl_{nes}$, then the internal differential of left multiplication by $\SA$ is also nearstandard.

The first formula, given directly below, is critical to the proof of the ad Lemma.
 Let $G\in\loc LG$ and $L=LA(G)$. Let $g\in G$, $v\in L$. Then $a_g:h\mapsto ghg^{-1}$ is an
  automorphism of $G$ (when restricted to a sufficiently small neighborhood of $0$, the identity) and as $d(a_g)_e:T_eG\ra T_eG$, in fact is an automorphism of $T_eG=L$, then, in particular, $Ad_g\dot=d(a_g)|T_eG\in Gl(L)$. This gives a smooth Lie group homomorphism $Ad:G\ra Gl(L)$ whose differential induces on the Lie algebra level the homomorphism  of Lie algebras $v\in L\mapsto ad_v\in\en(L)$ ( the homomorphism property is essentially the Jacobi identity) and it's not hard to prove that $\ad_v(w)=[v,w]$ .
 Given these preliminaries the first intertwining formula is, for $v\in L$ and $g\in G$
 \begin{align}\label{diagram: 1st intertwining formula}
     a_g(\exp v)=\exp(Ad_g(v)).
  \end{align}
defined for $g$ in a sufficiently small neighborhood of the identity and $v\in L$ sufficiently small so that $\exp(v)$ is well defined in the expression. This formula will be used to show sufficient regularity of $Ad$.

    Recall that $\SG\in{}^\s\loc SC^{0}\rz LG$ and $\Fg=\rz LA(\SG)$ is it's internal Lie algebra. In the following lemma, we will use the first intertwining formula above.

\begin{lem}[Ad lemma]\label{lem: Ad lemma}
  Suppose that $g\in\SG$ ($=\SG_{\nes}$) is sufficiently small. Then if $X\in\Fg$ with $X\sim 0$, then  $Ad_{g}X\sim0$.
\end{lem}

\begin{proof} (Note here we are working with the transfer of the above machinery although *'s are generally absent.) First note that if $g\in\SG$, then the mapping $a_g:\SG\to\SG$ given by $a_g(h)=ghg^{-1}$ is an $S$-homeomorphism. This follows from the fact that
$a_g=L_g\circ R_{g^{-1}}$ and that $L_g$ and $R_{g^{-1}}$ are
$S$-homeomorphisms, and composition of $S$-homeomorphisms are such.
This implies that if $h$, $h'\in\SG$, then $h\sim
h'\underset{\circledast}{\dllra}a_g(h)\sim a_g(h')$. Note also that
$$
a_g(e) = e.
$$
So if $X\in\Fg$ is such that $X\sim0$, then (by the $\mu-\exp$
lemma) $\exp X\sim e$.

So by $\circledast$ $a_g(\exp X)\sim a_g(e)=e$. But
$$
a_g(\exp X)=\exp(Ad_g(X)),
$$
so that $\exp(Ad_g(X))\sim e$. But then  $Ad_g(X)\sim0$ (again by
the $\mu-\exp$ lemma).
\end{proof}



\subsection{Translating to the general linear group}\label{sec: transl pf of ad is SCo to Gin}

In this part we move the problem from $\SG$ to $\rz Gl_n$.



 We will use the (*transfer of) the intertwining formula that connects the group/Lie algebra structure of $\SG$
 with the group/Lie algebra structure of a Euclidean group.
 One can check that the definition of $Ad_g$ determines a smooth Lie group homomorphism $Ad:G\ra Gl(L)$ whose differential induces on the Lie algebra level the homomorphism  of Lie algebras $ad:v\in L\mapsto ad_v\in\en(L)$.
 With this, we have the second  formula
    \begin{align}\label{diagram: 2nd intertwining formula}
      Ad_{\exp(v)}(w)=\expp(ad_v)(w)\notag,
    \end{align}
  where $\exp$ is the exponential map for $L$
and $\expp$ is the exponential map for $\en(L)$. So lemma \ref{lem: Ad lemma} (the Ad
lemma) along with this formula has as a consequence the following
restatement of lemma \ref{lem: Ad lemma}. Before we make the statement, we need some
notation. When we are looking at $Ad_{g}$, it will be for
$g=\exp(tv)$ for a fixed, small $v\in\Fg_{nes}$ with $v\nsim0 $, and
$t\in\rz U$, $U$ being a symmetric neighborhood of $0$ in $\bbr$. See
the $\mu$-exp lemma (\ref{lem: mu-exp lem}).

\begin{lem}[EXP(ad) is regular]\label{lem: EXP(ad) is regular}Let $v\in\Fg$ be  fixed (standardly) small enough,  and $t\in\rz U$.
Then $w\in\Fg_{nes}$ implies that $\expp(ad_{tv})(w) \in\Fg_{nes}.$
\end{lem}

\begin {proof} This follows from Lemma \ref{lem: Ad lemma}, from the intertwining
formula and from the next statement in the following manner. Let $v\in\rz\bbr^n$. Then
$v\in\rz\bbr^n_{nes}$ if and only if the following holds. For every
$\e\!\in\!\!\rz\bbr$\;,\;$\e\sim0$ if and only if $\e v \sim0$. Let
$tv\in\Fg_{nes}$ be small enough and suppose that $w\in \Fg_{nes}$.
Then if $\e\in\rz\bbr$ with $\e\sim0$ we have that $\e w\sim0$
in$\rz\Fg$. So Lemma \ref{lem: Ad lemma} implies that $Ad_{\exp(tv)}(\e w)\sim0$. That
is, $\e Ad_{\exp(tv)}(w)\sim0$. As this holds for all $\e\sim0$, we
have that $Ad_{\exp(tv)}(w)$ must be nearstandard.
\end {proof}

Considering $ad_{tv}$ as an element of $End(\Fg)$, we will show that if the conclusion of this lemma holds for
\textbf{any} $A\in\rz\en_n$, in particular, for any
$A\in\rz\en(\Fg)$, then that $A$ must be nearstandard.
\textbf{That is, Lemma \ref{lem: EXP(ad) is regular}, along with the next general result about
the relationship between nearstandard elements of $\rz\en(\bbr^n)$
and their exponentials in $\rz Gl_n$, will be all we need to prove
that $\ad$ is $SC^o$}. We need one more lemma and some elementary
differential geometry before we begin the theorem.

\begin{lem}\label{lem:local *gp hom is regular}Let $t\mapsto g_t:\rz U\rightarrow\rz Gl_{n,nes}$ denote a
$\rz$local one parameter subgroup, where $U$ is a symmetric
neighborhood of $0$ in $\bbr^n$. Suppose that $v\in\rz\bbr^n_{nes}$ with
$v\nsim0$. Then for $t\in\rz U$, $g_t(v)\in\rz\bbr^n_{nes}$.
\end{lem}

\begin {proof}
But, if $A\in\!\rz Gl_{n,nes}$ and $v\in\!\rz\bbr^n_{nes}$,then, by
definition, $A(v)\in\rz\bbr^n_{nes}$. In particular, this holds for
$A=g_t$, when $t\in\rz U$.
\end {proof}

\subsection{Differential geometry of the transferred general linear group}\label{sec: diff geom of *Gl_n}
\subsubsection{Restricted differential of the general linear group structure}

  The tangent bundle of $Gl_n$ has a canonical smooth
trivialization, ie., $TGl_n\cong Gl_n\times T_{Id}(Gl_n)$, and
$T_g(Gl_n)\cong \en_n$ One can, eg., get this via the embedding of
$Gl_n$ as an open (dense) subset of $\bbr^{n^2}$. (See \ref{sec: props *Gl_n,*End_n,*Exp} in the
paper) Also, using this identification, we have that $T_g(Gl_n)\cong
\en_n$, for any $g\in Gl_n$.
 We $*$transfer
these \textbf{standard} identifications to get that $\rz TGl_n\cong
\rz Gl_n\times \rz T_{Id}(Gl_n)\cong \rz Gl_n\times \rz\en_n$. In
particular, if $g\in\rz Gl_n$ we have an internal, \textbf{not
necessarily nearstandard}, identification $\rz T_g(Gl_n)=\rz\en_n$
But, as this is identification on the standard level was a
homeomorphism (in fact, an analytic diffeomorphism), then the
*transferred identification restricts to a \textbf{nearstandard}
identification of the nearstandard  parts, ie.,
$$
\rz(TGl_{n})_{nes}\cong \rz Gl_{n,nes}\times
\rz(T_{Id}(Gl_n))_{nes}\cong \rz Gl_{n,nes}\times \rz\en_{n,nes}
$$
and, in this case, when $g\in Gl_{n,nes}$, the identification
$\rz(T_gGl_n)_{nes}\cong \rz\en_{n,nes}$ is a \textbf{nearstandard}
internal isomorphism.
  (See the last part of \ref{subsec: def *End_n, *Gl_n}). Let
  $\mu:Gl_n\times Gl_n\rightarrow Gl_n$ denote the product map on
  $Gl_n$, and let
$$
  \rz d\mu_{(g,h)}: \rz T_g(Gl_n)\times \rz T_h(Gl_n)\rightarrow \rz T_{gh}(Gl_n)
  $$
  denote its internal differential at $(g,h)\in \rz Gl_n\times \rz Gl_n$.
  Then using the identification above, we have $\rz d\mu_{(g,h)}:\rz\en_n\times \rz\en_n\rightarrow
  \rz\en_n$. Let $\la_g:Gl_n\rightarrow Gl_n:h\mapsto gh$ denote left
  multiplication by $g$. Similarly let $\rho_h:Gl_n\rightarrow Gl_n:g\mapsto gh$ denote
  right
  multiplication by $h$.
    Then we
  have the following well known formula (See Greub, Halperin and Vanstone, \cite{GHV1973} p 25). Let $A\in T_gGl_n, B\in
  T_hGl_n$, then
  $$
d\mu_{(g,h)}(A,B)=d\la_g(B)+d\rho_h(A).
$$
But then, if we restrict this formula to the smooth submanifold of
$TGl_n\times TGl_n$ given by $Z_{Gl_n}\times T_{Id}Gl_n $ where
$Z_{Gl_n}$ is the zero section of $TGl_n$ ie., when $h=Id$ and
$A=0$, we get $d\la_g(B)=d\mu_{(g,Id)}(0,B)$. In particular, as
$d\mu$ is a smooth map, this implies the following lemma.

\begin{lem}\label{lem: differential of left transl is smth }
The differential of left translation by $g\!\in\! Gl_n$;
\begin{align}
 d\la_g :Gl_n\times \en_n\rightarrow \en_n
\end{align}
 is a\;
$C^{\infty}$ map, and so if $g\in \rz Gl_{n,nes}$,
$d\la_g|\rz\en_{n,nes}$ is a nearstandard linear map.
\end{lem}

\begin {proof}
See the argument before the lemma.
\end {proof}
\begin{remark}
    Essentially, this (almost trivial) lemma says that if a group element is nearstandard, then the *differential of (left) multiplication by this element is nearstandard. This, to some extent, reveals the regularity argument of this paper.
\end{remark}
 A standard analogue of the next result is a statement of the following sort. Suppose that $A_1,A_2,\ldots$ is a sequence of elements of $End_n$ with $g_j(t)\dot= EXP(tA_j)$ for all $j$ and suppose  the following holds. There is a symmetric neighborhood, $U$, of $0$ in $\bbr$ such that  the sequence of maps $t\mapsto g_j(t):U\ra Gl_n$ is equicontinuous in $t$. Then the sequence $A_1,A_2,\ldots$ has a subsequence converging to an element in $End_n$.
 We could not find a result like this in
the literature, and so proved the corresponding needed nonstandard result.
\subsubsection{EXP is regular: ad is S-continuous}\label{subsec: EXP is reg, ad is SCo}
\begin{thm}[EXP is regular]\label{thm: EXP is regular}
Let $A\in \rz\en_n$, and let $g_t$ denote $\rz\expp(tA)$. Let $\rz U$ be a standard,
symmetric neigborhood of\; $0$ in $\rz\bbr$. Suppose that the
*one parameter subgroup
$$
 t\mapsto g_t :\rz U\lra\rz Gl_n
 $$
 is a (${}^\sigma$\!\! local) subgroup of $\rz Gl_{n, nes}$. Then $A\in\rz\en_{n,nes}$.
\end{thm}

\begin {proof}Suppose to the contrary that, $A\in\rz\en_{n,\infty}$.
In the proof below, we will use the notation from the previous lemma. So
$\la_g:Gl_n\rightarrow Gl_n$ denotes left multiplication by an
element $g\in Gl_n$,
 and by Lemma \ref{lem: differential of left transl is smth }, if $g\in\rz Gl_{n,nes}$, then $d\la_g$ is a
 nearstandard linear map, ie.,
$d\la_g:\rz \en_{n,nes}\rightarrow \rz \en_{n,nes}$.
Let $\tilde{c}$ denote the ${}^\s$\!\! local $\rz$subgroup given by
the image of $t\mapsto g_t$. Then, as $\tilde{c}$ is a local
subgroup of $\rz Gl_{n,nes}$, we have that the restriction of
$d\la_{g_t}$ to the $\rz$tangent space of $\tilde{c}$ is a
nearstandard linear map. (This, again follows from Lemma \ref{lem: differential of left transl is smth }.) As the ${}^\s$local subgroup $\tilde{c}$
is abelian, $d\la_{g_t}|_{\rz T\tilde{c}}= \rz dg_t$, the
differential of multiplication by $g_t$. In particular, $dg_t$ is a
nearstandard linear map. But, for $g_t=\rz\expp(tA)$, $\rz
dg_t=A\cdot g_t\cdot dt$, (see Arnold's ODE book; he writes in the
usual way, $\frac{d}{dt}(g_t)=A\cdot g_t$). That is, $A\cdot g_t$ is
a nearstandard linear map. On the other hand, using the hypothesis, let
$v\in\rz\bbr^n_{nes}$ with $v\nsim0$ such that $z\doteq
A(v)\in\rz\bbr^n_{\infty}$ holds. Let $w\doteq g_{-t}(v)$, so that by
Lemma \ref{lem:local *gp hom is regular}, $w$ is nearstandard. Then $A\cdot g_t(w)=z$, ie., an
infinite vector, a contradiction.
\end {proof}


Our last result applies Theorem 1 to the case $A=\ad_{tv}$. We
retain the notation used above.

\begin{cor}[ad is nearstandard]\label{cor: ad is nearst}
 $\ad$ is nearstandard. That is, suppose  $v\in\rz\Fg_{nes}$ is  fixed small enough,  and
$t\in\rz U$. Then $w\in\Fg_{nes}$ implies that $\ad_{tv}(w)$ is
nearstandard.
\end{cor}

\begin {proof} Lemma \ref{lem: EXP(ad) is regular} implies that $\expp(\ad_{tv})(w)$ is
nearstandard for  $t\in \!\!\rz U$, the standard symmetric
neighborhood of $0$, $v$ and $w$ nearstandard. But applying the EXP is regular result, ie., Theorem \ref{thm: EXP is regular},
 to the case where $tA=t\ad_v=\ad_{tv}$, we get that $\ad_v$ is
nearstandard.
\end {proof}

\section{Standard part of the exponential is a local homeomorphism}\label{chap: st(exp) is loc homeo}

\subsection{Introduction and strategy}\label{sec: intro: st(exp) is S-homeo}
In this section we prove the following result.

\begin{thm}\label{thm: exp is SC^o} $^\circ\exp$ is a local homeomorphism.
\end{thm}

More precisely, we have the standard identification
 $\Fg\cong\rz\bbr^n$ as a $\rz\bbr$ vector space and $(\SG,e)\subset(\rz\bbr^n,0)$. Using these, we can view $\exp:\rz\bbr^n,0\to\rz\bbr^n,0$. With this set up, we will show that there exists $U$, $V$, neighborhoods of $0$ in $\bbr^n$ such that the following holds: $^\circ\!\exp$ exist, is a continuous map: $U\to\bbr^n$ and in fact is a homeomorphism onto $V$.

At this point, let's motivate the importance of the above result by indicating where we are going with it. Using the exponential map and its inverse to transport the Lie group structure on $\SG$ to the vector space of $\frak g$, we get a canonical Lie group structure (product) on $L$ that is a rather complicated power series which is essentially an infinite sum multilinear forms in the bilinear form given by the Lie bracket on $L$. One can observe that the convergence of this power series depends only on the norm of this bilinear form and therefore without loss of information one can view the Lie bracket simply as a bilinear map on $L$. Hence, the point of this transferral to the structure on $L$ (via the exponential map) is that this product structure is susceptible to estimates in terms of the norm of this bilinear form $[\ ,\ ]$ and we know that it is nearstandard as $[v,\,w\ ]=ad_v(w)$.

In order to prove that $\exp$ is an S-homeomorphism in this section, we need to do some estimates with this series, the
\textbf{Hausdorff series of }$\bsm{(\SG,\Fg,[\ ,\ ])}$, using the result of the last
section that $[\ ,\ ]:\rz\bbr^n\x\rz\bbr^n\to\rz\bbr^n$ is a nearstandard
bilinear form and also using a critical relation between the $\exp$ map and the H-series, see Lemma \ref{lem: S-lem}. Where defined, the Hausdorff series
{\bf ($\bsm{H}$-series)} is a $*$analytic map $H:\rz\bbr^n\x\rz\bbr^n\to\rz\bbr^n$
defined in terms of  $[\ ,\ ]$. Kirillov, \cite{Kirillov1976} p. 105, and Duistermaat and Kolk, \cite{DuistermaatKolk1999} p.29-31, give
the power series expansions for the H-series.  Bourbaki, \cite{Bourbaki1989} pp.165--168 give convergence estimates that are complicated. We will
give simplified bounding approximations for which the reader can
fill in the details.

\subsection{H-series estimates, S-lemma}\label{sec: H-series estimates, S-lemma}
 First note that we can assume that $H\neq0$.
For suppose that $H=0$, then $[\ ,\ ]=0$. If $[\ ,\  ]=0$, ie., if
our local group is abelian, then the
 following
argument establishes a ${}^\s$local analytic structure. If $[\ ,\
]=0$, then $exp: (\rz\!\bbr^n_{\nes},+)\to \SG$ is a group isomorphism.
That is, $exp(x+y)= exp(x)+ exp(y)$; so $exp(\mu (x))=exp(x)+exp(\mu
(e))=\mu (exp(x))$, ie., $exp$ is an $S-homeomorphism$. It's inverse
therefore gives a change of coordinates to the analytic structure
$(\rz\!\bbr^n_{\nes},+)$.

We will begin by proving the following estimates for the $H$-series.

\begin{lem}[H-lemma]\label{lem: H-lem }
 Suppose that $|x|$, $|y|\le 1/2$ and $0<t<1/B_0$, where $B_0$ is the norm of $[\ ,\ ]$. Then
$$
\f t2 |x + y|\le |H(tx,ty)|\le 2t|x+y|.
$$
\end{lem}

\begin{proof} We first need some estimates. Let $B_0=\|[\ ,\ ]\|\doteq\rz\!\sup_{|x|,|y|=1}|[x,y]|\in\rz\!\bbr_{+,\nes}$ as $[\ ,\ ]$ is $SC^{0}$. Then $|[x,y]|\le (B_0/2)|x-y| |x+y|$. This follows from expanding $[x+y,x-y]$ using bilinearity and antisymmetry. Using this, we get
\begin{align*}
\mid[x,[x,y]]\mid  &\le (B^2_0/2) |x| |x - y|\ |x + y|\\
\mid[y,[x,[x,y]]]\mid  &\le (B^3_0/2) |x|\ |y|\ |x - y|\ |x + y|,
\text{ etc.}
\end{align*}
Now,
\begin{align*}
H(tx,ty)  &= t(x + y) + \f12[tx,ty] + \f1{12}[tx,[tx,ty]] -  \f1{24}[tx,[ty,[tx,ty]]] + \cdots\\
&= t(x + y) + \f{t^2}2[x,y] + \f{t^3}{12}[x,[x,y]] - \f{t^4}{24}[x,[y,[x,y]]] + \cdots\\
&= t(x + y) + f_t(x,y).
\end{align*}
Using the above estimates we get the bound:
$$
|f_t(x,y)|\le\f{B_0t^2}4 |x - y|\ |x + y|\(1 + \f{B_0t}6 |x| +
\f{(B_0t)^2}{12} |y|\ |x| +\cdots\).
$$

Now suppose that $0<t\le\f1{B_0}$ (e.g., $0<B_0t\le1$) and that
$|x|$, $|y|\le1/2$. Then we get
\begin{align*}
|f_t(x,y)|  &\le\f{B_0t^2}4 |x - y|\ |x + y|\(1 + \f16\cd\f12+\f1{12}\cd\(\f12\)^2+\cdots\)\\
&\le\f{B_0t^2}2 |x - y|\ |x + y|\ \text{ as }\ (\ )<2.
\end{align*}
Using these estimates, we have that
\begin{align*}
|H(tx,ty)|\ &\le t|x + y| + |f_t(x,y)| \le t|x + y|\(1 + \f{B_0t}2 |x - y|\)\\
&\le t|x + y| (1 + 1)(\text{as $0<B_0t\le1$ and $|x-y|\le2$})\\
&= 2t|x + y|\ \text{ which is the RHS of our bound.}
\end{align*}
We have to establish the lower bound.
\begin{align*}
|H(tx,ty)|\ &\ge |t(x + y)| - |f_t(x,y)| \ge |t(x + y)| - \f{B_0t^2}2 |x - y|\ |x + y|\\
&= t|x + y| \(1 - \f{B_0t}2 |y - x|\)\\
&\ge t|x + y| \(1 - \f{|y-x|}2\)\text{ as }\  0<B_0t\le1\\
&\ge\f t2 |x + y|\text{ as }\ |x| + |y|\le 1,
\end{align*}
which finishes the proof of the H-Lemma.
\end{proof}

\begin{cor}[H-corollary]\label{cor: H-cor}
Suppose $0\overset{<}{\nsim}t<1/B_0$, $0\nsim|x|$, $|y|\le 1/2$. Then
$$
|H(tx, - ty)|\sim0\dllra x\sim y.
$$
\end{cor}

This follows directly from the H-lemma.

\begin{lem}[S-lemma]\label{lem: S-lem}
Suppose $0\nsim|X|$, $|Y|<1/2$ and $0\underset{\nsim}{<}t<1/B_0$. Then
$$
\exp(tX)\sim\exp(tY)\dllra X\sim Y.
$$\qed
\end{lem}

\begin{proof} We will use the formulas A) $\exp(tX) = \exp(H(tX,-tY))\exp(tY)$, which comes from eg., Duistermaat and Kolk, \cite{DuistermaatKolk1999} p.27, and B) if $g$, $\d\in\SG$, then $\d\cd g\sim g\dllra\d\sim e$, which just follows from the fact that $\SG$ is $SC^0$ group. So suppose that $X\sim Y$, $0\nsim|X|$, $|Y|\le1/2$. Then by the H-corollary this is equivalent to $|H(tX,-tY)|\sim0$. But the $\mu-\exp$-Lemma (Lemma \ref{lem: mu-exp lem}) implies that this is equivalent to  $\exp(H(tX,-tY))\sim e$, where here we might have to shrink our standard neighborhood so that the $\mu-\exp$ lemma applies. But then this is equivalent by B) above to
$$
\exp(H(tX, -tY))\cd\exp(tY)\sim\exp(tY).
$$
But by A), the left hand side is $exp(tX)$, as we wanted.
\end{proof}

\subsection{Exp is onto a neighborhood of e}\label{sec: exp is onto nbd of e}
 We want to next prove that the image of $\exp:\Fg\to\SG$ contains a
 standard neighborhood of $e$ in $\SG$ ($\subset\rz\!\bbr^n$).

 First, we need some preliminaries. We will leave the $\rz\!$ off the
 transfer of standard sets, e.g., $U$ instead of $\rz\!U$ will be used.

Given $\SG$ defined on a standard neighborhood $U$ of $e$ ($=0$),
 there exists standard
 neighborhoods $V$, $W$ of $e$ in $\SG$ such that $V\cd W\subset U$. (See Montgomery and Zippin, \cite{MontZip1955} p. 32).
  Let $Z=V\cap W$, a standard neighborhood of $e$. So $Z\cd Z\subset U$. For $q\in Z$,
   let $\ov Z$ be the standard closure of $Z$ and $L_g$ be the $\rz\!$local diffeomorphism: $\rz\!\ov Z\to\rz\!U$
    given by left multiplication by $g$. (see Olver, \cite{Olver1996} p. 30).  Some notation is needed before we proceed.
    If $0<\e\in \rz\!\bbr$ and $g\in\SG$, then $D_{\e}(g)$ denotes the $\rz\!$Euclidean ball with
 center $g$ and radius $\e$. With this, we have the following result.

\begin{lem}[Compactness lemma]\label{lem: compactness}
  If $0<\e\sim0$, there is $\d\in\rz\bbr$ with $0<\d\sim0$ such that $L_g(D_{\e}(e))\supset D_{\d}(g)$, $\forall g\in\rz\!\ov Z$.
\end{lem}
The proof is given after the next lemma.
We can now prove the following. For a mapping $f$, $\img(f)$ or $\img f$ will denote the image of $f$ in its range.

\begin{lem}[Onto lemma]\label{lem: Onto}
 Given the setup above, $\mu(e)\subset\img(\exp)$.
\end{lem}

\begin{proof} Suppose not. Then  there exists $g_0\in\mu(0)$, $g_0\not\in\img(\exp)$ with $|g_0|$ minimal for this. (This is the transfer of an internal statement, technically this is stated for all $g$ in some internal disc.) Therefore, there exists $g_1\in\img(\exp)$ such that $\dist(g_1,g_0)<\d/2$. But then as $L_{g_1}(D_{\e}(e))$ contains the ball of radius $\d$ centered on $g_1$, it follows that $g_0\in L_{g_1}(D_{\e}(e))$. That is, $g_0=g_1h$ some some $h\in\img(\exp)$. So we have that $g_1=\exp(v_1)$ and $h=\exp(w)$ for some $v_{1}$, $w\sim0$ in $\Fg$. That is
$$
g_0 = \exp(v_1)\exp(w) = \exp(H(v_1,w)),
$$
contradicting that $g_0\not\in\img(\exp)$, finishing the proof.
\end{proof}


\begin{proof}[Proof of compactness lemma] As $L_g$ is a *local diffeomorphism for $g\in\rz\!\ov Z$, there is a *neighborhood $U_g$ of $g$ in $\rz\!\ov Z$ and $0<\e_g\in\rz\!\bbr$ such that if $g'\in U_{g}$, $L_{g'}(D_{\e}(e))\supset D_{\e_g}(g')$. This just follows from the fact that $L_g$ is *locally aproximable by a *linear isomorphism.

Now $\SU\doteq\{U_g:g\in\rz\ov Z\}$ is a *open cover of $\rz\ov Z$ and
as $\rz\ov Z$ is $\rz\!$compact, there is $n\in\rz\bbn$ and
$U_{g_1},\dots,U_{g_n}\in\SU$ such that $\rz\!\ov Z\subset\bigcup^n_{j=1}U_{g_j}$.
Let $\d\doteq\rz\!\min\{\e_{g_j}:j=1,\dots,n\}$. Then if $\tl g\in\ov
Z$, there is $j_{0}\in\{1,\dots,n\}$ such that $\tl g\in U_{g_{j_0}}$. This
then implies that  $L_{\tl g}(D_{\e}(e))\supset D_{\e_{j_0}}(\tl g)\supset
D_\d(\tl g)$, as we wanted to prove.
\end{proof}

As a direct result of the Onto Lemma, we have the following statement.

\begin{cor}[Onto corollary]\label{cor: Onto}
 If $g\in\img\exp$, then $\mu(g)\subset\img\exp$.
\end{cor}

\begin{proof} Let $k\in\mu(g)$. The LTG condition implies that $L_g\mu(e)=\mu(g)$ and so there is $h\in\mu(e)$ such that $k=gh$. But the Onto Lemma implies $h=\exp(w)$ for some $w\in\mu(\Fg)$ and the hypothesis implies $g=\exp(v)$ for some $v\in\Fg$. Hence,
$$
k=gh=\exp(v)\cd\exp(w)=\exp(H(v,w))\in\img(\exp)
$$
as needed.
\end{proof}

Using the following result we can finish. We, again, sometimes
denote a standard neighborhood by $U$ instead of $\rz\!U$.

\subsection{Fact on S-homeomorphisms and finish of proof}\label{sec: fact on S-homeos and pf finish}

 We need a NSA result that verifies that particular
 properties of an internal map on the nonstandard level make its standard part a homeomorphism on
  some neighborhood of the origin. The following definition does this.
    Suppose that $U$ is a standard neighborhood of $0$ in $\rz\bbr^n$.
\begin{definition}[S-homeomorphism]\label{def: S-homeo}
     Suppose that  $f:(\rz\bbr^n,0)\to(\rz\bbr^n,0)$ is an internal map.
     Then we say that $f$ is an $S$-homeomorphism on $\rz U$ if for all
     $x, y\in\rz U,x\nsim y\dllra f(x)\nsim f(y)$  and (ii) If $y\in\img f$, then $\mu(y)\subset\img f$.
\end{definition}

\begin{lem}[S-homeomorphism lemma]\label{lem: S-homeo}
Suppose that $f:(\rz\bbr^n,0)\to(\rz\bbr^n,0)$ is an $S$-homeomorphism on $\rz U$.
 Let $W=\;^{\circ}\!f(U)$. Then $^{\circ}f: U\to\intt(W)$ is a homeomorphism.
\end{lem}

\begin{proof} First of all, it is easy to verify that
$$
^{\circ}f:x\in\intt(U)\to{}^{\circ}(f(\rz\!x))\in W
$$
 is a well defined map. We want to show $^{\circ}f$ is $C^0$.
  We will show that if $x_j\in U$, for $j\in\bbn$ and $x_j\to x_0\in U$ (converges to $x_0$),
   then $^{\circ}f(x_j)\to{}^{\circ}f(x_0)$. To do this we need two typical NSA facts.
  Let $\rz\bbn_{\infty}=\rz\bbn\smallsetminus {}^{\s}\bbn$ be the infinite natural numbers.
   (a) If $S\doteq\{a_j:j\in\bbn\}$ is a sequence in $\bbr^n $ such that $a_j\to d$ in $\bbr^n$ and
    if $\rz\!S=\{\bar a_j:j\in\rz\bbn\}$ is the corresponding internal sequence,
     then $\bar a_j\sim\rz\!d$ for $j\in\rz\bbn_{\infty}$. In the reverse direction
   we have the following result. (b) If $\{\la_j:j\in\rz\bbn\}$ is an internal sequence
    such that there exists $\mu\in\rz\bbr^n_{\nes}$ with $\la_j\sim\mu$ for $j\in\rz\bbn_{\infty}$,
   then $\{{}^{\circ}\la_j:j\in\rz\bbn\}$ converges to $^{\circ}\mu$ in $\bbr^n$. See Stroyan and Luxemburg, \cite{StrLux76} p. 73.
    We now turn to the proof.

Now if $x_j\in U$ as above such that $x_j\to x_0$, then (b)
implies that $\bar x_j\sim x_0$ for $j>\infty$. But then by hypothesis
$f(\bar x_j)\sim f(\rz\!x_0)$ for $j>\infty$. But then by (a) for
$j\in\bbn$, $^{\circ}(f(x_j))\to{}^{\circ}f(\rz\!x_0)$. So $^{\circ}f$ is
$C^0$.

Also $^{\circ}f$ is $1-1$ on $U$. Let $x\neq y$ in $U$. Then
$\rz\!x\nsim\rz\!y$ in $U$ and therefore by hypothesis $f(\rz\!x)\nsim
f(\rz\!y)$. But then $^{\circ}(f(\rz\!x))\neq{}^{\circ}(f(\rz\!y))$. So $f$ is
$1-1$. This implies that $(^{\circ}f)^{-1}:\wt W\to U$ is a well defined
map where $\wt W=\intt(W)$.

 We will show that $(^{\circ}f)^{-1}$ is $C^0$. Let $z\in\wt W$ and $J=\{y_j:j\in\bbn\}$ be a sequence in $\wt W$ such that $y_j\to z$. We want to prove that $(^{\circ}f)^{-1}(y_j)\to(^{\circ}f)^{-1}(z)$. Let $x_j$, $j\in\bbn$ and $w$ in $U$ be defined by $x_j=(^{\circ}f)^{-1}(y_j)$ and $w=(^{\circ}f)^{-1}(z)$; so that $f(\rz\!x_j)\sim y_j$ and $f(\rz\!w)\sim z$ by the definition of $^{\circ}f$.

Let $d(x,y)$ be the standard Euclidean distance between $x$ and $y$
in $\rz\bbr^n$. Now suppose that $C=\{c_j:j\in\bbn\}$ is defined by
$d(y_j,z)=c_j$, e.g., $c_j\to0$ as $j\to\infty$. Let $\hat z\in
f(U)$ with $\hat z\sim\rz\!z$ and let $\SY=\{\hat y_j:j\in\rz\bbn\}$ be an
internal sequence in $f(U)$ such that for $j\in\bbn$, $\hat
y_j\sim\rz\!y_j$ and $d(\hat y_j,\hat z)\le2\hat c_j$ where $\rz C=\{\hat
c_j:j\in\rz\bbn\}$ is the *transfer of $C$. Such a sequence exists
because of the hypothesis: if $y\in\img f$, then $\mu(y)\subset\img
f$. So there exist $\SX=\{\hat x_j:j\in\rz\!\bbn\}\subset U$ and $\hat w\in U$
such that $f(\hat x_j)=\hat y_j$, for all $ j\in\rz\bbn$ and $f(\hat
w)=\hat z$. So note by (a) that for $j>\infty$,  $\hat c_j\sim0$,
i.e., $j>\infty$ implies that $\hat y_j\sim\hat z$.

Now for $j>\infty$, $\hat y_j\sim\hat z$ implies that $\hat x_j\sim\hat
w$ by hypothesis (i) on $f$. But $\hat z\sim \rz\!z$ implies that $\hat
w\sim\rz\!w$ for the same reason. That is, $\hat x_j\sim\rz\!w$ for
$j>\infty$. But then by result (b), the standard sequence
$\circledast$\: \: ${}^{\circ}(\hat x_j)\to w$ as
$j\to{}^{\s}\infty$. But $\hat y_j\sim\rz\!y_j$ for
$j\in\rz\bbn$ implies that $\hat x_j\sim\rz\!x_j$ for $j\in\rz\bbn$, again by
hypothesis (1). But then this implies that ${}^{\circ}\hat x_j=x_j$ and so
$\circledast$ now reads $x_j\to w$ for $j\to\infty$ as needed.
\end{proof}

So now we have that $^{\circ}f:U\to\wt W$ and its inverse
$(^{\circ}f)^{-1}:\wt W\to\wt U$ are both continuous maps, hence
$^{\circ}f: U\to\wt W$ is a homeomorphism onto $\wt W$. \qed

\begin{proof}[Proof of theorem \ref{thm: exp is SC^o}] To finish the proof of the theorem, note that if we let
$\exp=f$ in the NSA fact (Lemma \ref{lem: S-homeo}) and $U=\{tx\in\Fg:t>0$ and
$|xt|<\f1{2B_0}\}$ then the {\bf S Lemma} (Lemma \ref{lem: S-lem}) and {\bf$\mu\exp$ Lemma} (Lemma \ref{lem: mu-exp lem})
give hypothesis (i) in the NSA fact and the {\bf Onto Corollary} (Corollary \ref{cor: Onto})
gives hypothesis (ii) in the NSA fact.
\end{proof}

\section{Main nonstandard regularity theorem and standard version}\label{chap: Main NS reg thm and stan version}

\subsection{The product is S-analytic after coordinate change}\label{sec: prdct is S-analy after coord change} In this section, let $[\ ,\ ]$ denote our nearstandard  $\rz\!$Lie bracket and
let $\psi:\SG\x\SG\to\rz\!\bbr_{\nes}$ denote our product map. In this
section we will finish the proof of the main theorem by using the
results that $[\ ,\ ]$ is $SC^0$ and $^{\circ}\exp$ is local
homeomorphic along with the Campbell-Hausdorff-Dynkin, CHD, series
expansion for the product map of our local group in ``canonical''
(i.e., $\log$) coordinates.

The CHD series is the Hausdorff series we have already seen. But we
need different results and so a different formulation (which we will
provide), along with how it fits in here. Then we will summarize the
proof. We find the passage in Kolar, Michor and Slovak (henceforth KMS), \cite{KMS1993} p. 40,41, most suitable for our
purpose. See also e.g., \cite{Kirillov1976}, p. 105,106. KMS states their results
for a global group but it holds locally in the same manner (see
\cite{Kirillov1976}, p105,106, and \cite{Bourbaki1989}).

 Our $\rz$Lie bracket $[\ ,\ ]$ is a nearstandard bilinear form:
$\rz(\bbr^n\x\bbr^n)_{\nes}\to\rz\bbr^n_{\nes}$ and it defines the usual
map $\ad:\rz\bbr^n_{\nes}\to\rz\gl(n)_{\nes}$ by $v\to\ad_v:
(w\to[v,w])$. For clarity of purpose, we will instead work with an arbitrary nearstandard
bilinear form $B: (\rz\bbr^n\x\rz\bbr^n)_{\nes}\to\rz\bbr_{\nes}$ and define
$\ad^B_v(w)=B(v,w)$ just as with $[\ ,\ ]$.

For a standard Lie bracket, $[\ ,\ ]$, the text of Kol\'{a}\v{r}, Michor and Slov\'{a}k \cite{KMS1993}, defines (p. 40) an analytic map
$H:\bbr^n\x\bbr^n\to\bbr^n$ by
\begin{align*}
H(X,Y) = H_{[\ ,\ ]}(X,Y)\doteq &Y + \int^1_0 f(e^{t\ad_X}
e^{\ad_Y})\cd X dt
\tag*{$\circledast$}\\
&\qquad\text{where }\ f(t) = \f{\log(t)}{t-1}.
\end{align*}
This is a nice closed form for the CHD series. We will substitute
$B$ for $[\ ,\ ]$ in this formula. The (long known: see \cite{Bourbaki1989}
Historical Notes) Baker, Campbell, Hausdorff (BCH) formula is given as follows.

If $g$, $h\in\SG$ and $\psi(g,h)$ are all defined and if $g$,$h$ are
in the range of $\exp$ then $\psi(g,h)=\exp(H_{[\ ,\ ]}(\exp^{-1}g,
\exp^{-1}h))$. (see \cite{Kirillov1976} p. 105]. We will prove that for $B=$ our
$[\ ,\ ]$, $H_{[\ ,\ ]}$ has analytic standard part.

At this point we need to make some clarifying remarks. First of all,
for an arbitrary $\rz$Lie algebra product $[\ ,\ ]$
$$
H_{[\ ,\ ]}: \rz\bbr^n\x \rz\bbr^n\lra\rz\bbr^n
$$
defines a $\rz\!$ analytic group structure on $\rz\bbr^n$ (with $0$ as the
identity) near $0$. That is, (see \cite{Bourbaki1989} p. 162) where defined
\begin{align*}
H(H(x,y),z)  &= H(x,H(y,z))\\
H(x,-x)  &= H(-x,x) = 0\\
H(x,0)  &= H(0,x) = x.
\end{align*}
Note here that $x^{-1}=-x$ with respect to this group structure. That is, the
inversion map for the $\rz$group structure $H$ is just $x\to-x$ which
is obviously $^\s$analytic and therefore e.g., $S$ analytic.
Therefore to prove that this $\rz\!$group structure (on $\rz\bbr^n$
defined by $H$) is $S$-analytic, we just need to prove that the
product map, i.e. $H$, is $S$-analytic.

As the $S$-analyticity of $H_{[\ ,\ ]}$ depends only on the
$^\s$boundedness of $[\ ,\ ]$ as a bilinear map and as we want to
look at the $S$-analyticity of one map in our proof in the context
of multilinear maps; for us, it is more clear to work through the
steps with a general $\rz$bilinear map $B$ substituted in place of $[\
,\ ]$.

The next lemma verifies that the map $H_B$ as defined by the
expression $\circledast$ is $S$-analytic by building the expression
for $H_B$ from clearly $S$-analytic simpler expressions.

\begin{lem}[$H_B$ analyticity]\label{lem: H_B analyticity}
 If $B$ is nearstandard, $H_B$ is $S$-analytic.
\end{lem}

\begin{proof} The proof will be a reconstruction of $H_B$ as the composition of
two $S$-analytic, $SA^w$, maps. Suppose that
$B:\bbr^n\x\bbr^n\to\bbr^n$ is a bilinear map. For each $x\in\bbr^n$,
$B$ defines an element $\ad^B_x\in\gl(\bbr^n)$ by $\ad^B_x(y)\doteq
B(x,y)$ for $y\in\bbr^n$. We therefore get a linear map
$$
\ad = \ad^B: \bbr^n\to\gl(n)\quad\text{by}\quad x\mapsto\ad^B_x.
$$
Then note that if the norm on $A\in\gl(n)$ is defined (typically) by
$$
\|A\| = \sup\{|A(x)|: |x| = 1\},
$$
then $\forall x$, $y\in\bbr^n$, $|\ad_x(y)|\le\|\ad_x\|$.

Note that if $A\in\gl(n)$, then $e^A (=\exp A)\in Gl(n)$ satisfies
$$
\|e^A\|\le e^{\|A\|}\le 1 + \|A\|\quad\text{if}\quad \|A\|\le1
$$
and if $p(x)$ is a real polynomial in $x$, then $p(e^A)\in\gl(n)$
and $\|p(e^A)\|\le p(1+\|A\|)$ again if $\|A\|\le1$. By continuity,
these estimates extend to convergent power series. So if $A_1$,
$A_2\in\gl(n)$ are such that $\|e^{A_1} e^{A_2}\|$ lies in the  domain of $\log$, then
$\log(e^{A_1} e^{A_2})$ is defined, as is $(e^{A_1}
e^{A_2}-1)^{-1}\log(e^{A_1} e^{A_2})$ if we  also have $e^{A_1}
e^{A_2}\neq1_{\bbr^n}$. Let $f(x)=(x-1)^{-1}\log(x)$. Then we have
the well defined element of $\gl(n)$, $g(A_1,A_2)$ defined to be
$f(e^{A_1} e^{A_2})$.

By construction $g:\gl(n)\x\gl(n)\to\gl(n)$ is an analytic function.
If $t$ is a real number, and $x$, $y\in\bbr^n$, $\bar
g(A_1,A_2,x,y)\doteq y+\int^1_0g(tA_1,A_2)x\cdot dt$ is therefore also
an analytic function:
$$
\bar g: \gl(n)\x\gl(n)\x\bbr^n\x\bbr^n\lra\bbr^n.
$$
Let $\wh B: \bbr^n\x\bbr^n\lra\gl(n)\x\gl(n)$ be defined by $\wh
B(x,y)=(\ad^B_x,\ad^B_y)$. Now $\rz$transfer the above construction. $\rz\bar g$ is
now a $^\s$analytic function (and therefore $S$-analytic) on
$\rz\!(\gl(n)\x\gl(n)\x\bbr^n\x\bbr^n)_{\nes}$ and if $B$ is nearstandard,
then $\wh B$ is an $S$-analytic map:
$\rz\!(\bbr^n\x\bbr^n)_{\nes}\lra\rz\!(\gl(n)\x\gl(n))_{\nes}$. But then
$\rz\!\bar g\circ(\wh B\x1_{\rz\!\bbr^n\x\bbr^n})$ is $S$-analytic (at
nearstandard points, of course) as it is the composition of
$S$-analytic maps. (See Fact B) and C) in Lemma \ref{lem: S-analytic maps} in the preliminaries.) But
this is just the mapping $H_B(x,y)=H(x,y)$ as defined in \cite{KMS1993} p.40,
where here we have not specialized our general bilinear form $B$ to
a Lie bracket.

That is, if $B$ is nearstandard, then
$$
H_B(x,y): \rz\!(\bbr^n\x\bbr^n)_{\nes}\lra\rz\bbr^n_{\nes}
$$
   is an S-analytic map, as we wanted to show.
\end{proof}
\subsection{Finish of proof of main nonstandard theorem}

Using this we can get the following result.

\begin{thm}[Product theorem]\label{thm: Product thm}
 Let $(\SG,\psi)\in{}^\s\loc SC^0\rz LG$ modeled on $\rz U\x\rz U$, where $U$ is a convex neighborhood of 0 in $\bbr^n$. Let $(L,[\ ,\ ])=\rz\!LA(\SG,\psi)$. Suppose that $[\ ,\ ]:\rz(\bbr^n\x\bbr^n)_{\nes}\lra\rz\bbr^n_{\nes}$ is $SC^0$. Suppose that $\exp$ is a $^\s$local $S$-homeomorphism. Suppose that $(\ov\SG,\ov\psi)$ is the representation of the $^\s\loc\rz\!LG(\SG,\psi)$ with respect to the local coordinates given by $\exp^{-1}$. Then $\ov\psi$ is $SA^w$ on $\exp^{-1}(\rz\wt U)\x\exp^{-1}(\rz\wt U)$, where $\wt U=U\cap V$. Here $\rz V$ is a standard neighborhood of $0$ on which $\exp^{-1}$ is an  $S$-homeomorphism (see chapter \ref{chap: st(exp) is loc homeo}).
\end{thm}

\begin{proof} The result follows from the above preliminaries, the $H_B$-lemma (Lemma \ref{lem: H_B analyticity}) and the fact that $\ov\psi(x,y)=H_{[\ ,\ ]}(x,y)$.
\end{proof}

We can now give our main nonstandard result. We continue to use the notation $\ov{\psi}$ from the product theorem above for the group structure in the new coordinates.

\begin{thm}[Main nonstandard theorem]\label{thm: Main nonst thm}
 Suppose that  $(\SG,\psi)$ is an $SC^0$ ${}^{\s}\loc\rz\!LG$. Then there exists standard neighborhoods\; $U$ and\; $V$ of \;$0$ in $\bbr^n$, such that $\eta\doteq\;^o(\exp^{-1}):( U,0)\ra (V,0)$ is a homeomorphism $^o\psi:V\x V\to\bbr^n$ is an analytic local group structure and $\eta\circ\ov\psi\circ(\eta^{-1}\times\eta^{-1}):V\x V\to\bbr^n$ is our original $^{\s}\loc\rz\!LG$ structure $\psi$.
\end{thm}

\begin{proof} This follows immediately from the $\exp$ is an $S$-homeomorphism theorem (see Theorem \ref{thm: exp is SC^o}) and from the $\ad$ is  $SC^0$ Corollary (see Corollary \ref{cor: ad is nearst} ) applied to  Theorem \ref{thm: Product thm} above.
\end{proof}

Suppose that $(\SG,\psi,\nu)$ is an $SC^0$ $^\s\loc\rz\! LG$ modeled on $(\bbr^n,0)$. Then we have proved the following: there is a (canonical) S-homeomorphic choice of coordinates in which our $^\s$local *Lie group is S-analytic. \textbf{Henceforth in this paper, if we have such a $^\s$local *Lie group with standard part a locally Euclidean local topological group, then the standard part is in fact S-analytic. For the purposes of this paper, this can be heuristically restated as follows.}
\begin{remark}
 Suppose that we have a $^\s$local *Lie group modeled on $\bbr^n$ lying (pointwise) infinitesimally close to a locally Euclidean local topological group on some standard neighborhood of $0$, then that local topological group is analytic.
\end{remark}

\subsection{Standard Consequences of Main Theorem}\label{sec: stan conseq of main reg thm}
  In this section, we will give some standard corollaries of the main regularity theorem.
    We start with a nonstandard result that is a statement about how a single nonstandard object with mild regularity in fact has strong regularity, and when we shift to the standard domain we get a result about the asymptotic regularity of families of objects. This is typical.
    Note that these formulations have the flavor of the theory of normal families of analytic functions; see Robinson's paper \cite{RobinsonNormalFamilies1965MR0188406}. But, here, instead of some kind of nonstandard Cauchy formula applied to nonstandard elements of a standard sequence of analytic functions, we work with the (transferred) properties that ideal (nonstandard) elements of families of local Lie groups might have and use our work to force properties on the family itself.

  We will be talking about families of local topological groups defined on some neighborhood of $0$ on some fixed $\bbr^n$. As such, we need to fix a convex  neighborhood $V$ of $0$ in $\bbr^n$, where our local groups will be defined. If $G=(\psi,\nu)$ is a local group defined on $\bbr^n$, let $\SD_G\subset\bbr^n$ denote a domain of definition for $G$, ie., an open neighborhood $N$ of $0$, such that $\psi$ is defined on $N\x N$ and $\nu$ on $N$    Given this, then note that as our choice for $V$ is arbitrary; then for a given family $\FG$ of local groups such that $\cap\{\SD_G:G\in\FG\}$ contains an open set, $\wh{N}$, then we can just choose $V$ to be a convex neighborhood of $0$ contained in $\wh{N}$.

   If $G_j=(\psi_j,\nu_j)$, $j=1,2$, are two local topological groups on $\bbr^n$ (here, the identity of any such group will always coincide with $0$), and $U$ is a neighborhood of $0$ in $\bbr^n$ with $U\subset\SD_{G_1}\cap\SD_{G_2}$, we say that they $G_1$ equals $G_2$ on $U$ if $\psi_1=\psi_2$ on $U\x U$ and $\nu_1=\nu_2$ on $U$; we say they are equal if there is such a $U$ on which they are equal.
   For perspective, we have included  the following lemma.
\begin{lem}
   Suppose that $U\subset V\subset\bbr^n$ are convex open neighborhoods of $0$ and $G_1,G_2$ are local topological groups such that $V\subset \SD_{G_1}\cap\SD_{G_2}$ and  $G_1=G_2$ on $U$, then $G_1=G_2$ on $V$.
\end{lem}
\begin{proof}
   We prove the statement for the products, $\psi_1,\psi_2$, the proof for the inversions following from this and using a similar open and closed argument. First of all, one can adapt the theorem on connected topological groups in \cite{MontZip1955} p.37 to get that elements of both $G_1$ and $G_2$ in $V$ are  products of elements of these groups in $U$. We therefore proceed by induction on the length of the product keeping in mind that convexity of $V$ implies that $k$-fold associativity holds for $G_1$ and $G_2$ on $V$ (see chapter \ref{chap: error in Jacoby} for why this could be a problem). Let's verify that if $z\in V\smallsetminus U$ is given by $\psi_j(x,y)$ for $j=1,2$ (as $G_1=G_2$ on $U$), then $w$ in some neighborhood of $0$, we have that $\psi_1(z,w)=\psi_2(z,w)$. Choose $\xi\in\mu(0)$; then $\rz\psi_j(\rz y,\xi)\in\rz U$ and so $\rz\psi_1(\rz\psi_1(\rz y,\xi))=\rz\psi_2(\rz\psi_2(\rz y,\xi))$. But the left hand side is (by associativity) $\rz\psi_1(\rz\psi_1(\rz x,\rz y),\xi)=\rz\psi_1(\rz z,\xi)$  and the right hand side is $\rz\psi_2(\rz\psi_2(\rz x,\rz y),\xi)=\rz\psi_2(\rz z,\xi)$. This holds for all $\xi\in\mu(0)$ and therefore by overflow (everything here being internal) there is a neighborhood $N$ of $0$ such that $\rz\psi_1(\rz z,\rz w)=\rz\psi_2(\rz z,\rz w)$ for all $w\in N$, ie., $\psi_1(z,w)=\psi_2(z,w)$ for $z\in U^2$ and $w\in N$. By doing the previous with $\rz\psi_j(\rz x,\xi)$ instead, we get $\psi_1(w,z)=\psi_2(w,z)$ for $w$ in a neighborhood of $0$ and $z\in U^2$.  So we have that if $\psi_1(z,w)=\psi_2(z,w)$ for some $z,w\in V$, then this holds on some neighborhoods of $z$ and $w$, ie., the set where they are equal is open.
    On the other hand, if $\psi_1=\psi_2$ on some set $S_1\x S_2$ and  $(\xi,\z)\in\rz S_1\x\rz S_2$, then transfer implies that $\rz\psi_1(\xi,\z)=\rz \psi_2(\xi,\z)$. But then if $x,y\in V$ are in the closure of $S_1\x S_2$, then choosing $(\xi,\z)\in\rz S_1\x\rz S_2$ with $^o\xi=x$ and $^o\z=y$, we have,  by the continuity of the $\psi_j$'s  on $V\supset S_1\x S_2$, that
 \begin{align}
    \psi_1(x,y)=\psi_1(^o\xi,\;^o\z)=\;^o(\psi_1(\xi,\z))=\;^o(\psi_2(\xi,\z))=\psi_2(^o\xi,\;^o\z)=\psi_2(x,y),
  \end{align}
   ie., they are equal on the closure of $S_1\x S_2$.
   So we know that the subset where $\psi_1=\psi_2$ contains $U\x U$ and is an open and closed subset of $V\x V$ and so as $V\x V$ is connected must be equal to $V\x V$.
\end{proof}


    As we will be talking about potential groups on this fixed neighborhood of $0$ of varying degrees of differentiability,  we will include some organizational definitions here. If $k\in\{0\}\cup\bbn$, we will let $\bsm{\wt{C}^k}=\wt{C}^k(V)$ denote $C^k(V\x V,\bbr^n)\x C^k(V,\bbr^n)$
    Recall that $Gp=Gp^k(V)$ denotes the family $C^k$ local Lie groups on $V$ for some fixed $2\leq k\leq\infty$ so that $G^k(V)\subset \wt{C}^k(V)$.  If $G\in Gp$  with $G=(\psi,\nu)$ and $(x,y),(\ov{x},\ov{y})\in V\x V$, then we write $|G(x,y)-G(\ov{x},\ov{y})|$ for $\max\{|\psi(x,y)-\psi(\ov{x},\ov{y})|,|\nu(x)-\nu(y)|\}$. Given an ordered pair of multiindices $(\a,\b)$, we will write $\p^{(\a,\b)}G$ for $(\p^\a_x\p^\b_y\psi,\p^\a\nu)$ where $\p^\a_x$ means taking the $\a$ partial derivative in the first coordinates of $\psi$ and $\p^\b_y$ the $\b$ partial derivative in the second coordinates. We now have our first result.
\begin{proposition}\label{prop: ptwise lim G_j in Gp is Gp}
    Suppose that $G^1,G^2,\ldots$ is a sequence in $Gp^k$ and $H\in\wt{C}^0$  such that for all $x,y\in V$ we have $|H(x,y)-G^j(x,y)|$ tends to $0$ as $j\ra\infty$. Then there are coordinates on $V$ in which $H$ can be given the structure of a local Lie group on $V$. Furthermore, in these coordinates we have that for each $(x,y)\in V\x V$ and all multiindices $(\a,\b)$ with $|\a|+|\b|\leq k$, we have  $|\p^\a\p^\b (G^j-H)(x,y)|$ tends to 0 as $j\ra\infty$ .
\end{proposition}
\begin{proof}
    Let $\om_0$ be any fixed infinite integer; then the transfer of the hypothesis gives  $|\rz H(\xi,\z)-\rz G^\om(\xi,\z)|\sim 0$ for all $\om\geq \om_0$. In particular, this implies that  $\rz G^{\om_0}$ is S-continuous on $\rz V$. But then theorem \ref{thm: Main nonst thm} implies that $\rz G^{\om_0}$ is a $^\s$local $SC^0$ *Lie group and therefore there are coordinates on $V$ so that $\rz G^{\om_0}$ is an S-analytic $^\s$local *Lie group. That is, $^o(\rz G^{\om_0})$ is a local analytic Lie group. Yet $\rz H(\xi,\z)\sim\rz G^{\om_0}(\xi,\z)$ for all $\xi,\z\in\rz V_{nes}$ means that $H=\;^o(\rz H)=\;^o(\rz G^{\om_0})$, eg., $H$ is an S-analytic group in these coordinates. Backing up, since in these coordinates, both $\rz G^\om$, for $\om\geq\om_0$, and $\rz H$ are eg., in $S\wt{C}^k(U)$, we have by theorem \ref{thmbasreg}  that the condition $\rz H(\xi,\z)\sim\rz G^\om(\xi,\z)$ for all $\xi,\z\in\rz V_{nes}$ (and all $\om\geq\om_0$) implies in fact that for all multiindex pair $(\a,\b)$ with $|\a|+|\b|\leq k$, we have that $\rz\p^{(\a,\b)}(\rz H)(\xi,\z)\sim\rz\p^{(\a,\b)}(\rz G^\om)(\xi,\z)$ for all $\xi,\z\in \rz V_{nes}$ (and all $\om\geq\om_0$). As this holds for all $\om\in\rz\bbn$ with $\om\geq\om_0$ and as $\rz V_{nes}=\cup\{\rz K:K\subset V\;\text{is compact}\}$, then we can rewrite this statement as follows. \textbf{(A)}: For every compact $K\subset V$ we have  $\rz\p^{(\a,\b)}(\rz H)(\xi,\z)\sim\rz\p^{(\a,\b)}(\rz G^\om)(\xi,\z)$ for all $\xi,\z\in \rz K$ and all $\om>\om_0$. Now consider, for each $r\in\bbr_+$ and compact $K\subset V$ the following set.
\begin{align}
     \FE_{K,\;r}=\{j_0\in\bbn:|\p^{(\a,\b)}G^j(x,y)-\p^{(\a,\b)}H(x,y)|<r\qquad\notag\\ \text{for all}\;x,y\in K,j\geq j_0\;\text{and}\;|\a|+|\b|\leq k \}.
\end{align}
    We claim that $\FE_{K,\;r}$ is nonempty; this will follow from reverse transfer when we verify that $\rz\FE_{K,\;r}$ is nonempty. By transfer, we have
\begin{align}
    \rz\FE_{K,\;r}=\{\la_0\in\rz\bbn:\rz|\rz\p^{(\a,\b)}(\rz G^\la)(\xi,\z)-\rz\p^{(\a,\b)}(\rz H)(\xi,\z)|<\rz r \qquad\notag \\ \text{for all}\;\xi,\z\in\rz K,\la\geq\la_0\;\text{and}\;|\a|+|\b|\leq k\};
\end{align}
    and note that statement (A) above implies, in particular, that $\rz G^{\om_0}\in\rz\FE_{K,\;r}$. So given this, if $K_1\subset K_2\subset\cdots$ is a sequence of compact subsets of $V$ with union $V$, then for each $t>0$ and $x,y\in V$, we have that there is $j_0\in\bbn$ such that $x,y\in K_j$ and $j\in\FE_{K_j,\;t}$ for all $j\geq j_0$. In other words, if $t>0$ and $x,y\in V$, we have that $|\p^{(\a,\b)}G^j(x,y)-\p^{(\a,\b)}H(x,y)|<t$ for all multiindex pairs $(\a,\b)$ with $|\a|+|\b|\leq k$ and $j\geq j_0$.
\end{proof}
   Let's give a more nuanced version of the previous proposition.  First note that the convergence condition in the above proposition is distinctly weaker than uniform convergence on $U$. Given this, we have a definition. (See Stroyan and Luxemburg, \cite{StrLux76}, p217 for a nonstandard rendition of the usual definition of equicontinuity.)
\begin{definition}\label{def: equicont fam}
   We say  that a family $\SF\subset \wt{C}^0(U)$  is equicontinuous if the following holds.   For each $x\in \ov{V}$ and each $r>0$, there is  a neighborhood of $x$, $U^x$, an element of $\SF$, $G^x$, and $s^x\in\bbr_+$ such that the following holds. If for all $y,\ov{y}\in U^x$, we have $|G^x(y,\ov{y})-G^x(x,x)|<s^x$, then for every $G\in\SF$, we have that $|G(y,\ov{y})-G(x,x)|<r$.
\end{definition}
   This is not the usual definition of equicontinuity; ours relates all elements of $\SF$ to some element of $\SF$ rather than to the universal function $d^x(y)=|y-x|$.
   This is a weak form of equicontinuity, a more typical form is a uniform version of this.  A typical example of of such a uniform equicontinuous subset of $\wt{C}^0(U)$ is given as follows. Let $\FM$ denote the set of functions  $m:\bbr_+\ra\bbr_+$  satisfying $\lim_{t\ra 0}m(t)$ exists and is $0$. Let $\SF_m$ denote the set of $G\in\wt{C}^0(U)$ such that there is  $b_G\in\bbr_+$ such that for all $x,y,\ov{x},\ov{y}\in U$ with $\max\{|x-\ov{x}|,|y-\ov{y}|\}<b_G$, we have $|G(x,y)-G(\ov{x},\ov{y})|<m(|(x,y)-(\ov{x},\ov{y})|)$.   Note also that if $m,m'\in\FM$ and $m$ decays no slower that $m'$, ie., $\lim_{t\ra 0}(m(t)/m'(t))$ exists and is finite, then $\SF_m\subseteq\SF_{m'}$, but note that this uniform condition implies that  $\wt{C}^0(U)\supsetneqq\cup\{\SF_m:m\in\FM\}$. A nonuniform version of the $\SF_m$ that still defines an equicontinuous set is as follows. Choose $M:U\x\bbr_+\ra\bbr_+$ such that for each $x\in U$,  $\lim_{t\ra 0}M(x,t)=0$ and define $\SF_M$ as follows: we say that $G\in\SF_M$ if for each $x\in U$, there is $K_{G,x}\in\bbr_+$ and $r_x>0$ such that for $\max\{|y-x|,|\ov{y}-x|\}<r_x$, we have that $|G(y,\ov{y})-G(x,x)|<K_{G,x}M(x,\max\{|y-x|,|\ov{y}-x|\})$ . Note that these sets $\SF_M$ are also subrings of $\wt{C}^0(U)$. Furthermore, we find that for these equicontinuous sets we do get exhaustion; ie., if $\bbm$ denotes the set of these $M$'s, then $\wt{C}^0(U)=\cup\{\SF_M:M\in\bbm\}$, which follows from the fact that  we can define an $M\in\bbm$ in terms of a given $G\in\wt{C}^0(U)$, in which case we will have  $G\in\SF_M$. Hence our equicontinuity can be defined in this manner; and note, in particular, that the rate of decay of $t\mapsto M(x,t)$ at $0$ can be arbitrarily slowly as $x\in U$ leaves compact subsets of $U$.
   In spite of this severely weak form of equicontinuity, we nonetheless have our second standard consequence of the main nonstandard theorem.
\begin{proposition}\label{prop: C^0 equicont fam is C^k equicont}
    Suppose that $\SF\subset Gp$ is an infinite equicontinuous subset. Then there are coordinates on $V$ and $c\in\bbr_+$ depending only on $\SF$ such that for all $G\in\SF$, we have that $\|G\|_k<c$.
\end{proposition}
\begin{proof}
    Suppose that the conclusion is false, so that for each $b\in\bbr_+$, there is $G_b\in\SF$ with $\|G_b\|_k>b$ for all choices of coordinates on $V$. Transferring this statement get an element $\ov{\SG}\in\rz\SF$ with $\rz\|\ov{\SG}\|_k$ infinite in all *coordinates. Given this, let $x\in V$ and consider $\xi,\z\in\mu(x)$. Choose an arbitrary $r\in\bbr_+$, so that $U^x, G^x, s^x$ exist by the hypothesis of equicontinuity. We have the statement $\FS(x,r,U^x,G^x,s^x)$ as follows: if $y,\ov{y}\in U^x$ with $|G^x(y,\ov{y})-G^x(x,x)|<s^x$, we have that for all $G\in\SF$ that $|G(y,\ov{y})-G(x,x)|<r$. Therefore the transfer, $\rz\FS(x,r,U^x,G^x,s^x)$, of this statement is the following. If $\Fy,\ov{\Fy}\in\rz U^x$, satisfy $\bsm{(\ddag)}$ $|\rz G^x(\Fy,\ov{\Fy})-\rz G^x(\rz x,\rz x)|<\rz s^x$, then for all $\SG\in\SF$, we have that $|\SG(\Fy,\ov{\Fy})-\SG(\rz x,\rz x)|<\rz r$. Yet as $\xi,\z\sim \rz x$, we have that $(\ddag)$ is satisfied for $\Fy=\xi,\ov{\Fy}=\z$. But $r>0$ was chosen arbitrarily and so in fact $|\SG(\xi,\z)-\SG(\rz x,\rz x)|\sim 0$. Yet as we chose $\xi,\z$ arbitrarily in $\mu(x)$ and we chose $x\in V$ arbitrarily, we have that $\SG$ is S-continuous in $V$. But then the main nonstandard theorem, \ref{thm: Main nonst thm}, implies that there are coordinates (in fact the standard part of $\log^\SG$ coordinates) for which $\SG$ is S-analytic. But then eg., we must have $\rz\|\SG\|_k$ is finite; in particular $\rz\|\ov{\SG}\|_k$ is finite, a contradiction.
\end{proof}
    From the previous fact, we have the following consequence. If $G\in\wt{C}^k(U)$, $G=(\psi,\nu)$ and if $\a=(\a_1,\a_2)$ is a multiindex with $|\a_1+\a_2|\leq k$, then $G^\a=G^{\a_1,\a_2}$ will denote $(\psi^{\a_1,\a_2},\nu^{\a_1})\in\wt{C}^{k-|\a|}(U)$ where eg., $\psi^{\a_1,\a_2}(x,y)=\p_y^{\a_2}\p_x^{\a_1}\psi(x,y)$.
\begin{cor}\label{cor: C^0 precpt fam in Gp is C^k precpt}
    Suppose that $\SF\subset Gp^k$ is equicontinuous, $k\in\bbn$ and $\SS\subset\SF$ is a sequence. Then there is a homeomorphic change of coordinates on $U$, a subsequence $\SS'=\{G_1,G_2,\ldots\}\subset\SS$ and $\wh{G}\in Gp^k$ (in these coordinates), such that for each $r\in\bbr_+$ and $x,y\in U$, there is $j_0\in\bbn$ such that  we have $|G_j^\a(x,y)-\wh{G}^\a(x,y)|<r$ for $j\geq j_0$ and $|\a|\leq k$ in these new coordinates.
\end{cor}
\begin{proof}
     Consider $\SG_0\in\rz\SS\smallsetminus\;^\s\SS$. Then by the previous proposition, we have that $\SG_0\in S\wt{C}^k(U)$, eg., that $G_0\doteq\;^o\SG_0$ is a $C^k$ group on $U$. Assume for the moment that we have verified the following statement $\bsm{(\flat)}$:  for each $\xi,\z\in \rz U_{nes}$, we have that  for $|a|\leq k$, $\rz G_0^\a(\xi,\z)\sim\SG_0^\a(\xi,\z)$. Fixing arbitrary compact $K\subset U$ and $r\in\bbr_+$ and define
\begin{align}
     \SB_{K,\;r}=\{G\in\SS:|G^\a(x,y)-G_0^\a(x,y)|<r\;\text{for}\;|\a|\leq k\;\text{and for all}\;x,y\in K\}.
\end{align}
   We claim that $\SB_{r,\;K}$ is nonempty and will show this by verifying that its transfer, $\rz\SB_{K,\;r}$, is nonempty. To this end, transfer gives
\begin{align}
    \rz\SB_{K,\;r}=\{\SG\in\rz\SS:|\SG^\a(\xi,\z)-\rz G_0^\a(\xi,\z)|<\rz r\;\text{for}\;|\a|\leq k,\;\text{for all}\;\xi,\z\in\rz K\}.
\end{align}
    But the expression $(\flat)$ implies that $\SG_0\in\rz\SB_{K,\;r}$ and therefore, eg., by reverse transfer $\SB_{K,\;r}$ is nonempty. So choosing successively, $r=1/2,1/3,\ldots$, $K=K_1,K_2,\ldots$ with $K_j\subset K_{j+1}$ for all $j$ and $\cup\{K_j:j\in\bbn\}=U$, the above conclusion implies that we can find $G_j\in\SS$ such that $ G_j\in\SB_{K_j,1/j}$. But given  arbitrary $x,y\in U$, we have that there is $j_0\in\bbn$ with $x,y\in K_j$ for $j\geq j_0$. So given $r>0$ and $x,y\in U$, there is big enough $j_0$ such that $x,y\in K_{j_0}$ and $1/j_0<r$. Given this, the above argument assures us that $G_j$ for $j\geq j_0$ satisfies our conclusion. So we just need to verify $(\flat)$ above. But as the previous proposition gets $\SG_0\in S\wt{C}^k$, this fact follows from theorem \ref{thmbasreg}.
\end{proof}


\section{Almost implies near and a partial solution to the approximation problem}\label{chap: partial solution to approx prob}

\subsection{Summary of Initial Approach}

The {\bf initial approach} to this part of the result was to use a
$\rz\!$transferred version of Anderson's almost  $\Rightarrow$ near
construction as follows. [See Anderson's paper (\cite{Anderson1986}) or Keisler's
jazzed up version, \cite{Keisler1995},  for the details of this {\bf unused} theory].

First note that his construction implies than an accumulating
sequence of almost (local) Lie groups has a local Lie group as an
accumulation point. From Anderson's perspective if an almost Lie
group is sufficiently close to being a Lie group (locally speaking),
then there is a local Lie group lurking quite close. (In the author's work
on this, this is all structured in terms of careful numerical estimates; we are giving a crude description here).

We then *transferred this entire argument. So that now we have that if
an *almost *Lie group is *sufficiently close to being a *Lie
group, then there is a *Lie group lurking *nearby. The author had set this
problem up so that he could apply the transfer of Anderson's nearby implies (now infinitesimally) close.
That is, we applied this (almost $\Rightarrow$ near) set up to a
*transferred sequence of almost Lie groups approximating a local
topological group. Fixing one of these *sufficiently close to
being a *local Lie group, and by construction infinitesimally
close to the top group, we could then find a *local Lie group
infinitesimally close to the approximating *almost Lie group. This
is where the *transfer of Anderson's result is used. Together
these approximations imply that the approximating *local Lie group
is infinitesimally close to our (standard) local topological group.
But this implies that our *local Lie group is $S$-continuous, the
hypothesis needed to get the main nonstandard theorem.

This is a crude statement of the argument. For example, we need to
work in $^\s$local instead of *local. Nonetheless, we found a
mistake in the estimates for our approximating sequences that we have
not been able to fix.  It seems that this approximation (density) result
must certainly be true and in fact should be in the literature. To date, we have not been able to locate or prove this result, although we now believe that it will follow from the work done here.
Furthermore, some discussion with experts in   this matter indicates a general sense of the plausibility of the density assertion.

Recently, a reconstruction of the strategy just described may potentially give a proof of the local Fifth problem. We will summarize the strategy here and prove big parts of it's components later. First of all, we will prove an almost implies near result for local Lie groups motivated more by the approach of \v{S}pakula, Zlato\v{s}, \cite{SpakulaZlatos2004}, buttressed by a nonstandard formulation of a equicontinuity type property of $C^k$ mappings.

\subsection{Almost implies near for local Lie groups}

\subsubsection{almost implies near}
     In this subsection, we will prove that if $(\psi,\nu)\in\wt{C}^k$ is almost a (local) group, then there is a local (Lie) group nearby in the $C^k$ topology. Let us first define a nearness notion for potential locally Euclidean topological groups.
\begin{definition}
  Suppose that $V$ is a neighborhoods of $0$ in $\bbr^n$. If $k\in\{0\}\cup\bbn$, recall that  $\wt{C}^k$ denotes the set of $(\psi,\nu)\in C^k(V\x V,\bbr^n)\x C^k(V,\bbr^n)$ and if $b\in\bbr$ is positive, let $\wt{C}^k_b$ denote those $(\psi,\nu)\in\wt{C}^k$ such that $\|\psi\|_k$ and $\|\nu\|_k$ are bounded by $b$, where the supremum is on $V$. If $G=(\psi,\nu)$, we may write $\|G\|_k$ for $\max\{\|\psi\|_k,\|\nu\|_k\}$.
  In particular, if $G_j=(\psi_j,\nu_j)\in\wt{C}^k_b$ for $j=1,2$, then $G_1-G_2\in\wt{C}^k$, and so we have $\|G_1-G_2\|_k=\max\{\|\psi_1-\psi_2\|_k,\|\nu_1-\nu_2\|_k\}$.
\end{definition}
  One can check that the triangle inequality holds for $\|\;\|_k:\wt{C}^k_b\ra[0,\infty]$ and that $\|G_1-G_2\|_k=0$ if and only if $G_1=G_2$ on $V$.
  \textbf{Henceforth, until the finish of this subsection, all smooth maps will be defined on $U$ or its Cartesian products as is appropriate and all norms, ie., as in the previous definition will be on $\ov{V}$.}

\begin{lem}
  Suppose that $\SG_1,\SG_2$ are in $S\wt{C}^0\cap\rz\wt{C}^k$ are such that $\rz\|\SG_1-\SG_2\|_0\sim 0$. Then, for $j=1,2$, $^o\SG_j$ exists are continuous and equal on $V$.
\end{lem}
\begin{proof}
   This is clear.
\end{proof}

 Let us next define a notion measuring how close an element of $\wt{C}^k$ is to being a (local) topological group (on $V$). Essentially, we take approximations to the defining relations, ie., definition \ref{def: loc Euclid top gp}. Note that the approximations are pointwise only and this is what makes the following almost implies near result a bit surprising.
\begin{definition}\label{def: s-almost gps}
   Suppose that $H=(\psi,\nu)\in \wt{C}^k$ and $B\subset V$ is compact. Let
\begin{enumerate}
  \item $D_1(H)_B=\|\psi\circ(Id\x\psi)-\psi\circ(\psi\x Id)\|_{B,\;0}$,
  \item $D_2(H)_B=\max\{\|\psi\circ(\nu\x Id)\|_{B,\;0},\|\psi\circ(Id\x\nu)\|_{B,\;0}\}$
  \item $D_3(H)_B=\max\{\|\psi\circ(Id\x 0)-Id\|_{B,\;0},\|\psi\circ(0\x Id)-Id\|_{B,\;0}\}$.
\end{enumerate}
  If $s\in\bbr$ is positive, then we say that $H$ is an s-almost (local) group on $B$ if $D_j$ is defined at all elements of $V, V\x V$, etc.,  and  $D_j(H)_B<s$ for $j=1,2,3$. Let's denote the set of ($C^k$) s-almost groups on $B$ by $\SA_s(V,B)$ and $\SA^b_s(V,B)=\SA_s(V,B)\cap\wt{C}^k_b(V)$. Let's also denote the subset of $\wt{C}^k_b$ of local groups on $B$, ie., such that $D_j(G)_B=0$ for $j=1,2,3$, by $Gp^k_b(B)=Gp^k_b(V,B)$ (or $Gp_b(B)$ if $k$ is understood). If we want to emphasize that $V$ is a domain of the mappings defining the  $B$-local Lie group $G$, then we may write $G\in Gp(V,B)$.
\end{definition}
\begin{remark}
  Note that we did not demand a $C^k$ closeness for our $C^k$ structures, nonetheless this will be sufficient.
  In particular, note that if $H\in\wt{C}^k$ is such that $D_j(H)_B=0$ for $j=1,2,3$, then we have the defining conditions, definition \ref{def: loc Euclid top gp}, for $H$ to be a local ($C^k$) group on $B$ although now $B$ is a compact set and $H$ is defined on the larger open $V$.
\end{remark}
   We can now state our first almost implies near result. Note the history of attempts to verify that differentiable manifolds that are sufficiently close to being Lie groups in some sense must therefore have Lie group structures. All previously proven facts follow from hypotheses that, crudely speaking, are stated in terms of global $C^1$ smallness conditions on tensor fields that obstruct (on the tangent space level) a group structure. Furthermore, the hypotheses assume a particular type of group structure. For example, see the papers of Ruh, \cite{Ruh1987} and \cite{Ghanaat1989}.    On the other hand, our conditions are pointwise (ie., much weaker) conditions and do not assume that we are approximating a particular type of group. Note that although our assertions are local, assuming our local groups are globalizable, they may imply the global results.

    Again note that to get a local Lie group $C^k$-near the given almost group, we are assuming only $C^0$ closeness to a group structure for this almost group. Although this seems to regularize the nearness condition, this result is independent of our main nonstandard theorem. Nevertheless, in the next subsection the $C^0$ closeness allows us to pair the following result (through a nonstandard argument) with the main nonstandard (regularity) theorem.

     Let $cpt(V)$ denote the set of compact subsets of $V$. Note that $V=\cup\{K:K\in cpt(V)\}$ and here we can always find a countable subset of $cpt(V)$, whose elements are increasing with respect to inclusion whose union is $V$.
\begin{proposition}[Almost implies near]\label{prop: standard almost->near}
    For each $(B,r,b)\in cpt(V)\x\bbr_+\x\bbn$, there is a positive $s\in\bbr$ such that if $H\in\SA_s(B)\cap\wt{C}^k_b(V)$, then there is a local Lie group $G\in Gp^k(B)\cap\wt{C}^k_b(V)$ such that $\|G-H\|_{B,\;k}<r$.
\end{proposition}
\begin{proof}
   By way of contradiction, suppose that the conclusion is false. That is, suppose that there is  $(B_0,r_0,b_0)\in cpt(V)\x\bbr_+\x\bbn$ such that the following statement $\bsm{S(B_0,r_0,b_0)}$  holds.  For all positive $s\in\bbr$ and for all $H\in\SA_s(B_0)\cap\wt{C}^k_{b_0}(V)$, we have that for every local Lie group $G\in Gp(B_0)\cap\wt{C}^k_{b_0}(V)$, we have $\|H-G\|_{B_0,\;k}\geq r_0$. Transferring  statement $S(B_0,r_0,b_0)$ gets the following internal statement. For all positive $\Fs\in\rz\bbr$, and $\SH\in\rz\SA_\Fs(B_0)\cap\rz\wt{C}^k_{*b}(V)$, we have that for all local *Lie groups $\SG\in\rz Gp(B_0)\cap\rz\wt{C}^k_{b_0}(V)$, we have  $\rz\|\SH-\SG\|_{\rz B_0,\;k}\geq\rz r_0$.   But then choosing an $\Fs\sim 0$ and an $\wh{\SH}\in\rz\SA_\Fs(B_0)\cap\rz\wt{C}^k_{b_0}(V)$, this must imply that for all local *Lie groups $\SG\in\rz Gp(B)\cap\rz\wt{C}^k_{*b_0}(V)$ we have $\rz\|\wh{\SH}-\SG\|_{\rz B_0,\;k}\geq\rz r_0$. Yet we  will now verify that $\Fs\sim 0$ implies that the standard part of $\wh{\SH}$ is a $C^k$ group on $B_0$. First of all, if $\wh{\SH}=(\psi,\nu)$, then $^o\psi$ and $^o\nu$ exist and are in $C^k_{b_0}(V)$. This is a direct consequence of Theorem \ref{thmbasreg}. It also follows from this theorem that for all $\xi,\z\in\rz V_{nes}$, we have that $\psi(\xi,\z)\sim\rz(^o\psi)(\xi,\z)$ and $\nu(\xi)\sim\rz(^o\nu)(\xi)$ and as these are S-continuous we have that compositions are also infinitesimally close; for example if $\eta$ is also in $\rz V$ and the products are defined, we have $\psi(\xi,\psi(\z,\eta))\sim\psi(\xi,\rz(^o\psi)(\z,\eta))\sim\rz(^o\psi)(\xi,\rz(^o\psi)(\z,\eta))$  again by S-continuity. But also we have
   that as $\Fs\sim 0$ and $(\psi,\nu)$ are in $\SA_\Fs(B)$, the $D_1$ condition gives $\psi(\xi,\psi(\z,\eta))\sim\psi(\psi(\xi,\z),\eta)$ for $\xi,\z,\eta\in\rz B_0$. Putting this together with the above expressions, we get that $\rz(^o\psi)(\xi,\rz(^o\psi)(\z,\eta))\sim\rz(^o\psi)(\rz(^o\psi)(\xi,\z),\eta)$ for all $\xi,\z,\eta\in\rz B_0$, and as both sides are standard if $\xi$, $\z$ and $\eta$ are standard, then we actually have $\rz(^o\psi)(\xi,\rz(^o\psi)(\z,\eta)) = \rz(^o\psi)(\rz(^o\psi)(\xi,\z),\eta)$ for all $\xi=\rz x$, $\z=\rz y$ and $\eta=\rz z$ with $x,y,z\in B_0$; this gets that $^o\psi:V\x V\ra V$ satisfies associativity on $B_0$. This was a consequence of S-continuity and $\rz D_1(\wh{\SH})\sim 0$. Similarly, we can get that $\rz D_2\sim 0$ and $\rz D_3\sim 0$ imply the other two conditions that $^o\wh{\SH}=(^o\psi,^o\nu)$ is a $C^k$ local group. To show that the verification of the other two is quite similar, let's verify that $\rz D_2\sim 0$ implies the second group condition for $\wh{\SH}$.  As above, $\Fs\sim 0$ applied to the first $\rz D_2(\wh{\SH})\sim 0$ condition gets that $\psi(\nu(\xi,\xi))\sim 0$ for all $\xi\in\rz V_{nes}$. This along with the S-continuity of $\psi$ and the facts that $\rz(^o\psi)\sim\psi$ and $\rz(^o\nu)\sim \nu$, gets
\begin{align}
   0\sim\psi(\nu(\xi),\xi)\sim\psi(\rz(^o\nu(\xi)),\xi)\sim \rz(^o\psi)(\rz(^o\nu)(\xi),\xi).
\end{align}
    Summarizing, we have that $^o\wh{\SH}$ is a local Lie group in $Gp(B_0)\cap\wt{C}^k_{b_0}(V)$ and so $\wt{\SH}=\rz(^o\wh{\SH})$ is a local $\rz\wt{C}^k_{*b_0}$ group on $B_0$. But note that Theorem \ref{thmbasreg} also implies that $\rz\|\rz(^o\psi)-\psi\|_{\rz B_0,\;k}\sim 0$ and $\rz\|\rz(^o\nu)-\nu\|_{\rz B_0,\;k}\sim 0$, ie., $\wt{\SH}$ lies in a $C^k$ infinitesimal neighborhood of $\wh{\SH}$ over $\rz B_0$ in $\rz\wt{C}^k_{*b_0}(V)$, eg., as $r_0$ is a positive standard number, it's certainly not true that $\rz\|\wh{\SH}-\wt{\SH}\|_{\rz B_0,\;k}\geq\rz r_0$,  contradicting our assumption of the contrary conclusion.
\end{proof}
\begin{definition}\label{def: (b,s) is r-good}
     If $K\in cpt(V)$ and $r\in\bbr_+$, we say that the pair $\bsm{(b,s)\in\bbn\x\bbr_+}$ \textbf{is} $\bsm{(K,r)}$\textbf{-good} if the $b$ and $s$ satisfy the hypothesis in the above proposition for the conclusion to hold for this $K$ and $r$.
\end{definition}
\begin{remark}
   Note that this proof would not work if the hypotheses in the theorem did not include the fixed positive bound $b>0$. (This result is a kind of nonstandard equicontinuity argument; see eg., \cite{StrLux76}, p.217 and \cite{SpakulaZlatos2004}.) Yet as our approximations to the given local Euclidean topological group sharpen, the $C^k$ norms of these approximations grow in an unbounded way.
   But this turns out to not be a problem. As noted earlier in this paper, the local *Lie groups in our regularity theorem may have no restrictions on how large (in $\rz\bbr_+$) their *derivatives may be.  Nonetheless, according to our regularity theorem, as long as they are infinitesimally close to a local topological group (on $\bbr^n$), this local group is, in fact, an analytic group.
   \textbf{This is the crux of our strategy to prove the local Fifth. See the next subsection for our principal result; which is a consequence of this strategy.}
\end{remark}
\subsubsection{Principal standard result}
   Consider the following corollary of the above proposition and our main nonstandard theorem. See the appendix, chapter \ref{chap: appendix: S-smoothness}, definition \ref{def: dSC^k}, for the definition of $dSC^k(U)$. The principal standard result is a direct corollary of this nonstandard result.
\begin{cor}\label{cor: NSA reg+almost->near}
    Let $k\geq 2$ be an integer. Suppose that $\FB\in\rz cpt(V)$ with $\FB\supset\rz V_{nes}$ and $\Fr\in\rz\bbr_+$ with $\Fr\sim 0$ and   $\Fb\in\rz\bbn$ is arbitrary (possibly infinite). Given these data, there is $\Fs\in\rz\bbr_+$  such that if $\SH\in\rz\SA_\Fs(\FB)\cap\rz\wt{C}^k_\Fb(V)$ is S-continuous, then there are coordinates on $V$ such that $\SH\in\rz\SA_\Fs(\FB)\cap\rz\wt{C}^k_{*c}(V)$ for some $c\in\bbr_+$, eg., $^o\SH$ is in $\wt{C}^k_c(V)$ and is, in fact, a local Lie group. In particular, if $\wt{\SH}\in S\wt{C}^0(V)$ is such that $\rz\|\SH-\wt{\SH}\|_{0,\;\rz V_{nes}}\sim 0$, then $\wt{\SH}\in dS\wt{C}^k$ with respect to these coordinates.
\end{cor}
\begin{proof}
   First, transfer proposition \ref{prop: standard almost->near} above. So: for all\; $(\FB,\Fr,\Fb)\in\rz cpt(V)\x\rz\bbr_+\x\rz\bbn$, there is $\Fs\in\rz\bbr_+$ such that if $\SH\in\rz\SA_\Fs(\FB)\cap\rz\wt{C}^k_\Fb(V)$, then there is $\SG\in\rz Gp(\FB)\cap\rz\wt{C}^k_{\Fb}(V)$ such that $\rz\|\SG-\SH\|_{\FB,\;k}<\Fr$. We will use this statement for the (fixed) $(\Fr,\Fb,\FB)$ in the hypothesis.
   So we have  $\SG\in \rz Gp(\FB)\cap\rz\wt{C}^k_\Fb(V)$ with $\rz\|\SH-\SG\|_{\rz \FB,\;k}<\Fr\sim 0$. But as  $\rz V_{nes}\subset\FB$, we have that the S-continuity of $\SH$ implies that $\SG$ is S-continuous on $\rz V_{nes}$. Given this, the main nonstandard theorem implies that there are S-homeomorphic *coordinates on $V$ so that in these *coordinates $\SG$ is S-analytic, ie., $^o\SG$ is analytic in the corresponding standard coordinates. Now $\rz\|\SH-\SG\|_{\FB,\;k}\sim 0$ implies, eg., that $^o\SH=\;^o\SG$ on $V$ (as $^o\FB=V$), eg., $^o\SH$ is a local Lie group (in these coordinates) on $V$.  With respect to $\wt{\SH}$ in the hypothesis, note that as $\SG$ is S-analytic in these coordinates, we have in particular that $\rz\|\SG\|_{k,\;\rz V_{nes}}<\rz c$ for some $c\in\bbr_+$ ie., $\SG\in S\wt{C}^k(V)$  and so from the above hypothesis on $\wt{\SH}$, we have $\rz\|\wt{\SH}-\SG\|_{\rz V_{nes},\;0}\leq \rz\|\wt{\SH}-\SH\|_{\rz V_{nes},\;0}+\rz\|\SH-\SG\|_{\FB,\;k}\sim 0$ in these new coordinates as homeomorphic coordinate change preserves infinitesimal distances (for pairs of nearstandard points). But then $\SG\in S\wt{C}^k$ and corollary \ref{cor: SC^j sim SC^k}, implies that $\wt{\SH}$ must be in $dS\wt{C}^k$.
\end{proof}
   We have the following standard corollary of the above nonstandard result. First we need some notation, we write $m\in\FL(V)$ if $m:V\x V\x \bbr_+\ra\bbr_+$ is a continuous map such that for each $x,y\in V$, we have $\lim_{t\ra 0}m(x,y,t)=0$.
\begin{cor}[Principal  result]\label{cor: standard reg + almost->near}
    Let $k\geq 2$ be an integer. Suppose that $\FH\subset \wt{C}^0(V)$ is an equicontinuous subset and $m\in\FL$ with the following properties. For each compact $K\subset V$ and $r>0$, there is $\wt{H}\in\FH$ such that the following holds. There is $(s,b,H)\in\bbr_+\x\bbn\x \wt{C}^k(V)$ such that (1) $|H(x,y)-\wt{H}(x,y)|<m(x,y,r)$ for all $x,y\in K$ and (2) $(b,s)$ is $(K,r)$-good with $H\in\SA_s(B)\cap\wt{C}^k_b(V)$.
    Then, there is a sequence $\FS=\{H_j:j\in\bbn\}\subset\FH$, a choice of coordinates on $V$ and a  $C^k$ local Lie group $G$ on $V$ so that for every $t>0$ and $x,y\in V$, there is $j_0\in\bbn$ so that we have $|H_j(x,y)-G(x,y)|<t$ for all $j\geq j_0$.
\end{cor}
\begin{proof}
    We will demonstrate the existence of the sequence with the asserted properties. Transferring the above hypothesis statement, we get the following statement.
    If $\FB\in\rz cpt(V)$ and $\Fr\in\rz\bbr_+$, then there is $\wt{\SH}\in\rz\FH$ such that the following holds. There is $(\Fs,\Fb,\SH)\in\rz\bbr_+\x\rz\bbn\x \rz\wt{C}^k_\Fb(V)$ such that \textbf{(*1)}: $\rz|\SH(\xi,\z)-\wt{\SH}(\xi,\z)|<\rz m(\xi,\z,\Fr)$ and \textbf{(*2)}: $(\Fb,\Fs)$ is $(\FB,\Fr)$-good with $\SH\in\rz\SA_\Fs(\FB)\cap\rz C^k_b(V)$. Now first note that the conditions on $m$ imply that if $K\subset V$ is compact then $\Fr\sim 0$ implies that $\rz m(\xi,\z,\Fr)\sim 0$ for all $\xi,\z\in\rz K$ and so as $\rz V_{nes}$ is the union of these $\rz K$'s, then for $\Fr\sim 0$, we have that $\rz m(\xi,\z,\Fr)\sim 0$ for all $\xi,\z\in\rz V_{nes}$.
    So fix some $\bsm{\FB}\in\rz cpt(V)$ with $\FB_0\supset\rz V_{nes}$, $\Fr_0\sim 0$ and $\bsm{\wt{\SH}_0}\in\rz\FH$ and $\bsm{\SH_0}\in\rz\SA_\Fb(\FB_0)\cap\rz\wt{C}^k_\Fb(V)$ satisfying the above (*1) and (*2).
     With this setup, the previous discussion implies the following. First, (*1) now implies the statement \textbf{ (*3)}: $\SH_0(\xi,\z)\sim\wt{\SH}_0(\xi,\z)$ for all $\xi,\z\in\rz V_{nes}$. But as $\FH$ is equicontinuous, we have that $\wt{\SH}_0$ is S-continuous on $\rz V_{nes}$, so that (*3) implies  \textbf{(*4)}: $\SH_0$ is S-continuous on $\rz V_{nes}$. But then, as $\Fr\sim 0$, (*2) and (*4) imply that $\SH_0$ satisfies the hypotheses on the $\SH$ in the previous corollary, \ref{cor: NSA reg+almost->near}; that is, after a change of coordinates, we have the statement \textbf{(*5)}: $G\dot=\;^0\SH_0$ is, eg., a $C^k$ local Lie group on $V$ and by (*3) and S-continuity, we have that $\wt{\SH}_0\sim \rz G$ on $\rz V_{nes}$, which is equivalent to saying that, \textbf{(*6)}: for each compact $K\subset V$, we have that $|\wt{\SH}_0(\xi,\z)-\rz G(\xi,\z)|\sim 0$ for each $\xi,\z\in\rz K$. Given this development, consider, for each $t\in\bbr_+$ and compact $K\subset V$, the following set:
\begin{align}
   \FD_{K,\;t}\dot=\{\wt{H}\in\FH:|\wt{H}(x,y)-G(x,y)|<t\;\text{for all}\;x,y\in K\}.
\end{align}
   We assert that for each $t>0$ and compact $K\subset V$, $\FD_{K,\;t}$ is nonempty. By reverse transfer, it suffices to prove that $\rz\FD_{K,\;t}$ is nonempty. But transfer of $\FD_{K,\;t}$ implies that
\begin{align}
    \rz\FD_{K,\;t}=\{\wt{\SH}\in\rz\FH:|\wt{\SH}(\xi,\z)-\rz G(\xi,\z)|<\rz t\;\text{for all}\;\xi,\z\in\rz K\}.
\end{align}
    And clearly, statement (*6) says that $\wt{\SH}_0\in\rz\FD_{K,\;t}$, eg., $\rz\FD_{K,\;t}$ is nonempty. It's now clear how to find the sequence $\FS$. For each $j\in\bbn$, let $K_j\subset V$ be compact so that $K_j\subset K_{j+1}$ for all $j$ and $V=\cup\{K_j:j\in\bbn\}$. Then for each $j\in\bbn$, we have just proved  that there is $H_j\in\FD_{K_j,\;1/j}$. So given $t>0$ and $x,y\in V$, there is a $j'\in\bbn$ so that $x,y\in K_{j'}$ and there is a $j''\in\bbn$ with $1/j''<t$; choosing $j_0\geq\max\{j',j''\}$, we get our result.
 \end{proof}
\begin{remark}[\textbf{Interpretation}]
   \textbf{The previous result uses our main nonstandard theorem to, in some sense, preserve the conclusion of the almost implies near result (in the limit) while greatly weakening the condition on the almost hypothesis. That is, although the $H_j$'s are not required to lie within an $\SA^b_s$ for $(b,s)$ $r$-good for $r$'s tending to $0$, eg., they aren't even differentiable, nevertheless, as they are (only) $C^0$ approximating an almost implies near sequence, this is sufficient to prove that they are getting arbitrarily close to (for the appropriate coordinates) local Lie groups.}
   The operative example of an equicontinuous sequence $\FH$ is one that for which every subsequence has a limit point in $\wt{C}^0(U)$; eg., a sequence in $\wt{C}^0$ that is approximating our local Euclidean topological group $\SM$.
\end{remark}

\subsubsection{Technical lemmas}\label{subsec: technical lems: almos-> near}
\begin{lem}\label{lem: given b,r_0,s_0->exists 0<s<s_0}
   Let $b\in\bbr_+$ with $b\geq 1$, $r_0,s_0\in\bbr_+$  and $R_{s_0}(b)\subset\bbr_+$ denote the set
\begin{align}
    \{0<s<s_0:\text{if}\;H\in\!\SA^b_s,\;\text{there is}\;\;G\in Gp^k_b\;\text{such that}\;\|H-G\|_k<r_0\}.\notag
\end{align}
    Then  $R_{s_0}(b)$ is nonempty.
\end{lem}
\begin{proof}
    By way of contradiction, suppose that there is $\ov{b}\in\bbr_+$ such that $R_{s_0}(\ov{b})$ is empty. So  the following statement $\ST(\ov{b})$ holds: for all $s\in\bbr_+$ with $s<s_0$, there exists $H\in\SA^{\ov{b}}_s$ such that for all $G\in Gp^k_{\ov{b}}$, $\|H-G\|_k\geq r_0$. Therefore, transfer implies that $\rz\ST(\ov{b})$ holds, ie., for all $\Fs\in\rz\bbr_+$ with $\Fs<\rz s_0$, there exists $\SH\in\SA^{*\ov{b}}_k$ such that for all $\SG\in\rz Gp^k_{*\ov{b}}$, we have $\rz\|\SH-\SG\|_k\geq r_0$. But choosing $\ov{\Fs}\sim 0$, then $\rz\ov{b}$ being finite implies that if $\SH\in\rz\SA^{*\ov{b}}_k$, then we know from the proof of the previous lemma that $^o\SH\in Gp^k_{\ov{b}}$ and also that $\rz\|\SH-\rz(^o\SH)\|_k\sim 0$ which is certainly less than $r_0$, contradicting the contrary conclusion.
\end{proof}
\begin{remark}
    If for some $s,b$, $\SA^b_s$ is empty and $s_0>s$, then $R_{s_0}(b)$ is nonempty. We will, of course, use this lemma in situations where $\SA_s^b$ is not empty. Note that if $s<s'$, then $\SA^b_s\subset\SA^b_{s'}$ and so if $s'\in R_{s_0}(b)$, then $s\in R_{s_0}(b)$.
\end{remark}
   In order to  expose a closer relationship between $s$ and $(b,r)$, we need a definition.
\begin{definition}
   Let $q(A,b,r)$ be the assertion: there exists $G\in Gp_b$ with $\|A-G\|_k<r$. Let
\begin{align}
    m(b,r)=\sup\{s\in\bbr_+:\text{if}\;A\in\SA^b_s, \text{then}\;q(A,b,r)\;\text{holds}\}.
\end{align}
\end{definition}
\begin{lem}\label{lem: monotonicity & lim props of m}
   The following holds.
\begin{enumerate}
   \item If $r_1,r_2,b\in\bbr_+$ with $r_1<r_2$, then $m(b,r_1)\leq m(b,r_2)$.
   \item If $r,b_1,b_2\in\bbr_+$ with $b_1<b_2$, then $m(b_1,r)\geq m(b_2,r)$.
   \item  For all $b\in\bbr_+$, we have $\lim_{r\ra 0}m(b,r)=0$.
\end{enumerate}
\end{lem}
\begin{proof}
    The first assertion follows from the fact that for a given positive $s$, we have that for some $A\in \wt{C}^k$, $q(A,b,r_1)$ holds implies that $q(A,b,r_2)$ holds. That is, the set of positive $s$ satisfying $A\in\SA^k_b$ implies that $q(A,b,r_1)$ holds is a subset of the set of positive $s$ satisfying $A\in\SA^k_b$ implies that $q(A,b,r_2)$ holds. The second statement follows similarly as $\SA^k_{b_1}\subset\SA^k_{b_2}$ and so, for fixed $r,s$, if $q(A,b_2,r)$ holds for all $A\in\SA^{b_2}_s$ then $q(A,b_1,r)$ holds for all $A\in\SA^{b_1}_s$.
    Suppose that the last statement is false, ie., there is $b_0\in\bbr_+$ such that $\lim_{r\ra 0}m(b_0,r)>2s_0$ for some positive $s_0\in\bbr_+$. So let $r_1>r_2>\cdots$ be a decreasing sequence of positive numbers with limit $0$; then we have the following statement: for each $j\in\bbn$ we have that for all $A\in\SA^b_{s_0}$, there exists $G\in Gp^k_b$ such that $\|A-G\|_k<r_j$. Tranferring this statement we have that the following holds. For each $\Fj\in\rz\bbn$, we have that for all $\mathscr{A}\in\rz\SA^{*b}_{*s_0}$, there exists $\mathscr{G}\in\rz Gp^{*k}_{*b}$ such that $\rz\|\mathscr{A}-\mathscr{G}\|_k<\rz r_\Fj$. Pick $\Fj\in\rz\bbn_\infty$, so that the previous statement implies that $\rz\|\mathscr{A}-\mathscr{G}\|_k\sim 0$ and choose $\mathscr{A}_0\in\SA^{*b}_{*s_0}$ so that $\rz D_k(\mathscr{A})>\rz s_0/2$ for at least one of $k=1,2,3$. If $k=1$ and writing $\mathscr{A}=(\phi,\eta)$, we have that there is $\xi_0,\z_0,\g_0\in\rz V$ such that $\bsm{(\dag)}$: $|\phi(\phi(\xi_0,\z_0),\g_0)-\phi(\xi_0,\phi(\z_0,\g_0))|>\rz s_0/2$. But we know that there is $\mathscr{G}=(\psi,\nu)\in\rz Gp^{*k}_{*b}$ such that $\rz\|\mathscr{A}-\mathscr{G}\|_k\sim 0$, eg., $\rz\|\phi-\psi\|_k\sim 0$. It's easy to check that this and S-continuity of $\psi$ and $\phi$ imply that for all $\xi,\z,\g\in\rz V$, we have $\phi(\phi(\xi,\z),\g)\sim\psi(\psi(\xi,\z),\g)$ and similarly $\phi(\xi,\phi(\z,\g))\sim \psi(\xi,\psi(\z,\g))$. But then these can be strung together with the associativity expression for $\psi$ to imply $\phi(\phi(\xi,\z),\g)\sim\phi(\xi,\phi(\z,\g))$, contradicting $(\dag)$. A similar argument applies if instead $\rz D_2(\mathscr{A})>\rz s_0/2$ or $\rz D_3(\mathscr{A})>\rz s_0/2$.
\end{proof}
    The following lemma (which is a consequence of the previous standard lemma) will give us the critical information on the asymptotic relation between $(b,r)$ and $s$ from a nonstandard perspective.
\begin{lem}\label{lem: NSA; LGps close to almost gps}
    Let $m:\bbr_+\x\bbr_+\ra\bbr_+$ be as defined above. Let $\rz m$ be its transfer and suppose that $\Fr_0\in\rz\bbr_+$ is infinitesimal. Then for any infinite $\Fb_0\in\rz\bbr_+$ we have that $0<\rz m(\Fb_0,\Fr_0)\sim 0$. In particular, if $\Fr\sim 0$, irrespective of the value of $\Fb\in\rz\bbn$, an $\Fs$ that is $(\Fr)$ good must be infinitesimal.
\end{lem}
\begin{proof}
    Fix $0<\Fr_0\sim 0$ and note that if $b\in\bbr_+$, then Item (3) in Lemma \ref{lem: monotonicity & lim props of m} implies $\rz m(\rz b,\Fr_0)\sim 0$. That is, we have a map $b\mapsto m(\rz b,\Fr_0):\bbr_+\ra\mu(0)_+$, ie., the set of image values is a family $\SF$ of positive infinitesimals (internally) parametrized by a standard set. But, then as our model is sufficiently saturated, $\SF$ is bounded above by a positive infinitesimal $\Fs_0$, ie., $\rz m(b,\Fr_0)<\Fs_0$ for all $b\in\bbr_+$. Next, note that also by the transfer of Item (2) of Lemma \ref{lem: monotonicity & lim props of m},
    if $\Fb\in\rz\bbr_+$ is larger than a given $b\in\bbr_+$, then $\rz m(\Fb,\Fr_0)\leq\rz m(\rz b,\Fr_0)<\Fs_0$, eg., this holds for all infinite $\Fb\in\rz\bbr_+$. That $\rz m(\Fb,\Fr_0)>0$ follows from the transfer of Lemma \ref{lem: given b,r_0,s_0->exists 0<s<s_0}.

\end{proof}

\section{The error in Jacoby's proof of the local Fifth according to Olver}\label{chap: error in Jacoby}

\subsection{Hilbert's Fifth: Local does not follow from global}
Let's first use a paragraph to give some idea about why the 1957 proof of the local Fifth fails and cannot be fixed.
Let {\bf GA} denote global associativity (defined on
the following pages) and {\bf GL} denote globalizability (also
defined on the following pages).
Olver finds in his J. of Lie Theory paper, \cite{Olver1996}, families of local Lie groups
that are not GL. So for local Lie groups and especially for local
topological groups GL does not hold in general. In fact, as mentioned early in the paper, Olver, \cite{Olver1996},
clarifies Mal'cev's GA condition for GL to hold.
Jacoby in his paper Jacoby, \cite{Jacoby1957}, proving the local fifth problem, assumes
that his local topological group has the GA property of (global)
topological groups in order to prove his result with a line of argument along
the line of Gleason, \cite{Gleason1952}. But as Olver, \cite{Olver1996} (p.28), makes clear, such an
assumption cannot be made and hence the proof cannot be fixed without throwing out the strategy.

Before we talk about how local Lie groups are not like global Lie groups, about GA and GL, let's be clear on how they are similar. (See, eg.,  Kirillov,\cite{Kirillov1976} p. 99 for the following facts.) First of all, if $(L,\Fl)$ is a local Lie group and its Lie algebra, and $(G,\Fg)$ is a (global) Lie group with its Lie algebra and we have a Lie algebra isomorphism $\phi:\Fl\ra\Fg$, then a (classical) consequence is that there is a neighborhood, $L'$, of the identity in $L$ and a (local) Lie group isomorphism $\Phi:L'\ra G$ such that $d\Phi=\phi$ (ie., Lie algebra isomorphisms induce (sufficiently local) Lie group isomorphisms). Furthermore, given any finite dimensional $\bbr$ Lie algebra, $\Fg$, there is a (real) Lie group $G$ with Lie algebra $\Fg$. It's clear that these two statements together imply that any local Lie group is (locally!) isomorphic to a Lie group, locally in the sense that we may have restrict to a smaller neighborhood of the identity of the local Lie group to get a Lie injection into a (global) Lie group.
Nonetheless, Olver, \cite{Olver1996} explicitly produces the worst case scenario: for every lie group $G$, and any local Lie subgroup $H\subset G$, there is a local subgroup $H'\subset H$ and a local group isomorphism $H'\ra \wt{H}$ where $\wt{H}$ cannot be embedded as a local subgroup of any (global) Lie group.
It's important to note that all of these groups have isomorphic Lie algebras (given by well defined maps) and so all of these groups and local groups are locally isomorphic.

The strategic approach of the present paper is totally different from, and hence uses different ingredients from that of of Gleason (and hence Jacoby) and
so GA is not an issue here. The local topological group that We am
approximating with an internal local Lie group is not assumed to be GA
and so the internal local Lie group, will generally not be GA either (a
variation on the easy arguments of the next pages). But it doesn't
matter, GA does not play a role in the arguments here.

Apparently, if attaching the assumption of GA to Jacoby's proof
somehow allowed it to be fixed, my result (along with a proved
conjecture) would still be more general --- {\bf my result would
prove the local fifth for the apparently larger class of non GA
local topological groups}. (Note here, for simplicity sake, we have
not included the niceness qualifiers for the local topological
group, nor the $^{\s}$local one for the infinitesimally
approximately $^{\s}$local *Lie group.)

On the next pages we prove that if the local topological group is
GA, then the standard part of the approximating internal local Lie
group is GA. We will then prove an almost implies near result, stated roughly:  if a local Lie group is close to being $p$-fold associative, then it must have a $p$-fold associative local Lie group nearby.

\subsection{Olver's construction}
To begin, we have the following definitions from Olver, \cite{Olver1996} p. 27. Let $G$ be a
$\loc LG$. $G$ is {\bf associative to order $n$}, denoted
$G\in\SA(n)$, if for all $ k$, $3\le k\le n$ and for every ordered
$k$-tuple $(x_1,\dots,x_k)\in G^k$, all $k$-fold products of
$x_1,\dots,x_k$ are equal. A $k$-fold product of $(x_1,\dots,x_n)$
is a sequence of choices of products of these elements (in the given
order) by pairing adjacent pairs, adjacent elements and parentheses,
adjacent parentheses, etc. For example the 4-fold products of the
ordered four tuple $(x_1,x_2,x_3,x_4)$ are $((x_1x_2)x_3)x_4)$,
$((x_1x_2)(x_3x_4))$, $((x_1(x_2x_3))x_4)$, $(x_1((x_2x_3)x_4))$,
and  $(x_1(x_2(x_3x_4)))$. $G$ is {\bf globally
associative, GA}, {\bf denoted} $G\in\SA(\infty)$ if $G\in\SA(n)\;\forall
n\in\bbn$.

Given this definition, let's sketch Olver's construction of nonGA local local Lie
groups. Remove a point $x_0$ from a neighborhood of the identity,
$e$, in a Lie group $(G,m,e)$. For this construction $G$ should be
two dimensional, but an analogous construction works for any Lie group of dimension larger than one.  Suppose also that $G$ is simply connected, so that
the fundamental group of $G$ with a point removed is isomorphic to
$\bbz$.
Olver's construction is an insightful integration of the
given group product with the monodromy associated with the
fundamental group of the punctured group into the making of a nonGA
local Lie group. In particular it displays a dependence of global associativity on semilocal topology.
 Let $\tl{G}$ denote the simply connected covering
space, with $\pi:\tl{G}\ra G$ the projection map. On a sufficiently
restricted neighborhood $U$ of $e$ in $G$, we can lift enough of the
group structure to get a 3-associative product, ie., a local Lie
group $(\tl{G},\tl{m},\tl{U},\tl{e})$  on a selected component of
$\tl{U}\subset\pi^{-1}(U)$, the identity being the unique point in
$\pi^{-1}(e)\cap U$. For  $n\in\bbn$, Olver constructed
$2n$-fold products of elements consisting of n-fold
 products of the same ordered $n$-tuple, but associated in reverse
 order:
 $m(m(\cdots m(x_1,x_2),\cdots),x_n)$ and
 $m(x_1,(m(x_2,\cdots,m(x_{n-1},x_n)\cdots)))$.
As summarized below, GA of $G$ implies these determine a closed loop
in $G$, ie., associating the sequence in reverse order gets the same element of $G$; but, for $n\geq 4$, we can get $n$-tuples of elements in $\wt{U}$ products enclosing $x_0$ which therefore prevents the lifted loop from closing.
 Associate to
$m(m(\cdots(m(x_1,x_2),\cdots),x_n)$ the polygonal path, $\SP_1$,
with ordered vertices given by
\begin{align}
 e\ra x_1\ra m(x_1,x_2)\ra   \notag
m(m(x_1,x_2),x_3)\ra\ldots \\   \ra m(m(\cdots
m(x_1,x_2),\ldots,x_{n-1}),x_n). &   \notag
\end{align}
Here the arrows indicate directed movement along the path from
vertex to vertex, beginning at $e$ and ending at the given n-fold
product.
 If care is taken with respect to the singularity $x_0$,
then the sequence of lifted products, and hence the lifted polygonal
path, $\tl{\SP}_1$ whose vertices are given by this ordered sequence
of $\tl{m}$-products of the lifted points $\tl{x_1},\
\ldots,\tl{x_n}$,
\begin{align}
\tl{e}\ra\tl{x_1}\ra\tl{m}(\tl{x_1},\tl{x_2})\ra\tl{m}(\tl{m}(\tl{x_1},\tl{x_2}),\tl{x_3})\ra\ldots\notag
\\
 \ra\tl{m}(\tl{m}(\cdots\tl{m}(\tl{x_1},\tl{x_2}),\ldots,\tl{x}_{n-1})\tl{x_n})&\notag
\end{align}
is well defined in $\tl{G}$ and unique once $\tl{e}\in\tl{U}$ is
chosen.

 Similarly, to the same ordered  n-tuple
$x_1,x_2,\ldots,x_n$, we assign the polygonal path, $\SP_2$, with
vertices given by doing the associations in the product in the
reverse order, namely
\begin{align}
e\ra x_n\ra m(x_{n-1},x_n)\ra m(x_{n-2},m(x_{n-1},x_{n}))\ra\cdots
\notag
\\
  \ra m(x_1,(m(x_2,\cdots,m(x_{n-1},x_n)\cdots))).& \notag
\end{align}
As with $\tl{\SP}_1$ giving the well defined polygonal lift of
$\SP_1$, we have the well defined lift $\tl{\SP}_2$ given by the
successive $\tl{m}$-products of lifted points
\begin{align}
\tl{e}\ra \tl{x}_n\ra \tl{m}(\tl{x}_{n-1},\tl{x}_n)\ra
\tl{m}(\tl{x}_{n-2},\tl{m}(\tl{x}_{n-1},\tl{x}_{n}))\ra\cdots \notag
\\
  \ra \tl{m}(\tl{x}_1,(\tl{m}(\tl{x}_2,\cdots,\tl{m}(\tl{x}_{n-1},\tl{x}_n)\cdots))).& \notag
\end{align}
 As Lie groups are GA, $\SP_1$ and $\SP_2$ begin and end at the same
 points,
 and so form a closed polygonal loop. Olver chooses the $x_i$'s, in his explicit example $n=4$, so that this loop
  encloses $x_0$. But then the lifted polygons
$\tl{\SP}_1,\tl{\SP}_2$ cannot form a closed loop; the endpoints lie
on different sheets over the $n$-fold product of $x_1,\ldots,x_n$ in
$G$. That is, the  n-fold $\tl{m}$- product associated in one way is
not equal to the n-fold $\tl{m}$-product associated in the other
manner. By an artful choice for $U$, Olver produces the
worst case scenario when $n=4$, as local associativity is tested with 3-fold products.

 From this construction, it should be clear
that if we restrict our domain to smaller $V\subset U$, then for product of elements in $\wt{V}=\pi^{-1}(V)$, we may get a higher associativity but global associativity still does not hold.


Olver's general assertion follows this example: here he is given  a (global) Lie group, $G$, chooses a neighborhood of the identity, $U$, that is punctured ie., not simply connected. From this he chooses a piece $\wt{U}$ (a part of a sheet over $U$) of a good covering space of the altered $G$ so that we get associativity of 3-fold products of elements of $\wt{U}$ but associativity fails for sufficiently long associations as `downstairs' we are getting noncontractible loops in $U$. . Note that $\wt{U}$ is diffeomorphic to $U$ via the covering map, and, in fact, with his choice of lifted group sturcture, this is a \textbf{local} isomorphism on a sufficiently small part of $U$. But we see that although locally isomorphic, $U$ is globally associative (it is a subset of a global group!) but, by construction, $\wt{U}$ is not.

Later in the paper, his Theorem 21 (p.43) proves that any local Lie group is fully covered by a partial covering group that also is locally a group isomorphism onto a neighborhood of the identity of a (global) Lie group. So although the local Lie group is only locally isomorphic to a Lie group, we can extend the domain of the isomorphism (to all of the local Lie group) by weakening the notion of Lie group equivalence (via the covering group intermediary).
 Olver gives a very careful definition of a local
group homeomorphism (p28). Suffice it to say that it is a smooth
group isomorphism with great care given to the intertwining of
domains of definition of product and inverse maps.  He then defines
a local Lie group to be {\bf globalizable} if there is a local group
homeomorphism onto a neighborhood of the identity in a Lie group.
His version of Mal'cev's theorem (p46) is as follows.
{\bf A connected Lie group is globalizable $\dllra$ it is globally
associative.} Note here that connected is a technical condition
(p31) concerning connectivity and generation.

\subsection{Almost associativity}
  It seems that nonstandard analysis can say something about the subject of k-fold associativity by asking questions about almost or near associativity. Our perspective is similar to the almost implies near question for topological groups.
 Let $\SG\in{}^{\s}\loc\rz\!LG$ so that its
representatives are defined on standard neighborhoods of $0$ in
$\rz\bbr^m$ ($m\in\bbn$). Fix a standard domain of definition $\rz U$ for $\SG$. Let $n\in\bbn$. Then $\SG$ is
said to be {\bf almost associative to order $n$ on $\rz U$}, denoted
$\SG\stackrel{\in}{\sim}\SA(n)$ if for all $ k$, $3\le k\le n$, and
ordered $k$-tuple $(x_1,\dots,x_k)\in U\x\cdots\x U$, all corresponding
$k$-fold products are defined and infinitesimally close to each other (in
$\rz\!\bbr^m$). Similarly, we say that $\SG$ is  {\bf almost globally
associative on $\rz U$}, denoted $\SG\stackrel{\in}{\sim}\SA(\infty)$, if
$\SG\stackrel{\in}{\sim}\SA(n)$ for all $ n\in{}^{\s}\bbn$. We know that if
$G\in\SA(n)$ and $k\le n$, then all $k$-fold products of ($x_i\in
G$, $\forall i$) $(x_1,\dots,x_k)$ are equal and so we can
unambiguously denote this by $[x_1,\dots,x_k]$. Similarly if
$\SG\stackrel{\in}{\sim}\SA(n)$, then all $k$-fold products, for $k\le n$, of
$(x_1,\dots,x_k)$ ($x_i\in\SG$) are in the same monad, so we will
denote this monad by $\mu(x_1,\dots,x_n)$. Now suppose that
$^{\circ}\SG=G$, then it follows that $\SG\sim\rz\!G$; that is if the
product in $G$, and in $\rz\!G$ is denoted $x\cd y$ (for $x$, $y\in G$
or $\rz\!G$) and in $\SG$ is denoted by $x*y$ and if $\rz\!U$ is a standard
representative neighborhood of $0$ in $\rz\!\bbr^m$, so that both $x\cd
y$ and $x*y$ is defined, then $x\cd y\sim x*y$ in  $\rz\!\bbr^n$. We will
prove the following result.

\begin{lem} Suppose that $\SG$ and $G$ are as given above, e.g.,
$\SG\sim\rz\!G$. Then
$G\in\SA(\infty)\Rightarrow\SG\stackrel{\in}{\sim}\SA(\infty)$.
\end{lem}

\begin{proof} We will actually prove that if $n\ge3$,
$\SG\stackrel{\in}{\sim}\SA(n)$. The result follows directly from
this. So suppose that $G\in\SA(\infty)$, then we will prove by
induction that $\SG\stackrel{\in}{\sim}\SA(n)$. The induction
beginning will be obvious (2 fold products) and the induction step
will be to assume that we have proved that $\SG\in\SA(n-1)$ and to
show that this assumption and $G\in\SA(\infty)$ implies that it
holds for $\SA(n)$.

So given an ordered $n$-tuple $(x_1,\dots,x_n)\in\SG^n$ we want to
show that all $n$-fold products are infinitesimally close to each
other. We do know that all $k$-fold products of a given ordered
$k$-tuple for $k<n$, are infinitesimally close. An $n$-fold product
of $(x_1,\dots,x_n)$ will be of the form $y*y_2$ where $y_1$ is a
$k$ fold product of $(x_1,\dots,x_k)$ and $y_2$ is an $n-k$ fold
product of $(x_{k+1},\dots,x_n)$ where $1<k<n$. Let $k'$ and $k''$
be two such $k$. So we have $y'_1$ a $k'$-fold product and $y'_2$ a
$n-k'$-fold product such that $y'_1*y'_2$ gives one of the $n$-fold
products. We also have $y''_1$ a $k''$-fold product and $y''_2$ a
$n-k''$-fold product such that $y''_1*y''_2$ gives another possible
$n$-fold product of $(x_1,\dots,x_n)$ in $\SG$. Let $\bar y'_1$,
$\bar y'_2$, $\bar y''_1$ and $\bar y''_2$ be the corresponding
products in $\rz\!G$. Then $y'_1\sim\bar y'_1$,  $y'_2\sim\bar y'_2$,
$y''_1\sim\bar y''_1$ and  $y''_2\sim\bar y''_2$, which follows from
the induction hypothesis. Now $G\in\SA(n)\Rightarrow (a)\bar
y'_1\cd\bar y'_2=\bar y''_1\cd\bar y''_2$, but as $\SG$ is
$SC^0$,(b) $y'_1*y'_2\sim\bar y'_1*\bar y'_2$ and (c)
$y''_1*y''_2\sim\bar y''_1*\bar y''_2$. Also (d) $\bar y'_1*\bar
y'_2\sim\bar y'_1\cd\bar y'_2$ and (e) $\bar y''_1*\bar
y''_2\sim\bar y''_1\cd\bar y''_2$. Putting all of these together, we
get that $ y'_1* y'_2\sim y''_1*y''_2$. Specifically, from
(b)$y'_1*y'_2\sim\bar y'_1*\bar y'_2$ which by (d) is $\sim$ $\bar
y'_1\cd\bar y'_2$. But by (a) this is$\sim$ $\bar y''_1\cd\bar
y''_2$. By (e), this last is $\sim\bar y''_1*\bar y''_2$ which
finally by(c) is $\sim$ $y''_1*y''_2$. That is, an arbitrary pair of
the $n$-fold products are infinitesimally close to each other as we
wanted to prove.
\end{proof}
\begin{lem}
  If $\SG$ is an $SC^k$ local group defined on $V$, and let $G=\;^o\SG$ denote the standard part restricted to $V$. Suppose that for some $p\in\bbn$ and an ordered $p$-tuple $(\xi_1,\ldots,\xi_p)$, we have  that all $p$-fold product associations in $\SG$ are defined and for two such associations $[\xi_1\cdots\xi_p]^\SG_j$ ($j=1,2$), we have $[\xi_1\cdots\xi_p]_1\sim [\xi_1\cdots\xi_p]^\SG_2$. Then, if $x_j\dot=\;^o\xi_j$, for $j=1,\ldots,p$, we have $[x_1\cdots x_p]^G_1=[x_1\cdots x_p]^G_2$.
\end{lem}
\begin{proof}
   The proof is by induction on the length $p$ of the product. The first nontrivial length is $p=3$ and this is just associativity. So it remains to prove that given the result holds for associations within $p-1$-fold products, $p>3$, it follows that it holds for associations within $p$-fold products.  But just as with the proof of the previous lemma, the result follows from the fact that associations in $p$-fold products decompose into associations within products of length less than $p$ (so that we may use the fact for $q$-fold products, $q<p$) and also from the S-continuity of the product.
\end{proof}
  In the other direction, we have the following statement. Recall the notation: $cpt(V)$ denotes the set of compact subsets of $V$.
\begin{definition}\label{def: t-close p-assoc. gps}
 For $B\in cpt(V)$, let $\bsm{Gp^k_b(V)\cap Asc_t(B,p)}$ denote the elements of $G\in Gp_b^k$ that are $t$-almost $p$-associative on $B$, ie., such that for each ordered ordered $p$-tuple $(x_1,\ldots,x_p)\in B\x\ldots\x B$, if $[x_1\cdots x_p]_j$, for $j=1,2$ are any two associations defining a $G$-product of the $x_j$'s (in this order), then $|[x_1\cdots x_p]_1-[x_1\cdots x_p]_2|<t$. If $G$ is actually $p$-fold associative on $V$ ($t=0$), then we will denote this by $G\in\bsm{Gp^k_b(V)\cap Asc(B,p)}$.
\end{definition}
\begin{proposition}[p-fold associativity: almost implies near]\label{prop: p-fold assoc: almost->near}
   Fix $p,n,k\in\bbn$ and $V$ a convex neighborhood of $0$ in $\bbr^n$. For each $B\in cpt(V)$ and positive $r\in\bbr_+$, there is \;$t>0$ in $\bbr$ such that if $G\in Gp^k_b(V)\cap Asc_t(B,p)$, eg., a local Lie group on $V$ that is  $t$-almost $p$-fold associative on $B$, then there is  $H\in Gp^k_b(V)\cap Asc(B,p)$ such that $\|G-H\|_{B,\;k}<r$.
\end{proposition}
\begin{proof}
   The proof is an analogue of our proof of the almost a group implies a group nearby result, proposition \ref{prop: standard almost->near}. Suppose that the conclusion does not hold, that is, suppose that there is compact $B_0\subset V$ and a nonzero positive $r_0\in\bbr$ satisfying the following statement. $\bsm{\bbs(B_0,r_0)}$: For all $s>0$, there is $G\in Gp^k_b(V)\cap Asc_t(B_0,p)$ with the property that for all $G'\in Gp^k_b(V)\cap Asc(B_0,p)$, we have $\|G-G'\|_{B_0,\;k}\geq r_0$. Therefore, the transfer of $\bbs(B_0,r_0)$ holds. $\rz\bbs(B_0,r_0)$: For all $\Fs\in\rz\bbr_+$, there is $\SG\in\rz Gp^k_b(V)\cap\rz Asc_\Fs(B_0,p)$ with the internal property $P$:  for all $\SG'\in\rz Gp^k_b(V)\cap\rz Asc(B_0,p)$, we have $\rz\|\SG-\SG'\|_{*B,\;k}\geq\rz r_0$.
   Given this, choose $\Fs$ to be a positive infinitesimal so that there is $\SG_0\in\rz Gp^k_b(V)\cap\rz Asc_\Fs(B_0,p)$ with the above property $P$. But, we claim that  $\Fs\sim 0$ implies that if we let $G_0=\;^o\SG_0$ restricted to $V$, then $G_0\in Gp^k_b(V)\cap Asc(B_0,p)$. To see this, first, we have  that the group properties in definition \ref{def: loc Euclid top gp} hold by the S-continuity of $\SG_0$, second, $G_0\in \wt{G}^k(V)$ by theorem \ref{thmbasreg} and finally, $G_0$ is $p$-associative follows from $\Fs\sim 0$. For suppose that  $(x_1,\ldots,x_p)\in  B\x\ldots\x B$ is an ordered $p$-tuple and $[\rz x_1\cdots\rz x_p]^{\SG_0}_j$, for $j=1,2$, are two product associations in $\SG_0$ of this ordered $p$-tuple. Then $\Fs\sim 0$ implies that $[\rz x_1\cdots\rz x_p]^{\SG_0}_1\sim [\rz x_1,\cdots\rz x_p]^{\SG_0}_2$ and so the lemma above implies that $[x_1\cdots x_p]^{G_0}_1=[x_1\cdots x_p]^{G_0}_2$, ie., $G_0$ is $p$-fold associative and so by transfer, we have that $\rz G_0\in\rz Gp^k_b(V)\cap\rz Asc(B,p)$. We  claim that $\|\SG_0-\rz G_0\|_{*B_0,k}\sim 0$ which is clear as S-continuity of $\SG_0$ implies that $\|\SG_0-\rz G_0\|_{*B_0}\sim 0$ and so theorem \ref{thmbasreg} implies the claim as both are $SC^k$. But then the existence of $\rz G_0$ with these properties violates our contrary conclusion that $\SG_0$ has property $\rz P$, as $r_0$ is noninfinitesimal.
\end{proof}

\section{Appendix 1: Nonstandard conditions of smooth equicontinuity}\label{chap: appendix: S-smoothness}

    In this part, we will prove that internally regular maps satisfying mild nonstandardly stated regularity properties have good (again nonstandardly stated) regularity properties, see theorem \ref{thmbasreg}. We then give standard corollaries of this result, see eg., corollary \ref{cor: 2nd stan cor of appendix thm}. We follow this with results on the the nonstandard class of maps, $dSC^k(U)$, followed again with standard corollaries, for example see corollary \ref{cor: SC^j sim SC^k} and \ref{cor: standard dSC^k result}.   More specifically, in this section we will present *smooth representations of maps whose standard parts are  $C^k$ for $1\leq k\leq\infty$ or asymptotically $C^k$. Theorem \ref{thmbasreg} is the principal result for these $SC^k$ internal maps: it says that  internal differentiation behaves nicely with respect to being infinitesimally close (pointwise!) and also with respect to the operation of taking standard parts. This theorem is used repeatedly in the previous chapters. (In later work, we will extend this to Lipschitz maps, to maps belonging to the Sobolev classes of maps and other classes of weakly differentiable functions.)
    These representations act in different ways and have different uses. The *
     representations of  standard smooth maps effectively give straightforward criteria for *smooth maps to actually have standard smooth parts; eg., these are regularity results. For example, in this paper, we have families of differentiable maps and we are looking at the asymptotic properties of these families. Here, we do this by transferring the families and looking at elements `at infinity'. Knowing that these are *differentiable, if their standard parts exist,  the results of this chapter imply that the asymptotic behavior of these families have certain regularities.

    The nonstandard representations for the various classes of maps (asymptotically) lacking differentiability properties are of a different nature: they give these standard mapping extra facility. These nonstandard representations of asymptotically nondifferentiable  families of maps have all of the operational properties of differentiable maps, although the standard parts of `asymptotic elements' are potentially quite wild. Nevertheless, the transfer of, eg., (typically nonlinear) differential equation type restrictions on these families (eg., the Maurer Cartan equations in this paper) can directly force (on the nonstandard level) certain regularizing behavior. This section gives one way of seeing when nonstandard maps have standard regularity properties and when these are preserved under nonstandard operations (eg., $\rz\f{\p}{\p x_j}$).
    (Parenthetically, although the results in this appendix play an important role in this paper, the primary motivation behind these results is the facilitation of direct, often nonlinear methods in investigations on partial differential equations.)

    The proof of the main theorem for smooth maps is rather involved. Most of the hard work will occur in the next central proposition.

\begin{proposition}\label{approxprop}                                      
   If $\frak f\in SC^1(U,\bbr)$,
    then $d(^o\Ff)_x$ exists (and is finite) for all $x\in U$, $^o(\rz d\Ff)_x=d(^o\Ff)_x$ for all $x\in U$,
   and finally $x\mapsto d(^o\Ff)_x$ is continuous on $U$. That is,  $^o\Ff\in C^1(U,\bbr)$.
   In particular, if $\f{\p}{\p x_1},\ldots,\f{\p}{\p x_n}$ is the  canonical basis for $TU$, the tangent space of $U$, with $\rz\f{\p}{\p x_1},\ldots,\rz\f{\p}{\p x_n}$ the transferred internal basis for $\rz TU$, then for each $x\in\bbr^n$, we have $^o(\rz\f{\p}{\p x_j}\Ff)(x)=\f{\p}{\p x_j}(^o\Ff)(x)$ and $x\mapsto \f{\p}{\p x_j}(^o\Ff)(x)$ is continuous on $U$.
\end{proposition}

\begin{proof}
We will first prove the existence of $d(^o\Ff)_x$ for all $x\in U$. We will then prove that $^o(\rz d\Ff)_x=d(^o\Ff)_x$ and finally we will prove the continuity statement.

First, we want to show that for every $x\in\bbr^m$ and $0<\d_0\in\bbr$,
there exists $0<\e_0\in\bbr$, and $L_x\in\hom(\bbr^m,\bbr^n)$ such
that if $v\in\bbr^m$, then
\begin{align}
\big|\f{ \Ff(x+\e v)- \Ff(x)}{\e }-\rz L_x(v)\big|<\d_0 \label{der1}
\end{align}
if $0<\e<\e_0$. We left the *'s off $x,v,\e,\d$ above.  (Without loss of generality, we may assume that
$|v|=1$.) To prove expression\; \ref{der1}, it suffices to prove
that if $0<\e\sim 0$ then
\begin{align}
\big|\f{\frak f(x+\e v)-\frak f(x)}{\e }-\rz L_x(v)\big|<\d_0
\label{est2}
\end{align}
as indicated by the following argument. (Here we are writing $x$ for
$\rz x$ and $|\;\;|$ for$\rz|\;\; |$.) Let $\d_0>0$ as given above
and let
$$
\frak S\doteq\big\{\e_0>0:0<\e<\e_0\Rightarrow\big|\f{\frak f(x+\e v)-\frak f(x)}{\e}-\rz
L_x(v)\big|<\rz\d_0\big\},
$$
  where we at this point choose $L_x(v)\dot=^o(\rz d\Ff)_v$ which is finite as $\xi\mapsto\rz d\Ff_\xi$ is S-continuous by hypothesis.
Then $\frak G$ is internal and $\{\e\in\rz\bbr:0<\e\sim
0\}\subset\frak S$ by(\ref{est2}). Therefore, by overflow, there
exists $\e_0$ with $0<\e_0\not\sim 0$ such that
$\e_0\in\frak S$. That is
\begin{align}
0<\e<\e_0\Rightarrow\big|\f{\frak f(x+\e v)-\frak f(x)}{\e }-\rz
L_x(v)\big|<\rz\d_0 \label{est3}
\end{align}
Note that as $\frak f$, $\rz L$, and $|\;|$ are $SC^o$, then
$^o(\frak f(x+\e v))=f(x+\;^o\e v)$, $^o(\rz L_x(v))=L_x(v)$, and
$^o(|w|)=|^ow|$ if $w\in\rz\bbr^n_{nes}$. Therefore taking standard
parts of expression (\ref{est3}), we get expression (\ref{der1}), as
we wanted.

We have reduced the proof of the first assertion to proving the assertion
(\ref{est2}). By hypothesis, if $u,v\in\rz\bbr^m_{nes}$, then $u\sim
v\Rightarrow\rz d\frak f_u\sim\rz d\frak f_v$. It follows that if
$0<\e\sim 0$, then $t\in[0,\e]\Rightarrow\rz d\frak f_{x+tv}\sim\rz
d\frak f_x$. In particular, as $L_x\doteq\; ^o(\rz d\frak
f_x)$, we get that $\rz d\frak f_{x+tv}\sim\rz L_x$ for $0\leq
t\leq\e$. So as $\{\|\rz d\frak f_{x+tv}-\rz L_x\|:t\in[0,\e]\}$ is
*compact and a  subset of $\mu(0)$, there is $\d$ with $0<\d\sim 0$ such
that $\|\rz d\frak f_{x+tv}-\rz L_x\|<\d$ for $0\leq t\leq\e$. (Here
$\|\;\|$ stands for $\rz\|\;\|$, the *transfer of the usual operator
norm on $\hom(\bbr^m,\bbr^n)$.)With this, we have that
\begin{align}
\bigg|\int_0^{\e}\rz d\frak f_{x+tv}dt-\int_0^{\e}\rz
L_x(v)dt\bigg|\leq\e |v|\cdot\|\rz d\frak f_{x+tv}-\rz L_x\|=\e\d
\label{est4}
\end{align}
But substituting $\int_0^{\e}\rz d\frak f_{x+tv}(v)dt=\frak f(x+\e
v)-\frak f(x)$ and $\int_0^{\e}L_x(v)dt=\e L_x(v)$ into expression
(\ref{est4}), we get expression (\ref{est2}).

Let's now prove the second assertion. We have preliminaries; let
$(x,v)\in\bbr^n$ and in the following, we shall let $x=\!\rz x$, and
$v=\!\rz v$, ie., use the same symbols whether in $\bbr^n$ or
$\rz\bbr^n$. Then we shall prove that
\begin{align}
^o(\rz d\Ff)_x(v)=d(\!\;^o\!\Ff)_x(v).  \label{disstan}
\end{align}
Now we have $\Ff\in\rz C^1(U,\bbr^m)$, and so if $0<r\in\; ^{\s} \bbr$, then
\begin{align}
\Ff(x+rv)-\Ff(x)=\int_0^r\rz d\Ff_{x+tv}(v)dt \label{basthm1.1}
\end{align}
where we leave (as usual) the * off the integral and also off $r$.
Now the fact that $\rz d\Ff$ is $SC^0$ implies that if $0<\d\in\bbr$,
then there exists $0<{\e}_1\in\bbr$ such that $|\rz d\Ff_{x+\e
v}(v)-\rz d\Ff_x(v)|<\f{\d}{2}$ if $\e<{\e}_1$. That is,
\begin{align}
\bigg|\int_0^{\e}\rz d\Ff_{x+tv}(v)-\e\rz
d\Ff_x(v)\bigg|<\f{\e\cdot\d}{2}\; \mbox{for}\; \e<{\e}_1
\label{basthm1.2}
\end{align}
On the other hand as $^of$ is differentiable, then for the given
$\d$ above, there exists $0<{\e}_2\in\bbr$ such that
\begin{align}
\bigg|d(^o\Ff)_x(v)-\f{^o\Ff(x+\e v)-^o\!\!\Ff(x)}{\e}\bigg|<\f{\d}{2}\;\text{for}\;0<\e<\e_2.
\label{basthm1.3}
\end{align}
But $^o\Ff(x+\e v)\sim \Ff(x+\e v)$ and $^o\Ff(x)\sim \Ff(x)$ and as $\e$ is
a standard positive number, then there exists $\eta\in\rz\bbr_+$
with $\eta\sim 0$ such that
\begin{align}
\bigg|\f{^o\Ff(x+\e v)-^of(x)}{\e}\;-\;\f{\Ff(x+ \e
v)-\Ff(x)}{\e}\bigg|<\eta. \label{basthm1.4}
\end{align}
But then expressions (\ref{basthm1.3}) and (\ref{basthm1.4}) give
\begin{align}
\bigg|d(^o\Ff)_x(v)-\f{\Ff(x+\e v)-\Ff(x)}{\e}\bigg|<\eta+\f{\d}{2}
\label{basthm1.5}
\end{align}
On the other hand, using expression (\ref{basthm1.1}) at $r=\e
(<{\e}_1)$, dividing it by $\e$ and combining it with expression
(\ref{basthm1.2}), we get
\begin{align}
\bigg|\f{\Ff(x+\e v)-\Ff(x)}{\e}-\rz d\Ff_x(v)\bigg|<\f{\d}{2}
\label{basthm1.6}
\end{align}
Finally, using the triangle inequality with expressions
(\ref{basthm1.5}) and (\ref{basthm1.6}), we get that
\begin{align}
\big|d(^o\Ff)_x(v)-\rz d\Ff_x(v)\big|<\d+\eta.\label{basthm1.7}
\end{align}
But in this inequality, the left side is independent of the
arbitrarily chosen standard positive number $\d$, that is,
$|\;d(^o\Ff)_x(v)-\rz d\Ff_x(v)|\sim 0$; which is the same as saying
$^o(\rz d\Ff_x(v))=d(^o\Ff)_x(v)$.
But the standard map $^o(\rz d\Ff):\bbr^m\x\bbr^m\ra\bbr^n$ is defined by
 $^o(\rz d\Ff)_x(v)\doteq^o(\rz d\Ff_x(v))$, ie., as standard maps
$:\bbr^m\x\bbr^m\ra\bbr^n$, $^o(\rz d\Ff)=d(^o\Ff)$ as we wanted.

Finally, we will prove the continuity of the map $x\ra\;^o(\rz d\Ff)_x$. It suffices to verify that if $\xi,\z\in\rz U_{nes}$ with $\xi\sim \z$, then $\rz(d(^o\Ff))_\xi\sim \rz(d(^o\Ff))_\z$. But note that we have proved that $^o(\rz d\Ff)=d(^o\Ff)$ and so, transferring, we have $\rz(^o(\rz d\Ff))=\rz(d(^o\Ff))$. Given this, it is sufficient to prove that $\rz(^o(\rz d\Ff))_\xi\sim\rz(^o(\rz d\Ff))_\z$.
But we know that the map $\Fv\mapsto\Fh(\Fv)\dot=\rz d\Ff_\Fv$ is S-continuous. It is basic in NSA that this implies two things: first, we have for all $\Fv,\Fw\in\rz U_{nes}$ with $\Fv\sim\Fw$, $\Fh(\Fv)\sim\Fh(\Fw)$;   and second, for all $\Fv\in\rz U_{nes}$, we have $\rz(^o\Fh)(\Fv)\sim\Fh(\Fv)$. Applying both of these statements to $\Fh(\Fv)\dot=\rz d\Ff_\Fv$, we have
\begin{align}
   \rz(^o(\rz d\Ff))_\xi\sim\rz d\Ff_\xi\sim\rz d\Ff_\z\sim\rz(^o(\rz d\Ff))_\z
\end{align}
   as we wanted; finishing the proof of the three assertions of the proposition.
   For the result on the coordinate derivatives, since we have $\bsm{(\diamondsuit)}$: $d(^o\Ff)_x=\;^o(\rz d\Ff)_x$ for each $x\in U$, and $\xi\mapsto\rz d\Ff_\xi$ is S-continuous and so $\bsm(\sharp)$: $\xi\mapsto\rz d\Ff_\xi(\rz\f{\p}{\p x_j})=\rz\f{\p}{\p x_j}(\Ff)(\xi)$ is S-continuous. With this we have
\begin{align}
   \f{\p}{\p x_j}\big(^o\Ff\big)(x)\stackrel{1}{=}d(^o\Ff)_x\big(\f{\p}{\p x_j}\big)\stackrel{2}{=}\;^o(\rz d\Ff)_{*x}\big(\f{\p}{\p x_j}\big)\stackrel{3}{=}\qquad\qquad\qquad \notag\\
    ^o\Big(\rz d\Ff_{*x}(\rz\f{\p}{\p x_j})\Big)\stackrel{4}{=}\;^o\Big(\rz\f{\p}{\p x_j}\big(\Ff\big)(\rz x)\Big)\stackrel{5}{=}\;^o\Big(\rz\f{\p}{\p x_j}\big(\Ff\big)\Big)(x)
\end{align}
   where equality (1) is basic vector calculus, (4) is its transfer, (2) is the above formula $(\diamondsuit)$ we just proved and both (3) and (5) follow from the S-continuity expressed in $(\sharp)$ above.
\end{proof}

Given the previous proposition, we can now prove the main theorem of this appendix.

 \subsection{ Basic S-smooth regularity theorem and consequences}                            
   The representation result for internal S-smooth maps and some related results are contained in the following theorem.
  \begin{thm}\label{thmbasreg}
  Let $\Ff:\rz U_{nes}\ra\rz\bbr^n$ be an internal map and $k\in\bbn\cup\{\infty\}$. Then the following statements hold.
\begin{enumerate}
  \item If $\Ff\in SC^k(U,\bbr^n)$, then $^o\Ff$ exists and is in
  $C^k(U,\bbr^n)$.\\Furthermore, if $\a$ is a $^\s$finite multiindex, with $|\a|\leq k$ then
  $^o(\rz\p^{\a}\Ff)=\p^{\a}(^o\Ff)$ on $U$.
  \item If $\Ff,\Fg\in SC^k(U,\bbr^n)$, and $\Ff(\xi)\sim \Fg(\xi)$ for all $\xi\!\in\!\rz U_{nes}$,
  then $\rz\p^\a(\Ff)(\xi)\sim\rz\p^\a(\Fg)(\xi)$ for all\; ${}^\s$\!\! finite
  multiindices \;$\a$ with $|\a|\leq k$ and $\xi\in\rz U_{nes}$.
  \item If $\Ff\in SC^k(U,\bbr^n)$, then $\rz\p^\a(\Ff)(\xi)\sim\rz\p^\a(\rz(^o\Ff))(\xi)$ for all $\a$ with $|\a|\leq k$ and $\xi\in\rz U_{nes}$.

\end{enumerate}
  \end {thm}

  \begin{proof}
  We shall show that 2) follows easily from 1) and   3)  from 1) and 2).

  As these assertions are  true if and only if they hold for points infinitesimally close to some point of $U$,
  then it suffices to prove these when
  the domain is $\rz\bbr^m_{nes}$.
  To verify 2), let $\Fh=\Ff-\Fg$, then $\Fh\in SC^k$ and $\Fh\sim 0$, ie., $^o\Fh=0$, eg., $\p^\a(^o\Fh)=0$ for all multiindices. In particular, using 1), and as eg., $\rz\p^\a\Ff$ and $\rz\p^\a\Fg$ are nearstandard, we see that this implies that, for $|\a|\leq k$, that
\begin{align}
  0=\;^o(\rz\p^\a\Fh)=\;^o(\rz\p^\a\Ff-\rz\p^\a\Fg)=\;^o(\rz\p^\a\Ff)-\;^o(\rz\p^\a\Fg),
\end{align}
  which implies 2).
  So now let us verify 3) assuming that 1) and 2) holds. We know that if
  $\Ff\in SC^0$, then $\Ff\sim\rz(^o \Ff)$. But 1) implies that
\begin{align}
  \Ff\in SC^k\Rightarrow\;^o \!\Ff\in C^k\Rightarrow
  \rz(^o\Ff)\in SC^k.
\end{align}
  Therefore 2) implies that for
   multiindices $\a$ with $|\a|\leq k$ that
  $\rz\p^{\a}\Ff\sim\rz\p^{\a}(\rz(^o\Ff))$; that is\quad
  $\rz\p^{\a}(\Ff-\rz(^o\Ff))\sim 0$ as we wanted.


To prove 1),
  we need the above technical result, which is essentially the induction step in the proof of 1) and the lemma below.
  The previous proposition gives statement (1) for $k=1$. We will first verify that $\Fg\in SC^k$ implies that $^o\Fg\in C^k$. Inductively assuming that we have the statement for all $k$ up to some value $l\in\bbn$, we will verify the statement for $k=l+1$. Supposing  that $\Fg\in SC^{l+1}(U)$, we know that this is equivalent to having the map $\xi\mapsto\rz d\Fg_\xi$ in $SC^l$.   Writing $\rz d\Fg_\xi=\rz\!\f{\p}{\p x_1}\Fg(\xi)\rz dx_1+\cdots+ \rz\!\f{\p}{\p x_n}\Fg(\xi)\rz dx_n$, we see that this is equivalent to $\xi\mapsto\rz\f{\p}{\p x_j}\Fg(\xi)$ being $SC^l$ for all $j$. By the induction hypothesis, this implies that $^o(\rz\f{\p}{\p x_j}\Fg)\in C^k$. But $l\geq 1$, and so proposition \ref{approxprop} implies that $^o(\rz\f{\p}{\p x_j}\Fg)=\f{\p}{\p x_j}(^o\Fg)$, and so we have that $x\mapsto\f{\p}{\p x_j}(^o\Fg)(x)\in C^l(U)$ for all $j$, which therefore implies that $^o\Fg\in C^{l+1}$, completing the induction step.
  Finally, we need to show that taking standard parts intertwines internal and standard partial derivatives. Again we will verify this by induction: we have the $k=1$ case in the proposition; suppose that we have the result up to $k=l$ for some $l\geq 1$ and we need to verify it for $k=l+1$. To this end, suppose that  $\Fg\in SC^{l+1}(U)$ and let $\b$ be a multiindex with $|\b|\leq l+1$. Writing $\b=\a_j$ where $|\a|\leq l$, and letting $\p^j$ denote $\f{\p}{\p x_j}$, we have the following.
\begin{align}
   ^o(\rz\p^\b\Fg)\stackrel{a}{=}\;^o(\rz\p^j(\rz\p^\a\Fg))\stackrel{b}{=}  \p^j(^o(\rz\p^\a\Fg))\stackrel{c}{=}\p^j(\p^\a(^o\Fg))\stackrel{d}{=}\p^\b(^o\Fg)
\end{align}
   where equalities (a) and (d) are obvious, equality (b) follows from the fact that as $|\a|\leq l$, then $\rz\p^\a\Fg$ is $SC^1$ and equality (c) follows from the fact that $\Fg$ is, in particular, in $SC^l$, $|\a|\leq l$ and the induction hypothesis.
\end{proof}

Before we proceed to a nonstandard corollary, let's demonstrate how this nonstandard theorem about individual mappings has standard consequences about asymptotic properties of families of maps. The next corollary is a consequence of the third statement of the theorem. It is close to (note we are working on an open set!) a $C^k$ version of a classic result on equicontinuous (as usually defined) families of continuous maps. For the closest nonstandard rendition, see \cite{StrLux76}, chapter 8.4. The corollary following this, which is a consequence of the second statement of the theorem, as stated, may be new. It's important to note at this point that theorem \ref{thmbasreg} is much more effective in this paper in its nonstandard form.

   Before we proceed to the corollaries, we will give some lemmas that simplifies a step in both corollaries.
   The first lemma, although a conceptually simple result, has a tedious proof. But it will streamline the proof of the second lemma and allow weaker hypotheses in the following two corollaries.
\begin{lem}\label{lem: SC^0,connected, finite at pt->bdd}
    Suppose that $U\subset\bbr^n$ is open and connected and $\Ff:\rz U\ra\rz\bbr$ is an internal function with the following properties. For each $\xi,\z\in\rz U_{nes}$ with $\rz|\xi-\z|\sim 0$, we have $\rz|\Ff(\xi)-\Ff(\z)|\sim 0$ (SC-criterion) and there is $\xi_0\in\rz U_{nes}$ with $\Ff(\xi_0)\in\rz\bbr_{nes}$. Then $\Ff$ is finite on $\rz U_{nes}$ and therefore S-continuous on $\rz U_{nes}$.
\end{lem}
\begin{proof}
    First, as $\xi_0\in\rz U_{nes}$, then for some $x_0\in U$, $\xi\sim\rz x_0$ and so the hypothesis implies that $\Ff(\xi_0)$ is finite if and only if $\Ff(\rz x)$ is finite; so we may assume that $\xi_0$ is standard and in fact, we can take it to be $0$ by composing with a standard translation.
    Next, it suffices to verify the conclusion for (the transfer of) connected, compact subsets of $U$.  Finally, it suffices to prove the result for $K$ a closed ball $B_r$ of radius $r\in\bbr_+$ centered at $0$, for one can cover $K$ by overlapping balls (contained in $U$) such that $|\Ff|$ can change a finite amount on each ball. Given this, suppose that we have $\Ff:\rz B_r\ra\rz\bbr$ with $\Ff(0)\in\rz\bbr_{nes}$ and such that, for each $\xi,\z\in\rz B_r$ with $|\xi-\z|\sim 0$, we have $|\Ff(\xi)-\Ff(\z)|\sim 0$. Assuming also that $\Ff(0)=0$ (by adding a finite constant to $\Ff$), suppose that the conclusion does not hold; ie., there is $\Fv_0\in\rz B_r$ with $\la_0\dot=|\Ff(\Fv_0)|$ infinite. So now define the internal $\Fg:\rz[0,1]\ra\rz\bbr_+$ by $\Fg(\Ft)=\rz\sup\{|\Ff(\Fr\Fv_0)|:0\leq\Fr\leq\Ft\}$. First, the hypothesis on $\Ff$ implies that if $\Ft_1,\Ft_2\in\rz\bbr_+$ with $0\leq\Ft_1<\Ft_2\leq 1$, then $\Fg(\Ft_1)\sim\Fg(\Ft_2)$. For, by definition, we have $\Fg(\Ft_1)\leq\Fg(\Ft_2)$ and on the other hand, we have $\Ff(\Ft_1\Fv_0)\sim\Ff(\Ft_2\Fv_0)$. Now let $\om\in\rz\bbn_\infty$ be the largest integer with $\om\leq\la_0$ and for $\Fj=1,\ldots,\om$, let $\a_\Fj=\Fj/\om$. Using the monotonicity of $\Fg$, we have
\begin{align}
    \la_0=\Fg(1)-\Fg(0)=\rz\sum_{\Fj=1}^\om\big(\Fg(\a_\Fj)-\Fg(\a_{\Fj-1})\big)\leq\om\cdot\rz\max\{\Fg(\a_\Fj)-\Fg(\a_{\Fj-1}):1\leq\Fj\leq\om\}.
\end{align}
    That is, $1\sim \la_0/\om\leq\rz\max\{\Fg(\a_\Fj)-\Fg(\a_{\Fj-1}):1\leq\Fj\leq\om\}$, eg., there is a $\Fj_0$ such that $\Fg(\a_{\Fj_0})-\Fg(\a_{\Fj_0-1})$ is noninfinitesimal, a contradiction.
\end{proof}
\begin{lem}\label{lem: basic equicont families lemma}
    Let $U\subset\bbr^n$ be open and connected, and suppose that $\{h_j\in C^0(U,\bbr):j\in J\}$ is an equicontinuous family with the property that there is $c\in\bbr_+$ and compact $B\subset U$ such that for each $j\in J$, there is $x_j\in B$ with $|h_j(x_j)|<c$. Then, for $\om\in\rz\bbn_\infty$, we have that $\rz h_\om$ is S-continuous.
\end{lem}
\begin{proof}
    That $\rz h_\om$ is S-continuous is basic in the NSA literature; but let's give a proof from basics. $\{h_j:j\in\bbn\}$ is equicontinuous means (in the usual rendering) that for each $x\in U$ and $r\in\bbr_+$, there is $s\in\bbr_+$ so that if $y\in U$ is such that $|y-x|<s$, then for all $j\in\bbn$, we have $|h_j(y)-h_j(x)|<r$. So we have a family of $(x,r,s)$ statements $S(x,r,s)$. Transferring each of these separately we get the following statement. $\bsm{\ST}$: For each $x\in U$ and $r\in\bbr_+$, there is $s\in\bbr_+$ such that if $\xi\in\rz U$ and $|\xi-\rz x|<\rz s$, then, for all $\la\in\rz\bbn$,  $|\rz h_\la(\xi)-\rz h_\la(\rz x)|<\rz r$. Given this, let $\e\in\rz\bbr_+$ be infinitesimal and consider the set $\FA=\{\d\in\rz\bbr_+:\;\text{if}\;|\xi-\rz x|<\e, \;\text{then}\;|\rz h_\la(\xi)-\rz h_\la(\rz x)|<\d\}$. Now $\FA$ is internal and contains arbitrarily small standard $\rz r$'s (as for any $r\in\bbr_+$, the corresponding $s\in\bbr_+$ in statement $\FA$ is bigger than $\e$). Therefore, underflow implies that $\FA$ contains infinitesimals, proving the SC-criterion (previous lemma) for each $\rz h_\la$ for $\la\in\rz\bbn$. But note, eg., that our boundedness hypothesis implies that there is  $\xi\in\rz U_{nes}$ with $\rz h_\la(\xi)$  finite  and therefore as  $\rz h_\la$ satisfies the SC-criterion  on $\xi\in\rz U_{nes}$, the previous lemma implies that $\rz h_\la$ is S-continuous on $\rz U_{nes}$.
\end{proof}
\begin{cor}\label{cor: 1st stan cor of appendix thm}
    Suppose that $\SF=\{f_l:l\in L\}$ is a family of functions in $C^k(U)$ satisfying the following. For each multiindex $\a$ with $|\a|\leq k$, the family $\{\p^\a f_l:l\in L\}$ is equicontinuous.
    Suppose also that there is $c\in\bbr_+$, a compact $K\subset U$ and $x_l\in K$ for each $l\in L$, such that $|\p^\a f_l(x_l)|<c$ for all $l\in L$ and $\a$ with $|\a|=k$.
    Then there is a sequence $\SS=\{f_i:i\in\bbn\}\subset\SF$ and $g\in C^k(U)$ such that for each $r>0$,  $x\in U$, we have that there is $j_0=j_0(r,x)\in\bbn$ with $|\p^\a f_j(x)-\p^\a g(x)|<r$ for all $\a$ with $|\a|\leq k$ and $j\geq j_0$.
\end{cor}
\begin{proof}
   The result will follow  from statement (1) of the above theorem and the above lemma. First, note that the condition on the $\p^\a f_l$'s with respect to the sequence of points $x_l$ implies that for $\Fl\in\rz L$, we have $\rz x_\Fl\in\rz K\subset\rz U_{nes}$, so that for each $\a$, $\p^\a$  The previous lemma applied, for each $\a$ with $|\a|\leq k$, to the the families $\{\p^\a f_l:l\in L\}$ implies that for $\Fl\in\rz L\smallsetminus\;^\s L$, we have that for each of these $\a$'s, $\rz\p^\a f_\Fl\in SC^0(U)$, ie., by definition  $\rz f_\Fl\in SC^k(U)$. But then statement (3) of the theorem implies the statement $\bsm{(\natural)}$:  $g=\;^o(\rz f_\Fl)$ is a well defined element of $C^k(U)$ and for all $\xi\in\rz U_{nes}$ and $\a$ with $|\a|\leq k$ we have $\rz\p^\a(\rz f_\Fl)(\xi)\sim\rz (\p^\a g)(\xi)$. Given this, let $K_j$ for $j\in\bbn$ be an increasing sequence of compact subsets of $U$ satisfying $U=\cup\{K_j:j\in\bbn\}$. For $j\in\bbn$, define the following  subset of $\SF$.
\begin{align}
   \SE_j=\{f\in\SF: |\p^\a f(x)-\p^\a g(x)|<1/j \;\text{for all}\;x\in K_j,\;\text{for all}\;\a\;\text{with}\;|\a|\leq k\}.
\end{align}
  We will show that $\SE_j$ is nonempty by verifying that $\rz\SE_j$ is nonempty. Now, by transfer, $\Ff\in\rz\SF$ is in $\rz\SE_j$ precisely if
\begin{align}
   |\rz\p^\a\Ff(\xi)-\rz(\p^\a g)(\xi)|<\rz\f{1}{j}\;\text{for all}\;\xi\in\rz K_j,\;\text{for all}\;\a\;\text{with}\;|\a|\leq k.
\end{align}
   But the statement  $(\natural)$ above implies that $\rz f_\Fl$ satisfies these properties as $\rz K\subset\rz U_{nes}$ and so eg., $\rz\SE_j$ is nonempty, hence $\SE_j$ is nonempty, ie., for each $j\in\bbn$, there is  $f_j\in\SF$ that is in $\SE_j$ . So given $r>0$ and $x\in U$, then first as the $K_j$'s are nondecreasing and $\cup\{K_j:j\in\bbn\}=U$, there is $j_1\in\bbn$ such that $x\in K_j$ for all $j\geq j_1$ and so choosing $j_0>\max\{j_1,1/r\}$ we have our assertion.
\end{proof}

  We now consider a second corollary to the above theorem. As far as the author can tell, this result, although of a basic nature, seems to be new.
\begin{cor}\label{cor: 2nd stan cor of appendix thm}
    Let $U\subset\bbr^n$ be connected and open. Suppose that for $j\in\bbn$, $\SF=\{f_j\}$ and $\SG=\{g_j\}$ are two sequences in $C^k(U)$ with the following properties.
    For each multiindex $\a$ with $|\a|\leq k$, the sequences $\{\p^\a f_i:i\in\bbn\}$ and $\{\p^\a g_i:i\in\bbn\}$ are equicontinuous sequences on $U$ and there is compact $B\subset U$, $c\in\bbr_+$ and  $x_{j,\;\a}\in B$ such that $\sup\{|\p^\a f_j(x_{j,\a})-\p^\a g_j(x_{j,\a})|:j\in\bbn,|\a|\leq k\}<c$.
    Given this, if, for each $r>0$ and $x\in U$ there $i=i_{x,\;r}\in\bbn$ such that $|f_i(x)-g_i(x)|<r$ for $i\geq i_{x,\;r}$,
    then, for each $s>0$ and $x\in U$, there is $j_{x,r}\in\bbn$ such that $|\p^\a f_j(x)-\p^\a g_j(x)|<s$  for every multiindex $\a$ with $|\a|\leq k$, for all $j\geq j_{x,r}$.
\end{cor}
\begin{proof}
   First of all, if we let $h_i=f_i-g_i$, then the above hypotheses imply that
   for each multiindex $\a$ with $|\a|\leq k$, the sequence $\{\p^\a h_i:i\in\bbn\}$ is  equicontinuous on $U$ and is appropriately bounded on a good sequence of points and so lemma \ref{lem: basic equicont families lemma} implies that  we can restate the first part of the hypothesis as statement \textbf{(SC)}: for each infinite $\la\in\rz\bbn$, $\rz h_\la\in SC^k(U)$.
   Now the second part of the hypothesis can be restated  as follows: for each $r>0$ and $x\in U$, there is $i_0=i_0(x,r)\in\bbn$ such that $|h_i(x)|<r$ for all $i\geq i_0$. Given this, if we can verify that for each $s>0$ and $x\in U$, there is $j_0\in\bbn$ with $|\p^\a h_j(x)|<s$ for each  multiindex $\a$ with $|\a|\leq k$ and all $j\geq j_0$, then the conclusion will clearly follow.

   Now, as the $h_i,i=1,2,\ldots$ form an equicontinuous sequence, a version of the second part of lemma \ref{lem: basic equicont families lemma} above implies that the second part of the hypothesis can be restated as \textbf{(2S)}: for each $r>0$ and compact $K\subset U$, there is $i_0\in\bbn$ such that we have that $|h_i(x)|<r$ for all $i\geq i_0$ and $x\in K$.
    But now fixing $r,K,i_0$ in (2S) and transferring we get statement $\bsm{(2N-r,K,i_0)}$: $|\rz h_\la(\xi)|<\rz r$ for all $\la\in\rz\bbn$ with $\la>\rz i_0$ and $\xi\in\rz K$. In particular, if $\la\geq\om\in\rz\bbn$ is infinite, we have that (2N-$r,K,i_0$) holds for all $r\in\bbr_+$ and compact $K\subset U$, eg., $\rz h_\la(\xi)\sim 0$ for all $\xi\in\rz U_{nes}$ and $\la\geq\om$. This, along with statement (SC), gives the hypotheses for statement (3) in the theorem above for all $\la\geq\om$. That is, we have the fact \textbf{(HT)}: $\rz\p^\a(\rz h_{\om'})(\xi)\sim 0$ for all $\xi\in\rz U_{nes}$ and $|\a|\leq k$ and $\om'\geq \om$. Given this, define for each $r\in\bbr_+$ and compact $K\subset U$ the following set:
\begin{align}
   \SB_{r,\;K}=\{j\in\bbn:\;\text{for all}\;j'\geq j, \text{for all}\;\a\;\text{with}\;|\a|\leq k, |\p^\a h_{j'}(x)|<r\;\text{for all}\;x\in K\}.
\end{align}
   Then, for each $r>0$ and compact $K\subset U$ it's clear from (HT) that $\om\in\rz\SB_{r,\;K}$, eg., $\rz\SB_{r,\;K}$ is nonempty and so by reverse transfer, $\SB_{r,\;K}$ is nonempty.   So  let $x\in U$ and $s>0$, then there is compact $K\subset U$ such that $x\in K$ and we have just verified that $\SB_{s,\;K}$ is nonempty, ie., by definition of $\SB_{s,\;K}$, there is $j_0\in\bbn$ such that $x\in K$ and so we have that $j\geq j_0$ implies $|\p^\a h_j(x)|<s$ for all $\a$ with $|\a|\leq k$.
\end{proof}
\subsection{dS-smoothness}

Returning to the nonstandard developments, we have some material relating internal functions of apparently different S-regularities. We begin with a definition.
Note that it's easy to see that a $\Fg\in SC^0(U)$ generally will not satisfy $\Fg\sim\Ff$ for any $\Ff\in SC^k(U)$, eg.,  for $\Ff\in\;^\s C^k(U) $ when $\Fg\notin SC^k(U)$. But, it is true that all standard iterated difference quotients up to degree $k$ must eventually approximate an $SC^0$ function on $\rz U_{nes}$. The corollary below gives a statement of this. We need to first give a definition of this (external) class of internal functions. Note that we are using the recipe discussed in subsection \ref{subsec: notion of S-property} and plays a role in the corollary, \ref{cor: NSA reg+almost->near}, that is the nonstandard version of the standard principal result that follows it.
\begin{definition}\label{def: dSC^k}
    For $k\in\bbn\cup\{\infty\}$, we will define the (external) class of internal functions $dSC^k(U)$ as follows. We say that $\Ff:\rz U\ra\bbr$ belongs to $dSC^1(U)$ if for each $x\in U$ there is an infinitesimal $\e_0\in\rz\bbr_+$ and a nearstandard  internal linear $\SL_\xi:\rz\bbr^n\ra\rz\bbr$ such that $\xi\mapsto \SL_\xi$ is S-continuous on $\rz U_{nes}$ the following is satisfied:
\begin{align}
    \bigg|\f{\Ff(\xi+\e\Fv)-\Ff(\xi)}{\e}-\SL_\xi(\Fv)\bigg|\sim 0\;\text{for all}\;\e_0<\e\sim 0\;\text{and all}\;\Fv\;\text{with}\;|\Fv|=1.
\end{align}
    More generally, we say that $\Ff\in dSC^k(U)$ if there are nearstandard $j$-*multilinear symmetric maps $\SL^j_\xi:\rz\bbr^n\x\cdots\x\rz\bbr^n\ra\rz\bbr$ for $j=1,\ldots,k$ such that $\xi\mapsto \SL^j_\xi$ is S-continuous on $\rz U_{nes}$ and there is an infinitesimal $\e_0\in\rz\bbr_+$ such that the following holds
\begin{align}\label{eqn: def expression for f in dSC^k}
    \text{for every}\;\Fv\in\rz\bbr^n\;\text{with}\;|\Fv|=1,\;\text{for every}\;\e\;\text{with}\;\e_0<\e\sim 0 \qquad\qquad\notag\\
    \f{1}{\e^k}\bigg| \Ff(\xi+\e\Fv)-\Ff(\xi)-\sum_{j=1}^k\e^j\SL^j_\xi(\Fv,\ldots,\Fv)\bigg|\sim 0\qquad
\end{align}
\end{definition}
   Note that  $^\s C^k(U)\subset SC^k(U)\subsetneqq dSC^k(U)$ but $dSC^k(U)\nsubseteq\rz C^k(U)$ as elements of $dSC^k(U)$ are not necessarily internally differentiable, eg., again consider the infinitesimal saw tooth function. We will show that \textbf{$\bsm{dSC^k(U)}$ is the maximal class of pointwise nearstandard internal functions whose standard part is in $\bsm{C^k(U)}$}.
\begin{lem}\label{lem: dSC^k properties}
     \textbf{(1)}: The property of being in $dSC^k(U)$ is stable under perturbations. That is,  if\; $\Ff\in dSC^k(U)$ and $\Fg:\rz V\ra\rz\bbr$ is an internal map with $\Ff(\xi)\sim\Fg(\xi)$ for all $\xi\in\rz V_{nes}$, then $\Fg\in dSC^k(U)$.
     \textbf{ (2)}: If $f:U\ra\bbr$ is such that $\rz f\in dSC^k(U)$, then $f\in C^k(U)$.
    \textbf{ (3)}:  In particular, if\; $\Ff\in dSC^k(U)$, we have $^o\Ff\in C^k(U)$.
\end{lem}
\begin{proof}
    For assertion (1), note that we just need to prove the statement on $\rz K$ for $K\subset U$ compact as $x\in U$ implies that $\mu(x)\subset\rz K$ for some such $K$. Given such a $K$, then we have that, as $\rz K$ is internal, $\FI=\{|\Ff(\xi)-\Fg(\xi)|: \xi\in\rz K\}$ is internal and therefore if it contained arbitrarily large infinitesimals, overflow would imply that $\FI$  contains noninfinitesimals, contrary to hypothesis; hence $\rz\sup\{|\Ff(\xi)-\Fg(\xi)|:\xi\in\rz K\}=\Fr_0\sim 0$. So now  choose our infinitesimal $\e_0$ in the definition that $\Ff\in dSC^k(U)$ above so that $\e_0^k/\Fr_0\sim 0$, ie., so that
\begin{align}\label{eqn: |f-g|<<e^k on K}
    \f{|\Ff(\z)-\Fg(\z)|}{\e_0^k}\sim 0\;\text{for all}\;\z\in\rz K.
\end{align}
     We will now apply the above expression for $\z=\xi$ and $\z=\xi+\e\Fv$. In expression \ref{eqn: def expression for f in dSC^k} let $T^k_\xi\Ff(\d,\Fv)$ denote the higher `Taylor' expansion for $\Ff$, ie., the summed term. in expression \ref{eqn: def expression for f in dSC^k} (starting with first order part). Then, we have that
\begin{align}
     \f{1}{\e^k}\big|\Fg(\xi+\e\Fv)-\Fg(\xi)-T^k_\xi\Ff(\e,\Fv)\big|\leq \qquad\qquad\qquad\qquad\qquad\qquad\qquad\qquad\notag\\
     \f{|\Fg(\xi+\e\Fv)-\Ff(\xi+\e\Fv)|}{\e^k}+\f{|\Fg(\xi)-\Ff(\xi)|}{\e^k}+\f{1}{\e^k}\big|\Ff(\xi+\e\Fv)-\Ff(\xi)-T^k_\xi\Ff(\e,\Fv)\big|.
\end{align}
     We see that this is infinitesimal for all $\e$ with  $0\sim \e\geq\e_0$, since the first two terms in the second line are infinitesmal from expression \ref{eqn: |f-g|<<e^k on K} and the last by hypothesis.

        To prove the second assertion, assume that for $x\in U$, $\Ff$ has the expansion as in expression \ref{eqn: def expression for f in dSC^k}, with the nearstandard multilinear maps $\SL^j_\xi$, and define an internal function $B$ as follows:
\begin{align}
     B(\e)\;\dot=\;\f{1}{\e^k}\bigg| \rz f(\xi+\e\Fv)-\rz f(\xi)-\sum_{j=1}^k\e^j\SL^j_\xi(\Fv,\ldots,\Fv)\bigg|.
\end{align}
     Now $B$ is an internal function on $\rz\bbr_+$  which, because $\rz f\in dSC^k(U)$, we have that for $\e_0\leq \e\sim 0$ , $B(\e)\sim 0$. Let $r\in\bbr$ be arbitrary positive, and define
\begin{align}
     \SB\;\dot=\;\{\e'\in \rz [\e_0,1]:\;\text{for all}\; \e\in [\e_0,\e'],\;B(\e)<r/2\}.
\end{align}
     Then all $\e'$ with $\e_0\leq\e'\sim 0$ is in $\SB$ and so by overflow, there is a positive noninfinitesimal $\Fa$ in $\SB$, eg., by the definition of $\SB$, if $s=\;^o\Fa/2$, then $\rz s\in\SB$. Given this, we have that
\begin{align}
    \f{1}{\e^k}\bigg| \rz f(\xi+\e\Fv)-\rz f(\xi)-\sum_{j=1}^k\e^j\SL^j_\xi(\Fv,\ldots,\Fv)\bigg|<\rz r/2\;\text{for}\;\e_0\leq\e<\rz\!s.
\end{align}
    But, as $\SL^j_\xi$ is a nearstandard multilinear linear operator and so has a standard part (at each $x\in U$), we have a multilinear operator $L^j_x(v,\ldots,v)\doteq\;^o(\SL^j_{*x}(\rz v,\ldots,\rz v))$  for  $v\in \bbr^n$. Furthermore, we have for each $j$ that $|\SL^j_{*x}(\rz v,\ldots,\rz v)- \rz L^j_x(v,\ldots,v)|\sim 0$, so that, using the triangle inequality for each $\SL^j_{*x}$, for each $v\in\bbr^n$ with $|v|=1$, the \textbf{above inequality restricted to standard values} becomes:
\begin{align}
       \f{1}{\rz t^k}\bigg| \rz f(\rz x+\rz t\rz v)-\rz f(\rz x)-\sum_{j=1}^k\rz t^j\rz L^j_{\rz x}(\rz v,\ldots,\rz v)\bigg|\quad\leq\qquad\quad\qquad\qquad\qquad\qquad\notag\\
       \sum_{j=1}^k\rz t^{j-k}|\SL^j_{*x}(\rz v,\ldots,\rz v)- \rz L^j_x(v,\ldots,v)|\quad+\qquad\qquad\qquad\qquad\qquad\qquad\\
        \f{1}{\rz t^k}\bigg| \rz f(\rz x+\rz t\rz v)-\rz f(\rz x)-\sum_{j=1}^k\rz t^j\SL^j_{\rz x}(\rz v,\ldots,\rz v)\bigg|<\rz r\qquad\qquad\qquad\notag
\end{align}
     $\text{for}\;t\in\bbr_+,t<s$ as the middle term is infinitesimal. But this is just the statement
\begin{align}
     \f{1}{ t^k}\bigg|  f(x+ tv)- f(x)-\sum_{j=1}^k t^j L^j_{ x}(v,\ldots,v)\bigg|<r,\;\text{for}\;t\in\bbr_+,t<s
\end{align}
     which, as $x\mapsto L^j_x$ is continuous for each $j$, is the statement that $f\in C^k(U)$.

     To prove the third assertion, first note that as $\Ff\in SC^0(U)$, then if $f\dot=\;^o\Ff$, we have  $\Ff\sim\rz f$ and so the first assertion implies that $\rz f\in dSC^k(U)$ and so the second assertion implies that $f\in C^k(U)$.
\end{proof}
\begin{cor}\label{cor: SC^j sim SC^k}
    Suppose that $k\in\{0,1,2,\ldots,\infty\}$ and $U\subset\bbr^n$ is open. Let $\Ff:\rz U\ra\rz\bbr$ be an internal map and $\Fg\in SC^k(U,\bbr)$ with  $\Ff(\xi)\sim\Fg(\xi)$ for all $\xi\in\rz U_{nes}$. Then $\Ff\in dSC^k(U,\bbr)$; in particular, $SC^k(U)\subset dSC^k(U)$ as noted earlier.
\end{cor}
\begin{proof}
   By (3) of the above lemma, $f\dot=\;^o\Ff\in C^k(U)$; and so $\rz f\in dSC^k(U)$, but on $\rz U_{nes}$ we have that $\Fg\sim \Ff\sim\rz f$; ie., $\Fg\sim \rz f$ with $\rz f\in dSC^k(U)$ and so by (1) in the above lemma, $\Fg\in dSC^k(U)$.

\end{proof}
   We have the following (almost) standard corollary of the previous result. To shorten notation, let $S^{n-1}$ denote the set of $v\in\bbr^n$ with $|v|=1$.
\begin{cor}\label{cor: props of asympt dSC^k families }
    If $\Ff\in dSC^k(U)$ and $\FF=\{f_j:j\in\bbn\}$ is a sequence of continuous maps $f_j:U\ra \bbr$ satisfying $\rz f_\la(\xi)\sim\Ff(\xi)$ for all $\xi\in\rz U_{nes}$ and $\la\in\rz\bbn_\infty$, and if for each $j=1,\ldots,k$ and $x\in U$, $L^j_x$ is the standard part of the nearstandard (multi)linear map $\SL^j_{*x}$ in the expansion for $\Ff$ as in expression \ref{eqn: def expression for f in dSC^k}, then the following holds. For every $x\in U$ and $r\in\bbr_+$, there is $j_{x,r}\in\bbn$ and $s_{x,r}\in\bbr_+$ with the following properties:
\begin{align}
    \text{for every}\;j>j_{x,r},\;\text{positive}\;s<s_{x,r}\;\text{and vector}\;v\in S^{n-1}\qquad\qquad \notag\\
    \f{1}{s^k}\bigg| f_j(x+sv)-f_j(x)-\sum_{i=1}^k s^i L^i_x(v,\ldots,v)\bigg|<r
\end{align}
\end{cor}
\begin{proof}
   From lemma \ref{lem: dSC^k properties} we know that, for infinite $\la\in\rz\bbn$, we have $\rz f_\la\in dSC^k(U)$. To simplify notation here for $\la\in\rz\bbn,\;\xi\in\rz U_{nes},\;\d\in\rz\bbr_+,\;\Fv\in\rz\bbr^n$ with $|\Fv|=1$, we will let
\begin{align}
     R(\la,\xi,\d,\Fv)\;\dot=\;\f{1}{\d^k}\bigg|\rz f_\la(\xi+\d\Fv)-\rz f_\la(\xi)-\sum_{j=1}^k\d^j\SL^j_\xi(\Fv,\ldots,\Fv) \bigg|.
\end{align}
    where $\SL^j_\xi$ are the nearstandard internal multilinear maps corresponding to  the same notation $\Ff$ in expression \ref{eqn: def expression for f in dSC^k}.   Next fix an arbitrary compact $K\subset U$, and let
\begin{align}
   \SR_K(\la,\d)\;\dot=\;\rz\sup\{ R(\la,\xi,\d,\Fv):\xi\in\rz K\;\text{and}\;|\Fv|=1\}.
\end{align}
   Now it's easy to see that that there is $0<\d_0\sim 0$ such that for $\la\in\rz\bbn_\infty$, $\d_0\leq\d\sim 0$, $\xi\in\rz K$ and $\Fv\in\rz S^{n-1}$, we have that $R(\la,\xi,\d,\Fv)\sim 0$. This follows from the same argument as used in the proof of statement (1) of lemma \ref{lem: dSC^k properties}: use the $\rz f_\la$ in place of $\Fg$ and $\d_0$ in the place of $\e_0$ in expression \ref{eqn: |f-g|<<e^k on K}. But then, as $\rz K\x\rz S^{n-1}$ is internal, $\rz f_\la$ is internal and $R$ is internal function of $\xi\in\rz K$ and $\Fv\in\rz S^{n-1}$, then  the set $\FR\dot=\{R(\la,\xi,\d,\Fv):\xi\in\rz K\;\text{and}\;\Fv\in\rz S^{n-1}\}$ is internal and so if $\FR$ has arbitrarily large infinitesimals, overflow would imply that $\FR$ would contain noninfinitesimals. If this were true, then there must be $\xi_0\in\rz K$ and $\Fv_0\in\rz S^{n-1}$ such that $R(\la,\xi_0,\d,\Fv_0)\nsim 0$ contrary to the above statement. Hence, we have statement \textbf{(I)}: $\SR(\la,\d)$ is infinitesimal for infinite $\la\in\rz\bbn$ and positive infinitesimal (bigger than $\d_0$) in $\rz\bbr_+$.
    Given this, if $r\in\bbr_+$ and $K\subset U$ is compact, we will define the following internal set:
\begin{align}
       \SO_{r,\;K}\;\dot=\;\{(\la',\d')\in\rz\bbn\!\x\![\d_0,\rz\infty):\text{if}\;(\la,\d)\in [\la',\rz\infty)\!\x\![\d_0,\d'],\;\text{then}\;\SR_K(\la,\d)<\rz r/2\}.
\end{align}
     Now statement (I) above implies that if $\la\in\rz\bbn$ is infinite and $\d_0$ is a positive infinitesimal at least as big as $\d_0$, then $(\la,\d)\in\SO_{r,\;K}$, and this statement implies the following. If $[a]$ denote the integer part of $a\in\bbr_+$ and if $\la_0\in\rz\bbn$ is the infinite integer given by $\la_0=[1/\d_0]+1$, then $(\la,1/\la)\in\SO_{r,\;K}$ for infinite $\la\in\rz\bbn$ with $\la\leq \la_0$. That is, $\{\la\in\rz\bbn:(\la,1/\la)\in\SO_{K,\;r}\}$ contains the external set of infinite integers less than $\la_0$ and so, as $\SO_{K,\;r}$ is internal, contains finite integers, eg., $\rz\! j_0$ for some $j_0\in\bbn$. By the definition of $\SO_{K,\;r}$, this says that if $(\la,\d)\in \rz [j_0,\infty)\x[\d_0,\rz 1/j_0]$, then $\SR_K(\la,\d)<\rz r/2$.
      That is, if $\rz|\Fv|=1$ and $\xi\in\rz K$, then $R_K(\la,\xi,\d,\Fv)<\rz r/2$. In particular, this holds for standard values in these ranges, ie., if $v\in\bbr^n$  with $|v|=1$, $x\in K$, $j\in\bbn$ with $j\geq j_0$ and $s\in\bbr_+$ with $s\leq s_0\;\dot=1/j_0$, then we have fact \textbf{(B)}: $R_K(\rz j,\rz x,\rz s,\rz v)<\rz r/2$. But, as already argued twice $L^j=\;^o\SL^j$ satisfies \textbf{(C)}: $|\rz L^j_{*x}(\rz v)-\SL^j_{*x}(\rz v)|\sim 0$ for all $x\in K$. Using (B) and (C) along with the triangle inequality, we can weaken the bound some, say to $r$, to get
\begin{align}
   \f{1}{\rz t^k}\bigg| \rz f_{*j}(\rz x+\rz t\rz v)-\rz f_{*j}(\rz x)-\sum_{j=1}^k\rz t^j\rz L^j_{\rz x}(\rz v,\ldots,\rz v)\bigg|<\rz r.
\end{align}
    But everything is standard in this statement and we finish just as we finished the proof of the second statement of lemma \ref{lem: dSC^k properties}.
\end{proof}
    The previous work now has the completely standard consequence. Again, consider a sequence of saw tooth functions converging pointwise to a constant function!
\begin{cor}\label{cor: standard dSC^k result}
    Suppose that $f\in C^k(U)$ and $\FF=\{f_1,f_2,\ldots\}$ is a sequence in $C^0(U)$ such that for each $x\in U$, $f_j(x)\ra f(x)$ as $j\ra\infty$. If we denote the  $j$-multilinear differential of $f$ at $x$ by $L^j_x$, then we have the conclusion of the previous result. That is, for every $x\in U$ and $r\in\bbr_+$, there is $j_{x,r}\in\bbn$ and $s_{x,r}\in\bbr_+$ such that
\begin{align}
    \text{for every}\;j>j_{x,r},\;\text{positive}\;s<s_{x,r}\;\text{and vector}\;v\in S^{n-1}\qquad\qquad \notag\\
    \f{1}{s^k}\bigg| f_j(x+sv)-f_j(x)-\sum_{i=1}^k s^i L^i_x(v,\ldots,v)\bigg|<r.
\end{align}
\end{cor}
\begin{proof}
    As the hypothesis on $\FF$ implies that for infinite $\la\in\rz\bbn$, we have $\rz f_\la(\xi)\sim\rz f(\xi)$ for all $\xi\in\rz U_{nes}$ and as $^\s C^k(U)\subset dSC^k(U)$, the result follows from the previous proposition.
\end{proof}
\subsection{Functorial expression of S-smoothness}
 Here we will give easy consequences of Theorem \ref{thmbasreg}in terms of relationships between our canonical maps.

The following diagram is an immediate corollary of Theorem \ref{thmbasreg}.

\begin{cor}                                      
For every $^\s$finite multiindex $\a$, the following diagram of maps
is commutative:
\begin{align}\label{diag1}
    \begin{CD}
                \dscmn   @>\rz\p^{\a}>>    \dscmn  \\
                @V\frak{st}VV              @V\frak{st}VV\\
                \dcmn    @>\p^{\a}>>        \dcmn
    \end{CD}
\end{align}
 where $\frak{st}(f)\doteq\;^o\!f$.
\end{cor}

Let $\SJ^k_{m,n}$ denote the affine bundle of $k$ jets of maps in
\cmn. We have the usual source projections $\pi_k:\SJ^k_{
m,n}\ra\bbr^m$ and target $\rho:\SJ^k_{m,n}\ra\bbr^n$ projections.
Let $C^{\infty}(\SJ^k_{m,n})$ denote the \cm  module of smooth
sections of $\SJ^k_{m,n}$. Let ${\frak j}_k:\dcmn\ra
C^{\infty}(\SJ^k_{m,n})$, denote the $k$ jet operator, given by
sending $f\in\dcmn$ to the map $x\mapsto j^k_xf$. Now *transfer this
setup.

\begin{cor}                                              
${\frak j}_k:\dscmn\ra  SC^{\infty}(\SJ^k_{m,n})$ satisfies
$\frak{st}\circ\rz{\frak j}_k={\frak j}_k\circ \frak{st}$, ie., we have
an abelian diagram
\begin{align}
\begin{CD}      SC^{\infty}(\SJ^k_{m,n})   @>\frak{st}>>    C^{\infty}(\SJ^k_{m,n})  \\
                @A{\rz\frak j}_kAA              @A{\frak j}_kAA\\
                \dscmn    @>\frak{st}>>        \dcmn
\end{CD}
\end{align}
\end{cor}
\begin{proof}
 This is just the jet version of the previous
corollary.
\end{proof}

\section{Good Hausdorff topologies on families of map germs}\label{chap: topologizing germs}
\subsection{Introduction}\label{sec: intro: top on map germs}
    Before we begin, we should once more point out that there is an updated version of this chapter on the arXiv, \cite{McGafGerms2012arXiv1206.0473M}. It includes simplified proofs of some of the early parts of this chapter as well as proof that the group of germs of homeomorphisms $(\bbr^n,0)\ra(\bbr^n,0)$ is a topological group, a thorough standard rendition of this topology and a classical (but surprising) topological setting for our topology. Nonetheless, although this chapter is somewhat rough around the edges in places, we believe that the hardy field constructions in section \ref{subsec: top independence from delta}  and the part on relationship with nongerm  convergence in \ref{sec: relationship with nongerm convergence} are important. There is a third version that includes other material that will appear in time.

    It is commonly believed that one cannot construct a nondiscrete ``good'' topology on the ring of germs at $0$ of smooth real valued functions on $\bbr^n$, much less the ring of germs of continuous real valued functions on $\bbr^n$.
    For example, Gromov, \cite{Gromov1986} remarks (p.36) that ``There is no useful topology in this space...of germs of [$C^k$] sections...'' over a particular set.
    Furthermore, there are hints in the literature, for example in the work of eg., Du Pleisses and Wall, \cite{DuPlessisWall1995},  on topological stability, see p.95 and chapter 5 (p.121-) on the great difficulties of working with germ representatives with respect to aspects of smooth topology (eg., how to define the stability of germs), but that there are no alternatives, eg., in working with the germs directly.
        In this chapter, using nonstandard analysis, we will give a construction of a good Hausdorff topology on the ring of germs of real valued functions on $\bbr^n$ at $0$ that has good convergence properties.

    More specifically, we give a construction of a nonmetrizable Hausdorf topology on the ring of real valued germs on $\bbr^n$ at $0$ that has the following properties:  Net convergence is akin to uniform convergence of continuous functions in the sense that a convergent net of germs of continuous functions has limit the germ of a continuous function. Moreover, germ composition is a continuous map with composition on the right by germs of homeomorphisms giving topological ring isomorphisms. In one very special instance, we show how this topology extends the usual norm topology if the `germs' come from functions all with a common domain. For example, we give theorems relating types of convergence of a family of functions all defined in a given ball to our germ convergence of the *finite extensions  of these families restricted to germ domains. As time allows, we will extend this work; eg., we will extend this framework to the context of the orbit space of the action of the topological group of homeomorphisms germs acting on locally Euclidean topological groups. The overall  intention  is to develop what might be called a categorical framework for germ topologies.

    Our constructions rely critically on nonstandard methods. To have some chance of success, we needed the following critical facts to make these results possible. First, the algebra of germs at $0$ is canonically isomorphic (via the domain restriction map) to the external algebra of standard functions on any infinitesimal ball about $0$, see corollary \ref{cor: SG_0 -> F(B_delts) is R alg isomorph}.  Second, we need that the germ topology be defined in terms of these nonstandard algebras of standard functions on these infinitesimal balls, see definition \ref{def: of topology  at 0 germ} on page \pageref{def: of topology  at 0 germ}. Third, we have a criterion, in the context of these  functions restricted to these infinitesimal balls,  to determine those germs that are germs of continuous functions, see proposition \ref{prop: germ continuity from k<<<d}  on page \pageref{prop: germ continuity from k<<<d}.

    Let us next summarize the strategies and results here. Our topologies are simply defined as norm topologies on infinitesimal balls. But unlike the standard case, it's  a serious problem \textbf{(1)}: to determine a choice of a family of bounds $\{\a\}$, so that as $\|\rz f\|$ is controlled by these specific values, we get a good notion of nearness. The second problem is, \textbf{(2)}: we want this topology to have good convergence properties, eg., we want a convergent net of continuous germs to have a continuous germ as the limit point.   The third problem is,   that $\|\;\|$ is defined as a norm over a ball with radius some positive infinitesimal $\d$, symbolically: $\rz\|\;\|_\d$, and thus we want \textbf{(3)}: this topology to be independent of the choice of this infinitesimal. Although, we don't, as yet, have a standard rendering of this topology, we still want \textbf{(4)}: to find relationships with standard convergence results. Finally, we hope that \textbf{(5)}: this topology has good properties with respect to ring operations and composition of germs.

    Let's describe how we solve all these problems.  First, problems (1) and (2) are intertwined and are solved in section \ref{sec: various radii, converg of C^0 is C^0} and section \ref{sec: topology-fixed target} up to page \pageref{subsec: top properties ring structure}. Simply posed: to get a good set of distances for $\rz\|\;\|_\d$,  we must, in fact,  define multiple families  of infinitesimals, all related to the $\SN_\d$ families, see eg., \pageref{def: SN_delta families of numbers} and show that they define ``equivalent'' sets of distances (this is defined in terms of several forms of the notion ``coinitial'' in the text), analogous to the trivial standard fact that eg., $1/2,1/3,1/4,\ldots$ and $1/4,1/9,1/16,\ldots$ form equivalent sets of distances for, say, the $\sup$ norm for the continuous functions on the unit ball (see the abstract definition on page \pageref{def: of coinitial partial order relations}). (Note that, unlike in such a simple example, our set of distances cannot be countable, eg',  is not metrizable;  see  corollary \ref{cor: countable nets can't converge in tau} and especially the construction in the previous lemma.) But, it turns out that for these infinitesimal balls, the families $\SN_\d$ do not have the right properties to prove that a limit point of a convergent net of continuous germs is a continuous germ. Therefore, we must define other families of infinitesimals, the various $\ov{\SS}^\d$ families (one, $\wh{\SS}^{\k,\d}$, which depends on two, incomparable, infinitesimals, again see page \ref{def: SS^delta families of numbers}), as well as the notion of $[f]$-good infinitesimals (see definition \ref{def: e is [f]-good}), specifically for this purpose. Essentially, convergence akin to uniform convergence is implied by the existence of $[f]$-good numbers in our sets of moduli, see theorem \ref{thm: convergnet of cont germs in SG converges to cont} and especially lemma \ref{lem: SN_d contains f-good numbers} on page \pageref{lem: SN_d contains f-good numbers}.  We must also show  the equivalence of the $\SN_\d$ and $\SS^\d$ families. This essentially occurs in lemma \ref{lem: SS_delta=SN_delta}, on page \pageref{lem: SS_delta=SN_delta},  but occupies several other lemmas, eg. see lemma \ref{lem:  SM^0 is coinitial in SM}, on page \pageref{lem:  SM^0 is coinitial in SM}, and its corollary.

    Problem (3) depends essentially on the existence of a simultaneously sufficiently numerous and sufficiently rigid family of positive functions defined near infinity in $\bbr$. Sufficiently numerous means that they define a coinitial subset $\SA_\d$ of the $\SN_\d$ moduli. This is the Hardy construction rendered in lemma \ref{lem: Hardy power series construction}, on page \pageref{lem: Hardy power series construction}, but we needed a systematic version, see definition \ref{def: SM_<,B and SM^om,n_<,B} on page \pageref{def: SM_<,B and SM^om,n_<,B}, the following lemmas and lemma \ref{lem: E(H(m))>m} on page \pageref{lem: E(H(m))>m}. Sufficiently rigid means a subset of these Hardy series gives a coinitial subset of $\SN_\d$ when evaluated at an infinitesimal $\d$ if and only if they give a coinitial subset of $\SN_{\d'}$ for any other infinitesimal $\d'$. This was accomplished by looking at the sequences of integer exponents defining them and proving that these have certain asymptotic rigidity properties, see eg., corollary \ref{cor: sequence ptwise cofinal->unif cofinal}. We then convert this sequential rigidity into a perturbation rigidity for values in the domain of the Hardy series, see lemma \ref{lem: Incr pwwise bdd->Hrdy ptwise bdd} and the consequence of this and the sequential rigidity just mentioned, corollary \ref{cor: FK_xi_0 cofin<->FK_xi_1 cofin}.  These elements are pulled together in lemma \ref{lem: un(n)is sequence in N corres to un(m)} and it's following corollary.

    The solutions of problems (4) and (5) can now be found.  Problems (4) needed a new definition for coinitial subsets of $\SN_\d$ (see definition \ref{def: convergently coinitial in SN_d}). We need this  for internal sequences of infinitesimals that are given by the values of the transfers of standard sequences of functions evaluated at an infinitesimal. These have properties quite different  from the moduli, $\SN_\d$; eg, see corollary \ref{cor: m_j(d_0)is coin SN_d_0 -> true for all d}, on page \pageref{cor: m_j(d_0)is coin SN_d_0 -> true for all d}.  Nonetheless, with a bridging definition for convergence of a sequence of functions in our germ topology  $\tau$ (see definition \ref{def: convergence of standard seq of fncs in germ top} and the following cautionary remark), we get a group of results for the relationships between standard convergence and $\tau$ convergence, of which proposition \ref{prop: corres. between unif. converg and tau converg} (for uniformly convergent sequences) and   proposition \ref{prop: corres: t->f_t standard converg<-> tau converg} (for one parameter families of maps) are representative examples. For problem (5), the topological aspects of the ring operations  for the space of continuous map germs occupy  subsection \ref{subsec: top properties ring structure}; eg see proposition \ref{prop: SG_0 is Hausdorff top ring} and the preceding lemma. The material on composition of continuous germs occurs in subsection \ref{subsec: top properties of germ composition}; see proposition \ref{prop: rc_h:G^0_n->G^0_n is C^0} for the continuity of right composition and for the statement of the more difficult left composition, see proposition \ref{prop: lc_h:G^0_n,p,0->G^0_n,p,0 is C^0}.

\subsection{Germs and their infinitesimal restrictions}\label{sec: various radii, converg of C^0 is C^0}
    Let $n\in\bbn$ and if $0<r\in\bbr$, respectively $0<\Fr\in\rz\bbr$, let $B_r=B^n_r=\{x\in\bbr^n:|x|\leq r\}$, respectively $\rz B^n_\Fr=\{\xi\in\rz\bbr^n:|\xi|\leq\Fr\}$. Let $\mu(0)=\mu_n(0)=\{\xi\in\rz\bbr^n:|\xi|\sim 0\}$ and $\mu(0)_+=\{\xi\in\mu(0):\xi>0\}$; we will sometimes write $0<\d\sim 0$ instead of $\d\in\mu(0)_+$
\begin{definition}
    Let $\un{F}=\un{F}(n,1)=$
\begin{align}
    \{(U,f):U\subset\bbr^n\;\text{is a convex neighborhood of}\;\;0\;\text{and}\;\;f:U\ra\bbr\}
\end{align}
     and $\un{F}(n,1)_0\subset\un{F}(n,1)$ denote the set of those $(U,f)$ such that $f(0)=0$. If $y$ is some point in the range, we may also use $\un{F}(n,1)_y$ for those $(f,U)$ with $f(0)=y$. For the associated set of equivalence classes of germs, let $\SG_0=\SG_{n,1}$ denote the ring of germs of $f:(\bbr^n,0)\ra(\bbr,0)$ at $0\in\bbr^n$, that $\SG_0^0\subset\SG_0$ the subring consisting  of germs of continuous functions.
\end{definition}
    Although elementary, the following basic result is apparently folklore. There are many variations of this; the statement below is needed for this paper.

\begin{lem}\label{lem: *A contains inf nbd-> has stand nbd}
    Suppose that $A\subset\bbr^n$ and $0<\d\sim 0$ are such that $\{\xi\in\mu(0):|\xi|>\d\}\subset\rz A$. Then there is $0<r\in\bbr$ such that $B_r\ssm\{0\}\subset A$. Similarly, if $B\subset\bbr^n$ is such that $\rz B_\d\subset\rz B$, then there is $0<r\in\bbr$ such that $B_r\subset B$.
\end{lem}
\begin{proof}
    First, it's clear that as $\rz A$ is internal, then overflow implies that there is a standard $a>0$ such that $\{\xi:\d<|\xi|\leq \rz a\}\subset \rz A$. Let $E$ denote $A\cap B_a$ and let $E^c$ denote the complement of $E$ in $B_a\ssm\{0\}$; so that $E\cup E^c=B_a\ssm\{0\}$.  We know that $\rz E^c\subset B_\d\ssm\{0\}$; that is, for $0<d\in\bbr$, we have the statement: $\xi\in\rz E^c\Rightarrow \xi\in\rz B_d\ssm\{0\}$. But then reverse transfer gives the statement: $x\in E^c\Rightarrow x\in B_d\ssm\{0\}$ and as $d>0$ in $\bbr$ was arbitrary, then we get that $E^c=\emptyset$ so that $E=B_d\ssm\{0\}$.
    To prove the second assertion, suppose the conclusion does not hold so that there is a maximum positive $\d\sim 0$ such that if $\SB_t$ is the set $\{t\in\bbr_+:B_t\subset B\}$, then
\begin{align}
    [0,\d)\subset\rz\SB\dot=\{\Ft\in\rz\bbr_+:\rz B_\Ft\subset\rz B\}.\notag
\end{align}
     But then $\{\Ft:\d<\Ft\sim 0\}\subset\rz\SB^c$ and so as $\d\sim 0$ and is nonzero, the first part ($n=1$ here) implies that $\{\Ft:0<\Ft\sim 0\}\subset\rz \SB^c$, forcing $[0,\d)\nsubseteq \rz \SB$, ie., $\rz B_\Ft\subsetneqq \rz B$ for $\Ft<\d$, a contradiction.
\end{proof}

     Let $\rz F(B_\d)$ denote the $\rz\bbr$ algebra of internal functions on $B_\d$ and $^\s F(B_\d)$ denote the (external) subring of standard functions on $B_\d$ which is clearly an $^\s\bbr$ algebra , and so can be viewed as an $\bbr$ algebra. Note that $\SG_0$ and its subring $\SG_0^0$ are $\bbr$ algebras.
     The above lemma has the following immediate consequence which is the critical fact that allows the characterizations of germs in this paper.
\begin{cor}\label{cor: SG_0 -> F(B_delts) is R alg isomorph}
    Suppose that $U\subset\bbr^n$ is a neighborhood of $0$ in $\bbr^n$, $f: U\ra\bbr$ is a function and $0<\d\sim 0$. Then if $\rz f(\xi)=0$ for all $\xi\in B_\d$, then there is another neighborhood of $0$, $V\subset\bbr^n$ such that $f|_V$ is identically zero; ie., $[f]\in\SG_0$ is the zero germ.
    That is, the map $\SR_\d:\SG_0\ra\;^\s F(B_\d):[f]\mapsto \rz f|_{B_\d}$ is an $\bbr$-algebra isomorphism.
\end{cor}
\begin{proof}
    Let $supp(f)\subset U$ be the set of $x\in U$ such that $f(x)\not=0$ and  $A\subset U$ denote $U\ssm supp(f)$. Then $B_\d\subset\rz A$ and so the above lemma implies that there is a positive $r\in\bbr$ such that $B_r\subset A$, eg., $f(x)=0$ for $x\in B_r$; eg., $[g]=0$.
    To verify that $\SR_\d$ is an $\bbr$-algebra homomorphism is straightforward as $[f],[g]\in\SG_0$ and $c\in\bbr$ satisfy $c[f]=[cf],[f]+[g]=[f+g]$ and $[f][g]=[fg]$ and eg., $(\rz f\cdot\rz g)|_{B_\d}=(\rz f|_{B_\d})(\rz g|_{B_\d})$ as internal functions on $B_\d$.
\end{proof}
   \textbf{ Given the above proposition, when we talk about germs or elements of $\bsm{\SG_0}$, we will usually be working with subalgebras of the external algebras $\bsm{^\s F(B_\d)}$.  In particular, all work on germs will occur in the algebras $\bsm{^\s F(B_\d)}$, for some infinitesimal $\bsm{\d}$.}
\subsubsection{Monadic regularity of standard functions}\label{subsec: monadic regularity of standard fcns}
    We begin with the following simple but surprising proposition.
\begin{proposition}
    Suppose that $0<\d\sim 0$ and $[f]\in\SG_0$ is such that $\rz f|_{B_\d}$ is *continuous on $B_\d$. Then $[f]\in\SG^0$.
\end{proposition}
\begin{proof}
     The proof is trivial: if $A=\{r\in\bbr_+:f|_{B_r}\;\text{is continuous on}\; B_r\}$, then $\rz A=\{\Fr\in\rz\bbr_+:\rz f|_{\rz B_\Fr}\;\text{is *continuous on}\rz B_\Fr\;\}$ and the hypothesis says that $\rz A\not=\emptyset$ and so $A\not=\emptyset$.

\end{proof}
\begin{remark}
     Analogues of these two results for various regularity notions, eg., for homeomorphism germs, or differentiability classes, eg., germs of $C^k$ submersions, hold by almost identical arguments. We will return to these and their implications in later sections and in following papers.
     These results will allow one to work on monads where domains and ranges for standard functions are remarkably well defined and then lift to local standard results.
\end{remark}

    The following corollary indicates that the topology we define on germs will be independent of the infinitesimal neighborhood.
\begin{cor}\label{cor: *cont on B epsilon iff *cont on B del}
    Suppose that $[f]\in\SG_0$ and $\e,\d$ are positive infinitesimals. Then $\rz f$ is *continuous on $B_\d$ if and only if it is *continuous on $B_\e$.
\end{cor}
\begin{proof}
    This is clear from the previous proposition.
\end{proof}
\begin{definition}\label{def: e is [f]-good}
    Let $\e,\d$ be positive infinitesimals,ie., $\e,\d\in\mu(0)_+$ with $\e\lll\d$. For a nonzero germ $[f]\in\SG_0$, we say that $0<\e\sim 0$ is $[f]$-good on $\rz B_\d$ if  the following holds. Suppose that for all \;$\xi,\z\in\rz B_\d$ with $|\xi-\z|$ sufficiently small, $|\rz f(\xi)-\rz f(\z)|<\e$ holds;  then $[f]$ has a continuous representative on some neighborhood $B_r$ of $0$.
    We say that $0<\e\in\rz\bbr$ is $\SG_0$-good if $\e$ is $[f]$-good for all nonzero $[f]$ in $\SG_0$.
\end{definition}
    Note that if $\e\in\rz\bbr$ is $[f]$-good (respectively $\SG_0$-good), and $0<\ov{\e}<\e$, then $\ov{\e}$ is $[f]$-good (respectively $\SG_0$-good).
    The proof of existence of $[f]$-good numbers of appropriate magnitudes will be carried out later; existence $\SG_0$-good numbers needs saturation.
    Clearly, but implicitly the magnitude of a $\SG_0$-good, or a $[f]$-good number, is dependent on the degree of pinching, $\e$, occurring on this ball; but also it depends on  the magnitudes of the size of the ball, $\d$, where this occurs: the relative magnitudes are critical. This should be kept in mind in the following

  Moreover, one can show that
{\it There exists  $0<\e\in\rz\bbr$ that is $\SG_0$-good.}
\noindent Since this fact will not be used here, the proof, which is an easy saturation argument, will be omitted.

\subsubsection{A good set of moduli}\label{subsec: good set of radii}
    We want to choose a collection of potential distances between germs on $B_\d$ that will separate germs without getting a discrete topology. One possible way is as follows. In the  world of standard functions on the unit interval $I\subset\bbr$, if $f$ is any nonzero bounded function on $I$, then we can always find a positive $r\in\bbr$ such that $r<\norm{f}=\sup\{|f(x)|:x\in I\}$, eg., a Hausdorff topology can be defined strictly in terms of positive numbers via $N_c=\{f:\norm{f}<c\}$ and these are in fact determined by eg., a (countable) sequence of positive numbers dense at $0$.   But in our case, there is no clear set of numerical moduli. Such a set must be infinite and as such cannot be internal (infinite *finite will not work; eg., the first obvious problem with such is that it will have a minimum!).
    Given this, we look to the standard functions (restricted to $\rz B_\d$) themselves for our set of moduli. The first guess would be to take the external set from  *supremums of collections of standard functions. This is quite analogous to the definition of the compact open topology on the space of continuous function between topological spaces, where here, the family of (guaging) open sets  in the range collapse to a single ideal infinitesimal element.

    Recall that we are assuming sufficient saturation in the following.
    Given sufficient saturation, there exists incomparably pairs of infinitesimals $\FI\subset\mu_+(0)\x\mu_+(0)$, defined as follows. Let $F(\bbr_+,0)$ denote the set of maps $f:(U,0)\ra (V,0)$ where $U,V$ are arbitrary interval neighborhoods of $0$ in $\bbr_+=\{r\in\bbr:r>0\}$.
\begin{definition}\label{def:  SM and SM^0}
\begin{enumerate}
\item    Let $\bsm{\SM}\subset F(\bbr_+,0)$ denote the set of $\{m:\bbr_+\ra\bbr_+:\;\text{if}\;r,s\in\bbr_+\;\text{then}\; r<t \Leftrightarrow m(r)<m(t)\;\text{and}\;t\ra 0\Leftrightarrow m(t)\ra 0\}$.
\item        Given $0<\d\sim 0$, we say that $\bsm{0<\e\sim 0}$ \textbf{is incomparably smaller than} $\bsm{\d}$ if for all $m\in\SM$, we have $\rz m(\d)>\e$. We may write this $\bsm{\e\lll\d}$ or write $(\d,\e)\in\FI$.
\item     Let $\bsm{\SM^0}=\{m\in \SM:m\;\text{is continuous on some neighborhood of}\;0\}$.
\item    Let $\bsm{\wt{\SM}}$ denote the set of $m\in F(\bbr_+,0)$ with possible value $0$ such that if $r,s$ are in the domain of $m$ with $r<s$, then $m(r)\leq m(s)$ and as above $\lim_{t\ra 0}m(t)=0$.
\end{enumerate}
\end{definition}

\begin{remark}\label{rem: e<<<d and f(d)<e -> f(d)=0}
    Clearly we have that $\wt{\SM}\supset\SM$. Note that if $f\in F(\bbr_+,0)$ has values in $[0,\infty)$, $\d\in\mu(0)_+$ and $\e\lll\d$, then by the definition, if $\rz f(\d)<\e$, then, in fact, $\rz f(\d)=0$.
    Also note that if $m\in\SM^0$, then $m^{-1}\in\SM$, where here $m^{-1}$ may be defined on an arbitrarily small neighborhood of $0$ in $\bbr_+$.
    Finally, note that if $m_1,m_2,\ldots$ is a sequence in $\SM$, then $\liminf_{j\ra\infty}m_j$ is an element of $\wt{\SM}$.
\end{remark}
    A proof of the existence of incomparable pairs of positive infinitesimals is an easy concurrence argument in an enlarged model.
    Given this, let's give a criterion for $[f]\in\SG_0$ to be a continuous germ.
    We first need a preparatory abstract lemma that gives a (new) standard interpretation of incomparable infinitesimals.
\begin{lem}\label{lem:  abstract k<<<d lemma}
    Suppose that $\om,\Om\in\rz\bbn$ are such that $\Om\ggg\om$, and let $r_j\in\bbr_+$ with $r_j\ra 0$ as $j\ra\infty$. Let $A:\bbr^n\x\bbr^n\ra\bbr_+$ and for $j\in\bbn$, let $S_A(j)$ denote the assertion:
\begin{align}
    \text{there is}\;r\in\bbr_+\;\text{such that}\;|x-y|<r\Rightarrow A(x,y)<r_j
\end{align}
    and for $n\in\bbn$, $\bbs_A(n,j)$ denote the statement $(x,y\in B_{r_n})\;\wedge\;S_A(j)$. Suppose that $\rz\bbs_A(\om,\Om)$ holds.
    Then there is $n_0\in\bbn$ such that $\bbs_A(n_0,j)$ holds for infinitely many $j\in\bbn$.
\end{lem}
\begin{proof}
    If the conclusion does not hold, then for each $n\in\bbn$, there are only a finite number of $j\in\bbn$ such that $\bbs_A(n,j)$ holds. Therefore, for each $n\in\bbn$, the integer $L(n)\dot=\max\{j:\bbs_A(n,j)\;\text{holds}\}$. That is, $L:\bbn\ra\bbn$ is a map such that $\Om\leq\rz L(\om)$, a contradiction.
\end{proof}
\begin{proposition}\label{prop: germ continuity from k<<<d}
     Suppose that $0<\d\sim 0$ and $[f]\in\SG_0$ satisfies the following condition.
    For $\xi,\z\in\rz B_\d$  sufficiently small, we have that  $|\rz f(\xi)-\rz f(\z)|\lll\d$. Then $[f]\in\SG^0_0$.
\end{proposition}
\begin{proof}
    Using the notation of the previous, let $A(x,y)\dot=|f(x)-f(y)|$ and let $S_f(j)$ denote the statement $S_A(j)$ of the previous lemma and $\bbs_f(n,j)$ the corresponding $\bbs_A(n,j)$. Then to say that $\bbs_f(\om,\Om)$ holds is precisely our hypothesis as $\rz r_\om\lll\rz r_\Om$. Hence, we have the conclusion: there is $n_0\in\bbn$ such that $\bbs_f(n_0,j)$ holds for infinitely many $j\in\bbn$. That is, there are $j_1,j_2,\ldots\in\bbn$, such that for each $k\in\bbn$ the following holds:
\begin{align}
    \text{there is}\;r\in\bbr_+\;\text{such that}\;|x|,|y|<r_{j_{n_0}},|x-y|<r\Rightarrow |f(x)-f(y)|<r_{j_k}
\end{align}
    That is, since $r_{j_k}\ra 0$ as $k\ra\infty$, this says that on the ball of radius $r_{j_{n_0}}$ intersected with the open set where the representative $f$ for $[f]$ is defined, we can, for any $k$,  make $|f(x)-f(y)|<r_{j_k}$ by choosing $|x-y|$ sufficiently small, ie.,  $f$ is continuous; eg., $[f]\in\SG^0_0$.
\end{proof}
    The above is some motivation for the next definitions.
\begin{definition}\label{def: J^delta_f numbers}
    Given our $0<\d\sim 0$, let $0<\k\sim 0$, $\k$ incomparably smaller than $\d$. Then for each $[f]\in\SG_0$ we have the following definitions.
\begin{enumerate}
  \item [(a)]   Define $\FJ_f^{\k,\d}=\rz\sup\{|\rz f(\xi)-\rz f(\z)|:|\xi-\z|<\k\;\text{for $\xi,\z$ sufficiently small and}\;\xi,\z\in\rz B_\d\}$.
  \item [(b)]     If\; $0<r\in\bbr$, let $\ov{J}_f^r=\lim_{t\ra 0}\sup\{|f(x)-f(z)|:|x-z|\leq t; x,z,x-z\in B_r\}$.
  \item [(c)]    For\; $0<\d\in\rz\bbr$ infinitesimal, define $\rz\ov{J}^\d_f=\rz\lim_{\Ft\ra 0}\sup\{|\rz f(\xi)-\rz f(\z)|:|\xi-\z|\leq\Ft:\xi,\z,\xi-\z\in B_\d\}$.
\end{enumerate}
\end{definition}

\begin{definition}\label{def: SS^delta families of numbers}
     Given this, define
\begin{enumerate}
 \item [(1)]    Define $=\wh{\SS}^{\k,\d}=\{\FJ_f^{\k,\d}:[f]\in\SG_0\}\smallsetminus\{0\}$.
 \item [(2)]    If\; $0<\d\in\rz\bbr$, let $\ov{\SS}^\d=\rz\{\ov{J}^\d_f:[f]\in\SG_0\}\ssm\{0\}$.
\end{enumerate}
\end{definition}
   {\it  We will later find that the previous collection of numbers is asymptotically comparable with the collection, $\SN_\d$, to be defined next. Those defined above will allow us to prove that convergence in the topology  $\tau_\d$  on $\SG_0$ (shortly to be defined), is like uniform convergence. The following collection of numbers will essentially play the role of the appropriate distances between germs, our set of moduli. It is critical that we show that  these sets are at least asymptotically intertwined so that our topology has the good properties that come from these sets being asymptotically equivalent moduli. The remainder of this section achieves this goal along with verifying the good convergence properties. }
\begin{definition}\label{def: SN_delta families of numbers}
     If $0<r\in\bbr$ and $g:B_r\ra\bbr$, write $\norm{g}_r\dot=\sup\{|g(x)|:x\in B_r\}$ so that $\rz\norm{g}_\d\dot=\rz\sup\{|\rz g(\xi)|:\xi\in B_\d\}$, we may write this as $\norm{g}_\d$.
     Let $\SN_\d$ denote the set $\{\rz\norm{\rz g}_\d:[g]\in\SG_0\}$ and $\SN_\d^0=\{\rz\norm{\rz g}_\d:[g]\in\SG^0_0\}$
     As the ring structure will later play a role, let $\wh{\SN}_\d=\SN_\d\cup-\SN_\d$ and define $\wh{\SN}_\d^0$ similarly.
\end{definition}
    Note that $\SN^0_\d$ consists of a set of positive infinitesimal and $0$, whereas $\SN_\d$ contains positive noninfinitesimals (equal to or infinitesimally close to standard values taken by elements of $\SG_0$ at $0$.) We also clearly have $\SN^0_\d\subset\SN_\d\cap\mu(0)$.

    We have the following easy information.
\begin{lem}
     Viewing $\rz\bbr$ as an $^\s\bbr$ algebra, we have that $\wh{\SN}_\d$ is an $^\s\bbr$ subalgebra of $\rz\bbr$ such that $^\s\bbr<\wh{\SN}$ where $\SN_\d$ is the subset of positive elements in $\rz\bbr$. $\wh{\SN}^0_\d$ is also a $^\s\bbr$ algebra and as it is a subring of $\mu(0)$, it does not contain $^\s\bbr$.
\end{lem}
\begin{proof}
      We just need to verify that if $\Fr,\Fs\in\SS$, then $\Fr+\Fs$ and $\Fr\Fs$ are also elements of $\SS$ as the rest is obvious or follows immediately from this.
\end{proof}

    We will now return to some properties of the various $J_f$'s.
    Note that if $\ov{\d}\ll\d$, in the case where we need $\la=\k$ to be incomparably smaller than $\la$, then we may need to choose a much smaller $\ov{\k}$ corresponding to it and then our $\SS^{\ov{\d}}_{\ov{\k}}$ will consist of a set of positive elements that will cluster around $0$ in a much tighter fashion.  In general, as $f$ is standard, we clearly we have  $\FJ_f^\d\geq\rz\ov{J}^\d_f$, but not necessarily related to $\ov{J}^r_f$ and so eg., if $[f]\in\SG_0^0$, then $\ov{J}^r_f$ may not be $0$ for some representatives, we will have $\rz \ov{J}^\d_f=0$ and $\ov{\FJ}^\la_f\sim 0$. Consider the following function. Let $S\subset\bbr$ such that both $S$ and $\bbr\ssm S$ are dense in $\bbr$. Define $f:\bbr\ra\bbr$ by $f(x)=x$ if $x\in S$ and $0$ otherwise and finally define $f$ to be constant for $|x|>r/2$. Then one can see that, if $0<\k\lll\d\sim 0$, then $0=\ov{J}^r_f$, $\d=\rz\ov{J}^\d_f\ggg\k=\FJ^\k_f>0$. Of course, $f$ defines a noncontinuous germ that is continuous at $0$.  Furthermore, again because $f$ is standard, we have the following tighter relation.

    Although the next lemma is obvious, it's important in the considerations on $[f]$-good numbers.
\begin{lem}\label{lem: 2J_f bndd below by J_f bar}
    For each $[f]\in\SG_0$, $\FJ^{\k,\d}_f\geq\rz\ov{J}^\d_f$.
\end{lem}
\begin{proof}
     This is clear
\end{proof}
\begin{lem}\label{lem: SN_d=SM_d}
     Given $\d\in\mu(0)_+$, if $\SM_\d=\{\rz m(\d):m\in\SM\}=\{\norm{\rz m}_\d:m\in\SM\}$, then $\SM_\d=(\SN_\d\cap\mu(0))\ssm\{0\}$.
     If $\wt{\SM}_\d$ is similarly defined, then $\wt{\SM}_\d=\SM_\d\cup\{0\}$.
\end{lem}
\begin{proof}
    Just note that if $[f]\in\SG_0$, and $f\in[f]$ is any representative with $\rz\norm{\rz f}_\d\sim 0$, then on the neighborhood where it's defined $t\mapsto m(t)=\norm{f}_t\in\SM$  so that $\rz m(\d)=\norm{\rz f}_\d$ giving $\SN_\d\subset\SM_\d$. On the other hand, if $m\in\SM$, then $f(x)= m(|x|)\in\un{F}(n,1)_0$ with $\norm{\rz f}_\d=\rz m(\d)\sim 0$.
     The proof of the last assertion is left to the appendix.

\end{proof}
\begin{lem}\label{lem: SS_delta=SN_delta}
    We have that $\ov{\SS}^\d=\SN_\d$; that is, if $\Fr\in \ov{\SS}$ ; then there is $[g]\in\SG_0$ such that $\norm{g}_\d=\Fr$ and conversely, if $\Fs\in\SN_\d$, then there is $[f]\in\SG_0$ such that $J_{[f]}=\Fs$.
    In particular, for every $\Fr\in\wh{\SS}^{\k,\d}$, there is $\Fs\in\SN_\d$ with $\Fs\leq\Fr$.
\end{lem}
\begin{proof}
    Letting, for any $0<r\in\bbr$, $f:B_r\ra\bbr$ be a representative for $\Fr=\rz\norm{f}_\d$, define $g:B_r\ra\bbr$ as follows:
\begin{align}
    g(x)=\lim_{t\ra 0}\sup\{|f(y)-f(z)|:|y-z|\leq t\;\text{with}\;y,z,y-z\in B_{|x|}\}.
\end{align}
    Then if $x,y, z\in B_r$ with $|x|\leq |y|$, and $|z|=r$, then we have $g(x)\leq g(y)\leq g(z)=\norm{g}_r$ which is, by definition equal to $\lim_{t\ra 0}\sup\{|f(y)-f(x)|:|y-x|\leq t,x,y,y-x\in B_r \}=J_f^r$.  But then, by transfer we have that $\rz\norm{\rz g}_\d=\ov{J}^\d_f$, ie., $\ov{\SS}^\d\subset \SN_\d$. Note that a simpler version of this argument gets $\wh{\SS}^\d\subset\SN_\d$ by instead letting $g(x)\dot=\sup\{|f(y)-f(z)|:y,z,y-z\in B_{|x|}\}$.

    On the other hand, if $\Fr\in\SN_\d$, we will find $[g]\in\SG_0$ such that $\ov{J}_g=\Fr$. From the constructions above, we may assume that the  $[f]\in\SG_0$ with $\norm{f}_\d=\Fr$ has representative $f$ that is continuous on some neighborhood $U$ of $0$. Let $K\subset U$ be a dense subset such that $U\smallsetminus K$ is also dense and $\Fr=\rz\sup\{\rz f(\xi):\xi\in\rz K\cap B_\d\}$. Define $h:U\ra \{0,1\}$ by $h(x)=1$ if $x\in K$ and $h(x)=0$ if $x\in U\ssm K$ and let $g(x)=f(x)h(x)$ for $x\in U$.
    Now, on the one hand, by density of $K$ and continuity of $f\geq 0$, we have that for each $x\in B_r$ that
\begin{align}\label{eqn: writing f(x) as a limsup}
    f(x)=\lim_{t\ra 0}\sup_y\{|g(x)-g(y)|:y,x-y\in B_r,|x-y|<t\}.
\end{align}
    On the other hand, if $V_t\dot=\{|g(x)-g(y)|:x,y,x-y\in B_r, |x-y|<t\}$, and for a given $x\in B_r$, $V_{x,t}\dot=\{|g(x)-g(y)|:y,x-y\in B_r,|x-y|<t\}$, then we have the decomposition $V_t=\cup\{V_{x,t}:x\in B_r\}$ so that
\begin{align}
   \sup_{x,y}V_t=\sup\{\sup_y V_{x,t}:x\in B_r\}
\end{align}
    and so
\begin{align}
    \lim_{t\ra 0}\sup_{x,y\in B_r}V_{t}=\lim_{t\ra 0}\sup\{\sup_y V_{x,t}:x\in B_r\}=\sup\{\lim_{t\ra 0}\sup_y V_{x,t}:x\in B_r\}.
\end{align}
    But, by  expression \ref{eqn: writing f(x) as a limsup} and the definition of $V_{x,t}$, the last term is just $\norm{f}_r$ and by definition, the first term is $\ov{J}^r_g$ for the function $g$ defined above in terms of the given $f$, ie., we have constructed $g$ so that $\ov{J}^r_g=\norm{f}_r$. But then our assertion follows by transfer.
    The last assertion follows from lemma \ref{lem: 2J_f bndd below by J_f bar} and the first part of this lemma.
\end{proof}
    The following lemma will be important in constructing a counterexample to $\tau$-continuity of composition with a general continuous germ.
\begin{lem}\label{lem: SN_d doesnot have d-incomparable no}
    $\SN_\d$ does not contain an incomparable range relative to $\d$. That is, for all $\Fr\in\SN_\d$, then  $(\d,\Fr)\not\in\FI$.
\end{lem}
\begin{proof}
    We must verify that if $\Fr\in\SN_\d$ that there is $f\in\SM$ such that $\rz f(\d)\leq \Fr$. We may assume that $\Fr=\norm{\rz g}_\d$ for some pseudomonotone $[g]\in\SG_0^0$ so that for $r<s$ sufficiently small, $r<s$, $\norm{g}_r\leq\norm{g}_s$. Given this, for $r\in\bbr_+$, sufficiently small define $f(r)=\f{1}{2}\norm{g}_r$. Then $f\in\SM$ by the previous sentence and upon transfer we get that $\rz f(\d)=\f{1}{2}\norm{\rz g}_\d<\Fr$.
\end{proof}


\begin{remark}
      For our given $\d\sim 0$, we have defined three (external) sets of numbers $\SS,\ov{\SS}$ and $\SN$. Above we have verified that $\ov{\SS}=\SN$. We also know from Lemma \ref{lem: 2J_f bndd below by J_f bar} that we can bound elements of $\SS$ below by elements of $\ov{\SS}$, hence if $\SS$ has $[f]$-good numbers, so does $\ov{\SS}=\SN$, the set of numbers that will form the moduli for our system of neighborhoods of the $0$ germ.
      It's not that important for our purposes that $\ov{\SS}$ contain $\SG_0$-good numbers, but, if we wish to prove that $\SG_0^0$ is closed in $\SG_0$, it's critical that it contain $[f]$-good numbers for each nonzero $[f]\in\SG_0$.
      Finally, note that the last lemma will allow us to give a characterizations of the topology $\tau_0$ invariant of the choice of the given infinitesimal $\d$.
\end{remark}
\begin{lem}\label{lem: SN_d contains f-good numbers}
    If $[f]\in\SG_0$ is nonzero, then  $\SN_\d$ contains $[f]$-good numbers for $\rz B_\d$.
\end{lem}
\begin{proof}
    We know the following: lemma \ref{lem: SS_delta=SN_delta} implies that  for every $\Fr\in\wh{\SS}^{\la,\d}$, there is  $\Fs\in\SN_\d$ with $\Fs\leq\Fr$, and also if $\Fr,\Fr'\in\mu(0)_+$ with $\Fr'<\Fr$ and $\Fr$ is $[f]$-good for $\rz B_\la$, then it follows that $\Fr'$ is $[f]$-good for $\rz B_\la$. Hence, it suffices to show that $\wh{\SS}^{\la,\d}$ contains $[f]$-good numbers. But note that, as $\d\lll\la$, if $|\rz f(\xi)-\rz f(\z)|$ is coinitial with $\SN_\d$ for $\xi,\z,\xi-\z\in\rz B_\la$ with $|\xi-\z|$ sufficiently small, then $[f]\in\SG^0_0$ by proposition \ref{prop: germ continuity from k<<<d}.
\end{proof}
    Assuming we wish to include magnitudes akin to those of $f$-good numbers in our range of neighborhood diameters, then with the previous lemma we get an upper bound on these moduli on  how finely we wish to resolve our germs. In the next subsection, we will see that this resolution range works well.

\subsection{Topology on germs with fixed target}\label{sec: topology-fixed target}
    This section is the lions' share of the work here. We will verify that a convergent net of continuous germs is a continuous germ.  We will prove that $\SG^0$ and its higher dimensional analogs have good topological algebraic properties in Subsections \ref{subsec: top properties ring structure} and \ref{subsec: top properties of germ composition}. With a fair amount of effort, we prove in Subsection \ref{subsec: top independence from delta} that the topology defined in terms of the given infinitesimal $\d$, ie., $\tau^\d$, is independent of the choice of infinitesimal. In Subsection \ref{sec: relationship with nongerm convergence}, we prove results giving  correspondences between convergence of a sequence of functions on a neighborhood of $0$ in $\bbr^n$ and $\tau$ convergence of an extended net of `germs'.

    Here we will develop our topology only on $\SG^0_0$

    For a given $0<\d\sim 0$, using $\SN_\d$  (or equivalently $\SN^0_\d$ as we shall) see we wish to construct a system of neighborhoods of $[0]$, the zero germ in $\SG_0$.
\begin{definition}\label{def: of topology  at 0 germ}
    Given $\Fr\in\SN_\d$, let $U_\Fr=U^\d_\Fr\subset\SG_0$ denote the (external) set $\{[f]\in\SG_0:\norm{\rz f}_\d<\Fr\}$ where $\norm{\rz f}_\d=\rz\sup_{\xi\in B_\d}|\rz f(\xi)|$ . By definition, $\Fr\in\SN_\d$ implies that $U_\Fr$ is nonempty, in fact, infinite.  Let $\tau_0=\tau^\d_0=\{U^\d_\Fr:\Fr\in\SN_\d\}$ and note that by the definition of $\SN_\d$, $\tau_0$ is closed under finite intersections. If $\d$ is fixed in the discussion, we will often write $\tau$ or $\tau_0$ and $U_\Fr$ leaving off the $\d$'s.
    Suppose that $(D,<)$ is a directed set and that $([f_d]:d\in D)$ is a $D$ net in $\SG_0$. Then $([f_d]:d\in D)$ converges in $\tau^\d$ to the zero germ  in $\SG_0$ if for each $\Fr\in\SN_\d$, there is $d_0\in D$ such that if $d\in D$ with $d>d_0$, then $\norm{\rz f_d}_\d<\Fr$, ie.,  $[f_d]\in U^\d_\Fr$.
\end{definition}
    We will work with the properties of $\tau^\d$ here, for the arbitrary positive infinitesimal $\d$ already knowing that, at least for continuity of germs, the choice of $\d$ is irrelevant and later find that the topology generated by $\tau^\d$ itself is invariant of the choice of $\d$.

\subsubsection{A good system of neighborhoods at the zero germ}\label{subsec: a good system of nbds}
\begin{lem}\label{lem: f cont, f-g  small implies g cont}
    Fix $0<\d\sim 0$ and suppose that $[f]\in\SG_0^0$ and $[g]\in\SG_0$. With the hypothesis that $0<\Fr\in\ov{\SS}^\d$ is $[g]$-good, we have that if $\rz f|_{B_\d}-\rz g|_{B_\d}\in U_{\Fr/3}$ then $[g]\in\SG_0^0$.
\end{lem}
\begin{proof}
    The proof is a  nonstandard version of the three epsilon argument with the use of good numbers. Suppose that $\ov{\Fr}$ is $[g]$-good and let $\Fr=\ov{\Fr}/3$. As $\rz f$ is *continuous, if $\xi,\z\in B_\d$ are sufficiently small, we have that $|\rz f(\xi)-\rz f(\z)|<\Fr/3$; so that if $\rz f|_{B_\d}-\rz g|_{B_\d}\in U_\Fr$, we have that
    $|\rz g(\xi)-\rz g(\xi)|\leq |\rz g(\xi)-\rz f(\xi)|+|\rz f(\xi)-\rz f(\z)|+|\rz g(\z)-\rz f(\z)|<\ov{\Fr}$. But as $\ov{\Fr}$ is $[g]$-good, we obtain our result from the definition, \ref{def: e is [f]-good}.
\end{proof}
    The following theorem along with the nondiscrete nature of our topology indicates that our topology has good properties; in particular this result indicates that $\tau^\d$ convergence is analogous to uniform convergence.
\begin{thm}\label{thm: convergnet of cont germs in SG converges to cont}
    Suppose that $D$ is a directed set and that $d\in D\mapsto [f^d]$ is a $D$-net in $\SG_0^0$ that is $\tau_0$ convergent to $[g]\in\SG_0$. Then $[g]\in\SG_0^0$.
\end{thm}
\begin{proof}
    The result follows easily from the above preliminaries. As $d\mapsto [f^d]$ is $\tau$ convergent to $[g]$, if $\Fr\in\ov{\SS}$, there is $d_0\in D$ such that $\rz f^{d}|_{B_\d}-\rz g|_{B_\d}\in U_\Fr$ for $d>d_0$. But then, by Lemma \ref{lem: SN_d contains f-good numbers},  $\ov{\SS}$ contains a $[g]$-good number $\ov{\Fr}$ and as just noted we know that there is $\ov{d}\in D$ such that if $d\in D$ with $d>\ov{d}$, we have that $\rz f^d|_{B_\d}-\rz g|_{B_\d}\in U_{\ov{\Fr}}$ and so by Lemma \ref{lem: f cont, f-g  small implies g cont}, we have that $[g]\in\SG_0^0$.
\end{proof}

\subsubsection{Convergence in our topology}\label{subsec: convergence}
 In this part $0<\d\sim 0$ is still fixed. We have defined a system of neighborhoods of the $0$ germ with the collection of sets $\tau_0$; that is, a neighborhood base for a topology for $\SG_0$ at $[0]$. 
    Given the above preliminaries, we will now make $\SG_0$ into a topological vector space in the usual way by translating the sets of $\tau_0$.
\begin{definition}\label{def: of full topology by translation}
    Given $\tau_0$ above and $[f]\in\SG_0$, let $\tau_f$ denote the $[f]$ translation of $\tau_0$; ie., $\tau_f=\{U+\rz f|_{B_\d}:U\in\tau_0\}$ and let $\tau$ denote the topology generated, in the usual way, by finite intersections of arbitrary unions of elements of $\tau_f$ as $[f]$ varies in $\SG_0$.
\end{definition}

    Before we can say anything more about this topology we need more formalities on orders. See Fuchs, \cite{Fuchs1963}, for a good coverage of the mathematics of ordered algebraic systems.  So returning to the discussion at the beginning of subsection \ref{subsec: good set of radii}, let's begin with a definition from the theory of ordered sets. Suppose that $(P,\leq)$ is a partially ordered set (ie., for all  $p,q,r\in P$ we have $p\leq p,p\leq q$ and $q\leq p$ implies $p=q$ and $p\leq q,q\leq r$ implies $p\leq r$) and $J\subset P$ with the induced partial order. We will often assume that our partially ordered sets are downward, respectively upward, directed, ie., if $p,q\in P$ then $\inf\{p,q\}\in P$, respectively $\sup\{p,q\}\in P$.
\begin{definition}\label{def: of coinitial partial order relations}
    If $(P,<)$ is a partially ordered, downward directed set and $J\subset P$, then we say that $\bsm{J}$\textbf{ is coinitial in }$\bsm{P}$ with respect to $<$, if for all $p\in P$, there is $a\in J$ such that $a\leq p$.
    Suppose that we have two subsets $J,K\subset P$. Then we say that $\bsm{J}$ \textbf{is coinitial with} $\bsm{K}$, written $J\risingdotseq K$ or $K\fallingdotseq J$ if for all all  $k\in K$, there is $j\in J$ such that $j\leq k$ and we say that $\bsm{J}$ \textbf{and} $\bsm{K}$ \textbf{are coinitial}, written $J:=:K$ if $J$ is coinitial with $K$ and $K$ is coinitial with $J$.
    If $\Fr,\d$ are positive infinitesimals, we say that $\Fr$ is almost in $\SN_\d$, if there is $\Fs,\Ft\in\SN_\d$ such that $\Fs\leq \Fr\leq \Ft$; we will write this as  $\Fr\; a\!\!\!\in\SN_\d$. If $\Fr,\Fs\in\mu(0)_+$ with $\Fr\;a\!\!\!\in\SN_\Fs$ and $\Fs\;a\!\!\!\in\SN_\Fr$, then we write $\bsm{\Fr\asymp\Fs}$.
\end{definition}
\begin{remark}\label{rem: coinitial is partial order}
     Note that $J\risingdotseq K$ defines a partial order (on subsets of $\mu(0)_+$) and $J:=:K$ defines an equivalence relation on subsets of $\mu(0)_+$. In particular, if $J_1\risingdotseq K_1$, $J_1:=:J_2$ and $J_2\risingdotseq K_2$, then $K_1:=:K_2$.
     Note that if $\Fr,\Fs$ and $\Ft\in\mu(0)_+$ and if $\Fr\asymp\Fs$, and $\Fs\asymp\Ft$, then $\Fr\asymp\Ft$; eg., $\asymp$ defines an (external) equivalence relation on $\mu(0)_+$.
     Note that it's possible (and even probable depending on the saturation) that $\Fr\asymp\Fs$, yet $\SN_\Fr\cap\SN_\Fs=\emptyset$. Nonetheless we have the following good asymptotic correspondence.
\end{remark}
    We need a further definition on convergence within the framework of coinitiality.
    Until now we were generally satisfied with considered subsets $S$ of $\SN_\d$ that are coinitial; this corresponds to the notion of accumulation. We have considered convergence of nets in eg., $\SG_0$, but have not formalized this notion with respect to the $\SN_\d$'s. We will do this now.

\begin{definition}
    Suppose that $(T,<)$ is a totally ordered, upwardly directed, set with $S\subset T$ given the restricted total order, $(D,<)$ is an upward directed set, and that $\SV=(\Fv_d:d\in D)$ is a $D$ net in $T$. Then we say that $\bsm{\SV}$ \textbf{is convergently coinitial in the range of} $\bsm{S}$, if for each $\Fs\in S$, there is $d_0\in D$ such that, if $d>d_0$ then $\Fv_d>\Fs$. If, in addition we have that for each $\Fv\in\SV$, there is $\Fs\in S$ with $\Fs\geq\Fv$, then we say that $\bsm{\SV}$ \textbf{is convergently coinitial with} $\bsm{S}$.
\end{definition}
    If we already know that the map $D\ra T:d\mapsto \Fv_d$ is (order reversing)\textbf{ monotone}, then clearly coinitiality implies coinitial convergence. 

    Given the above definitions, we can say a little more about the totally ordered sets defining our topology.  If $[m]\in\SM^0$, and $\Fr,\Fs$ are positive infinitesimals with  $\Fr<\Fs$, then $\norm{\rz m}_\Fr\leq\norm{\rz m}_\Fs$ and so eg., $\SN_\Fr$ is coinitial with $\SN_\Fs$. We also have the following.
\begin{lem}
    If $0<\d\sim 0$ and $\Fr\in\SN_\d$, then $\SN_\Fr\subset\SN_\d$.
\end{lem}
\begin{proof}
    We will show that if $[m]\in\SM^0$, then $\norm{\rz m}_\Fr\in\SN_\d$. But , by definition, letting $\Fr=\norm{\rz m'}_\d$ for some $m'\in\SM^0$ and we may assume, as $m$ is continuous, that $\norm{\rz m}_\Fr=\rz m(\z_0)$ where $\z_0=\rz m'(\xi_0)$ for some $\xi_0\in B_\d$. (This gives the third equality in the following expression, ie., the $\rz\sup$ occurs on the image of $\rz m'$.) Now recalling that all of our functions are positive
\begin{align}
    \norm{\rz m}_\Fr=\norm{\rz m}_{\|m'\|_\d}=\rz\sup\{\rz m(\z):\z\leq\norm{\rz m'}_\d\}\\
    =\rz\sup\{\rz m\circ m'(\xi):\xi\leq\d\}=\norm{\rz m\circ m'}_\d, \notag
\end{align}
    and of course $\norm{\rz m\circ m'}_\d\in\SN_\d$.
\end{proof}
\begin{remark}
    It will follow from the results in subsection \ref{subsec: top independence from delta} that if $\Fr,\Fs$ are positive infinitesimals such that $\Fr\asymp\Fs$; then $\SN_\Fr$ and $\SN_\Fs$ are coinitial in $\rz\bbr_{nes,+}$.
\end{remark}

     At this point, we need to refine a definition given in Definition \ref{def: of coinitial partial order relations }.
\begin{definition}
    Let $\d\in\mu(0)_+$ and suppose that $(D^1,<)$ and $(D^2,<)$ are directed sets. Suppose that $\FF_1=(\Fr_d:d\in D^1)$ and $\FF_2=(\Fs_d:d\in D^2)$ are nets in $\mu(0)_+$. Then we say that $\FF_1$ and $\FF_2$ are coinitial with each other in the range of $\SN_\d$ if for each $\Fr_{d_1}\in\FF_1$ with $\Fr_{d_1}>\Fu$ for some $\Fu\in\SN_\d$, there is $\Fs_{d_2}\in\FF_2$ such that $\Fs_{d_2}<\Fr_{d_1}$ and the analogous statement with the indices 1 and 2 switched holds also.
\end{definition}
    With this definition we need an elementary but useful fact.
\begin{lem}\label{lem: sets coin wrt N_d go to these under m in M}
    Suppose that $\FF_j\subset\mu(0)_+$, $j=1,2$ are coinitial with each other in the range of $\SN_\d$ and that $[m]\in\SM^0$. Then $\rz m(\FF_j)$, $j=1,2$ are coinitial with each other in the range of $\SN_\d$. In particular, $\rz m(\FF_1)$ is coinitial with $\SN_\d$ if and only if $\rz m(\FF_2)$ is coinitial with $\SN_d$.
\end{lem}
\begin{proof}
    As $\rz m(\SN_\d)\subset\SN_\d$ and as $\rz m$ is monotone and so eg., sends sandwiched elements to such respecting their orders, then this statement is clear.
\end{proof}
    We also have a convergence version of the above.
\begin{lem}
    Suppose that $(D^j,<)$ for $j=1,2$ are directed sets and that $\SX^j=(\Fu^j_d:d\in D^j)$ are nets in $\mu(0)_+$ such that the maps $d\mapsto\Fu^j_d$ are monotone (order reversing). If $[m]\in\SM^0$, let $m\circ\SX^j$, $j=1,2$ denote the net $d\in D^j\mapsto \rz m(\Fu^j_d)$. Suppose that $\SX^1$ and $\SX^2$ are convergently coinitial in the range of $\SN_\d$. Then $m\circ \SX^1$ is convergently coinitial in the range of $\SN_\d$ if and only if $m\circ\SX^2$ is convergently coinitial in the range of $\SN_\d$.
\end{lem}

\begin{lem}\label{lem: SN is gen by set of the f(d)'s}
     If $V(\d)$ is the semiring  $\{\rz m(\d):m\in\SM\}$, then $V(\d)=\SN_\d$.
\end{lem}
\begin{proof}
     For each $[f]\in\SG$, we have the element $[m_f]\in\SM$ given by $t\mapsto\norm{f}_t$, call this map $A$. Then $A$ is clearly a surjection, as given $[m]\in\SM$, we get a $[f]\in\SG$ such that $A([f])=[m]$ by defining $f(x)=m(|x|)$. But also it's clear by the definition of $A$ that if $[f]\in\SG$, then $\norm{f}_t=A([f])(t)$, and so by transfer if $\d\in\mu(0)_+$, then $\norm{\rz f}_\d=\rz A([f])(\d)$. But then $\{\norm{\rz f}_\d:[f]\in\SG\}=\{\rz A([f])(\d):[f]\in\SG\}$ and as the image of $A$ is all of $\SM$, this last set id just $\{\rz m(\d):[m]\in\SM\}$.
\end{proof}
     If $\Fr\in\SN_\d$, let $\SN_{\d,\Fr}\subset\SN_\d$ be the set of $\Fs\in\SN_\d$ such that $\Fs\leq \Fr$.
     Then a net $(D,\xi_d)\subset\SN_\d$ is coinitially convergent with $\SN_\d$ if and only if for each $\Fr\in\SN_\d$, there is $d_0\in D$ such that $\{\xi_d:d>d_0\}\subset\SN_{\d,\Fr}$.

\begin{cor}\label{cor: m(N_d)-N_d and N_(m(d))=N_d}
    Suppose that $0<\d\sim 0$ and $\un{m}\in\SM^0$. Then $V(\d)=V(\rz \un{m}(\d))$ as subsemirings of $\mu(0)_+$ and  $\SN_{*\un{m}(\d)}=\SN_\d$. Also $\rz m(\SN_\d)=\SN_\d$, ie., if $\Fr\in\SN_\d$, then $\SN_\Fr=\SN_\d$.
\end{cor}
\begin{proof}
    We know by definition that $V(\rz \un{m}(\d))\subset V(\d)$, but $\un{m}^{-1}\in\SM$ and so we have that $\SM=[\un{m}^{-1}]\circ\SM=\{[\un{m}^{-1}\circ m]:[m]\in\SM\}$. Therefore
\begin{align}
    V(\d)=\{\rz m(\d):[m]\in\SM\}=\{\rz\un{m}\circ\un{m}^{-1}\circ m(\d):m\in\SM\}\qquad\qquad\qquad\notag\\
    =\{\rz m\circ\wt{m}(\d):[\wt{m}]\in[\un{m}^{-1}]\circ\SM\}\quad\qquad\qquad\notag\\
    =\{\rz \un{m}\circ \wt{m}(\d):[\wt{m}]\in\SM\}=  V(\rz m(\d)).
\end{align}
    The last statement follows by an identical verification using instead that $\SM\circ[\un{m}]=\SM$.
\end{proof}
    Note that it possible that $0<\e\sim 0$ is not incomparably large or smaller than $\d$, ie., for some $m_1,m_2\in\SM$ we have $\rz m_1(\d)<\e<\rz m_2(\d)$ and yet $\SN_\e\cap\SN_\d=\emptyset$. One should check Puritz, \cite{Puritz1971}, for `discrete' versions of this and much other. Nonetheless, we have the following result.

\begin{lem}
    Suppose that $\Fr\in\SN_\d$, then $\SN_\d$ and $\{\Fu\in\SN_\d:\d<\Fu<\d+\Fr\}$  have the same cardinality.
\end{lem}
\begin{proof}
    It's easy to see that as $m$ varies in $\SM^0$, the map $T:\rz m(\d)\mapsto\d+\rz m(\d):\SN_\d\ra\SN_\d$ is one to one. Also if $\SM^0_\Fr=\{m\in\SM^0:\rz m(\d)<\Fr\}$, then it's clear that the cardinalities of $\{\rz m(\d):m\in\SM^0_\Fr\}$ and $\SN_\d$ are the same. Therefore, because of this and as $T$ is an injection, the cardinality of $T(\{\rz m(\d):m\in\SM^0_\Fr\})$ is the same as that of $\SN_\d$. But by construction $T(\{\rz m(\d):m\in\SM^0_\Fr\})$ lies in $\{\Fu\in\SN_\d:\d<\Fu<\d+\Fr\}$.
\end{proof}
\begin{proposition}
    Suppose that $\SR_1,\SR_2$ are positive subsemirings of $\rz\bbr_{nes,+}$ that don't contain incomparable ranges and are coinitial in $\rz\bbr_+$. Then $\tau(\SR_1)_0$ and $\tau(\SR_2)_0$ define equivalent neighborhood systems of the zero germ in $\SG_0$.
\end{proposition}
\begin{proof}
    This is straightforward. Suppose that $(D,<)$ is a directed set and $(f_d;d\in D)$ is a $\tau(\SR_1)_0$-convergent $D$-net; ie., for each $\Fr\in\SR_1$ there is $d_0\in D$ such that $f_d\in U_\Fr$ for $d\geq d_0$. But by hypothesis, if $\Fs\in\SR_2$, there is $\Fr\in\SR_1$ with $\Fr<\Fs$ and therefore a $d_0\in D$ such that $d\geq d_0$ implies that $d_d\in U_\Fr\subset U_\Fs$.
\end{proof}
\begin{lem}\label{lem:  SM^0 is coinitial in SM}
    $\SM^0\subset\SM$ is coinitial in $\SM$.
\end{lem}
\begin{proof}
    Let $[m]\in\SM$; we will find $[\tl{m}]\in\SM^0$ such that $\tl{m}(t)\leq m(t)$ for $t>0$ sufficiently small. Now, as $[m]$ is monotone, the germ at $0$  of the set of points of discontinuity, $S$, is discrete, eg countable with possible limit point at $0$; let $p_1>p_2>\cdots$ be a enumeration of these, noting that only the tail end is well defined. Define a piecewise linear $\tl{m}\in\SM^0$ as follows. For each $j\in\bbn$, let $m(p_j)_-$ denote the limit $m(t)$ as $t\uparrow p_j)$ and $m(p_j)_+$ denote the limit $m(t)$ as $t\downarrow p_j$ (exists by monotonicity and $m$ assumes one of these). We know that $m(p_j)_-\leq m(p_j)_+<m(t)$ for $p_j<t<p_{j-1}$; so there is $q_j\in (p_j,p_{j-1})$ such that if $m_j:[p_j,q_j]\ra\bbr_+$ is the affine map with graph the line segment connecting the two points $(p_j,f(p_j)_-)$ and $(q_j,f(q_j))$, then we have $m_j(t)\leq m(t)$ for $t\in [p_j,q_j]$.
    Therefore defining $\tl{m}$ to be $m_j$ on $[p_j,q_j]$ for all $j\in\bbn$ and to be $m$ on the complement of $\cup\{[p_j,q_j]:j\in\bbn\}$, we have defined a function $\tl{m}$ (representative) in $\SM^0$ with $\tl{m}(t)\leq m(t)$ for sufficiently small $t\in\bbr_+$.
\end{proof}
\begin{definition}
    Let $\SN^0_\d\subset\SN_\d$ denote the set $\{\Fr\in\SN_\d:\Fr=\norm{\rz f}_\d:\;\text{for some}\;[f]\in\SG_0^0\}$. It's easy to see, as with $\SN_\d$, that $\SN^0_\d$ is a semiring. Let $\hat{\SN}^0_\d=\SN^0_\d\sqcup-\SN^0_\d\sqcup\{0\}$ denote the subring of $\hat{\SN}_\d$ it generates.
\end{definition}
\begin{proposition}\label{prop: SN^0_d is coin in SN_d}
     $\SN^0_\d$ is coinitial in $\SN_\d$.
\end{proposition}
\begin{proof}
     This is a direct consequence of the proof of Lemma \ref{lem: SN_d=SM_d} and Lemma \ref{lem:  SM^0 is coinitial in SM}.
\end{proof}
    From work later in this paper (see Subsection \ref{subsec: top independence from delta}) we have some idea of the cardinality of minimal coinitial subsets of $\SN_\d$.
\begin{lem}
    There exists a coinitial subset of $\SN_\d$ with the cardinality of $\bbr$.
\end{lem}
\begin{proof}
    From Lemma \ref{lem: Hardy power series construction}, we know that the subset of $\SM$ given by germs  of analytic functions, $\SS\SM$, satisfies $\SA_\d=\{\rz m(\d):m\in\SS\SM\}$ is coinitial in $\SN_\d$ and we also know, by Corollary \ref{cor: E_d is an isomorphism onto A_d}, that the map $\SE_\d:\SS\SM\ra\SN_\d:[m]\mapsto\rz m(\d)$ is an injection onto this coinitial subset. The assertion follows as the cardinality of $\SS\SM$ is the same as that of $\bbr$.
\end{proof}
    So the previous fact gives a lower bound on the possible cardinalities of minimal coinitial subsets. We will now prove that this cardinality must actually occur.
    Now consider $F(I)$ the ring of real valued functions on the unit interval $I$ and in the $\sup$ topology on $F(I)$, consider the collection $(\SC,\subset)$ of all neighborhoods of the the zero function partially ordered by inclusion. $(\SC,\subset)$ is an uncountable partially ordered set, but it has a (many) countable coinitial subset $\{N_{1/n}:n\in\bbn\}$ where $N_{1/n}=\{f\in F(I):\norm{f}<1/n\}$ hence to prove completeness or not it suffices to test convergence for one of these countable coinitial sets. The critical fact here is that $(0,\infty)$ has countable coinitial subsets, but as we shall see $\SS$ does not, hence convergence in $(\SG,\tau_0)$ will need uncountable nets.

    The previous is elementary and standard, but motivates the following. We will show that $\SN_\d$ does not contain a countable coinitial subset and hence according to the above discussion, one needs nets on at least uncountable directed sets to give convergence. Consider the following example. Here we will be considering real valued germs of maps at $0$ in $\bbr$; fix our positive infinitesimal $\d$. For each $0<r\in\bbr$, let $f_r(x)=\exp(-1/x^r)$ for $x>0$, setting $f_r$ equal to $0$ for $x<0$. $f_r$ is monotone (and in fact smooth) and so it is straightforward that $\norm{\rz f_r}_\d=\rz\exp(-1/\d^r)$ and so as $c^{ab}=(c^a)^b$, we have that
\begin{align}
    \norm{\rz f_{r+t}}_\d=(\norm{\rz f_r}_\d)^{1/\d^t}.
\end{align}
    From this one can check that $\norm{\rz f_r}_\d\sim 0$ and that if $a>0$, then  $\norm{\rz f_{r+a}}_\d\ll (\norm{\rz f_r}_\d)^n$ for all $n\in\bbn$. That is, we have an uncountable subset
\begin{align}
    \{\ell_r\dot=\norm{\rz f_r}_\d:r\in\bbr_+\}\subset\SS\cap\mu(0)
\end{align}
    such that for each $r,s\in\bbr_+$ with $r<s$, $\ell_s<(\ell_r)^n$ for all $n\in\bbn$, yet there is $\Fr\in\SS$ such that $\Fr<\ell_r$ for all $r\in\bbr_+$. We will give a general proof of this below but here we will give an explicit bound. If $g(x)=\exp(-e^{1/x})$, it's straightforward to check that $\rz f_r(\d)>\rz g(\d)$ for all positive $r\in\bbr$.

    For what it's worth  the net of functions $f_r$ does indeed have a countable coinitial subset $f_n$ for $n\in\bbn$ ie., as $\bbn$ is cofinal in $\bbr_+$. But we will prove the following.
\begin{lem}\label{lem: SN does not contain countale coinit}
    There is no countable subset $\SC=\{\Fs_j:j\in\bbn\}\subset\SS$ with $\Fs_j>\Fs_{j+1}$ for all $j$, that is coinitial in $\SS$.
\end{lem}
\begin{proof}
    Suppose, by way of contradiction,  that such a sequence exists.  By Lemma \ref{lem: SS_delta=SN_delta}, we may work with the collection of $\SN_\d$, so suppose that $\Fr_1>\Fr_2>\ldots$ is a countable subset of $\SS$ such that, if $[g]\in\SG_0$, then there is $j\in\bbn$ with $\norm{\ov{f}^j}_\d<\norm{g}_\d$. Here, for each $j\in\bbn$, we have that $\ov{f}^j$ is any representative of a germ $[\ov{f}^j]\in\SG_0$ with $\Fr_j=\norm{\ov{f}^j}_\d$. Without loss of generality, we may redefine the $\ov{f}^j$'s preserving $\Fr_j=\norm{\rz \ov{f}^j}_\d$ by defining $f^j(x)=sup\{|\ov{f}^j(y)|:y\in B_{|x|}\}$ so that if $r>0$ such that $f^j$ is defined on $B_r$ and $x\in B_r$ has $|x|=r$, then $f^j(x)=\norm{f^j}_r=\norm{\ov{f_j}}_r$.
    (In this case, we have that the representatives $f^j$ are nonnegative and pseudo-monotone: ($|x|\leq |y|$ implies that $f^j(x)\leq f^j(y)$)). We may further assume that the representatives $f^j$ are defined on a ball $B_j$ for $j\in\bbn$ with radius $b_j$ such that $b_j\ra 0$ as $j\ra \infty$, so that now if $x\in\p B_j$, then $f^j(x)=\norm{f^j}_{b_j}$. Define $h:B_1\ra\bbr$ as follows. Writing $B_1=\sqcup\{B_j\ssm B_{j+1}:j\in\bbn\}\sqcup\{0\}$, a disjoint union, we will define $h$ as a step function undercutting successively more of the $f^j$'s. For $j\in\bbn$, define
\begin{align}
    h(z)=\sup\{\min\{f^j(z):1\leq j\leq k\}:z\in B_j\ssm B_{j+1}\},
\end{align}
    defining $h(0)=0$. We have that, as the $f^j$ are pseudo-monotone, that $h$ is pseudo-monotone on $B_1$ and constant on $B_j\ssm B_{j+1}$. In particular, if $D\subset\bbr^n$ is a closed ball centered at $0$ with $B_{j+1}\subsetneqq D\subset B_j$, then we have ($\bsm{\diamondsuit}$) $\norm{h}_D=\norm{h}_{b_j}$ by the following argument. First $h$ pseudomonotone on $D$ implies that $\norm{h}_D=\norm{h}_{D\ssm B_{j+1}}$, but $h$ is constant on $B_j\ssm B_{j+1}$ and so $\norm{h}_{D\ssm B_{k+1}}=\norm{h}_{B_j\ssm B_{j+1}}$ and again by pseudomonotonicity of $h$ we have $\norm{h}_{B_j}=\norm{h}_{B_j\ssm B_{j+1}}$.

     Transfer this setup. The transfer of the set $\{f_j:j\in\bbn\}$ defined on the balls $B_j$ will be denoted by $\{\rz f_\Fj:\Fj\in\rz\bbn\}$ where the transfer of the set of $B_j$'s will be denoted by $\rz B_\Fj$ for $\Fj\in\rz\bbn$ with the (internal) set of radii denoted $\rz b_\Fj$ satisfying $\rz\lim_{\Fj\ra\infty}\rz b_\Fj=0$; eg., for $\Fj$ large enough $\rz B_\Fj\subset B_\d$. Also, we have, by transfer, that if $\Fk\in\rz\bbn$ and $\xi\in\rz B_\Fk\smallsetminus\rz B_{\Fk+1}$, then $\rz h(\xi)=\rz\min\{\rz f^\ell(\xi):1\leq\ell\leq\Fk\}$.
    Now there is $\Fk$  big enough so that $\rz B_\Fk\subset B_\d\subset\rz B_{\Fk-1}$. (At this point, note that in our original choice of the radii $b_1>b_2>\cdots$, as we already have $\d$ in hand we need to be sure that $\d\not\in\rz\{b_1,b_2,\ldots\}$, which is no problem.)  So we have that
\begin{align}\label{eqn: *f^k, f^j inequalities}
    \rz\norm{\rz f^\Fk}_\Fk<\rz\norm{\rz f^\Fk}_\d<\rz\norm{\rz f^j}_\d=\Fr_j,\;\text{for all}\; j\in\bbn.
\end{align}
     But by the transfer of ($\diamondsuit$) above, we have
\begin{align}\label{eqn: h(xi) bounds norm(*f^k)below}
    \rz\norm{\rz h}_\d=\rz\norm{\rz h}_\Fk
\end{align}
    and so putting expressions \ref{eqn: *f^k, f^j inequalities} and \ref{eqn: h(xi) bounds norm(*f^k)below} together we get that $\norm{f^j}_\d\dot=\;\Fr_j>\rz\norm{\rz h}_\d$ for all $j\in\bbn$, contradicting the assertion on the coinitial sequence.
\end{proof}

\begin{cor}\label{cor: countable nets can't converge in tau}
    Suppose that $(D,<)$ is a countable directed set and $d\mapsto [f_d]$ is a $D$ net in $\SG_0$ that is not eventually constant. Then $([f_d]:d\in D)$ does not converge.
    In particular, this is also true if $D$ contains a countable cofinal subset.
\end{cor}
\begin{proof}
    Suppose that a countable noneventually constant net $([f_d]:d\in D)$ converges. Then there is a coinitial subset $\ST\subset\SN$ such that for each $\Ft\in\ST$, there is a $d\in D$ such that $[f_d]\in U_\Ft$. But this defines a map from $D$ onto $\ST$ which is impossible as $\ST$ is uncountable.
\end{proof}
    Although $\tau_0$ appears to be defined by a norm, as we have seen the associated  `metric' takes values in a set without a countable coinitial set.
\begin{cor}
    $(\SG_0,\tau_0)$ is not metrizable.
\end{cor}
\begin{remark}
    Note that later we will use a construction of Borel to show that any countable sequence in $\SG_0^0$ (and therefore $\SG_0^0$), there is a power series germ that blocks its coinitiality. But such limiting constructions lie outside the Hardy field constructions, hence it appears that such do not define consistent coinitial moduli.
\end{remark}

    The $\tau$ topology is about nearness in the sense of how closely pinched the graphs of map germs are. If we try to extend the topology on $\SG_0$ to a ring of map germs with arbitrary target value, eg., $\SG$, we get bad behavior. For example, if $[f_1],[f_2]\in\SG^0_n$ are germs such that $f_1(0)\not=f_2(0)$, then for all $\Fr,\Fs\in\SN_\d$, we have that $U^\d_\Fr([f_1])\cap U^\d_\Fs([f_2]=\emptyset)$. In particular, if $\SG(x)\subset\SG$ consists of those germs $[f]$ such that $f(0)=x$, and $\SG(-x)=\SG\smallsetminus\SG(x)$, then $\SG(x)$ and $\SG(-x)$ would be disjoint open sets in this extended $\tau^\d$ topology, that is, $\SG(x)$ would be both open and closed in this extended topology. The previous assertion follows easily from the observation: $[f]\in\SG(x)$ implies that $U_\Fr([f])\subset\SG(x)$ for all $\Fr\in\SN_\d$ and $\SG(-x)$ is just the union of such $\SG(y)$ for $y\not=x$.
    Hence we will extend the topology to all of $\SG=\cup_{x\in\bbr}\SG(x)$ in a different manner.
\begin{lem}
    With respect to the $\tau$ topology $\SG$ is the union of its disjoint open subsets $\SG(x)$ as $x$ varies in $\bbr^n$.
\end{lem}
    Although this seems to foreclose topological relations between the different $\SG(x)$ and therefore a good rendering of $\tau^d$ continuity of eg.,  families of map germs $t\mapsto [f_t]$ when $f_t(0)$ is not constant, we shall see that this is a fixable problem.

\subsection{Rigid coinitial subsets and topological independence from $\d$}
    The first subsection constructs coinitial subsets of $\SN_\d$ that come from increasingly rigid subfamilies of $\SM$. Initially, we are not able to force topological independence from the choice of infinitesimal, $\d$. We do get some good ordering properties, in particular a total ordering of analytic germs at good (generic) infinitesimals. In the second subsection, we prove independence by systematically exploiting the Hardy construction (from the first subsection) by using a critical (and somewhat surprising, see eg., corollary \ref{cor: FI_(om_0) cofin<->FI_(om_1) cofin}) fact about increasing sequences of integers.
\subsubsection{Coinitial subsets  and order}
   In this subsection, a warmup for the next, we investigate some coinitial subsets of $\SN_\d$ that are defined in terms of quite rigid families of functions. Before we proceed with the motivation for this part, we need a formal definition.
\begin{definition}
    If $\d\in\mu(0)_+$, let $\SE_\d:\SG_0\ra\rz\bbr$ denote the evaluation map $[m]\mapsto\rz m(\d)$.
    If $\SJ\subset\SM$, then we will often denote $\SE_\d(\SJ)$ by $\SJ_\d$.
\end{definition}
    The original  hope here was to prove that the topologies $\tau_d$ were independent of the infinitesimal $\d$ via an order preserving argument for some good coinitial subset $\SS\subset\SM^0$. That is, we wanted to find $\SS$ so that (1): $\SS\mapsto \SS_\d$ was injective for sufficiently numerous good infinitesimals $\d$, so that (2): $\SS_\d$ was coinitial in $\SN_\d$ for sufficiently many infinitesimals $\d$ and such that (3): the graphs $\d\mapsto\rz m(\d)$ as $\d$ varied over positive infinitesimals satisfied a local intersection property. Assuming the injectivity property for sufficiently numerous $\d$'s so that for each $\Fr\in\SE_\d(\SS)$, the element $m_\Fr=\SE^{-1}(\Fr)\in\SS$ is well defined, then the local intersection property can be stated as follows.
    Given a given good infinitesimal $\d_0$ and  $\Fr\in\SS_{\d_0}$, then there is an $\Fs_0\in\SS_{\d_0}$ such that for all $\Fs\in\SS_{\d_0}$ with $\Fs<\Fs_0$, we have that the graphs $\d\mapsto \rz m_\Fr(\d)$ and $\d\mapsto \rz m_\Fs(d)$ would be disjoint. After some thought, one could see that this strategy would allow us to prove that if $\d_0,\d_1$ are good infinitesimals and $\ST\subset\SS$ is such that $\ST_{\d_1}$ is coinitial in $\SS_{\d_1}$, then we would also have $\ST_{\d_2}$ coinitial in $\SS_{\d_2}$. From this fact, we could derive that our topology is independent of $\d$.

    This approach foundered. We found good subsets of $\SM^0$ satisfying conditions (1) and (2), but, we could not get a handle on condition (3). This subsection contains the results of this approach. In the next subsection, we use some the results here and a different strategy to reach our goal. We begin with some constructions of good $\SS$ and then consider some order properties for good infinitesimals.

   \begin{definition}
    Let $\SP\SL^0\subset\SM^0$ denote the set of piecewise affine germs in $\SM^0$. That is, an element $[m]\in\SM^0$ is in $\SP\SL^0$ if there is the germ of a countable discrete subset $S$ of points $p_1>p_2>\cdots$ whose only possible limit point is $0$ such that for all $j\in\bbn$, $m|_{[p_{j+1},p_j]}$ is an affine map. If $0<\d\sim 0$, let $\SP l_\d$ denote $\SE_\d(\SP\SL^0)\subset\mu(0)_+$.
\end{definition}
    Given this definition, we have the following construction of an element of $\SP\SL^0$.
    We may assume that if  $m\in\SM^0$, we have eg., $m(1/10)<1/2$. Fix any such $m$ and for $j\in\bbn$ greater that $10$, say, choose $e_j\in\bbn$ such that $(m(1/j))^{e_j}<m(1/(j+1))$. This is clearly possible, as $m(1/j)>0$ for all $j\in\bbn$ and for a given $j_0\in\bbn, j\geq 10$ say, we have $(m(1/j_0))^k\ra 0$ as $k\ra\infty$. Given this, note then that the line segment over the interval $[1/j+1,1/j]$ joining the two points $(1/(j+1),(m(1/(j+1)))^{e_{j+1}})$ and $(1/j,(m(1/j))^{e_j})$ lies below the graph of $m$ over the segment $[1/(j+1),1/j]$ as its lies below the `horizontal' line segment over $[1/(j+1),1/j]$ with `$y$-coordinate' the value $(m(j))^{e_j}<m(1/(j+1))$.
    Hence defining the continuous piecewise affine function $\tl{m}$ on the interval $(0,1/10)$, say,  such that for each $j\in\bbn$, $\tl{m}|_{[1/(j+1),1/j]}$ gives the graph just described, we find for all positive $t<1/10$ that $\tl{m}(t)<m(t)$. We have therefore proved the following.
\begin{lem}
    $\SP l_\d$ is coinitial in $\SN^0_\d$; that is, if $\Fr\in\SN_\d^0$, then there is $\Fs\in\SP l$ such that $\Fs\leq\Fr$.
\end{lem}
\begin{proof}
    If $\Fr\in\SN^0_\d$, there is $m\in \SM^0$ such that $\Fr=\rz m(\d)$, but then choosing $\tl{m}\in\SP\SL^0$ as defined above, transfer of the statement that $0<t<1/10\Rightarrow\tl{m}(t)<m(t)$ gets that $\Fs\;\dot=\rz\tl{m}(\d)<\rz m(\d)$.
\end{proof}
    Yet $\SP\SL^0$ does not either have sufficient rigidity for the evaluation map to be an injection. In order to find a sufficiently rigid semiring, we will use an old result of Borel, see Hardy, \cite{Hardy1924}. But, to make our constructions a little easier, we will, for the moment convert to asymptotics at $\infty$ instead of $0$ using the following recipe. We say that a real valued function $f$ is defined in a neighborhood of infinity if there is $c>0$ such that $f$ is defined on $(c,\infty)$. Note then that $x\mapsto f(x)$  is defined on a neighborhood of infinity if and only if $x\mapsto f(1/x)$ is defined on a (deleted) neighborhood of $0$ in $\bbr_+$ and that $x\mapsto f(x)$ is monotone increasing at infinity with limit infinity (ie., for some such $c$, $c<x<y$ implies $f(x)<f(y)\uparrow\infty$ as $y\uparrow\infty$) if and only if $x\mapsto 1/f(1/x)$ is monotone decreasing to $0$ with limit $0$.
    Given this, it is elementary that given $f,g$ monotone increasing to infinity with $f$ dominating $g$ at infinity, ie., on a possibly smaller neighborhood $(\ov{c},\infty)$ of infinity, we have $x>\ov{c}$ implies that $f(x)>g(x)$, then for all $x>0$ sufficiently small, we have that $1/f(1/x)<1/g(1/x)$. It is also elementary that $x\mapsto f(x)$ is continuous (a convergent power series) in a neighborhood of $\infty$ if and only if $x\mapsto 1/f(1/x)$ is continuous (a convergent power series) in a neighborhood of $0$ in $\bbr_+$.
    This and the previous follows easily once one considers $(\bbr_+,\cdot)$ as a totally ordered abelian group with inversion the (rational) order reversing isomorphism $I:\bbr_+\ra\bbr_+:t\mapsto t^{-1}$ giving a one to one correspondence between germs of functions at $0$ and germs of functions at $\infty$ by $f\mapsto I\circ f\circ I$ (as $I=I^{-1}$). We have the following obvious lemma.
\begin{lem}\label{lem: equiv of ordered germs at 0 and infty}
    Let $\d\in\mu(0)_+$ and $\SD_0\subset\mu(0)_+$ and let $\xi=I(\d)$ and $\SD_\infty=I(\SD_0)$. If $\SF_0$ is a family of germs at $0$ of nonvanishing functions and $\SF_\infty\dot=\{[I\circ f\circ I]:[f]\in\SF_0\}$, then $(\SF_0)_\d$ is coinitial in $\SD_0$ if and only if $(\SF_\infty)_\xi$ is cofinal in $\SD_\infty$.
    Furthermore, if $\SF'_0$ is a second family of germs at $0$ of nonvanishing functions, with $\SF'_\infty$ defined as $\SF_\infty$, then $\SF_0$ and $\SF'_0$ are coinitial in $\SD_0$ if and only if $\SF_\infty$ and $\SF'_\infty$ are cofinal in $\SD_\infty$.
\end{lem}
    Given this conversion recipe, we can prove the following.
\begin{definition}
    Let $\un{\SS}\SM$, respectively $\ov{\SS}\SM$, denote the ring of germs at $0$, respectively at $\infty$, of strictly positive monotone (increasing) convergent power series functions, $m$, with limit $0$ at $0$, respectively with limit $\infty$ at $\infty$, defined on some neighborhood of $0$ in $\bbr_+$, respectively neighborhood of $\infty$ in $\bbr_+$. such that all derivatives of $m$ are positive where $m$ is defined.
\end{definition}
\begin{lem}\label{lem: Hardy power series construction}
    If $[m]\in\SM^0$, then there is $[u]\in\un{\SS}\SM$ such that for some $a>0$ we have that if $0<x<a$, then $u(x)<m(x)$.
\end{lem}
\begin{proof}
    First of all, according to the above conversion recipe any germs in $\SM^0$ just the germ of a mapping $m$ given by $x\mapsto 1/\ov{m}(1/x)$ where $\ov{m}$ is representative of the germ $[\ov{m}]$ at infinity of a continuous monotone increasing mapping; and also according to the above recipe we need only find a monotone convergent power series function $\ov{u}$ at $\infty$ such that for sufficiently large $x$, $\ov{m}(x)<\ov{u}(x)$ for then for sufficiently small $x$, $x\mapsto u(x)\dot=1/\ov{u}(1/x)$ will be a convergent power series function in $\un{\SS}\SM$ with $u(x)<m(x)$.
    So given the $\ov{m}$, we will produce $\ov{u}$ via a construction that Hardy attributes to Borel (see Hardy's text, \cite{Hardy1924}) as follows. For $j\in\bbn$, choose $a_j\in\bbr_+$ with $a_j< a_{j+1}$, $a_j\ra\infty$ as $j\ra\infty$, choose $b_{j+1}\in\bbr_+$ with $a_j<b_{j+1}< a_{j+1}$ for all $j$ and finally we choose $n_j\in\bbn$, $n_j<n_{j+1}$ for all $j$ whose growth will be more closely specified shortly. First note that
\begin{align}
    \ov{u}(x)\;\dot=\;\sum_{j=1}^\infty\Big(\f{x}{b_j}\Big)^{n_j}\;\;\;\text{converges for all}\;x\in\bbr_+
\end{align}
    as can be verified with the root test. Next, note that if $x\in[a_{j},a_{j+1})$ we have $\ov{u}(x)>(a_j/b_j)^{n_j}$, but then assuming we have (inductively) chosen $n_1,\ldots,n_{j-1}$,  choose $n_j$ larger than these and also large enough so that
\begin{align}
    \f{a_j}{b_j}>\Big(\ov{m}(a_{j+1})\Big)^{\f{1}{n_j}}
\end{align}
   Assuming we have so chosen our $n_j$'s, we see that for $x\in [a_{j},a_{j+1})$
\begin{align}
    \ov{u}(x)\geq \ov{u}(a_j)>\Big(\f{a_j}{b_j}\Big)^{n_j}>\ov{m}(a_{j+1})>\ov{m}(x),
\end{align}
    giving our domination requirement. Note that as $\ov{u}$ converges uniformly on any closed interval, then it is indefinitely differentiable, its derivatives given by termwise differentiation. But then, by inspection, one can see that $\ov{u}^{(k)}(x)>0$ for all $x>0$; ie., $\ov{u}\in\ov{\SS}\SM$.
\end{proof}

We will now consider the relationship (with respect to the ordering induced from $\SA_\d$) between the evaluation map $\SE_\d$ and (internal) derivatives of elements of $\SS\SM$. Although this part does not play a role in the present work, it will play a role in later work on differentiable germs.
\begin{definition}
    If $0<\d\sim 0$, let $\SA_\d$ denote the image of $\SS\SM$ under the evaluation homomorphism $\SE_\d$.
    We say that an infinitesimal $\d\in\rz\bbr_+$ is generic if the map $\SE_\d|\SA$ is an injection.
\end{definition}
    In a later section we shall need the fact that $\SN_\d$ is given by the set of values $\rz m(\d)$ as $[m]$ varies in $\SM$. Here, we need a more precise determination of a coinitial subset of $\SN_\d$.
\begin{cor}\label{cor: E_d is an isomorphism onto A_d}
     $\SA_\d$ is cofinal in $\SN_\d$. Also, there exists generic infinitesimals $\d\in\rz\bbr_+$, ie., the semiring homomorphism $\SE_\d:\SS\SM\ra\SA_\d$ is an isomorphism.
\end{cor}
\begin{proof}
    The proof of the first statement is an immediate consequence of the previous lemma. The last part of the second statement follows from the existence of a generic infinitesimal $\d$.  The existence of $\d$ depends on the fact that if power series function germs $[m],[\ov{m}]\in\un{\SS}\SM$ are equal on the germ of a countable set with limit point $0$, then $[m]=[\ov{m}]$. This fact will allow the construction of a concurrent relation. Saturation of $\rz\bbr$ will imply  the existence of ideal points, ie., generic points, for this relation.

    Let $\SS$ denote the set of convergent power series $(m,(0,c))$ defined on some neighborhood $(0,c)$ of $0$ in $\bbr_+$. Let $\SS^{(2)}$ demote the set of unequal ordered pairs
\begin{align}
    \SS^{(2)}\;\dot=\;\{((m,(0,c)),(\ov{m},(0,\ov{c})))\in\SS\x\SS:m\not=\ov{m},c\not=\ov{c}\}
\end{align}
    and define a relation $\SR\subset\bbr_+\x\SS^{(2)}$ by
\begin{align}
    (s,((m,(0,c)),(\ov{m},(0,\ov{c}))))\in\SR\;\text{if}\; s<\min\{c,\ov{c}\},\;\text{and}\; m(s)\not=\ov{m}(s).
\end{align}
    This relation is concurrent; that is, if $k\in\bbn$ and  $((m_j,(0,c_j)),(\ov{m}_j,(0,\ov{c}_j)))$ for $j=1,\ldots,k$ are $k$ elements of $\SS^{(2)}$, then there is $s_0\in\bbr_+$ with $s_0<c_0\dot=\min_j\{c_j,\ov{c}_j\}$ such that $m_j(s_0)\not=\ov{m}_j(s_0)$ for all $j\leq k$. If
\begin{align}
    E=\{s\in (0,c_0):\;\text{there is}\;j\;\text{such that}\;m_j(s)=\ov{m}_j(s)\},
\end{align}
     then clearly $E$ is a countable subset of $(0,c_0/2)$, eg., there is $s_0\in (0,c_0/2)\ssm E$ which by definition satisfies our condition. The result now follows from saturation, ie., there is $\d\in\rz\bbr_+$, such that for every $(m,(0,c))\in \SS$, we have $(\d,(\rz m,\rz(0,c)))\in\rz\SR$. This means that $\d<\rz c$ for every $c\in\bbr_+$ (ie., is infinitesimal) and second, if $[m_1],[m_2]\in\un{\SS}\SM$ are distinct germs, then $\rz m_1(\d)\not=\rz m_2(\d)$.
\end{proof}
\begin{definition}
    The isomorphism $\SE_\d$ mapping $\SS\SM$ onto a semiring of positive elements of $\rz\bbr_{nes,+}$ induces a total ordering $\stackrel{\d}{\prec}$ on $\SS\SM$ by $[m_1]\stackrel{\d}{\prec} [m_2]$ if and only if $\SE_\d(m_1)<\SE_\d(m_2)$.
\end{definition}

\begin{remark}
    This total order extends the total ordering of all field extensions of $\bbr_+$ by fields of functions (Hardy type fields) on $(\bbr_+,0)$. We don't yet know how this fit into the current framework (see eg., Boshernitzan, \cite{BoshernitzanOrdOfInfinty1982}, or Aschenbrenner and van den Dries, \cite{vandendriesAschenHfieldExtensions2002}) and  will return to this point later.
    The problem with this isomorphism for our needs is that as we vary our generic  $\d$'s, although it leaves the ordering of these field extensions of $\bbr$ unchanged, it is shuffling a large number of the $\SE_\d(m)$'s around for arbitrary $m\in\SM^0$. We can only fix  this problem in an asymptotic sense, ie., as we vary our infinitesimal (generic or not) the shuffling does not alter asymptotic behavior of families in $\SM^0$.
    This will be sufficient to prove that our topology $\tau_\d$ is independent of $\d$.
\end{remark}
     We need a further restriction of our coinitial structures in order to prove the main result of this subsection.
\begin{definition}
     Let $\wt{\SS\SM}=\{[m]\in\SS\SM:\rz m(\xi)=o(\xi)\;\text{for}\;\xi\in\mu(0)_+\}$ and $\wt{\SA}_\d=\SE_\d(\wt{\SS\SM})$.
\end{definition}
     Let's collect some useful properties of $\wt{\SS\SM}$.
\begin{lem}
     $\wt{\SS\SM}$ is a convex ideal in the semiring $\SS\SM$ as is $\wt{\SA}_\d$ in $\SA_\d$ for any positive infinitesimal $\d$. In particular, $\wt{\SS\SM}$ is coinitial in $\SS\SM$ and $\tl{\SA_\d}$ is coinitial in $\SA_\d$.
     Also if $[m]\in\wt{\SS\SM}$, then $m'(\xi)\sim 0$ for all  $\xi\in\mu(0)_+$ and we also have that for all $\xi\in\mu(0)_+$, $m(\xi)<\xi m'(\xi)$.
\end{lem}
\begin{proof}
     Let $\z\in\mu(0)_+$ and as $m(\xi)=o(\xi)$ for all $\xi\in\mu(0)_+$, then choosing  $\la\in\mu(0)_+$ such that $\la \gg \z$ and so $m(\la)\ll \la$ implies that $m(\la)\ll \la-\z$ and therefore
\begin{align}
     m'(\z)< m'(\ov{\la})=\f{m(\la)-m(\z)}{\la-\z}<\f{m(\la)}{\la-\z}\sim 0,
\end{align}
     where we used the transferred mean value theorem to get $\ov{\la}\in (\z,\la)$ for the equality above; the rest follows by induction.
\end{proof}
     The previous result can be improved with the following lemma.
\begin{lem}
     Suppose that $[m]\in\SS\SM$ and $\d\in\mu(0)_+$ and $m(\d)=o(\d)$. Then $m(\xi)=o(\xi)$ for all $\xi\in\mu(0)_+$; ie., $[m]\in\wt{\SS\SM}$.
\end{lem}
\begin{proof}
     First of all, note that as the graph of $m$ is convex on $\mu(0)_+$ (as $m''>0$ on $\mu(0)_+$), then for all $\xi\in\mu(0)_+$ with $\xi<\d$, we have $m(\xi)=o(\xi)$. But if $0<c\in\bbr$, and $B_c=\{t\in\rz\bbr:m(t)<ct\}$, then $(0,\d)\subset\rz B_c$ and so by lemma \ref{lem: *A contains inf nbd-> has stand nbd}, $\rz B_c$ contains a standard neighborhood, eg., $\mu(0)_+\subset\rz B_c$, and as $c>0$ was arbitrary in $\bbr_+$, we are finished.
\end{proof}
\begin{cor}
    $\wt{\SS\SM}=\SE_\d^{-1}(\wt{\SA_\d})$.
\end{cor}
\begin{proof}
    This follows from Corollary \ref{cor: E_d is an isomorphism onto A_d} and the previous lemma.
\end{proof}

\subsubsection{Topological independence from delta}\label{subsec: top independence from delta}

    Although $\SN=\SN^\d$ depends on the `size' $\d$ of our infinitesimal disk, we will see that all topological properties are independent of the choice for $\d$. This approach has the following ingredients. First of all, we break the Hardy construction in two parts: we get a map $\SH:\SM^0\ra Incr$ where $Incr$ is the set of strictly increasing sequences of integers, the set of sequences of exponents for the $\ov{u}$'s of the previous subsection, and we have the map $E:Incr\ra \FH_{<,\;B}$ where $\FH_{<,\;B}$ is the set of Hardy series and $E$ assigns to an increasing sequence the corresponding series. Second, we prove an asymptotic growth result for $Incr$, lemma \ref{lem: increasing sequence bdd at point}, that says that given a subset $\FI\subset Incr$ and an infinite $\om_0\in\rz\bbn$ with the property that the set of values $\rz\sbbn(\om_0)$ as $\sbbn$ varies in $\FI$ is bounded in the set $\{\rz\sbbn(\om_0):\sbbn\in Incr\}$, then there is an element $\wt{\sbbn}\in Incr$ such that $\sbbn(j)<\wt{\sbbn}(j)$ for all $\sbbn\in\FI$ and $j\in\bbn$. Next, we prove that this pointwise-bound-implies-uniform-bound in a family of exponents of Hardy series implies a a pointwise implies uniform result for the graphs of families of Hardy series, see lemma \ref{lem: Incr pwwise bdd->Hrdy ptwise bdd}. Combining this with the  systematic bounding of elements of $\SM^0$ by elements Hardy series, we are able to prove our main assertion: that the topology given by $\tau_{\d_0}$ and $\tau_{\d_1}$ for any positive infinitesimals $\d_0,\d_1$ are homeomorphic.

    We need a specific subspace of $\ov{\SS}\SM$ taylored for the considerations here. In order to prove it has certain properties, we need a result about boundedness of increasing sequences of integers that are all bounded at a given infinite value. We need some notation.
\begin{definition}
    Let $Incr=\{\sbbn:\bbn\ra\bbn\;|\;\sbbn(j)<\sbbn(j+1)\;\text{for all}\;j\in\bbn\}$ and if $k\in\bbn$, $\un{\sbbn}\in Incr(\bbn)$, let $Incr_{k,\un{\sbbn}}=\{\sbbn\in Incr: \sbbn(k)\leq \un{\sbbn}(k)\}$ and note that if $\om\in\rz\bbn$, we still have a well defined standard set $Incr_{\om,\un{\sbbn}}=\{\sbbn\in Incr:\rz \sbbn(\om)\leq\rz\un{\sbbn}(\om)\}$.
    If $\sbbn_1,\sbbn_2\in Incr$, we write $\sbbn_1<\sbbn_2$, and say that $\sbbn_1$ is smaller than $\sbbn_2$, if for sufficiently large $j_0\in\bbn$, we have $\sbbn_1(j)<\sbbn_2(j)$ for $j>j_0$.
    If $\SL\subset\FZ\dot=\;^\s Incr(\a)=\{\rz\sbbn(\a):\sbbn\in Incr,\a\in\rz\bbn_\infty\}$ and $\om\in\rz\sbbn$, then we say that a subset $\FI\subset Incr$ is $\om$-cofinal with $\SL$ if for each $\la\in\SL$, there is $\sbbn\in \FI$ such that $\rz\sbbn(\om)>\la$. If $\SL=\FI ncr_\om\dot=\{\rz\sbbn(\om):\sbbn\in Incr\}$, we say that $\FL$ is $\om$-cofinal.
    If for every $\sbbn\in Incr$, there is $\ov{\sbbn}\in\FI$ and $j_0\in\bbn$ with $\sbbn(j)<\ov{\sbbn}(j)$ for $j>j_0$, then we say that $\FI$ is cofinal.
\end{definition}
    We chose $\SL$ as a subset of $\FZ$ instead of $\rz\bbn$ as $\FZ$ is far from being all of $\rz\bbn$ and we are considering the analogs of the $\SN_d$'s for strategic reasons. Now for clarity, first note that  $\sbbn_1<\sbbn_2$ precisely when $\rz\sbbn_1(\xi)<\rz\sbbn_2(\xi)$ for all (sufficiently) small $\xi\in\rz\bbn_\infty$. Next,  note that a subset $\FI\subset Incr$ is not $\om$-cofinal precisely when it
    Given these definitions, we have a lemma.
\begin{lem}\label{lem: increasing sequence bdd at point}
    Let $\om\in\rz\bbn_\infty$ and $\un{\sbbn}\in Incr$. Then there is $\wt{\sbbn}\in Incr$ such that for each $\sbbn\in \rz Incr_{\om,\un{\sbbn}}$ we have $\sbbn(j)<\wt{\sbbn}(j)$ for all $j\in\bbn$.
\end{lem}
\begin{proof}
    Let $X\subset Incr$ denote $Incr_{\om,\un{\sbbn}}$ and let
\begin{align}
    P=\{k\in\bbn:\sbbn(k)/\un{\sbbn}(k)\leq 1\;\text{for all}\;\sbbn\in X\}.
\end{align}
    Then, clearly $\om\in\rz P$ and so $P$ is infinite (see Hirschfeld, \cite{Hirschfeld1988}, for this interesting hook). Let $k_1<k_2<\cdots$ denote an enumeration of the elements of $P$\;; so for each $j\in\bbn$, and $\sbbn\in X$ we have $\sbbn(l)<\sbbn(k_j)\leq \un{\sbbn}(k_j)$ for $l<k_j$. So given this, define $\wt{\sbbn}:\bbn\ra\bbn$ as follows: for $l<k_1$, define $\wh{\sbbn}(l)=\un{\sbbn}(k_1)$ and inductively for $i\geq 1$, suppose that we have defined $\wh{\sbbn}(l)$ for $l<k_i$. Then,  for $l$ such that $k_i\leq l<k_{i+1}$, define $\wh{\sbbn}(l)=\un{\sbbn}(k_{i+1})$. With this, $\wh{\sbbn}(j)\leq \wh{\sbbn}(j+1)$ for all $j$ and has the property that if $\sbbn\in X$, then $\sbbn(j)\leq\wh{\sbbn}(j)$ for all $j$. Therefore, defining $\wt{\sbbn}(j)\dot=\wh{\sbbn}(j)+j$, we have $\wt{\sbbn}\in Incr$  and if $\sbbn\in X$, $l\in\bbn$, then for some $i\in\bbn$, $k_i\leq l\leq k_{i+1}$ and so $\sbbn(l)<\sbbn(k_{i+1})\leq\un{\sbbn}(k_{i+1})=\wh{\sbbn}(l)<\wt{\sbbn}(l)$, as we wanted.
\end{proof}

 We have a consequence of this lemma will be critical to proving the invariance of our topology $\tau_\d$ of the choice of infinitesimal $\d$.
\begin{cor}\label{cor: FI_(om_0) cofin<->FI_(om_1) cofin}
     Suppose that $\FI\subset Incr$ and $\om_0,\om_1\in\rz\bbn_\infty$. Then $\FI_{\om_0}$ is cofinal in $\FI ncr_{\om_0}$  if and only if $\FI_{\om_1}$ is cofinal in $\FI ncr_{\om_1}$.
\end{cor}
\begin{proof}
    Suppose, by way of contradiction, that $\FI_{\om_0}$ is cofinal in $\FI ncr_{\om_0}$, but $\FI_{\om_1}$ is not cofinal in $\FI ncr_{\om_1}$. So we have that there is $\un{\sbbn}\in Incr$ such that $\FI\subset \rz Incr_{\om_1,\;\un{\sbbn}}$ and therefore the previous lemma implies that there is $\wt{\sbbn}\in Incr$ such that for every $\sbbn\in \rz Incr_{\om_1,\un{\sbbn}}$, we have $\sbbn(j)<\wt{\sbbn}(j)$, and so by transferring this statement for each given $\sbbn\in Incr_{\om_1,\;\un{\sbbn}}$, we have that for each $\sbbn\in Incr_{\om_1,\;\un{\sbbn}}$ that $\rz\sbbn(\xi)<\rz\wt{\sbbn}(\xi)$ for each $\xi\in\rz\bbn$; eg., for each such $\sbbn$, we have that $\rz\sbbn(\om_0)<\rz\wt{\sbbn}(\om_0)$. But this says, in particular that $\{\rz\sbbn(\om_0):\sbbn\in\FI\}$ has all elements less than $\rz\wt{\sbbn}(\om_0)$, contradicting our hypothesis.
\end{proof}
  From this corollary follows another rather surprising consequence.
\begin{cor}\label{cor: sequence ptwise cofinal->unif cofinal}
    Suppose that $\FI\subset Incr$ with $\FI_{\om_0}$ cofinal in $Incr_\om$. Then for every $\sbbn\in Incr$, there is $\wt{\sbbn}\in\FI$ such that $\rz\sbbn(\om)<\rz\wt{\sbbn}(\om)$ for all $\om\in\rz\bbn_\infty$.
\end{cor}
\begin{proof}
    Suppose not, ie., there is $\sbbn_0\in Incr$, such that for every $\wt{\sbbn}\in\FI$, there exists $\om_1\in\rz\bbn_\infty$ with $\rz\sbbn_0(\om_1)\geq\wt{\sbbn}(\om_1)$. But this just says that $\FI_{\om_1}$ is not cofinal which by the previous corollary implies that $\FI_{\om_0}$ is not cofinal, a contradiction.
\end{proof}
  Given the above initial work on increasing sequences of integers, we will now formally systematize the Hardy growth properties in these families of analytic maps by carrying the asymptotic ordered properties of $Incr$ into the analytic families.
\begin{definition}\label{def: SM_<,B and SM^om,n_<,B}
    Given $B$ (and $A$) as above, $\un{\sbbn}\in Incr(\bbn)$ and $\om_0\in\rz\bbn_\infty$, define the $B$-Hardy class of analytic maps
\begin{align}
    \FH_{<,B}=\{ u\in\ov{\SS}\SM:u(x)=u_{\sbbn}(x)=\sum_{j=1}^\infty b_j^{-\sbbn(j)}x^{\sbbn(j)}\;\text{for}\;\sbbn \in Incr\}\qquad\qquad\qquad\qquad\qquad\quad \\ \notag
    \text{and}\qquad\;\;\FH^{\om_0,\un{\sbbn}}_{<,B}=\{u_\sbbn\in\ov{\SS}\SM_{<,B}:\rz \sbbn(\om_0)\leq\rz\un{\sbbn}(\om_0)\}=E(Incr_{\un{\sbbn},\;\om_0})\qquad\qquad\qquad\qquad\qquad\qquad\qquad
\end{align}
    where $E:Incr\ra\ov{\SS}\SM_{<,\;B}$ denote the map $\sbbn\mapsto u_\sbbn$.
\end{definition}
    With this, we have the following lemma.
\begin{lem}\label{lem: B power series with a bdd exponent}
      $E$ is a strict order preserving  bijection onto $\FH_{<,\;B}$. Furthermore, if $\om_0\in\rz\bbn_\infty$ and  $u\in\FH^{\om_0,\un{\sbbn}}_{<,B}$ for some $\un{\sbbn}\in Incr$,  then there is $\wt{\sbbn}\in Incr$ such that  $\FH^{\om_0,\;\un{\sbbn}}_{<,\;B}\subset \FH^{\om,\;\wt{\sbbn}}_{<,\;B}$ for all $\om\in\rz\bbn_\infty$. So  there is $\wt{u}\in\FH_{<,\;B}$ such that $u(x)<\wt{u}(x)$ for all $u\in\FH_{<,\;B}^{\om_0,\;\un{\sbbn}}$ for all sufficiently large $x\in\bbr_+$.
\end{lem}
\begin{proof}
     If $\sbbn,\un{\sbbn}\in Incr$ with $\sbbn(j)<\un{\sbbn}(j)$ for all $j\in\bbn$, then $E(\sbbn)(x)<u(\un{\sbbn})(x)$ for all $x\in\bbr_+$. It's easy to see that $E$ is a bijection which is order preserving in the sense that if $\sbbn,\un{\sbbn}\in Incr$ with $\sbbn(j)<\un{\sbbn}(j)$ for all $j\in\bbn$, then $u(\sbbn)(x)<u(\un{\sbbn})(x)$ for all $x\in\bbr_+$. Note that as $\FH^{\om,\un{\sbbn}}_{<,B}=E(Incr_{\om,\un{\sbbn}})$, then as the $\wt{\sbbn}$ found in Lemma \ref{lem: increasing sequence bdd at point} satisfies $\wt{\sbbn}(j)>\sbbn(j)$ for all $\sbbn\in Incr_{\om,\un{\sbbn}}$, then $\wt{u}\dot=E(\wt{\sbbn})$ satisfies $\wt{u}(x)>u(x)$ for all $u\in\FH^{\om,\un{\sbbn}}_{<,B}$ and $x\in\bbr_+$, as we wanted to show.
\end{proof}

\begin{definition}
   If $\xi\in\rz\bbr_+$ is infinite and $\FK\subset\FH_{<,\;B}$, let $\FH_B^\xi=\{\rz u(\xi):u\in\FH_B\}$ and $\FK_\xi\subset\FH^\xi_B$ denote $\{\rz u(\xi):u\in \FK\}$.
\end{definition}

\begin{lem}\label{lem: Incr pwwise bdd->Hrdy ptwise bdd}
    Let $\xi_0\in\rz\bbr_+$ be infinite and $\SA\subset\FH_{<,\;}$. Then there is $\om_0\in\rz\bbn_\infty$ such that   $\SA_{\xi_0}$ is not cofinal in $\FH^{\xi_0}_B$ if and only if $E^{-1}(\SA)\subset Incr_{\un{\sbbn},\;\om_0}$ for some $\un{\sbbn}\in Incr$. In other words, $\SA_{\xi_0}$ is cofinal in $\FH_B^{\xi_0}$ if and only if $E^{-1}(\SA)_{\om_0}$ is cofinal in $\FI ncr_{\om_0}$.
\end{lem}
\begin{proof}
    It suffices to prove that if $\SA_{\xi_0}$ is not cofinal in $\FH^{\xi_0}_B$, then there is $\om_0\in\rz\bbn_0$ such that $E^{-1}(\SA)\subset Incr_{\un{\sbbn},\om_0}$. Suppose that this is not true. Then, by corollary \ref{cor: sequence ptwise cofinal->unif cofinal},  for every $\sbbn \in Incr$, there is $\wh{\sbbn}\in E^{-1}(\SA)$ such that $\rz\sbbn(\om)<\rz\wh{\sbbn}(\om)$ for all $\om\in\rz\bbn_\infty$. But as $E$ is order preserving (lemma \ref{lem: B power series with a bdd exponent}), then this implies that for every $\sbbn\in Incr$, there is $\wh{\sbbn}\in E^{-1}(\SA)$ such that $\rz E(\sbbn)(\xi)<\rz E(\wh{\sbbn})(\xi)$ for all infinite $\xi\in\rz\bbr_+$. But  as $\FH_{<,\;B}=E(Incr)$, this says that for all $u\in\FH_{<,\;B}$, $\rz u(\xi)<\rz E(\wh{\sbbn})(\xi)$ for all infinite $\xi\in\rz\bbr_+$, an absurdity.
\end{proof}
\begin{cor}\label{cor: FK_xi_0 cofin<->FK_xi_1 cofin}
     If $\FK\subset\FH_{<.\;B}$ and $\xi_0,\xi_1\in\rz\bbr_+$ are infinite, then $\FK_{\xi_0}$ is cofinal in $\FH^{\xi_0}_B$ if and only if $\FK_{\xi_1}$ is cofinal in $\FH^{\xi_1}_B$
\end{cor}
\begin{proof}
    By the previous lemma there exists $\om_0$ and $\om_1$ in $\rz\bbn_\infty$ such that $\FK_{\xi_j}$ is cofinal in $\FH^{\xi_j}$ if and only if $E^{-1}(\FK)_{\om_j}$ is cofinal in $Incr_{\om_j}$ for $j=1,2$. But by corollary \ref{cor: FI_(om_0) cofin<->FI_(om_1) cofin}, $E^{-1}(\FK)_{\om_0}$ is cofinal in $Incr_{\om_0}$ if and only if $E^{-1}(\FK)_{\om_1}$ is cofinal in $Incr_{\om_1}$.
\end{proof}

    Returning to the Hardy series bounding above a given element of $\SM^0$, with a bit more care on the determinations of the exponents of the power series $u$ in Lemma \ref{lem: Hardy power series construction}, we will give a bound on a particular exponent in the power series expansion in terms of the magnitude of the monotone function $m$ at a value associated with the given index of this exponent.
    This will give us a map from subsets of $\SM^0$ to subsets of $\ov{\SS}\SM$, so that by factoring through the map $E$, this map will allow a strong correspondence between the growth of a given subset of $\SM^0$ at infinity, and the corresponding subset of $\FH_{<,\;B}$.

    Without loss of generality in the Hardy construction, we will assume that $q=a_j/b_j>1$ is constant. For example, if we want $q=2$, we can let $a_j=e^j$ and $b_j=e^j/2$, getting $a_{j}<b_{j+1}<a_{j+1}$ for all $j\in\bbn$. As is easy to see, this has no effect on the above construction. This will ease the process of finding an explicit formula for our sequence of exponents.
    Given $[m]\in\SM^0$, with $m\in [m]$, and $j\in\bbn$, taking the logarithm of the expression $q^{n_j}>m(a_{j+1})$, we find that for this expression to hold and for $n_j>n_{j-1}$ to be satisfied, it suffices to define out exponent $n_j$ as follows:  we find  that if we define $\SH(m):\bbn\ra\bbn$ by

\begin{align}\label{eqn: def of n^m_j for m in SM^0}
     \SH(m)(j)\;\dot=\;\wt{m}(j)\;\dot=\;\ell\bigg(\f{\log(m(a_{j+1}))}{\log(q)}\bigg)+j.
\end{align}
    where if $r\in\bbr_+$, $\ell(r)$ is the least integer $k$ such that $k\geq r$. To explain, clearly $1<a_j<a_{j+1}$ and monotonicity of $m$ implies that  $\log (m(a_j))<\log(m(a_{j+1}))$, and the assertion follows.
     So given the $m\in\SM^0$, we have defined an element $\wt{m}\in Incr$ by defining $\wt{m}(j)$ to be the integer $n_j$ above and  with $\wt{m}$ so defined in terms of $m$, we get a specific form of the power series constructed in Lemma \ref{lem: Hardy power series construction}
\begin{align}\label{eqn: def of bnding exp as integer of log}
    u_m(x)=\sum_{j=1}^\infty \bigg(\f{x}{b_j}\bigg)^{\wt{m}(j)}
\end{align}
    which as noted in that lemma clearly converges uniformly on compact intervals, and so eg., is analytic on $\bbr_+$.
    Formalizing this, we have
\begin{lem}\label{lem: E(H(m))>m}
    Let $\SH:\SM^0\ra Incr$ denote the map $m\ra\wt{m}$ and $E:Incr\ra\FH_{<,\;B}$ the map defined in definition \ref{def: SM_<,B and SM^om,n_<,B}. Then for all sufficiently large $x\in\bbr_+$, we have that $E(\SH(m))(x)>m(x)$. In particular, if $\xi\in\rz\bbr_+$ is infinite, then $\rz E(\SH(m))(\xi)>\rz m(\xi)$.
\end{lem}
    Let $\un{m}\in\SM^0$ and $\xi_0\in\rz\bbr_+$ be infinite and define $\bsm{\SJ_{\xi_0,\un{m}}}=\{m\in\SM^0:\rz m(\xi_0)\leq\rz\un{m}(\xi_0)\}$.
    Suppose now that $\SJ\subset\SM^0$ is such that for infinite $\xi_0\in\rz\bbr_+$, $\SJ_{\xi_0}$ is not cofinal in $\SN_{\xi_0}$; that is, there is $\un{m}\in\SM^0$ such that $\un{m}(\xi_0)>m(\xi_0)$ for all $m\in\SJ$. Now for each $m\in\SM^0$, we have $\SH(m)=\wt{m}\in Incr$ given by the increasing sequence $\wt{m}(1)<\wt{m}(2)<\cdots$, and we also have $\xi_0\in\rz[a_\om,a_{\om+1})$ for some $\om\in\rz\bbn_\infty$.

     Given this we have the following assertion.
\begin{lem}\label{lem: un(n)is sequence in N corres to un(m)}
   Suppose that $\SP\subset\SM^0$ and $\xi_0\in\rz\bbr_+$ is infinite. Suppose also that $\SP_{\xi_0}$ is not cofinal in $\SN_{\xi_0}$. Then $\SP_{\xi}$ is not cofinal in $\SN_\xi$ for all infinite $\xi\in\rz\bbr_+$.
\end{lem}
\begin{proof}
   As $\SP$ is not cofinal in $\SN_{\xi_0}$, there is $\un{m}\in\SM^0$ so that $\rz\un{m}(\xi_0)>\rz m(\xi_0)$ for all $m\in\SP$. But by lemma \ref{lem: E(H(m))>m}, this says that $\rz u_{\un{m}}(\xi_0)\dot=\rz E(\SH(\un{m}))(\xi_0)>\rz u_{m}(\xi_0)$ for all $u_m\in E(\SH(\SP))$. That is, $E(\SH(\SP))_{\xi_0}$ is not cofinal in $\FH_B^{\xi_0}$. But by corollary \ref{cor: FK_xi_0 cofin<->FK_xi_1 cofin}, this implies that $E(\SH(\SP))_{\xi}$ is not cofinal in $\FH_B^{\xi}$ for all infinite $\xi\in\rz\bbr_+$. Yet again by lemma \ref{lem: E(H(m))>m} (applied now for all infinite $\xi$), this says that $\SP_\xi$ is not cofinal in $\SN_\xi$ for all infinite $\xi\in\rz\bbr_+$.
\end{proof}
   This has an easy but important consequence.
\begin{cor}\label{cor: cofinal independ of delta}
    Suppose that $m\in\SJ_{\xi_0,\un{m}}$ for some $\un{m}\in\SM^0$ and some infinite $\xi_0\in\rz\bbr_+$. Then there is $\wt{u}\in\FH$ such that $m\in\SJ_{\xi,\wt{u}}$ for all infinite $\xi\in\rz\bbr_+$. Said differently, if $\SA\subset\SM^0$, and $\xi_0,\xi_1\in\rz\bbr_+$ are infinite, then $\SA$ is cofinal at $\xi_0$ if and only if $\SA$ is cofinal at $\xi_1$.
\end{cor}
\begin{proof}
    Let $\wt{u}$ be equal to $E(\sbbn_{\un{m}})$ in the previous lemma. Then the transfer of the conclusion of that lemma implies the following statement. If $m\in\SJ_{\xi_0,\;\un{m}}$ for some $\xi_0$ and $\un{m}\in\SM^0$, then $\rz m(\xi)<\rz\wt{u}(\xi)$ for all $\xi\in\rz\bbr_+$, in particular, the infinite $\xi$; ie., $m\in\SJ_{\xi,\;\wt{u}}$. The second statement follows easily: by symmetry we need only show cofinal at $\xi_0$ implies cofinal at $\xi_1$. Suppose not, ie., we have cofinal at $\xi_0$ and not $\xi_1$, ie., $\SA\subset\SJ_{\xi_1,\un{m}}$ for some $m\in\SM^0$, but then the conclusion of first part says $\SA$ is not cofinal at all $\xi$, eg., at $\xi_0$, a contradiction.
\end{proof}

    Summarizing our framework, we have the following. We have the mapping system
\begin{align}\label{diag: cont monotone to monotone analytic}
    \xymatrix{\SM^0\ar@{->}[r]^{\SH}& Incr\;\ar@{>->>}[r]^{E} &\FH_{<,B}}\; :[m]\mapsto (\sbbn_m(j))_{j\in\bbn}\mapsto\sum_j b^{-\sbbn_m(j)}x^{\sbbn_m(j)}
\end{align}
    where for $[m]\in\SM^0$, $E(\SH([m]))=[u_m]$ the element of $\SM_{<,B}$ satisfying $\rz u_m(\xi)>\rz m(\xi)$ for all $\xi\in\mu(0)_+$. With the above results we can conclude the following.
\begin{cor}\label{cor: relating bnd in Incr to bnd in SM}
    If $\xi_0\in\rz\bbr_+$ is infinite, $\un{m}\in\SM^0$, $\om\in\rz\bbn_\infty$ is the unique integer such that $\xi_0\in\rz[a_\om,a_{\om+1})$ and $\un{n}\in Incr$ is the sequence of integers defined in Lemma \ref{lem: un(n)is sequence in N corres to un(m)}, then  $E\circ\SH(\SJ_{\xi,\un{m}})\subset\FH^{\om,\un{n}}_{<,B}$. In particular, if $[m]\in\SM^0$ is such that $\rz m(\xi_0)\leq \rz \un{m}(\xi_0)$, then $u_m(x)<E(\un{n})(x)$ for all sufficiently large $x\in\bbr_+$; eg., $\rz u_m(\xi_0)<\rz E(\un{n})(\xi_0)$.
\end{cor}
\begin{proof}
    This follows from the previous work; ie., by Lemma \ref{lem: un(n)is sequence in N corres to un(m)}, we have $\SH(\SJ_{\xi_0,\un{m}})\subset Incr_{\om,\un{n}}$ and by Lemma \ref{lem: B power series with a bdd exponent}, $E(Incr_{\om,\un{n}})\subset\FH^{\om,\un{n}}_{<,B}$. The last part is just a restatement of this.
\end{proof}

\begin{cor}\label{cor: dominating analytic family shadow subset M^0}
    Suppose that $\d\in\mu(0)_+$ and that $\SJ\subset\SM^0$. Then there is $\SK\subset\un{\SS}\SM$ such that for each $0<\d\sim 0$ we have that $\SK_\d=\SE_\d(\SK)$ is coinitial with $\SJ_\d=\SE_\d(\SJ)$.
    That is, $\SK$ can be chosen so that for any $\d\sim 0$, $\SJ_\d$ is coinitial in $\SN_\d$ if and only if $\SK_\d$ is.
\end{cor}
\begin{proof}
    Applying lemma \ref{lem: equiv of ordered germs at 0 and infty} to the above corollary,  to get the corresponding result for germs at $0$.
\end{proof}
    From this point until the end of the paper, we will use $\SS\SM$ instead of $\un{\SS}\SM$.

\begin{thm}\label{thm: Id map: tau_d ot tau_d'is homeo}
    If $0<\d,\d'\sim 0$, with $\d$ generic, the identity map $\SI:(\SG^0,\tau_\d)\ra (\SG^0,\tau_{\d'})$ is a  homeomorphism. Hence, for all positive infinitesimals $\d,\d'$, we have that $\SI:(\SG^0,\tau_\d)\ra (\SG^0,\tau_{\d'})$ is a homeomorphism.
\end{thm}
\begin{proof}
    We will show that a net in $\SG$ converges with respect to the $\tau_\d$ topology if and only if it converges with respect to the $\tau_{\ov{\d}}$ topology.
\end{proof}
    We have the useful corollary.
\begin{cor}\label{cor: f_d converges in tau_d iff in tau_d' }
    Suppose that $[g]\in\SG^0$,\; $(D,<)$ is a directed set and $([f_d]:d\in D)$ is a net in $\SG^0$ and $\d,\d'$ are positive infinitesimals. Then $\norm{\rz f_d-\rz g}_\d$ converges to $0$ if and only if $\norm{\rz f_d-\rz g}_{\d'}$ converges to $0$; that is, $[f_d]\ra [g]$ in $\tau_\d$ if and only if $[f_d]\ra [g]$ in $\tau_{\d'}$.
\end{cor}

\subsection{Relationship with nongerm  convergence}\label{sec: relationship with nongerm convergence}
    First of all, whether talking about a sequence of germs at $0$ or a sequence of continuous functions defined on a neighborhood of $0$, we will fail at finding a relation with $\tau$ convergence. For germs, see below. As far as a sequence of functions defined on some fixed neighborhood of zero, we must still deal with the fact that there are no countable neighborhood bases of the zero germ in the $\tau$ topology, and so unless we can extend the sequence to an uncountable net without a countable coinitial subnet, we are stuck. Of course, we can transfer the sequence noting that $\rz\bbn$ does not have a countable coinitial subset. Although $\rz\bbn$ is too large, *finite initial intervals $\{1,2,\ldots,\om\}$ don't have countable coinitial subsets, and are suitable for what we need. We do introduce the problem that our nets are no longer standard, but as we will see below, the relationship of these *finite nets to the $\tau$ is directly related to standard behavior of the original sequence. Note that, we could have also considered `*finite' directed index sets of the form $\{\Fj\in\rz\bbn:\Fj\lll\Fz\}$ where $\Fz$ is an infinite integer. Given mild saturation, these also do not have countable coinitial subsets and in a sense are more natural, but they are not internal, a needed property.

    If we have a sequence of germs $\{[f_j]:j\in\bbn\}$ in $\SG^0_n$, and consider the transferred sequence $\{\rz[f_\Fj]:\Fj\in\rz\bbn\}$, we see immediately that for any given $0<\d\sim 0$, there is $\om\in\rz\bbn$ large enough to that representatives of $\rz[f_\om]$ are not well defined on $B_\d$. For example, if $\chi_{B_r}:\bbr\ra\{0,1\}$ is the indicator function of $B_r$, let $\wt{f}_j(x)=\chi(2^jx)f_j(x)$ where we extend the function $f_j$ arbitrarily outside of $B_r$ so that $\wt{f}_j$ will be well defined on $B_r$. Then for all $j\in\bbn$, $[f_j]=[\wt{f}_j]$ as elements of $\SG$, ie., their values on $\mu(0)$ are the same. But for $\Fj\in\rz\bbn_\infty$, $\rz f_\Fj$ and $\rz\wt{f}_\Fj$ are not equal on $\mu(0)$.
    On the other hand, if the germs of mappings arise from standard mappings all defined on some standard ball $B_r$ centered at $0$ and if we have a sequence of such $f_j$, $j\in\bbn$, then clearly all $\rz f_\Fj$ for $\Fj\in\rz\bbn$ are defined on all of $\mu(0)$, but note that even if $\rz f_\Fj(\xi)$ is nearstandard for $\xi\sim 0$, typically $\rz f_\Fj$ is not nearstandard for any nonzero $x\in B_r$; eg., consider the sequence $\wt{f}_j(x)=x^j$ and the transfer of the dilated sequence $f_j(x)=\wt{f}_j(2^jx)$.

    Nonetheless, we shall prove two basic results giving a correspondence between the convergence of a sequence of functions and $\tau$ convergence of extensions of these sequences restricted to germ type, ie., monadic domains. First, we shall see (in Proposition \ref{prop: corres. between unif. converg and tau converg}) that in the case that $\norm{f_j}_r\ra 0$ as $j\ra \infty$, then in fact the `germs' associated to any extended *finite sequence $\rz f_\Fj$, $\Fj=1,2,\ldots,\om$ do, in fact, `converge in the topology $\tau$', once properly interpreted. In contrast to the uniform situation just mentioned, we will also give a pointwise correspondence, Proposition \ref{prop: corres. between ptwise converg and tau converg}. Given the setting, it turns out that the pointwise and the uniform versions are equivalent; the nonstandard setting allows a sort of uniform convergence at a point phenomena.   We will also give  converses (to both): if this internally extended family converges in $\tau$ on $B_\d\subset\mu(0)_+$ for some $\d\sim 0$, then $\{f_j:j\in\bbn\}$ is convergently coinitial (see below) in a sufficiently small neighborhood of $0$.

\begin{definition}\label{def: convergently coinitial in SN_d}
    Fix $\d\in\mu(0)_+$. Suppose that $(D,<)$ is a (upward) directed set and $\Xi=(\xi_d:d\in D)$ is a net in $\mu(0)_+$. Then we say that $\Xi$ is convergently coinitial in the range of $\SN_\d$ if for each $\Fr\in\SN_\d$, there is $d_0\in D$ that if $d>d_0$ then $\xi_d<\Fr$. It will be said to be convergently coinitial with $\SN_\d$ if in addition we  have that $\SN_\d$ is coinitial with $\Xi$.
\end{definition}
    Note that $(\rz\bbn,\rz\!\!<)$ (and all of its subsets) is a directed set so that if $\om\in\rz\bbn$ and $\xi_\Fj\in\rz\bbr$ for $1\leq\Fj\leq\om$, we may consider $\{\xi_1,\xi_2,\ldots,\xi_\om\}$ as a net in $\rz\bbr$.
    It's clear that if $\Xi$ is monotone decreasing and coinitial in the range of $\SN_\d$, then it is convergently coinitial in the range of $\SN_\d$.
    If $\Xi=(\xi_d:d\in D)$ is convergently coinitial with $\SN_\d$, then it is roughly monotone, ie., for each $d_0\in D$, there is $d_1\in D$, such that if $d>d_1$, then $\xi_d<\xi_{d_0}$.
    We need to first state some lemmas giving useful properties of these nets that arise by extensions of standard sequences.

\begin{lem}
    Let $\SL\in\un{M}^0(0,a)$. Given a sequence $A=\{m_1,m_2,\ldots\}$ in $\un{M}^0(0,a)$, we can always find another sequence $\wt{A}=\{\wt{m}_1,\wt{m}_2,\dots\}$ in $\un{M}^0(0,a)$ such that $\wt{m}_1(t)>\wt{m}_2(t)>\cdots$ for all $t$  and for all $\d\in\mu(0)_+$, $A_\d$ is coinitial in $\SL_\d$ if and only if $\wt{A}_\d$ is coinitial in $\SL_\d$. So  $A_\d$ is coinitial with $\SL_\d$ if and only if $\wt{A}_\d$ is convergently coinitial with $\SL_\d$.
\end{lem}
\begin{proof}
    It's easy to see that defining $\wt{m}_j(t)=\min\{m_1(t),\ldots,m_j(t)\}$ for $j\in\bbn$ and $0<t<a$. It's clear that $\wt{m}_j(t)>\wt{m}_{j+1}(t)$ for all $t$ and $j$. As $\wt{m}_j(t)\leq w_j(t)$ for all $j$ and $t$ so that this holds after transfer, we need only verify that $\wt{A}_\d$ is coinitial in $\SL_\d$ implies that $A_\d$ is coinitial in $\SL_\d$. But if $\Fr\in\SL_\d$, then there is $\Fj\in\rz\bbn$ such that $\rz\wt{m}_\Fj(\d)<\Fr$ and clearly as $\rz\min\{\rz m_\Fi(\d):1\leq\Fi\leq\Fj\}=\rz\un{m}_{\Fj}(\d)$, then there is $\Fi\leq\Fj$ with $\rz m_\Fi(\d)<\Fr$.
\end{proof}
\begin{lem}
    Suppose that $m_1,m_2,\ldots$ is a sequence in $\SM$, $\d\in\SN_\d$, $\om\in\rz\bbn$ and $\rz m_1(\d)>\rz m_2(\d)>\cdots$. Then the following are equivalent.
\begin{enumerate}
     \item[$\imath$)]      $\{\rz m_1(\d),\ldots,\rz m_\om(\d)\}$ is convergently coinitial in the range of $\SN_\d$.
     \item[$\imath\imath$)] $\rz m_\om(\d)<\Fr$ for all $\Fr\in\SN_\d$.
     \item[$\imath\imath\imath$)] $\rz m_\om(\d)\lll\Fr$ for all $\Fr\in\SN_\d$.
\end{enumerate}
\end{lem}
\begin{proof}
    For $\imath$)  implies $\imath\imath$), assume $\imath$) holds but $\imath\imath$) doesn't hold, ie., there is $\Fr_0\in\SN_\d$ such that $\rz m_\om(\d)\geq\Fr_0$, but then for all $\Fj\in\{1,2,\ldots,\om\}$, $\rz m_\Fj(\d)\not<\Fr_0$, a contradiction.  Next we will verify $\imath\imath$) implies $\imath\imath\imath$), ie., that  $\rz m_\om(\d)<\Fr$ for all $\Fr\in\SN_\d$ implies that $\rz m(\d)\lll\Fr$ for all $\Fr\in\SN_\d$. Suppose this is not true, ie., there is $\Fr\in\SN_\d$ and $[m]\in\SM^0$ such that $\rz m(\Fr)<\rz m_\om(\d)$ and by definition $\Fr=\rz m'(\d)$ for some $m'\in\SM^0$. But then we have $\rz m\circ m'(\d)=\Ft\in\SN_\d$ satisfying $\Ft<\rz m_\om(\d)$, a contradiction.
    $\imath\imath\imath$) implies $\imath$) is obvious.
\end{proof}

\begin{remark}
    As the generic infinitesimal $\d$ grows, ie., as we evaluate our semiring of monotone function germs at larger infinitesimal values, our set of moduli $\SN_\d$ are gradually getting `closer' to noninfinitesimal values. To give a sense of how the $\SN_\d$ are becoming unbounded in $\mu(0)_+$ as $\d$ is increasing in an unbounded way in $\mu(0)_+$, we have the following lemma.
\end{remark}
\begin{lem}
    Suppose that $0<\e\sim 0$. Then there is a $0<\d\sim 0$ such that $\e<\Fr$ for all $\Fr\in\SN_\d$.
\end{lem}
\begin{proof}
    This lemma follows from the fact that given a positive infinitesimal $\e$, there is another positive infinitesimal $\d$ incomparably larger than $\e$. Assuming this for the moment, then Lemma \ref{lem: SN_d doesnot have d-incomparable no} implies that the elements of $\SN_\d$ are all larger than $\e$. So given $0<\e\sim 0$, if $m\in\SM^0$ and $j\in\bbn$, let $\SL_{m,j}=\{\Fr\in\rz\bbr_+:\rz m(\e)<\Fr<1/\rz j\}$. It's clear that the (external) set $\FL=\{\SL_{m,j}:[m]\in\SM^0\;\text{and}\;j\in\bbn\}$ has the finite intersection property, so that saturation (see eg.,  Stroyan and Luxemburg, \cite{StrLux76}, p181) implies that $\cap\FL$ is nonempty, ie., there is $\d\in\rz\bbr_+$ such that $\d<1/\rz j$ for all $j\in\bbn$ and $\rz m(\e)<\d$ for all $m\in\SM^0$.
\end{proof}

    Nonetheless there is are technical results that dramatically shows how different internal sets of the form $\{\rz f_\Fj(\d):\Fj\in\rz\bbn\}$ (where $\{f_j:j\in\bbn\}$ is a sequence of functions converging to zero, in some sense, on some neighborhood of $0$) are from our sets $\SJ_\d$, where $\SJ\subset\SM^0$ is such that $\SJ_\d$ is coinitial in $\SN_\d$, in the sense of the above noted properties. In the next lemma we show that for our nets that the  truncated *sequences that are convergently coinitiality in $\SN_\d$ for some $\d$ have the full *sequences coinitial in $\rz\bbr_+$.  (As it is usually true that relations between the characteristics of convergence and those of coinitiality can be captured by simplifying to the extremal values of function-germs in $\SM^0$, the following is stated in that context.)
\begin{lem}\label{lem: *K_d coin(mu(0)) iff *K^om_d coin N_d}
    Suppose that $\SK=\{m_j:j\in\bbn\}$ is a sequence in $\un{M}^0(0,a)$,  and for $\om\in\rz\bbn$ let $\SK^\om_\d$ denote $\{\rz m_\Fj(\d):1\leq\Fj\leq\om\}$. Then the following are equivalent.
\begin{enumerate}
    \item[a)] There is $\d\in\mu(0)_+$ such that $\rz \SK_\d$ is convergently coinitial in $\rz\SN_\d$.
    \item[b)] There is $\d\in\mu(0)_+$ such that for some $\om\in\rz\bbn$,  $\SK^\om_\d$ is convergently coinitial in the range of $\SN_\d$.
    \item[c)] There is $r_0\in\bbr_+$ such that $\lim_{j\ra\infty}m_j(r)=0$ for $r\leq r_0$.
\end{enumerate}

\end{lem}
\begin{proof}
    We will first prove (a) holds if and only if (b) holds. Suppose that (a) is true. As $\rz\SN_\d$ is coinitial with $\mu(0)_+$, we have that there is $\Fz\in\rz\SN_\d$ such that $\Fz$ is less than all elements of $\SN_\d$ and by hypothesis $\rz\SK_\d$ is coinitial with $\rz\SN_\d$, then there is $\la\in\rz\bbn$ such that $\rz m_\la(\d)<\Fz$, so that $\{\rz m_1,\ldots,\rz m_\la(\d)\}$ is certainly coinitial in the range of $\SN_\d$ and if monotone decreasing is convergently coinitial there. If (b) holds, we know from the previous lemma that $\rz m_\om(\d)\lll\Fr$ for all $\Fr\in\SN_\d$. Given this, let $\hat{m}(t)=\lim_{j\ra\infty}\inf m_j(t)$ so that we have $\hat{m}(t)\geq 0$ where defined. But $\rz\hat{m}(\d)=\rz\lim_{\Fj\ra\infty}\inf m_\Fj(\d)\leq\rz m_\om(\d)\lll\d$ and so by remark \ref{rem: e<<<d and f(d)<e -> f(d)=0}, we must have $\rz\hat{m}(\d)=0$. But this says that $\rz\SK_\d$ is convergently coinitial in $\rz\bbr_+$, ie. in $\rz\SN_\d$.

    To finish, it suffices to verify that (c) holds if and only if (b) holds. But if $P=\{r\in\bbr_+:\SK_r\;\text{is convergently coinitial in}\;\bbr_+\}$, then (b) is equivalent to $\rz P\not=\emptyset$ which then implies that $P$ is nonempty. On the other hand if $r\in P$ and $r'\in\bbr_+$, $r'<r$, then $m_j\in\SM^0$ implies $r'\in P$; transferring the statement: $r'\in\bbr_+$ with $r'\leq r_0$ implies $r'\in P$ gets (c).
\end{proof}
    The next lemma also shows the special nature of these transferred sequences: if the full transferred sequence is convergently coinitial, then in fact a *finite truncation is.
\begin{lem}\label{lem: K_d converg coin SN_d-> K^om_d converg coin in SN_d}
     Suppose we have the same notation of the previous lemma. If $\SK_\d$ is convergently coinitial in $\SN_\d$, then there is $\om\in\rz\bbn$ such that $\SK^\om_\d$ is convergently coinitial in $\SN_\d$.
\end{lem}
\begin{proof}
    Fix $a\in\bbr_+$ and let $m\in\SM^0$ and consider  the following internal set
\begin{align}
    \SK_m=\{\om\in\rz\bbn\;|\;\{\rz m_\Fj(\d):1\leq\Fj\leq\om\}\;\text{is coinitial with}\;[\rz m(\d),a]\}.
\end{align}
     First of all, (a) implies that $\SK_m\not=\emptyset$ for each $m\in\SM^0$ by definition of coinitiality: for any $m\in\SM^0$, there is $\om\in\rz\bbn$ such that $\rz m_\om(\d)<\rz m(\d)$. We will prove that the set of internal sets $\FK\dot=\{\SK_m:[m]\in\SM^0\}$ has the finite intersection property. Suppose that $m^1,\ldots,m^k\in\SM^0$; then there is $\un{m}\in\SM^0$ such that $\rz\un{m}(\d)<\rz m^j(\d)$ for $j=1,\cdots,k$. Now $\SK_{\un{m}}$ is nonempty and if $\un{\om}\in\SK_{\un{m}}$, it's clear that $\un{\om}\in\SK_{m^j}$ for all $j$, ie., $\un{\om}$ is in their intersection. Hence, the elements of $\FK$ have the finite intersection property and as the cardinality of $\SM^0$ is bounded above by that of $\SP(\bbr)$, then saturation implies that $\cap\{\SK_m:m\in\SM^0\}\not=\emptyset$. That is, there is $\om_0\in\SK_m$ for all $m\in\SM^0$ or equivalently $\{\rz m_\Fj(\d):1\leq\Fj\leq\om_0\}$ is coinitial in $\SN_\d$.
\end{proof}
    As we saw in the previous section, for a given $\SJ\subset\SM^0$ and $\d,\d'\in\mu(0)_+$, there is no good reason to believe that $\SJ_{\d'}$ is coinitial in $\SN_{\d'}$ if $\SJ_\d$ is coinitial in $\SN_\d$. With these extended sequences, this does occur.
\begin{cor}\label{cor: m_j(d_0)is coin SN_d_0 -> true for all d}
    Keeping the notation of the previous two lemmas, we have the following. If for some $\d\in\mu(0)_+$ we have that $\SK_\d$ is convergently coinitial in $\SN_\d$, then in fact $\SK_{\d'}$ is convergently coinitial in $\SN_{\d'}$ for all $\d'\in\mu(0)_+$.
\end{cor}
\begin{proof}
    The hypothesis and the (a) if and only (b) equivalence of Lemma \ref{lem: *K_d coin(mu(0)) iff *K^om_d coin N_d} gets $\rz\SK_\d$ is convergently coinitial in $\rz\SN_\d$. And as the last is coinitial in $\rz\bbr_+$, then, this implies that $\rz\lim_{\Fj\ra*\infty}\rz m_\Fj(\d)=0$. That is, if $Q=\{r\in\bbr_+:\lim_{j\ra\infty}m_j(r)=0\}$, this says that $\d\in\rz Q$ and so $Q$ has infinitely many $r\in\bbr_+$ accumulating at $0$. Pick $r\in Q$ and note that monotonicity implies the statement: $r'\in\bbr_+$ with $r'\leq r$ implies that $r'\in Q$. Transfer of this statement gets $\rz\lim_{\Fj\ra*\infty}\rz m_\Fj(\d')$ for all $\d'\in\rz\bbr_+$ with $\d'\leq\rz r$. But then applying (a) implies (b) of Lemma \ref{lem: *K_d coin(mu(0)) iff *K^om_d coin N_d} for any such $\d'\in\mu(0)_+$, we get our conclusion.
\end{proof}

    Suppose now that $0<r\in\bbr$ and that for $j\in\bbn$, $f_j:B_r\ra\bbr$ is a sequence of continuous functions on $B_r$.
    Recall that $\{f_j:j\in\bbn\}$ converges (uniformly) on $B_r$ to $0$ if $\norm{f_j}_r\ra 0$ as $j\ra\infty$ (equivalent to pointwise convergence of the values as $B_r$ is compact). We will define another type of convergence related to the topology $\tau$.
    Given a sequence $\{f_j\}$ as above, we have the *transfer $\{\rz f_\Fj:j\in\bbn\}$. We need the following definition in order to relate our two topologies.

\begin{definition}\label{def: convergence of standard seq of fncs in germ top}
    Let $r\in\bbr_+$, $F=\{ f_j:j\in\bbn\}$ be a sequence of real valued functions on $B_r$. We say that the extension of $\{f_j:j\in\bbn\}$ converges in the topology $\tau$  (to the  zero germ),  if for $\d\in\mu(0)_+$, there is $\om\in\rz\bbn$  such that the internal $\om$ sequence of numbers $\{\norm{\rz f_\Fj}_\d:1\leq\Fj\leq\om\}$ is cnvergently coinitial in the range of $\SN_\d$.
\end{definition}
    For example, the sequence of constant functions $x\mapsto f_j(x)\equiv 1/j$, as a sequence of functions uniformly defined on $B_r$, converges in the topology $\tau$ at each $x\in Int(B_r)$. Note also that the *sequence of constant functions $\rz f_\Fj:\xi\in\mu(0)\ra1/\Fj\in\rz\bbr $ converge in the $\tau$ topology, a hint at what follows.
\begin{remark}
    This definition is not possible if we are speaking instead of a sequence of germs. First, recall that for a given $\d$, the corresponding $\SN_\d$ cannot carry an incomparably range of numbers; eg., $\SN_\d$ is contained in a very narrow interval in $\mu(0)_+$ (there are *infinitely many). Given this, we note that the transfer of a sequence of domains of representatives of the germs may shrink through infinitesimal disks contained in $\mu(0)$ so rapidly that the common domain of the first $\om$ maps, for a given $\om\in\rz\bbn_\infty$, may be a $B_\d$ for $\d$ so small that the corresponding $\SN_\d$ consists of infinitesimals that are too small for the $\d$-norms of this extended sequence to be coinitial.
\end{remark}

    We will  now prove some  results that give  correspondences between  uniform convergence (to the zero function) on some neighborhood of $0$ and $\tau$ convergence of the extended sequence.
\begin{proposition}\label{prop: corres. between unif. converg and tau converg}
    Suppose that $F=\{f_j:j\in\bbn\}$ is a sequence of functions defined on $B_{r_0}$. The following are equivalent.
\begin{enumerate}
    \item[a')] There is $\ov{r}\in\bbr_+$, $\ov{r}\leq r$, such that $\norm{f_j}_{\ov{r}}\ra 0$ as $j\ra\infty$.
    \item[b')] For some $\d\in\mu(0)_+$, there is $\om\in\rz\bbn$ such that $\{\norm{\rz f_\Fj}_\d:1\leq\Fj\leq \om\}$ is convergently coinitial with $\SN_\d$.
    \item[c')] For each $\d\in\mu(0)_+$, there is $\om\in\rz\bbn$ such that $\{\norm{\rz f_\Fj}_\d:1\leq\Fj\leq \om\}$ is convergently coinitial with $\SN_\d$.
\end{enumerate}
\end{proposition}
\begin{proof}
    The hypothesis for the first claim clearly implies the statement: for all $r\in\bbr_+$ with $r\leq\ov{r}$ $\{\norm{f_j}_r:j\in\bbn\}$ is convergently coinitial in $\bbr_+$ whose transfer gives: for all $\Fr\in\rz\bbr_+$ with $\Fr\leq\rz\ov{r}$, $\{\norm{\rz f_\Fj}_\Fr:\Fj\in\rz\bbn\}$ is convergently coinitial in $\rz\bbr_+$.
    But then Lemma \ref{lem: *K_d coin(mu(0)) iff *K^om_d coin N_d} implies, in particular, that if $\d\in\mu(0)_+$, then there is $\om\in\rz\bbn$ such that $\{\norm{\rz f_\Fj}_\d:1\leq\Fj\leq\om\}$ is convergently coinitial in $\SN_\d$.
    Conversely, if given $\d\in\mu(0)_+$, there is $\om$ such that $\{\norm{\rz f_\Fj}_\d:1\leq\Fj\leq\om\}$ is convergently coinitial in $\SN_\d$, then again by Lemma \ref{lem: *K_d coin(mu(0)) iff *K^om_d coin N_d} we have that $\{\norm{\rz f_\Fj}_\d:\Fj\in\rz\bbn\}$ is convergently coinitial in $\mu(0)_+$, ie., in $\rz\bbr_+$ for all $\d\sim 0$. But then as this is, for each $\d$, an internal statement, overflow implies that it holds for all $\d$ less than some noninfinitesimal $\la>0$, eg., for some standard $\rz r<\la$. But then, reverse transfer of this statement for $\rz r$ gives our conclusion.

\end{proof}
    We have the following equivalent `pointwise' formulation of the previous proposition.

\begin{proposition}\label{prop: corres. between ptwise converg and tau converg}
    Suppose that $r_0\in\bbr_+$ and $F=\{f_j:j\in\bbn\}$ is a sequence of real valued functions on $B_{r_0}$; then the following are equivalent to a') and b') of the previous proposition.
\begin{enumerate}
    \item[a'')]There is $r\in\bbr_+,r\leq r_0$ such that for each $x\in B_r, f_j(x)\ra 0$ as $j\ra\infty$.
    \item[b'')]Given $\d\in\mu(0)_+$, there is $\om\in\rz\bbn$ such that if $\xi\in B_\d$, then $\{ |\rz f_{\Fj}(\xi)|:1\leq\Fj\leq\om\}$ is convergently coinitial in $\SN_\d$.
\end{enumerate}
\end{proposition}
\begin{proof}
    We will show a''), respectively b''), is equivalent to a'), respectively c'), of the previous proposition. The equivalence of a') and a'') is clear and c') clearly implies b''), so it suffices to verify b'') implies c').
    By way of contradiction, suppose b'') holds but that c') does not hold. That is, there is $\d\in\mu(0)_+$ such that for each $\om\in\rz\bbn$, $\{\norm{\rz f_\Fj}_\d:1\leq\Fj\leq\om\}$ is not convergently coinitial in $\SN_\d$; ie., given $\om\in\rz\bbn$, there is $\Fr_0\in\SN_\d$  such that $\norm{\rz f_\Fj}_\d>\Fr_0$ for all $\Fj\leq\om$. In particular, for $\Fj\leq\om$, there is $\xi_\Fj\in B_\d$ such that $|\rz f_\Fj(\xi_\Fj)|\geq\Fr_0$. But choosing our $\om$ to be the element of $\rz\bbn$ asserted in b''), then b'') implies, as $\xi_\om\in B_\d$, that $\{|\rz f_\Fj(\xi_\om)|:\Fj\leq\om\}$ is convergently coinitial in $\SN_\d$, in particular, there is $\Fj_0<\om$ such that if $\Fj>\Fj_0$, then $|\rz f_\Fj(\xi_\om)|<\Fr_0$, eg, this must be true for $|\rz f_\om(\xi_\om)|$, a contradiction.
\end{proof}

    Here we will give a correspondence result that relates `pointwise convergence' that is  analogous to the way the previous proposition gives a correspondence for `uniform convergence'.

    We want to consider the analogous situation of one parameter families of functions.
\begin{definition}\label{def: continuity in each variable}
    We say that a map $F:(\bbr,0)\ra C^0(B^n_r)$ is continuous at $t=0$ if the corresponding map $\wt{F}:\bbr^{n+1}\ra\bbr$, $\wt{F}(t,x)=F(t)(x)$ is continuous in  $x$ and in $t$ at $t=0$ separately with respect to the usual topology on the Euclidean spaces. We will sometimes write $F(t)=f_t$ and say that, eg., the map $t\mapsto f_t(x)$ is continuous for $x\in\bbr^n$ such that $f_t(x)$ is defined for $|t|$ small.
\end{definition}
    We first need a lemma.
\begin{lem}
    If $\d\in\mu(0)_+$ let  $\SU_\d\subset\mu(0)_+$ denote the set $\{\d'\in\mu(0)_+:\d<\d'\}$, and suppose that  $c\in\SM^0$. Then for each $\d\in\mu(0)_+$, there is $\e\in\mu(0)_+$ such that $\rz c(\SU_\e)$ is convergently coinitial with $\SN_\d$.
\end{lem}
\begin{proof}
    This is clear as for a given $\d\in\mu(0)_+$ we have $\rz c(\SN_\d)=\SN_\d$ so that just pick $\e\lll\d$.
\end{proof}
\begin{proposition}\label{prop: corres: t->f_t standard converg<-> tau converg}
    Suppose that, for some $r\in\bbr_+$, $t\mapsto f_t:\bbr\ra F(B_r,\bbr)$ is such that $f_t(0)=0$ for all $t\in\bbr$ where defined and $f_0$ is the zero function. Then the following are equivalent.
\begin{enumerate}
    \item[$\a$)]The map $t\mapsto f_t$ is a continuous at $t=0$.
    \item[$\b$)]The extended map $\Ft\in\rz\bbr_+\mapsto\rz f_\Ft$ satisfies the following. For each $\d\in\mu(0)_+$, there is $\a\in\rz\bbr_+$ such that $\{\norm{\rz f_\Fs}_\d:\Fs>\a\}$ is convergently coinitial in $\SN_\d$.
\end{enumerate}
\end{proposition}
\begin{proof}
     Since our condition in $\b$) is the same as saying for each $\e\in\mu(0)_+$, there is $\om\in\rz\bbn$, such that the *sequence $\{1/\Fj:1\leq\Fj\leq\om\}$ is convergently coinitial in $\SU_\e$,
     then our condition is equivalent to proving that there is $\om\in\rz\bbn$ such that the sequence $j\in\bbn\mapsto f_{1/j}$ satisfies $\{\norm{\rz f_{1/\Fj}}_\d:1\leq\Fj\leq\om\}$ is convergently coinitial in $\SN_\d$.
     With the same reparameterization, we similarly get that for each $x\in B_r$, $f$ is continuous at $t=0$, ie., $f_t(x)\ra 0$ as $t\ra 0$ for each $x\in B_r$ is equivalent to $f_{1/j}(x)\ra 0$ as $j\ra\infty$ at each such $x$.
     But now the equivalence of $\a$) and $\b$) has been reduced to that of a'') and b'') in Proposition \ref{prop: corres. between ptwise converg and tau converg}.
\end{proof}

\subsection{Continuity of germ operations}\label{sec: continuity of germ ops}
     In this section, armed with all of the preliminaries, eg., with the relationships among all of the various families of infinitesimals functioning as moduli, we can prove that our spaces of continuous germs have good algebraic properties. That is, in the next subsection, we verify good ring properties and in the following good compositional properties.

\subsubsection{Topological properties of the ring operations of germs}\label{subsec: top properties ring structure}
    In this subsection we will verify that if $[f]\in\SG_0$, then the maps $+_{[f]}:\SG_0\ra\SG_0:[g]\mapsto [f]+[g]$ and $\x_{[f]}:\SG_0\ra\SG_0:[g]\mapsto [f][g]$ are continuous in the topology $\tau$.
\begin{lem}
    Given $\Fr,\Fs\in\SN_\d$ with $\Fr<\Fs$, there is $\Ft\in\SN_\d$ such that $\Fs\Ft<\Fr$; that is, $\Fs\SN_\d\subset\SN_\d$ is coinitial.
\end{lem}
\begin{proof}
    As each element $\Fs\in\SN$ is bounded above by an element $\rz c$ for some $c\in\bbr_+$, and as such $\rz 1/c\in\SN_\d$, then a good  choose for $\Ft$ is $\Ft=(\rz 1/c)\Fr$.
\end{proof}
    It should be clear that we cannot weaken the previous lemma to assume $\Fr<\Fs$ and only one of $\Fr$ or $\Fs$ is in $\SN_\d$; an easy example is to choose $\Fs\in\SN_\d$ and $\Fr\in\mu(0)_+$ with $\Fr\lll\Fs$ (eg.,  $\Fr$ and $\Fs$ cannot be incomparable).

\begin{proposition}\label{prop: SG_0 is Hausdorff top ring}
    $(\SG_0,\tau)$ is a Hausdorff topological ring.
\end{proposition}
\begin{proof}
    We need to show that the ring operations are continuous and as the vector space addition is continuous by the definition of the topology, we need only prove the product is continuous. First left and right multiplication by a given element of $\SG_0$ is continuous. By the previous lemma, for a given $[f]\in\SG_0$, multiplication on the right $\x_{[f]}$ (and so multiplication on the left ${}_{[f]}\x$) is continuous. That is, if $\Fs\in\SN_\d$ is $\norm{f}_\d$ and given $\Fr\in\SN$, then there is $\Ft\in\SN_\d$ such that $\Fs\Ft<\Fr$, ie., $\x_{[f]}(U_\Ft)\subset U_\Fr$. In fact, this shows that given $\Ft\in\SN_\d$, then certainly there are $\Fr,\Fs\in\SN_\d$ with $\Fr\Fs<\Ft$ and so $U_\Fr\cdot U_\Fs=\{f\cdot g:f\in U_\Fr,g\in U_\Fs\}\subset U_\Ft$.
    This shows that $([f],[g])\mapsto [f][g]$ is continuous in the following special case: $([f_d]:d\in D)$ and $([g_d]:d\in D)$ are nets converging to $[0]$ in $\tau^\d$, then $d\mapsto [f_d][g_d]$ converges to $[0]$ in $\tau^\d$. Given this let's verify that if we have nets $[f_d]\ra [f]$ (in $\tau^\d$) and $[g_d]\ra [g]$ (also in $\tau^\d$), then $[f_d][g_d]\ra[f][g]$ (in $\tau^\d$ also). First, $\tau^\d$ continuity of left and right multiplication implies that $[f]([g_d]-[g])\ra [0]$ in $\tau^\d$ and $([f_d]-[f])[g]\ra [0]$ in $\tau^\d$. But the previous assertion says also that $([f_d]-[f])([g_d]-[g])\ra [0]$ in $\tau^\d$. Adding the previous three expressions (noting that addition is a continuous operation in $\SG_0$) gets $[f_d][g_d]\ra[f][g]$ (in $\tau^\d$) as we wanted.
\end{proof}
\begin{proposition}
    $\SG^0_0$ is a closed subring of $\SG_0$.
\end{proposition}
\begin{proof}
    We just need to prove that $\SG_0^0$ is a closed subspace of $\SG_0$. But this is the import of Theorem \ref{thm:  convergnet of cont germs in SG converges to cont}.
\end{proof}

\begin{definition}
    For $n,p\in\bbn$, let $\SG_{n,p}$ denote the germs of maps at $0$ of maps $f:(\bbr^n,0)\ra \bbr^p$, \; $\SG^0_{n,p}$ these map germs that are germs of continuous maps and $\SG_{n,p,0}\subset \SG_{n,n}$ those map germs sending $0$ to $0$.
    Considering $\SG_{n,p,0}$ as a Cartesian product of $p$ copies of the topological ring $\SG_{n,0}$, we give it the natural product topology. We give the subsets  $\SG_{n,n,0}$, $\SG^0_{n,n,0}$, etc., the subspace topology.
\end{definition}
    It is clear from the definition above that the product $\tau_\d$ topology defined on $\SG^0_{n,p,0}$ is generated by translates of open neighborhoods of the zero germ and that the system of open neighborhoods of the zero germ is generated by Cartesian products of the form $U^\d_{\vec{\Fr}}=U^\d_{\Fr_1}\x\cdots\x U^\d_{\Fr_n}$ where $\vec{\Fr}$ denotes the ordered $n$-tuple $(\Fr_1,\ldots,\Fr_n)\in\SN_\d^n$.
    Of course,   given such $\vec{\Fr}$, if $\un{\Fr}<\Fr_j$ for each $j$, then $U_{\un{\Fr}}\x\cdots\x U_{\un{\Fr}}\subset U_{\vec{\Fr}}$, eg., when checking convergence we may test with these diagonal neighborhoods. Similarly, given any of the typical norm functions $N:\bbr^n\ra [0,\infty)$ on $\bbr^n$, the family of nested neighborhoods of the $[0]$ germ defined by $U_{N,\Fr}=\{[f]\in\SG: \rz N(f(\xi))<\Fr\;\text{for each}\;\xi\in B_\d\}$ as $\Fr$ varies in $\SN_\d$, gives a coinitial family of neighborhoods of $[0]$ in $\tau^\d$.

\begin{proposition}
    $\SG^0_{n,p,0}$ is a Hausdorff $\SG^0_{n,0}$-module.
\end{proposition}
\begin{proof}
    This is clear: as a finite product of Hausdorff spaces (with the product topology), it is clearly Hausdorff. Also the continuity of the module operation $\SG_{n,0}\x\SG_{n,p,0}\ra\SG_{n,p,0}$ is clear as it's just the ring operation $\SG_{n,0}\x\SG_{n,0}\ra\SG_{n,0}$ on each of the $p$ coordinates.
\end{proof}

\subsubsection{Topological properties of germ composition}\label{subsec: top properties of germ composition}
    We return to the context of subsection \ref{subsec: top properties ring structure} and consider the morphisms of these topological rings induced by continuous map germs.

    The composition of germs of maps is well known and routine and we will assume the general definitions and facts known. Heuristically, composition of functions obviously distorts domains and range and so it will be here, but cutting down domains will not be a problem as our monad representatives remain well defined as such when domains are repeated expanded and contracted within $\mu(0)$. If $n,p\in\bbn$, let $\SG_{n,p,0}\subset\SG_{n,p}$ denote the set of germs at $0$ in $\bbr^n$ of maps $f:(\bbr^n,0)\ra (\bbr^p,0)$.
    We begin by noting the obvious problem: if $h\in\SG_{n,n,0}$ or even in $\SG^0_{n,n,0}$, and $f\in\SG$, then $\rz f|_{B_\d}\circ \rz h|_{B_\d}$ is often not defined. On the other hand, if $[h]\in\SG^0_{n,n,0}$ and $[f]\in\SG_n$, then $\rz(f\circ h)_{B_\d}$ is always defined even when $\rz h(B_\d)\nsubseteq B_\d$ as $h(\mu(0))\subset\mu(0)$ and although we are defining the composition with $[f]$ in terms of its representative that is uniquely defined on all of $\mu(0)$.
    If $[h]\in\SG^0_{n,n,0}$, we will denote the map $\SG_n\ra\SG_n:[f]\mapsto [f]\circ[h]$ by $rc_{[h]}$ and if $[g]\in\SG$

    To begin with, we look at the effect of composition on our sets of moduli. As right composition carries algebraic operations, we start there and as moduli are determined in terms of one dimensional mappings we will consider both the right and left actions of $\SM^0$ on inself.
    We will begin with some preliminaries on the effects of compositions on our semirings of moduli.

   We begin with a corollary of results in subsection \ref{subsec: convergence}.

\begin{cor}\label{cor: SN_r is coin with SN_d if r=f(d)}
    Suppose that $[f]\in\SG^0$, $0<\d\sim 0$, and $\Fr=\norm{\rz f}_\d$.
    Then for each $U^\d_\Fr\in\tau^\d_0$ with $\Fr\in\SN_\d$, there is $U^\d_\Fs\in\tau^\d_0$ with $\Fs\in\SN_\Fr$ such that $U^\d_\Fs\subset U^\d_\Fr$, eg., the subset of $\tau^\d_0$ given by the set of $U^\d_\Fs$ with $\Fs$ varying in $\SN_\Fr$ is a subbase at $0$ for $\tau^\d_0$.
\end{cor}
\begin{proof}
    Of course, according to Remark \ref{rem: coinitial is partial order} it suffices to prove that $\SN^0_\Fr$ is coinitial with $\SN^0_\d$.
    But  this is a consequence of corollary \ref{cor: m(N_d)-N_d and N_(m(d))=N_d} and the fact that $\SN^0_\d$ is coinitial in $\SN_\d$ which is proposition \ref{prop: SN^0_d is coin in SN_d}.
\end{proof}
    Before we proceed, we need to point out an obvious fact.
\begin{lem}\label{lem: m in M^0 sends coin pairs to coin pairs}
    Suppose that $(\La,<)$ and $(\G,<)$ are directed sets and that $\{\Fr_\la:\la\in\La\}$ and $\{\Fs_\g:\g\in\G\}$ are subsets of $\mu(0)_+$ that are coinitial with each other in the range $\SN_\d$. Then if $[m]\in\SM$, we have that $\{\rz m(\Fr_\la):\la\in\La\}$ is coinitial with $\SN_\d$ if and only if $\{\rz m(\Fs_\g):\g\in\G\}$ is coinitial with $\SN_\d$.
\end{lem}
    Given the previous, we will see that continuous germs sending zero to zero send coinitial magnitudes to coinitial magnitudes.
\begin{cor}\label{cor: Im of coin set under G^0_n,p,0 map is coin}
    Suppose that $[h]\in\SG^0_{n,p,0}$ and that $\{\Fv_\la:\la\in\La\}\subset\mu_n(0)$ is such that $\{|\Fv_\la|:\la\in\La\}$ is coinitial with $\SN_\d$. Then $\{\rz h(\Fv_\la):\la\in\La\}\subset\mu_p(0)$ satisfies $\{|\rz h(\Fv_\la)|:\la\in\La\}$ is coinitial in $\SN_\d$.
\end{cor}
\begin{proof}
    Without loss of generality, we may assume that $[h]$ is pseudomonotone (recall that this means that if $|\xi|\leq|\z|$, then $|\rz h(\xi)|\leq |\rz h(\z)|$).
    We can also assume that $p=1$ and
    we can therefore assume we are working with $[m]\in\SM^0$, for which the problem becomes: if $\SX=\{\xi_\la:\la\in\La\}$ is coinitial in $\SN_\d$, then $\{\rz m(\xi_\la):\la\in\La\}$ is coinitial in $\SN_\d$.
    But as $\SX$ and $\SN_\d$ are coinitial with each other in the range of $\SN_\d$, then, by Lemma \ref{lem: sets coin wrt N_d go to these under m in M}, this is equivalent to showing that $\rz m(\SN_d)$ is coinitial with $\SN_d$, and in fact they are equal by Corollary \ref{cor: m(N_d)-N_d and N_(m(d))=N_d}.
\end{proof}

    Note that the previous result certainly does not imply that $\{|\Fv_\la|:\la\in\La\}$ is coinitial with $\{|\rz h(\Fv_\la)|:\la\in\La\}$, eg., this certainly does not hold if $[h]$ is the zero germ.
\begin{definition}
    We say that $[h]\in\SG^0_{n,n,1}$ is regular if $\rz h(\mu(0))=\mu(0)$ and there is $n\in\bbn$ such that for each $\xi\in\mu(0)$, $\rz h^{-1}(\xi)\cap\mu(0)$ has cardinality at most $n$.
\end{definition}
    So clearly homeomorphisms are regular. Note that we can replace this last condition with the same condition on inverse images of points of $B_\d$.
\begin{lem}
    If $\Fr\in\SN_\d$ with $\Fr<\d$, respectively $\Fr>\d$, then there is $m\in\SM^0$ such that $\rz m(\d)\leq\Fr$, respectively $\rz m(\d)\geq\Fr$.
\end{lem}
\begin{proof}
    Suppose that there is $\Fr\in\SN_\d$ with $\Fr<\d$, but there is no $m\in\SM$ such that $\rz m(\d)\leq\Fr$. Then we have that $\Fr$ and $\d$ are incomparable and we know this can't happen by eg., Lemma \ref{lem: SN_d doesnot have d-incomparable no}. The other statement has the same proof.
\end{proof}

\begin{proposition}\label{prop: rc_h:G^0_n->G^0_n is C^0}
    If $[h]\in\SG^0_{n,n,0}$, then $rc_{[h]}:\SG^0_n\ra\SG_{n}^0$ is a continuous homomorphism.
\end{proposition}
\begin{proof}
    If $0<\d\sim 0$, $[f]\in\SG^0_n$ and  $[h]\in\SG^0_{n,n,0}$ and $\rz h(B_\d)\subset B_\e$ for some positive $\e\sim 0$, then $|\rz f\circ h(\xi)|\leq \norm{\rz f}_\e$ for $\xi\in B_\d$, so that if $\Fr=\norm{\rz h}_\d$, then $\xi\in B_\d$ implies that $|\rz f\circ h(\xi)|\leq\norm{\rz f}_\Fr$ and so $\norm{\rz f\circ h}_\d\leq\norm{\rz f}_\Fr$.

    Given this, if $([f_d:d\in D])$ is a net in $\SG^0_n$ converging to the zero germ $[0]$ in say the $\tau_\d$ topology, it is sufficient to show that it follows that $([f_d\circ h]:d\in D)$ converges also (in some $\tau_\Fs$ topology for some $0<\Fs\sim0$, as topology is independent of $\Fs$). The above estimate gives $\norm{\rz f_d\circ h}_\d\leq\norm{\rz f_d}_\Fr$ for all $d\in D$. Now by above (Corollary \ref{cor: SN_r is coin with SN_d if r=f(d)}) $\SN^0_\Fr$ is coinitial in $\SN^0_\d$ and $[f_d]$ converges in the $\tau_\Fr$ topology and so in the $\tau_\d$ topology and so $\{\norm{\rz f_d}_\Fr:d\in D\}$ is coinitial in $\SN^0_\d$. But the estimates above then imply that $\{\norm{\rz f_d\circ h}_\d:d\in D\}$ is coinitial in $\SN^0_\d$, as we wanted.
\end{proof}
\begin{cor}\label{cor: rc_h:G^0_n,p,0->G^0_n,p,0 is C^0}
    Suppose that $[h]\in\SG^0_{n,n,0}$. Then $rc_{[h]}$ is a continuous $\SG^0_n$ module homomorphism of $\SG^0_{n,p,0}$.
\end{cor}
\begin{proof}
    This is clear from the previous proposition.
\end{proof}
    Before we proceed to proving that left composition is a continuous operation, we want a more explicit description of the convergence of a net $([f_d]:d\in D)\subset\SG^0_{n,p,0}$ to $[f]\in\SG^0_{n,p,0}$. This result once more indicates the uniform convergence flavor of $\tau$ convergence.
\begin{lem}\label{lem: ptwise condition for [f_d]->[f_0] in tau}
    Suppose that $([f_d]:d\in D)$ is a net in $\SG^0_{n,p,0}$ and $[f]\in\SG^0_{n,p,0}$. Then $[f_d]\ra[f]$ in $\tau$ if and only if the following holds. Let $\d\in\mu(0)_+$. Given $\Fr$ in $\SN_\d$, then there is $d_0\in D$ such that if $\xi\in\mu_n(0)$ satisfies $\xi\in B_\d$, then $\rz f_d(\xi)\in\mu_p(0)$  satisfies $|\rz f_d(\xi)-\rz f(\xi)|\leq\Fr$ for all $d>d_0$.
\end{lem}
\begin{proof}
     This is a direct consequence of the definition.
\end{proof}
     Left composition by an element of $\SG^0_{n,n,0}$ acting on $\SG^0_{n,p,0}$ is obviously not a homomorphism, but we have a good topological result. Note also that, unlike proving the $\tau$ continuity of right composition, proving the continuity of left composition does not follow immediately from such at $[0]$ upon translation.
     The proof of left continuity will follow after a bit more formal development. We begin with a definition.
\begin{definition}
     We say that $[h]\in\SG_{n,n,0}$ is uniformly continuous at $0$ if it satisfies the following. There is $r\in\bbr_+$ such that for each $m\in\SM(B_r)$, there is $\ov{m}\in\SM(B_r)$ with the following property. For each $x,y\in B_r$ with $|x-y|<\ov{m}(|x|)$, we have $|h(x)-h(y)|<m(|x|)$.
\end{definition}
     So if $[h]$ is uniformly continuous at $0$, then there is $G_h:\SM(B_r)\ra\SM(B_r)$ for $r$ small enough defined by $G_h(m)=\ov{m}$.
\begin{lem}\label{lem: h unif cont at 0 -> h satisfies N_d cont est}
     Suppose that $[h]$ is uniformly continuous at $[0]$ and that $\d\in\mu(0)_+$. The the following holds. For each $\Fr\in\SN_\d$, there is $\ov{\Fr}\in\SN_\d$ such that if $\xi,\z\in B_\d$ with $|\xi-\z|<\ov{\Fr}$, then $|\rz h(\xi)-\rz h(\z)|<\Fr$.
\end{lem}
\begin{proof}
     Transfer the statement for the $r\in\bbr_+$ for which uniform continuity at $0$ holds.  That is, if $\xi,\z\in\rz B_r$ with $|\xi-\z|<\rz G_h(m)(|\xi|)$, then $|\rz h(\xi)-\rz h(\z)|<\rz m(|\xi|)$ and note that if $\Fr\in\SN_\d$, then there is $m\in\SM(B_r)$ with $\rz m(\d)=\Fr$  so that we can use the corresponding $\ov{\Fr}=\rz G_h(m)(\d)\in\SN_\d$ satisfying the transfer of the properties of $G_h$.
     That is, if $|\xi|=\d$, so that any $\Fr\in\SN_\d$ is given by $\Fr=\rz m(\d)$ and therefore to get $|\rz h(\xi)-\rz h(\z)|<\Fr$, we need only to choose $\ov{\Fr}=\rz G_h(m)(\d)$ for the inequality  $|\xi-\z|<\ov{\Fr}$ to imply the needed inequality.
\end{proof}
\begin{lem}\label{lem: elts of G^0 are unif cont at 0}
    Suppose that $[h]\in\SG^0_{p,p,0}$, then any representative $h\in[h]$ is uniformly continuous at $0$.
\end{lem}
\begin{proof}
    Suppose that $h$ is a representative for $[h]$, so that there is $r\in\bbr_+$ such that $h$ is a continuous function on the ball $B_r$ and therefore, $B_r$ being compact, uniformly continuous there. That is, building a parameter into our statement of uniform continuity,  if $m\in\SM_r$, then there is $\ov{m}\in\SM_r$ such that for $x,y\in B_r$ and a given real $t$ with $0<t<r$, if  $|x-y|<\ov{m}(t)$, then $|h(x)-h(y)|<m(t)$. But as $t$ is an arbitrary parameter in $(0,r)$, then for $0<|x|<r$, we have that if $|x-y|<\ov{m}(|x|)$, then $|h(x)-h(y)|<m(|x|)$, as we wanted to prove.
\end{proof}
    We are now in a position to prove the following proposition.

\begin{proposition}\label{prop: lc_h:G^0_n,p,0->G^0_n,p,0 is C^0}
    Suppose that $[h]\in\SG^0_{p,p,0}$, then $lc_{[h]}:\SG^0_{n,p,0}\ra\SG^0_{n,p,0}$ is a continuous map.
\end{proposition}
\begin{proof}
    Suppose that $([f_d]:d\in D)$ is a net in $\SG^0_{n,p,0}$ such that $[f_d]\ra [f]\in\SG^0_{n,p,0}$ in the topology $\tau$. We want to show that $[h]\circ [f_d]=[h\circ f_d]\ra [h\circ f_d]$ in $\tau$.
    By Lemma \ref{lem: ptwise condition for [f_d]->[f_0] in tau}, it suffices to prove given a fixed $\d\in\mu(0)_+$ the following statement. Given $\Fr\in\SN_\d$, there is $d_0\in D$ such that for $d>d_0$ and each $\xi\in B_\d$, we have $|\rz h\circ f_d(\xi)-\rz h\circ f(\xi)|<\Fr$. Fix this $\Fr\in\SN_\d$ and notice that Lemmas \ref{lem: h unif cont at 0 -> h satisfies N_d cont est} and \ref{lem: elts of G^0 are unif cont at 0} together  imply the following. $\bsm{(\diamondsuit)}$: Given the fixed $\Fr$, there is $\ov{\Fr}\in\SN_\d$ such that if $\xi\in B^n_\d$ and $\z\in B^p_\d$ satisfy $|\z-\rz f(\xi)|<\ov{\Fr}$ then we have that $|\rz h(\z)-\rz h(f(\xi))|<\Fr$.
    But then applying the hypothesis in the guise given by Lemma \ref{lem: ptwise condition for [f_d]->[f_0] in tau} once more, we know that  there is $d_0\in D$, such that for $\xi\in B^n_\d$ and $d>d_0$, we have that $|\rz f_d(\xi)-\rz f(\xi)|<\ov{\Fr}$. But this is precisely the condition, with $\z=\rz f_d(\xi)$ for $d>d_0$, required for the previous statement $(\diamondsuit)$ to hold.
\end{proof}
    We will show that if $[h]\in\SG^0_{n,n,0}$ is the germ of a homeomorphism, then composition on the right is a topological ring isomorphism on $\SG^0_n$.

\begin{definition}
    We say that $[f]\in\SG^0_{n,n,0}$ is a homeomorphism germ if there exists $[g]\in\SG^0_{n,n,0}$, its inverse, satisfying $[f]\circ[g]=[id]=[g]\circ [f]$. This set of germs is clearly a group, which we will denote by $\SH^0_n$.
\end{definition}

\begin{proposition}
    Germs of homeomorphisms $[h]\in\SG^0_{n,n,0}$ give topological ring isomorphisms $rc_{[h]}:\SG^0\ra\SG^0$. If we are considering $rc_{[h]}:\SG^0_{n,p}\ra\SG^0_{n,p}$, then $rc_{[h]}$ is a $\SG^0_n$ module isomorphism.
\end{proposition}
\begin{proof}
    This follows immediately from Proposition \ref{prop: rc_h:G^0_n->G^0_n is C^0} and Corollary \ref{cor: rc_h:G^0_n,p,0->G^0_n,p,0 is C^0} as follows: both $rc_{[h]}$ and $rc_{[h^{-1}]}$ are continuous homomorphisms, and as $rc_{[h]}\circ rc_{[h^{-1}]}([g])=[g]\circ[h^{-1}]\circ[h]=[g\circ h^{-1}\circ h]=[g]$ we have $rc_{[h^{-1}]}=(rc_{[h]})^{-1}$ and so $rc_{[h]}\circ(rc_{[h]})^{-1}=rc_{[h]}\circ rc_{[h^{-1}]}=rc_{[id_n]}=Id=rc_{[id]}=rc_{[h^{-1}]}\circ rc_{[h]}=(rc_{[h]})^{-1}\circ rc_{[h]}$.
\end{proof}
    Note that only right composition is a ring homomorphism, nonetheless left composition satisfies $lc:\SG^0_{p,p}\x\SG^0_{n,p}\ra\SG^0_{n,p}$ and a formal proof similar to that above gives the following.
\begin{proposition}
    Germs of homeomorphisms $[h]\in\SG^0_{p,p,0}$ give homeomorphisms $lc_{[h]}:\SG^0_{n,p,0}\ra\SG^0_{n,p,0}:[f]\mapsto[h\circ f]$.
\end{proposition}
    Note that the following gives a numerical bound on regularity growth and decay of a homeomorphism germ independent of the particular element of $\SH^0_n$.
\begin{proposition}
    Suppose that $[h]\in\SH^0_n$ is the germ of a homeomorphism. Then there are $\Fr,\Ft\in\SN_\d$ $\Fr<\Ft$, such that $B_\Fr\subset\rz h(B_\d)\subset B_\Ft$. Equivalently, if $\Fz\lll\d\lll\Fw$, then for every $[h]\in\SH^0_n$, $B_\Fz\subset\rz h(B_\d)\subset B_\Fw$.
\end{proposition}
\begin{proof}
    We will just verify that if $\Fz\lll\d$, then $B_\Fz\subset\rz h(B_\d)$. As $\rz h|_{B_\d}$ is a *homeomorphism sending $0$ to itself, then transfer implies their is a *neighborhood $\SV$ of $0$ such that $\SV\subset\rz h(B_\d)$. We also know that if the *boundary (ie., *frontier) of $B_\d$, is  denoted by $\rz\p B_\d$, then $\rz\p \rz h(B_\d)= \rz h(\rz\p B_\d )$. Now if $m(r)=\inf\{|h(x)|:|x|=r\}$ we have $m(r)\geq r$ for all $r$ and as $h$ is a bijection for sufficiently small $r\in\bbr_+$, we have $m(r)>0$ for small $r$. But then we know that $m(\d)$ is coinitial with $\SN_\d$, eg., that $\rz m(\d)\ggg\Fz$. Of course, this says that if $\xi\in\rz\p B_\d$, then $|\rz h(\xi)|\ggg\Fz$, eg., $\rz\p\rz h(B_\d)$ is disjoint from $B_\Fz$. Now by the transfer of the Jordan separation theorem (for a sufficiently large ball) $\rz\p\rz h(B_\d)$ *separates $B_\Fw$ (for some $\Fw\ggg\d$) and as $\SV$ is disjoint from $\rz\p\rz h(B_\d)$ and contained in $\rz Int \rz h(B_\d)$, one of these *components, and as the frontier of $\rz h(B_\d)$, ie., $\z\in B_\Fw$ of the form $\rz h(\xi)$ for $\xi\in\rz\p B_\d $ have length greater than $\Fz$, then $B_\Fz\subset \rz h(B_\d)$.
\end{proof}


\section*{Shorthand Notation and Abbreviations}

(Included is is a basic set of notations defined and used  in
this work.)
\begin{align*}
m_{*,x}\text{ or }dm_x  &\equiv\ \text{the differential at $x$ of a differentiable map (might}\\
&\qquad\text{be the internal differential with $\rz\!$ ($\rz\!$transfer) suppressed}\\
v_g  &\equiv\ \text{a tangent vector at $g$}\\
FDVS_k    &\equiv\ \text{finite dimensional vector space over $K$}\\
VS(\Fg)  &\equiv\ \text{if $\Fg$ is a Lie algebra, this is the underlying vector space of $\Fg$}\\
!  &\equiv\ \text{unique}\\
LA  &\equiv\ \text{Lie algebra}\\
LG  &\equiv\ \text{Lie group}\\
LA(G,\phi)=(L,[\ ,\ ])  &\equiv (G,\phi)\ \text{is a local Lie group and the associated Lie algebra}\\
&\qquad\text{ is }(L,[\ ,\ ])\\
\rz\!LA  &\equiv\ \text{is this functor star transferred}\\
\rz\!FDLA_K  &\equiv\ \text{an internal $^{\s}$finite dimensional Lie algebra (over the}\\ &\qquad\text{internal field $K$, usually $\rz\!\bbr$)}\\
\eta<\Fg  &\equiv\eta\ \text{is a Lie subalgebra of the Lie algebra $\Fg$}\\
\end{align*}

\bibliographystyle{amsplain}
\bibliography{nsabooks}

\end{document}